\documentclass[a4paper,11pt]{amsart}
\usepackage{amssymb}                                    
\usepackage{mathrsfs}                    
\usepackage{amsmath}                    
\usepackage{amsfonts}                   
\usepackage{amsthm}                     
\usepackage[latin1]{inputenc}           
\usepackage[arrow, matrix, curve]{xy}    
\usepackage[english]{babel}                
\usepackage{bibgerm}				 
\usepackage{graphics}
\usepackage[normalem]{ulem}
\usepackage{xcolor}
\usepackage{bbm}
\usepackage{soul}

\usepackage{hyperref}

\makeatletter

\def\chaptermark#1{}
\newcounter{appendix}

\def\chapter{%
  \if@openright\cleardoublepage\else\clearpage\fi
  \thispagestyle{plain}\global\@topnum\z@
  \@afterindenttrue \secdef\@chapter\@schapter}
\def\@chapter[#1]#2{\ifx\chaptername\appendixname\refstepcounter{appendix}
\else \refstepcounter{chapter}\fi
  \ifnum\c@secnumdepth<\z@ \let\@secnumber\@empty
  \else \let\@secnumber\thechapter \fi
  \typeout{\chaptername\space\@secnumber}%
  \def\@toclevel{0}%
  \ifx\chaptername\appendixname \@tocwriteb\tocappendix{chapter}{#2}%
  \else \@tocwriteb\tocchapter{chapter}{#2}\fi
  \chaptermark{#1}%
  \addtocontents{lof}{\protect\addvspace{10\p@}}%
  \addtocontents{lot}{\protect\addvspace{10\p@}}%
  \@makechapterhead{#2}\@afterheading}
\def\@schapter#1{\typeout{#1}%
  \let\@secnumber\@empty
  \def\@toclevel{0}%
  \ifx\chaptername\appendixname \@tocwriteb\tocappendix{chapter}{#1}%
  \else \@tocwriteb\tocchapter{chapter}{#1}\fi
  \chaptermark{#1}%
  \addtocontents{lof}{\protect\addvspace{10\p@}}%
  \addtocontents{lot}{\protect\addvspace{10\p@}}%
  \@makeschapterhead{#1}\@afterheading}
\newcommand\chaptername{Chapter}

\def\@makechapterhead#1{\global\topskip 7.5pc\relax
  \begingroup
  \fontsize{\@xivpt}{18}\bfseries\centering
    \ifnum\c@secnumdepth>\m@ne
      \leavevmode \hskip-\leftskip
      \rlap{\vbox to\z@{\vss
          \centerline{\normalsize\mdseries
              \uppercase\@xp{\chaptername}\enspace\thechapter}
          \vskip 3pc}}\hskip\leftskip\fi
     #1\par \endgroup
  \skip@34\p@ \advance\skip@-\normalbaselineskip
  \vskip\skip@ }
\def\@makeschapterhead#1{\global\topskip 7.5pc\relax
  \begingroup
  \fontsize{\@xivpt}{18}\bfseries\centering
  #1\par \endgroup
  \skip@34\p@ \advance\skip@-\normalbaselineskip
  \vskip\skip@ }
\def\appendix{\par
  \c@appendix\z@ \c@section\z@
  \let\chaptername\appendixname
  \def\thechapter{\@Alph\c@appendix}
  }

\newcounter{chapter}

\newif\if@openright

\makeatother

\numberwithin{section} {chapter}
\theoremstyle{plain}
\newtheorem{thm}{Theorem}[section]
\newtheorem{corollary}[thm]{Corollary}
\newtheorem{proposition}[thm]{Proposition}
\newtheorem{lemma}[thm]{Lemma}

\newtheorem{definition}[thm]{Definition}

\theoremstyle{definition}
\newtheorem{remark}[thm]{Remark}
\newtheorem{construction}[thm]{Construction}
\newtheorem{example}[thm]{Example}

\newcommand{\PO}{\mathrm{PoSet}}
\newcommand{\calS}{\mathcal{S}}
\newcommand{\calE}{\mathcal{E}}

\newcommand{\R}{\mathbb{R}}

\newcommand{\N}{\mathbb{N}}

\newcommand{\Z}{\mathbb{Z}}

\newcommand{\E}{\mathbb{E}}
\newcommand{\Cat}{\mathrm{Cat}}
\newcommand{\Hom}{\mathrm{Hom}}

\newcommand{\Fun}{\mathrm{Fun}}
\newcommand{\Top}{\mathrm{Top}}
\newcommand{\bS}{\mathbb{S}}
\newcommand{\calC}{\mathcal{C}}
\newcommand{\calD}{\mathcal{D}}
\newcommand{\Sp}{\mathrm{S}\mathrm{p}}

\newcommand{\id}{\mathrm{id}}

\newcommand{\colim}{\mathrm{colim}}
\DeclareMathOperator*{\hocolim}{hocolim}
\DeclareMathOperator{\invlim}{\underleftarrow{\mathrm{lim}}}
\DeclareMathOperator{\dirlim}{\underrightarrow{\mathrm{lim}}}

\newcommand{\Alg}{\mathrm{Alg}}
\newcommand{\CoAlg}{\mathrm{CoAlg}}
\newcommand{\Fix}{\mathrm{Fix}}
\newcommand{\xto}{\xrightarrow}
\newcommand{\Cyc}{\mathrm{Cyc}\Sp^{\mathrm{gen}}}
\newcommand{\CycO}{\mathrm{Cyc}\Sp^O}
\newcommand{\Cycn}{\mathrm{Cyc}\Sp}
\newcommand{\Map}{\mathrm{Map}}

\newcommand{\Fp}{\mathbb{F}_p}
\newcommand{\F}{{\mathcal{F}}}

\newcommand{\map}{\mathrm{map}}

\newcommand{\THH}{\mathrm{THH}}
\newcommand{\TC}{\mathrm{TC}}
\newcommand{\THC}{\mathrm{TC}}
\newcommand{\TR}{\mathrm{TR}}

\newcommand{\Aut}{\mathrm{Aut}}
\newcommand{\can}{\mathrm{can}}
\newcommand{\KU}{\mathrm{KU}}

\newcommand{\Q}{\mathbb{Q}}
\newcommand{\C}{\mathbb{C}}
\newcommand{\bF}{\mathbb{F}}
\newcommand{\HH}{\mathrm{HH}}
\newcommand{\HC}{\mathrm{HC}}
\newcommand{\op}{\mathrm{op}}

\newcommand{\Tor}{\mathrm{Tor}}

\newcommand{\lax}{\mathrm{lax}}
\newcommand{\calO}{\mathcal{O}}

\newcommand{\act}{\mathrm{act}}
\newcommand{\Free}{\mathrm{Free}}
\newcommand{\Ex}{\mathrm{Ex}}
\newcommand{\Lex}{\mathrm{Lex}}
\newcommand{\ev}{\mathrm{ev}}
\newcommand{\pt}{\mathrm{pt}}
\newcommand{\B}{\mathrm{B}}
\newcommand{\bP}{\mathbb{P}}
\newcommand{\fib}{\mathrm{fib}}
\newcommand{\tr}{\mathrm{tr}}
\newcommand{\T}{\mathbb{T}}
\newcommand{\Tp}{{C_{p^\infty}}}

\newcommand{\gen}{\mathrm{gen}}

\newcommand{\Fin}{\mathrm{Fin}}

\newcommand{\cofib}{\mathrm{cofib}}

\newcommand{\Ass}{\mathrm{Ass}}

\newcommand{\Nm}{\mathrm{Nm}}
\newcommand{\Cut}{\mathrm{Cut}}
\newcommand{\Mod}{\mathrm{Mod}}
\newcommand{\Ind}{\mathrm{Ind}}
\newcommand{\ind}{\mathrm{ind}}
\newcommand{\calM}{\mathcal{M}}
\newcommand{\calN}{\mathcal{N}}

\newcommand{\sd}{\mathrm{sd}}

\newcommand{\GL}{\mathrm{GL}}
\newcommand{\cyc}{\mathrm{cyc}}
\newcommand{\nat}{\rho}
\newcommand{\Sk}{\mathrm{Sk}}
\newenvironment{altenumerate}
   {\begin{list}
      {\textup{(\theenumi)} }
      {\usecounter{enumi}
       \setlength{\labelwidth}{0pt}
       \setlength{\labelsep}{2pt}
       \setlength{\leftmargin}{0pt}
       \setlength{\itemsep}{\the\smallskipamount}
       \renewcommand{\theenumi}{\roman{enumi}}
      }}
   {\end{list}}

\begin{document}

\title{On Topological Cyclic Homology}
\author{Thomas Nikolaus and Peter Scholze}
\begin{abstract}
Topological cyclic homology is a refinement of Connes-Tsygan's cyclic homology which was introduced by B\"okstedt--Hsiang--Madsen in 1993 as an approximation to algebraic $K$-theory. There is a trace map from algebraic $K$-theory to topological cyclic homology, and a theorem of Dundas--Goodwillie--McCarthy asserts that this induces an equivalence of relative theories for nilpotent immersions, which gives a way for computing $K$-theory in various situations. The construction of topological cyclic homology is based on genuine equivariant homotopy theory, the use of explicit point-set models, and the elaborate notion of a cyclotomic spectrum.

The goal of this paper is to revisit this theory using only homotopy-invariant notions. In particular, we give a new construction of topological cyclic homology. This is based on a new definition of the $\infty$-category of cyclotomic spectra: We define a cyclotomic spectrum to be a spectrum $X$ with $S^1$-action (in the most naive sense) together with $S^1$-equivariant maps $\varphi_p: X\to X^{tC_p}$ for all primes $p$. Here $X^{tC_p}=\cofib(\Nm: X_{hC_p}\to X^{hC_p})$ is the Tate construction. On bounded below spectra, we prove that this agrees with previous definitions. As a consequence, we obtain a new and simple formula for topological cyclic homology. 

In order to construct the maps $\varphi_p: X\to X^{tC_p}$ in the example of topological Hochschild homology we introduce and study Tate diagonals for spectra and Frobenius homomorphisms of commutative ring spectra. In particular we prove a version of the Segal conjecture for the Tate diagonals and relate these Frobenius homomorphisms to power operations.
\end{abstract}
\date{\today}
\maketitle

\tableofcontents

\section*{Introduction}
\renewcommand{\thesection}{1}

This paper grew out of an attempt to understand how much information is stored in topological cyclic homology, or more precisely in the cyclotomic spectrum calculating topological cyclic homology.

Let $A$ be an associative and unital ring. The $K$-theory spectrum $K(A)$ of $A$ can be defined as the group completion of the $\E_\infty$-monoid of finite projective $A$-modules. This is an important invariant of $A$ that is very hard to compute in practice. For this reason, various approximations to $K(A)$ have been studied, notably Connes-Tsygan's cyclic homology $\HC(A)$, and its variant, negative cyclic homology $\HC^-(A)$. These are obtained from the Hochschild homology $\HH(A)$, which (if $A$ is flat over $\mathbb Z$) is obtained as the geometric realization of the simplicial object
\[\xymatrix{
 \cdots \ar[r]<4.5pt>\ar[r]<1.5pt>\ar[r]<-4.5pt>\ar[r]<-1.5pt> & A \otimes_{\mathbb Z} A \otimes_{\mathbb Z} A  \ar[r]<3pt>\ar[r]\ar[r]<-3pt>  & A\otimes_{\mathbb Z} A \ar[r]<1.5pt>\ar[r]<-1.5pt> & A\ .
}\]
In fact, this is a cyclic object in the sense of Connes, which essentially means that there is a $\mathbb Z/(n+1)$-action on the $n$-th term, which commutes in a suitable way with the structure maps; here, this action is given by the obvious permutation of tensor factors. On the geometric realization of a cyclic object, there is a canonical continuous action of the circle group $\mathbb T = S^1$, so $\mathbb T$ acts on the topological space $\HH(A)$. One can also regard $\HH(A)$, via the Dold--Kan correspondence, as an object of the $\infty$-derived category $\mathcal D(\mathbb Z)$. One can then define cyclic homology as the homotopy orbits
\[
\HC(A) = \HH(A)_{h\mathbb T}
\]
taken in the $\infty$-derived category $\mathcal D(\mathbb Z)$, and similarly negative cyclic homology as the homotopy fixed points
\[
\HC^-(A) = \HH(A)^{h\mathbb T}
\]
of the circle action.\footnote{For a comparison with classical chain complex level definitions, cf.~e.g.~\cite{Hoyoiscyclic}.} A calculation of Connes shows that if $A$ is a smooth (commutative) $\mathbb Q$-algebra, then $\HC^-(A)$ is essentially given by de Rham cohomology of $A$. In this way, $\HC^-(A)$ can be regarded as a generalization of de Rham cohomology to noncommutative rings.

The following important theorem is due to Goodwillie. To state it we consider the object $\HC^-(A) \in \calD(\Z)$ as a generalized Eilenberg-MacLane spectrum. 

\begin{thm}[Goodwillie, \cite{Goodwillie}]\label{thm:goodwillie} There is a trace map
\[
\tr: K(A)\to \HC^-(A)\ ,
\]
which is functorial in $A$. If $A\to \overline{A}$ is a surjection of associative and unital 
algebras with nilpotent kernel, then the diagram
\[\xymatrix{
K(A)_\Q\ar[r]\ar[d] & \HC^-(A \otimes \Q)\ar[d]\\
K(\overline{A})_\Q\ar[r] & \HC^-(\overline{A}\otimes\Q)
}\]
is homotopy cartesian. Here $-_\Q$ denotes rationalization of a spectrum.
\end{thm}

The trace map $K(A)\to \HC^-(A)$ is often referred to as the Jones--Goodwillie trace map; the composite $K(A)\to \HH(A)$ with the projection $\HC^-(A) = \HH(A)^{h\T}\to \HH(A)$ was known before, and is called the Dennis trace map.

An important problem was to find a generalization of this theorem to the non rational case. This was eventually solved through the development of topological cyclic homology.

As it is more natural, we start from now on directly in the more general setting of an associative and unital ring spectrum $A$, i.e.~an $\E_1$-algebra in the $\infty$-category of spectra $\Sp$ in the language of \cite{HA}. An example is given by Eilenberg-MacLane spectra associated with usual associative and unital rings, but for the next step it is important to switch into the category of spectra.

Namely, one looks at the topological Hochschild homology $\THH(A)$ of $A$, which is given by the geometric realization of the simplicial object
\[\xymatrix{
 \cdots \ar[r]<4.5pt>\ar[r]<1.5pt>\ar[r]<-4.5pt>\ar[r]<-1.5pt> & A \otimes_{\mathbb S} A \otimes_{\mathbb S} A  \ar[r]<3pt>\ar[r]\ar[r]<-3pt>  & A\otimes_{\mathbb S} A \ar[r]<1.5pt>\ar[r]<-1.5pt> & A\ .
}\]
Here, we write $\otimes_{\mathbb S}$ for the symmetric monoidal tensor product of spectra, which is sometimes also called the smash product; the base $\mathbb S\in \Sp$ denotes the sphere spectrum. Again, this simplicial object is actually a cyclic object, and so there is a canonical $\mathbb T$-action on $\THH(A)$.

We remark that we have phrased the previous discussion in the $\infty$-category of spectra, so we get $\THH(A)$ as a $\mathbb T$-equivariant object in the $\infty$-category $\Sp$, i.e.~as an object of the functor $\infty$-category $\Fun(B\mathbb T,\Sp)$, where $B\mathbb T\simeq \mathbb C P^\infty$ is the topological classifying space of $\T$.

One could then define ``topological negative cyclic homology'' as
\[
\THC^-(A) = \THH(A)^{h\mathbb T}\ ;
\]
we warn the reader that this is not a standard definition, but a close variant has recently been investigated by Hesselholt, \cite{2016arXiv160201980H}. There is a still a trace map $\tr: K(A)\to \THC^-(A)$, but the analogue of Theorem~\ref{thm:goodwillie} does not hold true. However, $\THC^-(A)$ is interesting: If $A$ is a smooth $\mathbb F_p$-algebra, then $\THC^-(A)$ is essentially given by de Rham--Witt cohomology of $A$ over $\mathbb F_p$. This fact is related to computations of Hesselholt, \cite{HesselholtdRW}, but in a more precise form it is \cite[Theorem 1.10]{2018arXiv180203261B}.

It was an insight of B\"okstedt--Hsiang--Madsen, \cite{BHM}, how to correct this. Their definition of topological cyclic homology $\TC(A)$,\footnote{We warn the novice in the theory that topological cyclic homology does not relate to topological Hochschild homology in the same way that cyclic homology relates to Hochschild homology; rather, topological cyclic homology is an improved version of ``topological negative cyclic homology'' taking into account extra structure.} however, requires us to lift the previous discussion to a $1$-categorical level first. Namely, we use the symmetric monoidal $1$-category of orthogonal spectra $\Sp^O$, cf.~Definition~\ref{def:orthspectrum} below. We denote the symmetric monoidal tensor product in this category by $\wedge$ and refer to it as the smash product, as in this $1$-categorical model it is closely related to the smash product of pointed spaces. It is known that any $\E_1$-algebra in $\Sp$ can be lifted to an associative and unital ring in $\Sp^O$, and we fix such a lift $\widetilde{A}$. In this case, we can form the cyclic object
\[\xymatrix{
 \cdots \ar[r]<4.5pt>\ar[r]<1.5pt>\ar[r]<-4.5pt>\ar[r]<-1.5pt> & \widetilde{A} \wedge \widetilde{A} \wedge \widetilde{A}  \ar[r]<3pt>\ar[r]\ar[r]<-3pt>  & \widetilde{A} \wedge \widetilde{A} \ar[r]<1.5pt>\ar[r]<-1.5pt> & \widetilde{A}\ ,
}\]
whose geometric realization defines a $\mathbb T$-equivariant object $\THH(\widetilde{A})$ in the $1$-category $\Sp^O$.\footnote{We need to assume here that $A$ is cofibrant, otherwise the smash products need to be derived.} The crucial observation now is that this contains a wealth of information.

In fact, B\"okstedt--Hsiang--Madsen use a slightly different cyclic object, called the B\"okstedt construction; i.e., they are using a different $1$-categorical model for the smash product, cf.~Definition~\ref{def:boekstedt} below.\footnote{With this modification, one does not need to assume that $A$ is cofibrant, only a very mild condition on basepoints (see Lemma \ref{lemproper}).} The relation between these constructions is the subject of current investigations, and we refer to \cite{sixauthors} and \cite{TCcomp} for a discussion of this point. In the following, we denote by $\THH(\widetilde{A})$ the realization of the cyclic object defined through the B\"okstedt construction.

Now the surprising statement is that for any $n\geq 1$, the \emph{point-set fixed points}
\[
\THH(\widetilde{A})^{C_n}\in \Sp^O
\]
under the cyclic subgroup $C_n\subseteq \mathbb T$ of order $n$, are well-defined in the sense that a homotopy equivalence $\widetilde{A}\to \widetilde{A}^\prime$ induces a homotopy equivalence $\THH(\widetilde{A})^{C_n}\to \THH(\widetilde{A}^\prime)^{C_n}$.\footnote{For this statement, it is necessary to use the B\"okstedt construction; a priori, there is no reason that the $C_n$-fixed points have any homotopy invariant meaning, and different constructions can lead to different $C_n$-fixed points. Additionally one has to impose further point-set conditions on $A$, cf.~\cite[Definition 4.7 and Theorem 8.1]{MR3558224}, or derive the fixed points functor.} If $n=p$ is prime, this follows from the existence of a natural cofiber sequence
\[
\THH(\widetilde{A})_{hC_p}\to \THH(\widetilde{A})^{C_p}\to \Phi^{C_p} \THH(\widetilde{A})\ ,
\]
where $\Phi^{C_p}$ denotes the so-called geometric fixed points, and a natural $\mathbb T$-equivariant equivalence
\[
\Phi_p: \Phi^{C_p} \THH(\widetilde{A})\buildrel\simeq\over\longrightarrow \THH(\widetilde{A})\ ,
\]
which is a special property of $\THH(\widetilde{A})$. This leads to the statement that $\THH(\widetilde{A})$ is an orthogonal cyclotomic spectrum, which is a $\mathbb T$-equivariant object $X$ of $\Sp^O$ together with commuting $\mathbb T$-equivariant equivalences
\[
\Phi_p: \Phi^{C_p} X\buildrel\simeq\over\longrightarrow X\ ,
\]
cf.~Definition~\ref{def:orthogonalcyclo}. Here, on the left-hand side, there is a canonical action of $\mathbb T/C_p$, which we identify with $\mathbb T$ via the $p$-th power map. The above construction gives a functor
\[
\widetilde{A}\mapsto (\THH(\widetilde{A}),(\Phi_p)_{p\in \bP}): \Alg(\Sp^O)\to \CycO
\]
from associative and unital rings in $\Sp^O$ to the category of orthogonal cyclotomic spectra $\CycO$. Here and in the following, $\bP$ denotes the set of primes. Defining a suitable notion of weak equivalences of orthogonal cyclotomic spectra, this functor factors over the $\infty$-category of $\E_1$-algebras in $\Sp$, and we denote this functor by
\[
A\mapsto (\THH(A),(\Phi_p)_p): \Alg_{\E_1}(\Sp)\to \Cyc\ ,
\]
where $\Cyc$ denotes the $\infty$-category obtained from $\CycO$ by inverting weak equivalences; we refer to objects of $\Cyc$ as genuine cyclotomic spectra.

One possible definition for the topological cyclic homology $\TC(A)$ is now as the mapping space
\[
\TC(A) = \Map_{\Cyc}((\mathbb S,(\Phi_p)_{p\in \bP}),(\THH(A),(\Phi_p)_{p\in \bP}))
\]
in the $\infty$-category $\Cyc$, where $(\mathbb S,(\Phi_p)_{p\in \bP})$ denotes the cyclotomic sphere spectrum (obtained for example as $\THH(\mathbb S)$); in fact, this mapping space refines canonically to a spectrum.\footnote{This mapping spectrum $\TC(A)$ is equivalent to Goodwillie's integral TC, see Remark~\ref{integralTC}.} A more explicit definition directly in terms of $\THH(\widetilde{A})^{C_n}$ for varying $n$ was given by B\"okstedt--Hsiang--Madsen; we refer to Section~\ref{sec:tccomp} for the definition. We remark that there is a natural map
\[
\TC(A)\to \THC^-(A) = \THH(A)^{h\mathbb T}
\]
which comes from a forgetful functor from $\Cyc$ to the $\infty$-category of $\mathbb T$-equi\-va\-riant objects in $\Sp$.

Finally, we can state the analogue of Theorem~\ref{thm:goodwillie}.

\begin{thm}[Dundas--Goodwillie--McCarthy, \cite{Local}]\label{thm:DGMcC} There is a trace map
\[
\tr: K(A)\to \TC(A)\ ,
\]
called the cyclotomic trace map, functorial in $A$. If $A\to \overline{A}$ is a map of connective associative and unital algebras in $\Sp$ such that $\pi_0 A\to \pi_0 \overline{A}$ has nilpotent kernel, then
\[\xymatrix{
K(A)\ar[r]\ar[d] & \TC(A)\ar[d]\\
K(\overline{A})\ar[r] & \TC(\overline{A})
}\]
is homotopy cartesian.
\end{thm}

Despite the elegance of this definition, it is hard to unravel exactly how much information is stored in the cyclotomic spectrum $(\THH(A),(\Phi_p)_p)$, and what its homotopy invariant meaning is. Note that for applications, one usually assumes that $A$ is connective, as only then Theorem~\ref{thm:DGMcC} applies. In the connective case, it turns out that one can completely understand all the structure.

\begin{definition} A cyclotomic spectrum is a $\mathbb T$-equivariant object $X$ in the $\infty$-category $\Sp$ together with $\mathbb T/C_p\simeq \mathbb T$-equivariant maps $\varphi_p: X\to X^{tC_p}$ for all primes $p$. Let $\Cycn$ denote the $\infty$-category of cyclotomic spectra.
\end{definition}

Here,
\[
X^{tC_p} = \cofib(X_{hC_p}\xto{\Nm} X^{hC_p})
\]
denotes the Tate construction. Note that if $(X,(\Phi_p)_{p\in \bP})$ is a genuine cyclotomic spectrum, then $X$ endowed with the maps
\[
\varphi_p: X\simeq \Phi^{C_p} X\to X^{tC_p}
\]
is a cyclotomic spectrum, using the natural maps $\Phi^{C_p} X\to X^{tC_p}$ fitting into the diagram
\[\xymatrix{
X_{hC_p}\ar@{=}[d]\ar[r] & X^{C_p}\ar[d]\ar[r] & \Phi^{C_p} X\ar[d]\\
X_{hC_p}\ar[r] & X^{hC_p}\ar[r] & X^{tC_p}\ .
}\]
Note that contrary to the case of orthogonal cyclotomic spectra, we do not ask for any compatibility between the maps $\varphi_p$ for different primes $p$.

Our main theorem is the following.

\begin{thm}\label{thm:intromain} The forgetful functor $\Cyc\to \Cycn$ is an equivalence of $\infty$-categories when restricted to the full subcategories of bounded below spectra.
\end{thm}

In particular, we get the following alternative formula for $\TC(A)$.

\begin{corollary}\label{cor:formulatc} For any connective ring spectrum $A\in \Alg_{\E_1}(\Sp)$, there is a natural fiber sequence
\[
\TC(A)\to \THH(A)^{h\mathbb T}\xto{ (\varphi_p^{h\T}-\can)_{p\in \bP}} \prod_{p\in \bP} (\THH(A)^{tC_p})^{h\T}\ ,
\]
where
\[
\can: \THH(A)^{h\T}\simeq (\THH(A)^{hC_p})^{h(\T/C_p)}=(\THH(A)^{hC_p})^{h\T}\to (\THH(A)^{tC_p})^{h\T}
\]
denotes the canonical projection, using the isomorphism $\T/C_p\cong \T$ in the middle identification.
\end{corollary}

\begin{remark}\label{rem:equalizer} One could rewrite the above fiber sequence also as an equalizer
\[
\TC(A) = \mathrm{Eq}\left(\xymatrix{\THH(A)^{h\mathbb T}\ar[rr]<1.5pt>^-{(\varphi_p^{h\T})_{p\in \bP}}\ar[rr]<-1.5pt>_-{\can} &&\prod_{p\in \bP} (\THH(A)^{tC_p})^{h\T}}\right)\ ,
\]
which may be more appropriate for comparison with point-set level descriptions of spectra, where it is usually impossible to form differences of maps of spectra. As this is not a concern in the $\infty$-category of spectra, we will stick with the more compact notation of a fiber sequence.
\end{remark}

In fact, cf.~Lemma~\ref{lemtate2}, the term $(\THH(A)^{tC_p})^{h\T}$ can be identified with the $p$-completion of $\THH(A)^{t\T}$, and thus $\prod_p (\THH(A)^{tC_p})^{h\T}$ can be identified with the profinite completion of $\THH(A)^{t\T}$. Thus there is a functorial fiber sequence
\[
\TC(A) \to \THH(A)^{h\T} \to \left(\THH(A)^{t\T}\right)^\wedge \ .
\]

Intuitively, Theorem~\ref{thm:intromain} says that the only extra structure present on $\THH(A)$ besides its $\mathbb T$-action is a Frobenius for every prime $p$. If $p$ is invertible on $A$, this is actually no datum, as then $\THH(A)^{tC_p} = 0$. For example, if $A$ is a smooth $\mathbb F_p$-algebra, we see that the only extra structure on $\THH(A)$ besides the $\T$-action is a Frobenius $\varphi_p$. Recall that $\THH(A)$ with its $\T$-action gave rise to $\THC^-(A)$, which was essentially the de Rham--Witt cohomology of $A$. Then the extra structure amounts to the Frobenius on de Rham--Witt cohomology. Under these interpretations, the formula for $\TC(A)$ above closely resembles the definition of syntomic cohomology, cf.~\cite{FontaineMessing}. We note that the analogy between $\TC$ and syntomic cohomology has been known, and pursued for example by Kaledin, \cite{Kaledin}, \cite{Kaledin2}, see also at the end of the introduction of \cite{MR1072979} for an early suggestion of such a relation. As explained to us by Kaledin, our main theorem is closely related to his results relating cyclotomic complexes and filtered Dieudonn\'e modules.

By Theorem~\ref{thm:intromain}, the information stored in the genuine cyclotomic spectrum $(\THH(A),(\Phi_p)_p)$ can be characterized explicitly. In order for this to be useful, however, we need to give a direct construction of this information. In other words, for $A\in \Alg_{\E_1}(\Sp)$, we have to define directly a $\T/C_p\cong \T$-equivariant Frobenius map
\[
\varphi_p: \THH(A)\to \THH(A)^{tC_p}\ .
\]
We will give two discussions of this, first for associative algebras, and then indicate a much more direct construction for $\E_\infty$-algebras.

Let us discuss the associative case for simplicity for $p=2$. Note that by definition, the source $\THH(A)$ is the realization of the cyclic spectrum
\[\xymatrix{
 \cdots \ar[r]<3pt>\ar[r]\ar[r]<-3pt>  & A\otimes_{\mathbb S} A\ar@(ul,ur)^{C_2} \ar[r]<1.5pt>\ar[r]<-1.5pt> & A\ .
}\]
By simplicial subdivision, the target $\THH(A)^{tC_2}$ is given by $-^{tC_2}$ applied to the geometric realization of the $\Lambda_2^\op$-spectrum starting with
\[\xymatrix{
 \cdots \ar[r]<7.5pt>\ar[r]<-7.5pt>\ar[r]<4.5pt>\ar[r]<1.5pt>\ar[r]<-4.5pt>\ar[r]<-1.5pt>& A\otimes_{\mathbb S} A\otimes_{\mathbb S} A \otimes_{\mathbb S} A\ar@(ul,ur)^{C_4}  \ar[r]<4.5pt>\ar[r]<1.5pt>\ar[r]<-4.5pt>\ar[r]<-1.5pt> & A\otimes_{\mathbb S} A\ar@(ul,ur)^{C_2}\ ;
}
\]
here, $\Lambda_2^\op$ denotes a certain category lying over the cyclic category $\Lambda^\op$. We will in fact construct a map from the $\Lambda^\op$-spectrum
\[\xymatrix{
 \cdots \ar[r]<3pt>\ar[r]\ar[r]<-3pt>  & A\otimes_{\mathbb S} A\ar@(ul,ur)^{C_2} \ar[r]<1.5pt>\ar[r]<-1.5pt> & A
}\]
towards the $\Lambda^\op$-spectrum starting with
\[\xymatrix{
 \cdots \ar[r]<3pt>\ar[r]\ar[r]<-3pt>& (A\otimes_{\mathbb S} A\otimes_{\mathbb S} A \otimes_{\mathbb S} A)^{tC_2}\ar@(ul,ur)^{C_2} \ar[r]<1.5pt>\ar[r]<-1.5pt> & (A\otimes_{\mathbb S} A)^{tC_2}\ .
}\]
Here, the number of arrows has reduced (corresponding to factorization over the projection $\Lambda_2^\op\to \Lambda^\op$), as some arrows become canonically equal after applying $-^{tC_2}$. Constructing such a map is enough, as the geometric realization of this $\Lambda^\op$-spectrum maps canonically towards $\THH(A)^{tC_2}$ (but the map is not an equivalence), as the geometric realization is a colimit.

We see that for the construction of the map $\varphi_2$, we need to construct a map
\[
A\to (A\otimes_{\mathbb S} A)^{tC_2}\ ;
\]
this will induce the other maps by the use of suitable symmetric monoidal structures. This is answered by our second theorem. In the statement, we use that $X\mapsto X^{tC_p}$ is a lax symmetric monoidal functor, which we prove in Theorem~\ref{thm:tatelaxsymm}.

\begin{thm}\label{thm:introtatediag} For a prime $p$, consider the lax symmetric monoidal functor
\[
T_p: \Sp\to \Sp: X\mapsto (\underbrace{X\otimes_{\mathbb S}\ldots\otimes_{\mathbb S} X}_{p\ \mathrm{factors}})^{tC_p}\ ,
\]
where $C_p$ acts by cyclic permutation of the $p$ factors. Then $T_p$ is exact, and there is a unique lax symmetric monoidal transformation
\[
\Delta_p:\id\to T_p: X\to (X\otimes_{\mathbb S}\ldots\otimes_{\mathbb S} X)^{tC_p}\ .
\]

Moreover, if $X$ is bounded below, i.e.~there is some $n\in \mathbb Z$ such that $\pi_i X = 0$ for $i<n$, then the map
\[
X\to T_p(X) = (X\otimes_{\mathbb S}\ldots\otimes_{\mathbb S} X)^{tC_p}
\]
identifies $T_p(X)$ with the $p$-completion of $X$.
\end{thm}

This is related to results of Lun\o e-Nielsen--Rognes, \cite{LNR}. They call the functor $T_p$ the topological Singer construction, and they prove the second part of the theorem when $X$ is in addition of finite type. Note that the second part of the theorem is a generalization of the Segal conjecture for the group $C_p$ (proved by Lin, \cite{Lin}, and Gunawardena, \cite{gunawardena}) which states that $\mathbb S^{tC_p}$
is equivalent to the $p$-completion of the sphere spectrum $\mathbb S$.

Now we can explain the direct construction of $\THH(A)$ as a cyclotomic spectrum in case $A$ is an $\E_\infty$-algebra. Indeed, by a theorem of McClure-Schw\"anzl-Vogt, \cite{MR1473888}, $\THH(A)$ is the initial $\T$-equivariant $\E_\infty$-algebra equipped with a non-equivariant map of $\E_\infty$-algebras $A\to \THH(A)$. This induces a $C_p$-equivariant map $A\otimes_{\mathbb S}\ldots\otimes_{\mathbb S} A\to \THH(A)$ by tensoring, and thus a map of $\E_\infty$-algebras
\[
(A\otimes_{\mathbb S}\ldots\otimes_{\mathbb S} A)^{tC_p}\to \THH(A)^{tC_p}\ .
\]
On the other hand, by the Tate diagonal, we have a map of $\E_\infty$-algebras
\[
A\to (A\otimes_{\mathbb S}\ldots\otimes_{\mathbb S} A)^{tC_p}\ .
\]
Their composition gives a map of $\E_\infty$-algebras
\[
A\to (A\otimes_{\mathbb S}\ldots\otimes_{\mathbb S} A)^{tC_p}\to \THH(A)^{tC_p}\ ,
\]
where the target has a residual $\T/C_p\cong \T$-action. By the universal property of $\THH(A)$, this induces a unique $\T$-equivariant map $\varphi_p: \THH(A)\to \THH(A)^{tC_p}$ of $\E_\infty$-algebras, as desired.\\

Finally, let us give a summary of the contents of the different chapters; we refer to the individual introductions to the chapters for more details. In Chapter 1, we prove basic results about the Tate construction. Next, in Chapter 2, we prove Theorem~\ref{thm:intromain}. In Chapter 3, we discuss the cyclotomic spectrum $\THH(A)$ intrinsically to our new approach and relate it to previous constructions. In the final Chapter 4, we discuss various examples from the point of view of this paper.

In place of giving a Leitfaden for the whole paper, let us give an individual Leitfaden for several different topics covered in this paper. In all cases, the Tate construction as summarized in I.1, I.3, is relevant.\\

\noindent \begin{tabular}{ll}
New construction of $\THH$: & II.1 $\to $ III.1 $\to $ III.2 $\to $ III.3\\
Classical construction of $\THH$: & II.2 $\to $ II.3 $\to $ III.4 $\to $ III.5\\
Comparison results: & I.2 $\to $ II.4 $\to $ II.5 $\to $ II.6 $\to $ III.6\\
$\THH$ for commutative rings: & III.1 $\to $ Proposition III.3.1 $\to $ IV.2 $\to $ IV.4 $\leftarrow $ I.2\\
\end{tabular}\\

Moreover, there are three appendices. In Appendix~\ref{app:symmmon}, we review some facts about symmetric monoidal $\infty$-cate\-gories that we use. Most importantly, we prove that for 
a symmetric monoidal model category $\calM$ with associated $\infty$-category $\mathcal C=N(\calM)[W^{-1}]$, the $\infty$-category $\mathcal C$ is again symmetric monoidal, and the localization functor $N(\calM)\to \mathcal C$ is lax symmetric monoidal (and symmetric monoidal when restricted to cofibrant objects), cf.~Theorem~\ref{DKlocalizationsymmon}. We deduce this from a general functoriality assertion for the Dwyer--Kan localization, which seems to be new. In Appendix~\ref{app:cyclic}, we recall some basic results about Connes' cyclic category $\Lambda$, and related combinatorial categories, as well as the geometric realization of cyclic spaces. Finally, in Appendix~\ref{app:colim} we summarize some facts about homotopy colimits in various categories of spaces and spectra.
\\

{\bf Acknowledgements:}
We thank Tobias Barthel, David Gepner, John Greenlees, Lars Hesselholt, Mike Hopkins, Dmitry Kaledin, Jacob Lurie, Wolfgang L\"uck, Akhil Mathew, Matthew Morrow, Irakli Patchkoria, Charles Rezk, John Rognes, Steffen Sagave, Peter Teichner and Marco Varisco for fruitful discussions and help concerning different aspects of this project. Moreover, we thank Tobias Barthel, Vladimir Hinich, Akhil Mathew, Irakli Patchkoria, Steffen Sagave and Peter Teichner for comments on a draft and an anonymous referee for many helpful remarks. This work was started while the second author was a Clay Research Fellow.

\renewcommand{\thesection}{\Roman{chapter}.\arabic{section}}
\chapter{The Tate construction}\label{ch:tate}

Recall that if $G$ is a finite group acting on an abelian group $M$, then the Tate cohomology $\widehat{H}^i(G,M)$ is defined by splicing together cohomology and homology. More precisely, $\widehat{H}^i(G,M) = H^i(G,M)$ if $i>0$, $\widehat{H}^{i-1}(G,M) = H_{-i}(G,M)$ if $i<0$, and there is an exact sequence
\[
0\to \widehat{H}^{-1}(G,M)\to M_G\buildrel{\Nm_G}\over\longrightarrow M^G\to \widehat{H}^0(G,M)\to 0\ .
\]
Our goal in this chapter is to discuss the $\infty$-categorical analogue of this construction.

In Section~\ref{sec:tateconstruction} we recall the norm map $\Nm_G: X_{hG}\to X^{hG}$ defined for any finite group $G$ acting on a spectrum $X$, following the presentation of Lurie, \cite[Section 6.1.6]{HA}. Using this, we can define the Tate construction $X^{tG} = \cofib(\Nm_G: X_{hG}\to X^{hG})$. In Section~\ref{sec:tateorbitlemma}, we prove a very concrete result about the vanishing of the Tate construction which is the essential computational input into Theorem~\ref{thm:intromain}. In Section~\ref{sec:tatemonoidal}, we prove that the Tate construction is lax symmetric monoidal, which is required for many arguments later. This is a classical fact, however all constructions of this lax symmetric monoidal structure that we know use some form of genuine equivariant homotopy theory. By contrast, our construction is direct, and moreover proves that the lax symmetric monoidal structure is unique (and therefore necessarily agrees with other constructions). The argument uses Verdier quotients of stable $\infty$-categories, and we prove some basic results about them in Theorem~\ref{thm:verdier} and Theorem~\ref{thm:verdiermult}, amplifying the discussion in \cite[Section 5]{BlumbergGepnerTabuada}. In Section~\ref{sec:farrelltate}, we use these ideas to construct norm maps in much greater generality, following ideas of Klein, \cite{KleinDualizing}, and verify that the resulting Tate constructions, classically known as Farrell--Tate cohomology, are again lax symmetric monoidal in many cases. This is used later only in the case of the circle $\T$.

\section{The Tate construction for finite groups}\label{sec:tateconstruction}

We start with some brief recollections on norm maps in $\infty$-categorical situations. For this, we follow closely \cite[Section 6.1.6]{HA}. The following classes of $\infty$-categories will be relevant to us.

\begin{definition} Let $\calC$ be an $\infty$-category.
\begin{altenumerate}
\item[{\rm (i)}] The $\infty$-category $\calC$ is \emph{pointed} if it admits an object which is both initial and final; such objects are called zero objects.
\item[{\rm (ii)}] The $\infty$-category $\calC$ is \emph{preadditive} if it is pointed, finite products and finite coproducts exist, and for any two objects $X,Y\in \calC$, the map
\[
\left(\begin{array}{cc} \mathrm{id}_X & 0 \\ 0 & \mathrm{id}_Y\end{array}\right): X\sqcup Y\to X\times Y
\]
is an equivalence. Here, $0\in \Hom_{\calC}(X,Y)$ denotes the composition $X\to 0\to Y$ for any zero object $0\in \calC$.
\end{altenumerate}
\emph{If $\calC$ is preadditive, we write $X\oplus Y$ for $X\sqcup Y\simeq X\times Y$. If $\calC$ is preadditive, then $\pi_0 \Hom_\calC(X,Y)$ acquires the structure of a commutative monoid for all $X,Y\in \calC$.}
\begin{altenumerate}
\item[{\rm (iii)}] The $\infty$-category $\calC$ is \emph{additive} if it is preadditive and $\pi_0 \Hom_\calC(X,Y)$ is a group for all $X,Y\in \calC$.
\item[{\rm (iv)}] The $\infty$-category $\calC$ is \emph{stable} if it is additive, all finite limits and colimits exist, and the loop functor $\Omega: \calC\to \calC: X\mapsto 0\times_X 0$ is an equivalence.
\end{altenumerate}
\end{definition}

We refer to \cite[Chapter 1]{HA} for an extensive discussion of stable $\infty$-categories. These notions are also discussed in \cite{GGN}. Note that in \cite[Section 6.1.6]{HA}, Lurie uses the term semiadditive in place of preadditive.

\begin{definition}\label{def:equivobject} Let $G$ be a group, and $\calC$ an $\infty$-category. A $G$-equivariant object in $\calC$ is a functor $BG\to \calC$, where $BG$ is a fixed classifying space for $G$. The $\infty$-category $\calC^{BG}$ of $G$-equivariant objects in $\calC$ is the functor $\infty$-category $\Fun(BG,\calC)$.
\end{definition}

\begin{remark} We are tempted to write $\calC^G$ in place of $\calC^{BG}$, which is closer to standard usage by algebraists. However, this leads to conflicts with the notation $\calC^X$ for a general Kan complex $X$ that we will use momentarily. This conflict is related to the fact that algebraists write $H^\ast(G,M)$ for group cohomology of $G$ acting on a $G$-module $M$, where topologists would rather write $H^\ast(BG,M)$.
\end{remark}

\begin{remark} Applying this definition in the case of the $\infty$-category of spectra $\calC=\Sp$, one gets a notion of $G$-equivariant spectrum. This notion is different from the notions usually considered in \emph{equivariant stable homotopy theory},\footnote{It is even more naive than what is usually called naive.} and we discuss their relation in Section~\ref{sec:genuine} below. To avoid possible confusion, we will refer to $G$-equivariant objects in $\Sp$ as \emph{spectra with $G$-action} instead of $G$-equivariant spectra.
\end{remark}

\begin{definition}\label{deforbits} Let $G$ be a group, and $\calC$ an $\infty$-category.
\begin{enumerate}
\item[{\rm (i)}] Assume that $\calC$ admits all colimits indexed by $BG$. The \emph{homotopy orbits} functor is given by
\[
-_{hG}: \calC^{BG}\to \calC: (F: BG\to \calC)\mapsto \colim_{BG} F\ .
\]
\item[{\rm (ii)}] Assume that $\calC$ admits all limits indexed by $BG$. The \emph{homotopy fixed points} functor is given by
\[
-^{hG}: \calC^{BG}\to \calC: (F: BG\to \calC)\mapsto \lim\nolimits_{BG} F\ .
\]
\end{enumerate}
\end{definition}

\begin{remark}
Note that in the setting of Definition \ref{deforbits} it might be tempting to drop the word `'homotopy'', i.e. to refer to these objects as ``orbits'' and ``fixed points'' and denote them by  $-_G$ and $-^G$ since in an $\infty$-categorical setting this is the only thing that makes sense. However, later in the paper we will need some elements of equivariant homotopy theory, so that there is another spectrum referred to  as ``fixed points''. In order to avoid confusion we use the prefix ``homotopy''. 
\end{remark}

Now assume that $G$ is finite, and that $\calC$ is a sufficiently nice $\infty$-category, in a sense that will be made precise as we go along. We  construct a norm map
\[
\Nm_G: X_{hG}\to X^{hG}
\]
as a natural transformation of functors $-_{hG}\to -^{hG}: \calC^{BG}\to \calC$. The construction will be carried out in several steps.

For any Kan complex $X$, let $\calC^X=\Fun(X,\calC)$ be the functor $\infty$-category. For any map $f: X\to Y$ of Kan complexes, there is a pullback functor $f^\ast: \calC^Y=\Fun(Y,\calC)\to \Fun(X,\calC)=\calC^X$. We denote the left, resp.~right, adjoint of $f^\ast$ by $f_!$, resp.~$f_\ast$, if it exists. In light of \cite[Proposition 4.3.3.7]{HTT}, we will also refer to $f_!$, resp.~$f_\ast$, as the left, resp.~right, Kan extension along $f: X\to Y$.

As an example, note that if $f: BG\to \ast$ is the projection to a point, then the resulting functors $f_!,f_\ast: \calC^{BG}\to \calC$ are given by $-_{hG}$ and $-^{hG}$, respectively.

We will often use the following construction, where we make the implicit assumption that all functors are defined, i.e.~that $\calC$ has sufficiently many (co)limits.

\begin{construction}\label{constr:norm} Let $f: X\to Y$ be a map of Kan complexes, and let $\delta: X\to X\times_Y X$ be the diagonal. Assume that there is a natural transformation
\[
\Nm_\delta: \delta_!\to \delta_\ast
\]
of functors $\calC^X\to \calC^{X\times_Y X}$, and that $\Nm_\delta$ is an equivalence.

Let $p_0,p_1: X\times_Y X\to X$ denote the projections onto the first and second factor. We get a natural transformation
\[
p_0^\ast\to \delta_\ast \delta^\ast p_0^\ast\simeq \delta_\ast\buildrel{\Nm_\delta^{-1}}\over\simeq \delta_!\simeq \delta_! \delta^\ast p_1^\ast\to p_1^\ast\ ,
\]
and by adjunction a map $\id_{\calC^X}\to p_{0\ast} p_1^\ast$. Now consider the diagram
\[\xymatrix{
X\times_Y X\ar[r]^{p_0}\ar[d]_{p_1} & X\ar[d]_f\\
X\ar[r]^f& Y.
}\]
By \cite[Lemma 6.1.6.3]{HA}, cf.~\cite[Definition 4.7.4.13]{HA} for the definition of right adjointable diagrams, the natural transformation $f^\ast f_\ast\to p_{0\ast} p_1^\ast$ is an equivalence. We get a natural transformation $\mathrm{id}_{\calC^X}\to f^\ast f_\ast$ of functors $\calC^X\to \calC^X$, which is adjoint to a natural transformation $\Nm_f:  f_!\to f_\ast$ of functors $\calC^X\to \calC$.
\end{construction}

Recall that a map $f: X\to Y$ of Kan complexes is $(-1)$-truncated if all fibers of $f$ are either empty or contractible.\footnote{Here and in the following, all fibers of maps of Kan complexes are understood to be homotopy fibers.} Equivalently, $\delta: X\to X\times_Y X$ is an equivalence. In this case, $\Nm_\delta$ exists tautologically.

\begin{lemma}[{\cite[Proposition 6.1.6.7]{HA}}] If $\calC$ is a pointed $\infty$-category, then the functors $f_!, f_\ast$ exist for all $(-1)$-truncated maps $f: X\to Y$, and $\Nm_f: f_!\to f_\ast$ is an equivalence.$\hfill \Box$
\end{lemma}

Thus, if $\calC$ is pointed, we can now play the game for $0$-truncated maps $f: X\to Y$, as then $\delta: X\to X\times_Y X$ is $(-1)$-truncated, and so $\Nm_\delta$ exists and is an equivalence. We say that a $0$-truncated map $f: X\to Y$ has finite fibers if all fibers of $f$ are equivalent to finite sets.

\begin{lemma}[{\cite[Proposition 6.1.6.12]{HA}}] If $\calC$ is a preadditive $\infty$-category, then the functors $f_!, f_\ast$ exist for all $0$-truncated maps $f: X\to Y$ with finite fibers, and $\Nm_f: f_!\to f_\ast$ is an equivalence.$\hfill \Box$
\end{lemma}

Therefore, if $\calC$ is preadditive, we can go one step further, and pass to $1$-truncated maps $f: X\to Y$. We say that a $1$-truncated map $f: X\to Y$ is a relative finite groupoid if all fibers of $f$ have finitely many connected components, and each connected component is the classifying space of a finite group. Note that if $f: X\to Y$ is a relative finite groupoid, then $\delta: X\to X\times_Y X$ is a $0$-truncated map with finite fibers.

\begin{definition}\label{def:norm} Let $\calC$ be a preadditive $\infty$-category which admits limits and colimits indexed by classifying spaces of finite groups. Let $f: X\to Y$ be a map of Kan complexes which is a relative finite groupoid. The norm map
\[
\Nm_f: f_!\to f_\ast
\]
is the natural transformation of functors $\calC^X\to \calC^Y$ given by Construction~\ref{constr:norm}.
\end{definition}

\begin{example} Let $\calC$ be a preadditive $\infty$-category which admits limits and colimits indexed by $BG$ for some finite group $G$. Applying the previous definition in the special case of the projection $f: BG\to \ast$, we get a natural transformation
\[
\Nm_G: X_{hG}\to X^{hG}
\]
of functors $\calC^{BG}\to \calC$.
\end{example}

\begin{example} Let $\calC$ be a preadditive $\infty$-category which admits limits and colimits indexed by $BG$ for some finite group $G$. Assume that $G$ is a normal subgroup of some (topological) group $H$. In this case, for any $H$-equivariant object $X\in \calC^{BH}$, the $G$-homotopy orbits and $G$-homotopy fixed points acquire a remaining $H/G$-action. More precisely, consider the projection $f: BH\to B(H/G)$. Then, by \cite[Proposition 6.1.6.3]{HA}, the functors
\[
f_!,f_\ast: \calC^{BH}\to \calC^{B(H/G)}
\]
sit in commutative diagrams
\[\xymatrix{
\calC^{BH}\ar[r]^{f_!}\ar[d] & \calC^{B(H/G)}\ar[d] & & \calC^{BH}\ar[r]^{f_\ast}\ar[d] & \calC^{B(H/G)}\ar[d]\\
\calC^{BG}\ar[r]^{-_{hG}} & \calC & & \calC^{BG}\ar[r]^{-^{hG}} & \calC,
}\]
where the vertical functors are the forgetful functors. By abuse of notation, we sometimes denote these functors simply by
\[
-_{hG}, -^{hG}: \calC^{BH}\to \calC^{B(H/G)}\ .
\]
We claim that in this situation, the natural transformation $\Nm_G: -_{hG}\to -^{hG}$ of functors $\calC^{BG}\to \calC$ refines to a natural transformation $\Nm_f: f_!\to f_\ast$ of functors $\calC^{BH}\to \calC^{B(H/G)}$, i.e.~$\Nm_G$ is $H/G$-equivariant. Indeed, $\Nm_f$ is a special case of Definition~\ref{def:norm}.
\end{example}

\begin{definition}\label{def:tate} Let $\calC$ be a stable $\infty$-category which admits all limits and colimits indexed by $BG$ for some finite group $G$. The Tate construction is the functor
\[
-^{tG}: \calC^{BG}\to \calC: X\mapsto \cofib(\Nm_G: X_{hG}\to X^{hG})\ .
\]
If $G$ is a normal subgroup of a (topological) group $H$, we also write $-^{tG}$ for the functor
\[
-^{tG}: \calC^{BH}\to \calC^{B(H/G)}: X\mapsto \cofib(\Nm_f: f_! X\to f_\ast X)\ ,
\]
where $f: BH\to B(H/G)$ denotes the projection.
\end{definition}

We will use this definition in particular in the case where $\calC=\Sp$ is the $\infty$-category of spectra. More generally, one can apply it to the $\infty$-category $R\text{-}\Mod$ of module spectra over any associative ring spectrum $R\in \Alg(\Sp)$; this includes the case of the $\infty$-derived category of $S$-modules for some usual associative ring $S$ by taking $R=HS$. However, limits and colimits in $R-\Mod$ are compatible with the forgetful functor $R-\Mod\to \Sp$, as are the norm maps, so that the resulting Tate constructions are also compatible.

If one takes the Tate spectrum of an Eilenberg-MacLane spectrum $HM$ for a $G$-module $M$ then one gets back classical Tate cohomology, more precisely $\pi_i(HM^{tG}) \cong \widehat{H}^{-i}(G,M)$. As an example we have $\widehat{H}^*(C_p,\Z) \cong \bF_p[t^{\pm 1}]$ with $t$ of degree 2. We see in this example that there is a canonical ring structure on Tate (co)homology. This is in fact a general phenomenon, which is encoded in the statement that the functor
\[
-^{tG}: \Sp^{BG}\to \Sp
\]
admits a canonical lax symmetric monoidal structure. We will discuss this further below in Section~\ref{sec:tatemonoidal}. For now, we will ignore all multiplicative structure.
\section{The Tate Orbit Lemma}\label{sec:tateorbitlemma}

In this section, we prove the key lemma in the proof of Theorem~\ref{thm:intromain}, which we call the Tate orbit lemma. Let $X$ be a spectrum with a $C_{p^2}$-action for a prime $p$. The homotopy orbits $X_{hC_p}$ and the homotopy fixed points $X^{hC_p}$ have a remaining $C_{p^2}/C_p$-action which allows us to take Tate spectra. 

\begin{lemma}[Tate orbit lemma]\label{lem:tateorbit} Let $X$ be a spectrum with a $C_{p^2}$-action. Assume that $X$ is bounded below, i.e.~there exists some $n\in \mathbb Z$ such that $\pi_i(X) = 0$ for $i<n$. Then
\[
(X_{hC_p})^{t(C_{p^2}/C_p)} \simeq 0.
\]
\end{lemma}

\begin{lemma}[Tate fixpoint lemma]\label{lem:tatefixpoint} Assume that $X$ is bounded above, i.e.~there exists some $n\in \mathbb Z$ such that $\pi_i(X) = 0$ for $i>n$. Then
\[
(X^{hC_p})^{t(C_{p^2}/C_p)} \simeq 0.
\]
\end{lemma}

\begin{example} \label{counterexamples}
We give some examples to show that neither the Tate orbit lemma nor the Tate fixpoint lemma are satisfied for all spectra.\footnote{One can, however, check that they hold for all chain complexes, as then both $(X_{hC_p})^{t(C_{p^2}/C_p)}$ and $(X^{hC_p})^{t(C_{p^2}/C_p)}$ are modules over the $\E_\infty$-ring spectrum $(H\Z^{hC_p})^{t(C_{p^2}/C_p)}\simeq 0$.}
\begin{altenumerate}
\item For the sphere spectrum $\bS$ with trivial $C_{p^2}$-action we claim that $(\bS^{hC_p})^{tC_{p^2}/C_p}$ is equivalent to the $p$-completion of the sphere spectrum. In particular, it is nontrivial and $\bS$ does not satisfy the conclusion of the Tate fixpoint lemma.

To see this, note that there is a fiber sequence
\[
\bS_{hC_p}\to \bS^{hC_p}\to \bS^{tC_p}\ ,
\]
and it follows from the Segal conjecture that $\bS^{tC_p}\simeq \bS^\wedge_p$ is the $p$-completion of the sphere spectrum, cf.~Remark~\ref{remark:segal} below. Since the map $\bS \to \bS^{tC_p}$ that leads to this equivalence is equivariant for the trivial action on the source (cf.~Example \ref{examplesTHH}~(ii)) it follows that the residual $C_{p^2}/C_p$-action on $\bS^{tC_p}$ is trivial. Applying $-^{t(C_{p^2}/C_p)}$ to the displayed fiber sequence, we get a fiber sequence
\[
(\bS_{hC_p})^{t(C_{p^2}/C_p)}\to (\bS^{hC_p})^{t(C_{p^2}/C_p)}\to (\bS^\wedge_p)^{t(C_{p^2}/C_p)}\ .
\]
Here, the first term is $0$ by Lemma~\ref{lem:tateorbit}. The last term is given by
\[
(\bS^\wedge_p)^{t(C_{p^2}/C_p)}=(\bS^\wedge_p)^{tC_p}\simeq \bS^{tC_p}\simeq \bS^\wedge_p\ ,
\]
using Lemma~\ref{lem:tatecomplete} in the middle identification, and the Segal conjecture in the last identification. Thus, we see that
\[
(\bS^{hC_p})^{t(C_{p^2}/C_p)}\simeq \bS^\wedge_p\ ,
\]
as desired.

\item Dualizing the previous example, one can check that the Anderson dual of the sphere spectrum does not satisfy the Tate orbit lemma. We leave the details to the reader.

\item Let $X = \KU$ be periodic complex $K$-theory with trivial $C_{p^2}$-action. We claim that $X$ does not satisfy the Tate orbit lemma. We will prove in Lemma~\ref{lemtet} below that for every spectrum $X$ that satisfies the Tate orbit lemma, we get an equivalence $(X^{tC_p})^{h(C_{p^2}/C_p)} \simeq X^{tC_{p^2}}$. We now check that this cannot be the case for $X=\KU$. First, we compute that
$$
\pi_*(\KU^{tC_p}) \cong \KU_*((x)) / [p](x) \cong \KU_*((x)) / \big((x-1)^p-1\big)
$$
where we have used a well known formula for Tate spectra of complex oriented cohomology theories, see e.g.~\cite[Proposition 3.20]{MR1361893} or \cite[Section 2]{MR1638030}. This is a rational spectrum as one can see by a direct computation. The homotopy groups are given by the cyclotomic field $\Q_p(\zeta_p)$ in even degrees and $0$ in odd degrees. The remaining $C_{p^2}/C_p$-action on this spectrum extends to a $\T/C_{p}$-action. This implies that it acts trivially on homotopy groups. Since we are in a rational situation this implies that taking further homotopy fixed points does not have any effect, i.e.
$$\pi_*((\KU^{tC_p})^{h(C_{p^2}/C_p)}) \cong \KU_* \otimes \Q_p(\zeta_p)$$
The Tate-spectrum for the $C_{p^2}$-action on the other hand has homotopy groups
$$
\pi_*(\KU^{tC_{p^2}}) \cong \KU_*((x)) / [p^2](x) 
$$
The $p^2$-series of the multiplicative group law is given by 
$ [p^2](x) = (x-1)^{p^2} - 1$. This spectrum is again rational, as in this ring $x^{p^2}$ is divisible by $p$, so that inverting $x$ also inverts $p$. In fact, one computes that
$$\pi_*(\KU^{tC_{p^2}}) \cong \KU_* \otimes (\Q_p(\zeta_p)\oplus \Q_p(\zeta_{p^2}))$$
But the homotopy groups of this spectrum have a different dimension over $\Q_p$ than the homotopy groups of $(\KU^{tC_p})^{h(C_{p^2}/C_p)}$. Thus we get that $(\KU^{tC_p})^{h(C_{p^2}/C_p)} \not\simeq \KU^{tC_{p^2}}$. Therefore $\KU$ with trivial $C_{p^2}$-action does not satisfy the Tate orbit lemma.
\end{altenumerate}
\end{example}

We try to keep our proof of Lemma~\ref{lem:tateorbit} and Lemma~\ref{lem:tatefixpoint} as elementary as possible; in particular, we want to avoid the use of homotopy coherent multiplicative structures. If one is willing to use the multiplicativity of the Tate construction (see Section \ref{sec:tatemonoidal}), one can give a shorter proof.\footnote{Indeed, as below, one reduces the problem to Eilenberg--MacLane spectra $X$, in which case both $(X_{hC_p})^{t(C_{p^2}/C_p)}$ and $(X^{hC_p})^{t(C_{p^2}/C_p)}$ are modules over the $\E_\infty$-ring spectrum $(H\Z^{hC_p})^{t(C_{p^2}/C_p)}$. Thus, it suffices to see that $(H\Z^{hC_p})^{t(C_{p^2}/C_p)}\simeq 0$, which follows from half of Lemma~\ref{lem:tateabgroup}.}

The following lemma is the key computation in the proof of Lemma~\ref{lem:tateorbit} and Lemma~\ref{lem:tatefixpoint}.

\begin{lemma}\label{lem:tatefp} Let $X=H\Fp$ with trivial $C_{p^2}$-action. For any integer $i\geq 0$, we have
\[
(\tau_{[2i,2i+1]} X_{hC_p})^{t(C_{p^2}/C_p)}\simeq 0\ ,\ (\tau_{[-2i-1,-2i]} X^{hC_p})^{t(C_{p^2}/C_p)}\simeq 0\ .
\]
Moreover, the conclusion of Lemma~\ref{lem:tateorbit} and Lemma~\ref{lem:tatefixpoint} holds for $X$.
\end{lemma}

\begin{proof} In the proof, we use freely the computation of $H^i(C_p,\Fp)=H_i(C_p,\Fp)=\Fp$, the ring structure on cohomology, and the module structure on homology, as well as the similar results for $C_{p^2}$. These structures are standard and can be found in most books about cohomology of groups, e.g. \cite[Chapter XII, �7]{MR0077480}. 

We start with the following vanishing result.

\begin{lemma} Let $A\in \calD(\Fp)^{BC_p}$ be a $C_p$-equivariant chain complex of $\Fp$-vector space whose only nonzero cohomology groups are $H^0(A,\Fp)=H^1(A,\Fp)=\Fp$, with (necessarily) trivial $C_p$-action. Assume that the extension class in 
\[
\pi_0 \Map_{\calD(\Fp)^{BC_p}}(\Fp[-1],\Fp[1]) = H^2(C_p,\Fp)=\Fp\ .
\]
corresponding to $A$ is nonzero.
Then $A^{tC_p}\simeq 0$.
\end{lemma}

\begin{proof} We note that there is a unique $A$ up to  equivalence with the given properties. Consider the complex
\[
B: \Fp[C_p]\buildrel {\gamma - 1}\over\longrightarrow \Fp[C_p]\ ,
\]
where $\gamma$ acts by the cyclic permutation action on $C_p$. The $C_p$-action on $B$ is given by the regular representation $\Fp[C_p]$ in both degrees. One checks directly that the only nonzero cohomology groups of $B$ are $H^0(B,\Fp)=H^1(B,\Fp)=\Fp$, and $B^{tC_p}\simeq 0$, as $\Fp[C_p]^{tC_p}\simeq 0$ because $\Fp[C_p]$ is an induced $C_p$-module. In particular, $B$ must correspond to a nonzero class in $H^2(C_p,\Fp)$, as otherwise $B^{tC_p}\simeq \Fp^{tC_p}\oplus \Fp[-1]^{tC_p}\neq 0$. Thus, $A\simeq B$, and $A^{tC_p}\simeq B^{tC_p}\simeq 0$.
\end{proof}

Now take $A=\tau_{[2i,2i+1]} \mathbb F_{p,hC_p}[-2i-1]$ respectively $A=\tau_{[-2i-1,-2i]} \Fp^{hC_p}[2i]$. By the usual computation of $H^i(C_p,\Fp)=\Fp$ and $H_i(C_p,\Fp)=\Fp$, we see that the only nonzero cohomology groups of $A$ are given by $H^0(A,\Fp)=H^1(A,\Fp)=\Fp$. It remains to see that $A$ is not $C_{p^2}/C_p$-equivariantly split.

Note that there is a class $\alpha\in H_{2i+1}(C_{p^2},\Fp)=\Fp$ respectively $\alpha\in H^{2i}(C_{p^2},\Fp)=\Fp$ which projects onto a generator of $H_{2i+1}(C_p,\Fp)$ respectively $H^{2i}(C_p,\Fp)$. Cap respectively cup product with $\alpha$ induces a $C_{p^2}/C_p$-equivariant equivalence
\[
\tau_{[-1,0]} \Fp^{hC_p}\simeq A\ .
\]
Thus, it suffices to prove that $\tau_{[-1,0]} \Fp^{hC_p}$ is not $C_{p^2}/C_p$-equivariantly split.

Now we look at the contributions in total degree $1$ to the $E_2$-spectral sequence
\[
H^i(C_{p^2}/C_p,H^j(C_p,\Fp))\Rightarrow H^{i+j}(C_{p^2},\Fp)\ .
\]
If $\tau_{[-1,0]} \Fp^{hC_p}$ is $C_{p^2}/C_p$-equivariantly split, the differential
\[
H^0(C_{p^2}/C_p,H^1(C_p,\Fp))\to H^2(C_{p^2}/C_p,H^0(C_p,\Fp))
\]
is zero, which implies that both
\[
H^0(C_{p^2}/C_p,H^1(C_p,\Fp))=\Fp\ ,\ H^1(C_{p^2}/C_p,H^0(C_p,\Fp))=\Fp
\]
survive the spectral sequence, which would imply that $H^1(C_{p^2},\Fp)\cong \Fp^2$. However, $H^1(C_{p^2},\Fp)\cong \Fp$ is $1$-dimensional, which is a contradiction.

It remains to prove the last sentence of the lemma. Note that by splicing together the result for different $i$, we see that for all $i\geq 0$,
\[
(\tau_{\leq 2i+1} X_{hC_p})^{t(C_{p^2}/C_p)}\simeq 0\ ,\ (\tau_{\geq -2i-1} X^{hC_p})^{t(C_{p^2}/C_p)}\simeq 0\ .
\]
Now the result follows from the convergence lemma below.
\end{proof}

\begin{lemma}\label{lem:tateconvergence} Let $Y$ be a spectrum with $G$-action for some finite group $G$.
\begin{altenumerate}
\item The natural maps
\[\begin{aligned}
Y^{hG}&\to \lim\nolimits_n (\tau_{\leq n} Y)^{hG}\ ,\\
Y_{hG}&\to \lim\nolimits_n (\tau_{\leq n} Y)_{hG}\ ,\\
Y^{tG}&\to \lim\nolimits_n (\tau_{\leq n} Y)^{tG}
\end{aligned}\]
are equivalences.
\item The natural maps
\[\begin{aligned}
\colim_n (\tau_{\geq -n} Y)^{hG}&\to Y^{hG}\ ,\\
\colim_n (\tau_{\geq -n} Y)_{hG}&\to Y_{hG}\ ,\\
\colim_n (\tau_{\geq -n} Y)^{tG}&\to Y^{tG}
\end{aligned}\]
are equivalences.
\end{altenumerate}
\end{lemma}

\begin{proof} In the first part, the result for $-^{hG}$ is clear as limits commute with limits. It remains to prove the result for $-_{hG}$, as then the case of $-^{tG}$ follows by passing to the cofiber. But for any $n$, the map
\[
Y_{hG}\to (\tau_{\leq n} Y)_{hG}
\]
is $n$-connected, as $Y\to \tau_{\leq n} Y$ is $n$-connected and taking homotopy orbits only increases connectivity. Passing to the limit $n\to \infty$ gives the result. The second part follows in the same way.
\end{proof}

\begin{lemma}\label{lem:tateabgroup} Let $M$ be an abelian group with $C_{p^2}$-action, and let $X=HM$ be the corresponding Eilenberg-MacLane spectrum with $C_{p^2}$-action. Then $X$ satisfies the conclusion of Lemma~\ref{lem:tateorbit} and Lemma~\ref{lem:tatefixpoint}.
\end{lemma}

\begin{proof} We start with the Tate fixpoint lemma. First, observe that the functor taking an abelian group $M$ with $C_{p^2}$-action to
\[
(HM^{hC_p})^{t(C_{p^2}/C_p)} = \cofib((HM^{hC_p})_{h(C_{p^2}/C_p)}\to HM^{hC_{p^2}})
\]
commutes with filtered colimits. Indeed, it suffices to check this for the functors sending $M$ to $HM^{hC_p}$ and $HM^{hC_{p^2}}$ (as homotopy orbits and cofibers are colimits and thus commute with all colimits), where it is the observation that group cohomology $H^i(C_p,M)$ and $H^i(C_{p^2},M)$ commutes with filtered colimits.

Thus, we can assume that $M$ is a finitely generated abelian group. We can also assume that $M$ is torsion-free, by resolving $M$ by a $2$-term complex of torsion-free abelian groups with $C_{p^2}$-action. Now $M/p$ is a finite-dimensional $\Fp$-vector space with $C_{p^2}$-action. As $C_{p^2}$ is a $p$-group, it follows that $M/p$ is a successive extension of $\Fp$ with trivial $C_{p^2}$-action. Thus, $M/p$ satisfies the Tate fixpoint lemma by Lemma~\ref{lem:tatefp}. It follows that multiplication by $p$ induces an automorphism of $(HM^{hC_p})^{t(C_{p^2}/C_p)}$. Thus, we can pass to the filtered colimit $M[\tfrac 1p]$ of $M$ along multiplication by $p$, and assume that multiplication by $p$ is an isomorphism on $M$. Thus, the same is true for $Y=HM^{hC_p}$, and then $Y^{t(C_{p^2}/C_p)} = 0$ by the following standard observation.

\begin{lemma}\label{lem:taterational} Let $Y$ be a spectrum with $C_p$-action such that multiplication by $p$ is an isomorphism on $\pi_i Y$ for all $i\in \Z$. Then $Y^{tC_p}\simeq 0$.
\end{lemma}

\begin{proof} The assumptions imply that $\pi_i Y^{hC_p} = (\pi_i Y)^{C_p}$, $\pi_i Y_{hC_p} = (\pi_i Y)_{C_p}$, and the norm map $(\pi_i Y)_{C_p}\to (\pi_i Y)^{C_p}$ is an isomorphism, giving the claim.
\end{proof}

For the Tate orbit lemma, we do not see a direct way to prove that one can pass to filtered colimits. However, in this case, $(HM_{hC_p})^{t(C_{p^2}/C_p)}$ is $p$-complete by the next lemma.

\begin{lemma}\label{lem:tatecomplete}
Let $X$ be a spectrum with $C_p$-action which is bounded below. Then $X^{tC_p}$ is $p$-complete and equivalent to $(X^\wedge_p)^{tC_p}$.
\end{lemma}

\begin{proof}
Since $-^{tC_p}$ is an exact functor it commutes with smashing with the Moore spectrum $\bS/p$. Thus the canonical map
$$
X^{tC_p} \to (X^\wedge_p)^{tC_p}
$$
is a $p$-adic equivalence. If $X$ is bounded below then so is $X^\wedge_p$. Thus it suffices to show the first part of the lemma.

By Lemma~\ref{lem:tateconvergence} and the fact that limits of $p$-complete spectra are $p$-complete, we can assume that $X$ is bounded. We can then filter $X$ by Eilenberg-MacLane spectra, and reduce to an Eilenberg-MacLane spectrum $X=HM[i]$ concentrated in a single degree $i$. But the Tate cohomology $\widehat{H}^\ast(C_p,M)$ is $p$-torsion. Thus also $X^{tC_p}$ is $p$-torsion, and in particular $p$-complete.
\end{proof}

As $(HM_{hC_p})^{t(C_{p^2}/C_p)}$ is $p$-complete, it suffices to show that
\[
(HM_{hC_p})^{t(C_{p^2}/C_p)} / p\simeq 0
\]
in order to prove $(HM_{hC_p})^{t(C_{p^2}/C_p)}\simeq 0$. As above, we can also assume that $M$ is $p$-torsion free. In that case,
\[
(HM_{hC_p})^{t(C_{p^2}/C_p)} / p \simeq (H(M/p)_{hC_p})^{t(C_{p^2}/C_p)}\ ,
\]
so we can assume that $M$ is killed by $p$. If $\gamma$ is a generator of $C_{p^2}$, it follows that $(\gamma-1)^{p^2}=\gamma^{p^2}-1=0$ on $M$, so that $M$ has a filtration of length at most $p^2$ whose terms have trivial $C_{p^2}$-action. Therefore, we can further assume that $M$ has trivial $C_{p^2}$-action.

Thus, $M$ is an $\Fp$-vector space with trivial $C_{p^2}$-action. Applying Lemma~\ref{lem:tatefp}, we see that for all $i\geq 0$
\[
(\tau_{[2i,2i+1]} HM_{hC_p})^{t(C_{p^2}/C_p)}\simeq 0\ ,
\]
as this functor commutes with infinite direct sums. This implies
\[
(HM_{hC_p})^{t(C_{p^2}/C_p)}\simeq 0
\]
as in the final paragraph of the proof of Lemma~\ref{lem:tatefp}.
\end{proof}

\begin{proof}[Proof of Lemma~\ref{lem:tateorbit}, Lemma~\ref{lem:tatefixpoint}.] Consider first Lemma~\ref{lem:tateorbit}. We have to show that $(X_{hC_p})^{t(C_{p^2}/C_p)} = 0$ for $X$ bounded below. Equivalently, we have to prove that the norm map
\[
X_{hC_{p^2}}\to (X_{hC_p})^{hC_p}
\]
is an equivalence. By Lemma~\ref{lem:tateconvergence}~(i), both sides are given by the limit of the values at $\tau_{\leq n} X$. Thus, we can assume that $X$ is bounded. Filtering $X$, we may then assume that $X=HM[i]$ is an Eilenberg-MacLane spectrum, where the result is given by Lemma~\ref{lem:tateabgroup}.

For Lemma~\ref{lem:tatefixpoint}, one argues similarly, using Lemma~\ref{lem:tateconvergence}~(ii) and Lemma~\ref{lem:tateabgroup}.
\end{proof}

\section{Multiplicativity of the Tate construction}\label{sec:tatemonoidal}

Let $G$ be a finite group. Recall that, as a right adjoint to the symmetric monoidal functor $\Sp\to \Sp^{BG}$ sending a spectrum to the spectrum with trivial $G$-action, the homotopy fixed point functor $-^{hG}: \Sp^{BG}\to \Sp$ admits a canonical lax symmetric monoidal structure, cf.~\cite[Corollary 7.3.2.7]{HA}. We refer to Appendix \ref{app:symmmon} for a general discussion of symmetric monoidal $\infty$-categories. In this section, we show that 
the functor $-^{tG}: \Sp^{BG}\to \Sp$ admits a unique lax symmetric monoidal structure which makes the natural transformation $-^{hG}\to -^{tG}$ lax symmetric monoidal. More precisely:

\begin{thm}\label{thm:tatelaxsymm} 
The space consisting of all pairs of a lax symmetric monoidal structure on the functor $-^{tG}$ together with a lax symmetric monoidal refinement of the natural transformation $-^{hG}\to -^{tG}$ is contractible.
\end{thm}

It is well-known in genuine equivariant homotopy theory that the Tate construction admits a canonical lax symmetric monoidal structure, cf.~the discussion after Proposition~\ref{prop:hesselholtmadsen} below; however, we are not aware that the uniqueness assertion was known before. Moreover, we do not know a previous construction of the lax symmetric monoidal structure on the Tate construction that avoids the use of genuine equivariant homotopy theory. 

For the proof, we need to use some properties of Verdier localizations of stable $\infty$-categories. These are discussed in \cite[Section 5]{BlumbergGepnerTabuada}, but we need some finer structural analysis.

\begin{definition}\label{def:tensorideal} Let $\calC$ be a stable $\infty$-category.
\begin{altenumerate}
\item A stable subcategory of $\calC$ is a full subcategory $\calD\subseteq \calC$ such that $\calD$ is stable and the inclusion $\calD\subseteq \calC$ is exact.
\item Assume that $\calC$ is a stably symmetric monoidal stable $\infty$-category\footnote{By ``stably symmetric monoidal'' we mean that the tensor product is exact in each variable separately.}. A stable subcategory $\calD\subseteq \calC$ is a $\otimes$-ideal if for all $X\in \calC$ and $Y\in \calD$, one has $X\otimes Y\in \calD$.
\end{altenumerate}
\end{definition}

We need the following results about the Verdier localization of stable $\infty$-categories.

\begin{thm}\label{thm:verdier}
Let $\calC$ be a small, stable $\infty$-category and $\calD\subseteq \calC$ a stable sub\-cate\-gory.\begin{altenumerate}
\item Let $W$ be the collection of all arrows in $\calC$ whose cone lies in $\calD$. Then the Dwyer--Kan localization $\calC/\calD := \calC[W^{-1}]$ 
\footnote{For an $\infty$-category $\calC$ and a class of edges $W \subseteq \calC_1$ the Dwyer--Kan localization $\calC[W^{-1}]$ is the universal $\infty$-category with a functor $\calC \to \calC[W^{-1}]$ that sends $W$ to equivalences. We use the slightly non standard term Dwyer--Kan localization to distinguish this concept from the related notion of Bousfield localization, which is called localization in \cite{HTT}.} is a stable $\infty$-category, and $\calC\to \calC/\calD$ is an exact functor. If $\calE$ is another stable $\infty$-category, then composition with $\calC\to\calC/\calD$ induces an equivalence between $\Fun^\Ex(\calC/\calD,\calE)$ and the full subcategory of $\Fun^\Ex(\calC,\calE)$ consisting of those functors which send all objects of $\calD$ to $0$.
\item Let $X,Y\in \calC$ with image $\overline{X},\overline{Y}\in \calC/\calD$. The mapping space in $\calC/\calD$ is given by
\[
\Map_{\calC/\calD}(\overline{X},\overline{Y})\simeq \colim_{Z\in \calD_{/Y}} \Map_\calC(X,\cofib(Z\to Y))\ ,
\]
where the colimit is filtered. In particular, the Yoneda functor
\[
\calC/\calD\to \Fun(\calC^\op,\calS) : \overline{Y}\mapsto (X\mapsto \Map_{\calC/\calD}(\overline{X},\overline{Y}))
\]
factors over the $\Ind$-category,
\[
\calC/\calD\to \Ind(\calC)\subseteq \Fun(\calC^\op,\calS)\ .
\]
This  is an exact functor of stable $\infty$-categories, and sends the image $\overline{Y}\in \calC/\calD$ of $Y\in \calC$ to the formal colimit $\colim_{Z\in \calD_{/Y}} \cofib(Z\to Y)\in \Ind(\calC)$.
\item Assume that $\calE$ is a presentable stable $\infty$-category. Then the full inclusion $\Fun^\Ex(\calC/\calD,\calE)\subseteq \Fun^\Ex(\calC,\calE)$ is right adjoint to a localization in the sense of~\cite[Definition 5.2.7.2]{HTT}. The corresponding localization functor
\[
\Fun^\Ex(\calC,\calE)\to \Fun^\Ex(\calC/\calD,\calE)\subseteq \Fun^\Ex(\calC,\calE)
\]
is given by taking an exact functor $F: \calC\to \calE$ to the composite
\[
\calC/\calD\to \Ind(\calC)\xto{\Ind(F)} \Ind(\calE)\to \calE\ ,
\]
where the first functor comes from Theorem~\ref{thm:verdier}~(ii) and the last functor is taking the colimit in $\calE$.
\end{altenumerate}
\end{thm}

We note that it follows from the definition of $\calC/\calD$ that on homotopy categories $h(\calC/\calD)=h\calC/h\calD$ is the Verdier quotient of triangulated categories. Part (iii) says that any exact functor $F: \calC\to \calE$ has a universal approximation that factors over the quotient $\calC/\calD$.

\begin{proof} First, we prove part (ii). For the moment, we write $\calC[W^{-1}]$ in place of $\calC/\calD$, as we do not yet know (i).

We start with a general observation about mapping spaces in $\calC[W^{-1}]$ for any small $\infty$-category $\calC$ with a collection $W$ of arrows. Recall that there is a model of $\calC[W^{-1}]$ which has the same objects as $\calC$, and we will sometimes tacitly identify objects of $\calC$ and $\calC[W^{-1}]$ in the following. By definition, the $\infty$-category $\calC[W^{-1}]$ has the property that for any $\infty$-category $\calD$,
\[
\Fun(\calC[W^{-1}],\calD)\subseteq \Fun(\calC,\calD)
\]
is the full subcategory of those functors taking all arrows in $W$ to equivalences. Applying this in the case $\calD=\calS$ the $\infty$-category of spaces, we have the full inclusion
\[
\Fun(\calC[W^{-1}],\calS)\subseteq \Fun(\calC,\calS)
\]
of presentable $\infty$-categories (cf.~\cite[Proposition 5.5.3.6]{HTT}) which preserves all limits (and colimits). By the adjoint functor theorem, \cite[Corollary 5.5.2.9]{HTT}, it admits a left adjoint
\[
L: \Fun(\calC,\calS)\to \Fun(\calC[W^{-1}],\calS)\ .
\]
The following lemma gives a description of the mapping spaces in $\calC[W^{-1}]$ in terms of $L$.

\begin{lemma}\label{lem:compmapspacedwyerkan} For every object $X\in \calC$ with image $\overline{X}\in \calC[W^{-1}]$, the functor
\[
\Map_{\calC[W^{-1}]}(\overline{X},-)
\]
is given by $L(\Map_\calC(X,-))$. Moreover, the diagram
\[\xymatrix{
\calC^\op\ar[r]\ar[d] & \Fun(\calC,\calS)\ar^L[d] \\
\calC[W^{-1}]^\op\ar[r] & \Fun(\calC[W^{-1}],\calS)
}\]
of $\infty$-categories commutes, where the horizontal maps are the Yoneda embeddings.
\end{lemma}

\begin{proof} For any object $X\in \calC$ and functor $F: \calC[W^{-1}]\to \calS$, we have
\[
\Map_{\Fun(\calC[W^{-1}],\calS)}(L(\Map_\calC(X,-)),F) = \Map_{\Fun(\calC,\calS)}(\Map_\calC(X,-),F|_\calC) = F(\overline{X})
\]
by adjunction and the Yoneda lemma in $\calC$. In particular, if $f: X\to Y$ is a morphism in $W$ such that $\overline{X}\to \overline{Y}$ is an equivalence, then
\[
L(\Map_\calC(Y,-))\to L(\Map_\calC(X,-))
\]
is an equivalence by the Yoneda lemma in $\Fun(\calC[W^{-1}])$. This implies that the composite
\[
\calC^\op\to \Fun(\calC,\calS)\xto{L} \Fun(\calC[W^{-1}],\calS)
\]
factors uniquely over $\calC[W^{-1}]$. To see that the resulting functor
\[
\calC[W^{-1}]^\op\to \Fun(\calC[W^{-1}],\calS)
\]
is the Yoneda embedding, it suffices to check on the restriction to $\calC^\op$ (by the universal property of $\calC[W^{-1}]^\op$). But this follows from the first displayed equation in the proof, and the Yoneda lemma in $\Fun(\calC^\op[W^{-1}],\calS)$.
\end{proof}

In our case, we can give a description of $L$. Indeed, given any functor $F: \calC\to \calS$, we can form a new functor $L(F): \calC\to \calS$ which sends $X\in \calC$ to
\[
\colim_{Y\in \calD_{/X}} F(\cofib(Y\to X))\ .
\]
Note that as $\calD$ has all finite colimits, the index category is filtered. Also, there is a natural transformation $F\to L(F)$ functorial in $F$. If $F$ sends arrows in $W$ to equivalences, then $F\to L(F)$ is an equivalence, and in general $L(F)$ will send arrows in $W$ to equivalences. From this, one checks easily using \cite[Proposition 5.2.7.4]{HTT} that $L: \Fun(\calC,\calS)\to \Fun(\calC,\calS)$ is a localization functor with essential image given by $\Fun(\calC[W^{-1}],\calS)$, i.e.~$L$ is the desired left adjoint. Now part (ii) follows from Lemma~\ref{lem:compmapspacedwyerkan}.

Now we prove part (i). Note that it suffices to see that $\calC[W^{-1}]$ is a stable $\infty$-category for which the functor $\calC\to \calC[W^{-1}]$ is exact, as the desired universal property will then follow from the universal property of $\calC[W^{-1}]$. It follows from the formula in (ii) that $\calC\to \calC[W^{-1}]$ commutes with all finite colimits, as filtered colimits commute with finite limits in $\calS$. Applying this observation to the dual $\infty$-categories, we see that $\calC\to \calC[W^{-1}]$ also commutes with all finite limits. In particular, one checks that $\calC[W^{-1}]$ is pointed, preadditive and additive. As any pushout diagram in $\calC[W^{-1}]$ can be lifted to $\calC$, where one can take the pushout which is preserved by $\calC\to \calC[W^{-1}]$, it follows from~\cite[Corollary 4.4.2.4]{HTT} that $\calC[W^{-1}]$ has all finite colimits. Dually, it has all finite limits. The loop space functor $X\mapsto 0\times_X 0$ and the suspension functor $X\mapsto 0\sqcup_X 0$ are now defined on $\calC[W^{-1}]$ and commute with the loop space and suspension functors on $\calC$. To check whether the adjunction morphisms are equivalences, it suffices to check after restriction to $\calC$, where it holds by assumption. This finishes the proof of (i).

Part (iii) follows by a straightforward check of the criterion \cite[Proposition 5.2.7.4 (3)]{HTT}.
\end{proof}

In \cite[Section 5]{BlumbergGepnerTabuada}, the Verdier quotient is defined indirectly through the passage to $\Ind$-categories. We note that our definition agrees with theirs. More precisely, one has the following result.

\begin{proposition}\label{prop:verdierquotind} Let $\mathcal C$ be a small, stable $\infty$-category and let $\mathcal D\subseteq \mathcal C$ be a stable subcategory. Then the $\Ind$-category $\Ind(\mathcal C)$ is a presentable stable $\infty$-category, $\Ind(\mathcal D)\subseteq \Ind(\mathcal C)$ is a presentable stable subcategory, and the canonical functor
\[
\Ind(\mathcal C)\to \Ind(\mathcal C/\mathcal D)
\]
has kernel\footnote{The kernel is the full subcategory of $\Ind(\mathcal C)$ consisting of all objects that are mapped to a zero object under the functor $\Ind(\mathcal C)\to \Ind(\mathcal C/\mathcal D)$.} given by $\Ind(\mathcal D)$. Moreover, the functor $\Ind(\mathcal C)\to \Ind(\mathcal C/\mathcal D)$ is a localization in the sense of \cite[Definition 5.2.7.2]{HTT}, with fully faithful right adjoint
\[
\Ind(\mathcal C/\mathcal D)\to \Ind(\mathcal C)
\]
given by the unique colimit-preserving functor whose restriction to $\mathcal C/\mathcal D$ is given by the functor in Theorem~\ref{thm:verdier}~(ii).
\end{proposition}

\begin{proof} Note first that by passing to $\Ind$-categories, we get a colimit-preserving map
\[
\Ind(\calC)\to \Ind(\calC/\calD)
\]
between presentable stable $\infty$-categories; by the adjoint functor theorem \cite[Corollary 5.5.2.9]{HTT}, it has a right adjoint $R$. On $\calC/\calD\subseteq \Ind(\calC/\calD)$, the functor $R$ is given by the functor from Theorem~\ref{thm:verdier}~(ii). We claim that $R$ commutes with all colimits; as it is exact, it suffices to check that it commutes with all filtered colimits. This is a formal consequence of the fact that $\Ind(\calC)\to \Ind(\calC/\calD)$ takes compact objects to compact objects \cite[Proposition 5.5.7.2]{HTT}: It suffices to check that for any $C\in \calC$ and any filtered system $\overline{C}_i\in \calC/\calD$, the map
\[
\colim_i \Hom_\calC(C, R(\overline{C}_i))\to  \Hom_{\calC/\calD}(\overline{C},\colim_i \overline{C}_i)
\]
is an equivalence, where $\overline{C}\in \calC/\calD$ is the image of $C$; indeed, the left-hand side is $\Hom_\calC(C,\colim_i R(\overline{C}_i))$, and the right-hand side is $\Hom_\calC(C,R(\colim_i \overline{C}_i))$, and the adjunction for general objects of $\Ind(\calC)$ follows by passing to a limit. But the displayed equation is an equivalence by using adjunction on the left-hand side, and that $\overline{C}\in \calC/\calD\subseteq \Ind(\calC/\calD)$ is compact.

Now it follows from \cite[Proposition 5.2.7.4]{HTT} that $\Ind(\calC)\to \Ind(\calC/\calD)$ is a localization, as the required equivalences on the adjunction map can be checked on $\calC$, where they follow from the description of $R$ in Theorem~\ref{thm:verdier}~(ii). Thus, we can regard $\Ind(\calC/\calD)\subseteq \Ind(\calC)$ as the full subcategory of local objects.

Finally, we show that the kernel of $\Ind(\calC)\to \Ind(\calC/\calD)$ is exactly $\Ind(\calD)$. Indeed, for any filtered colimit $\colim_i C_i\in \Ind(\calC)$ with $C_i\in \calC$, one can compute the localization as $R(\colim_i \overline{C}_i) = \colim_i R(\overline{C}_i)$, where the  fiber of $C_i\to R(\overline{C}_i)$ is a filtered colimit of objects of $\calD$ by Theorem~\ref{thm:verdier}~(ii). Passing to a filtered colimit, we see that the fiber of $\colim_i C_i\to R(\colim_i \overline{C}_i) = \colim_i R(\overline{C}_i)$ is in $\Ind(\calD)$; in particular, if $R(\colim_i \overline{C}_i)\simeq 0$, then $\colim_i C_i\in \Ind(\calD)$.
\end{proof}

Moreover, we need the following multiplicative properties of the Verdier quotient.

\begin{thm}\label{thm:verdiermult} Let $\calC$ be a small, stably symmetric monoidal stable $\infty$-category and $\calD\subseteq \calC$ a stable subcategory which is a $\otimes$-ideal.
\begin{altenumerate}
\item There is a unique way to simultaneously endow the Verdier quotient $\calC/\calD$ and the functor $\calC\to \calC/\calD$ with a symmetric monoidal structure. If $\calE$ is another symmetric monoidal stable $\infty$-category, then composition with $\calC\to\calC/\calD$ induces an equivalence between $\Fun_\lax^\Ex(\calC/\calD,\calE)$ and the full subcategory of $\Fun_\lax^\Ex(\calC,\calE)$ of those functors which send all objects of $\calD$ to $0$.
\item Assume that $\calE$ is a presentably symmetric monoidal stable $\infty$-category. Then the full inclusion $\Fun_\lax^\Ex(\calC/\calD,\calE)\subseteq \Fun_\lax^\Ex(\calC,\calE)$ is right adjoint to a localization in the sense of~\cite[Definition 5.2.7.2]{HTT}. The corresponding localization functor
\[
\Fun_\lax^\Ex(\calC,\calE)\to \Fun_\lax^\Ex(\calC/\calD,\calE)\subseteq \Fun_\lax^\Ex(\calC,\calE)
\]
is given by taking an exact lax symmetric monoidal functor $F: \calC\to \calE$ to the composite
\[
\calC/\calD\to \Ind(\calC)\xto{\Ind(F)} \Ind(\calE)\to \calE\ ,
\]
where the first functor comes from Theorem~\ref{thm:verdier}~(ii) and the final functor is taking the colimit in $\calE$; both functors are canonically lax symmetric monoidal.
\end{altenumerate}
\end{thm}

In other words, combining part (ii) with Theorem~\ref{thm:verdier}~(iii), we see that for any exact lax symmetric monoidal functor $F: \calC\to \calE$, the universal approximation that factors over $\calC/\calD$ acquires a unique lax symmetric monoidal structure for which the relevant natural transformation is lax symmetric monoidal.

\begin{proof} We apply the results of Appendix \ref{app:symmmon}, more specifically Proposition~\ref{propHinich}. The $\infty$-category $\calC/\calD$ is by definition the Dwyer-Kan localization $\calC[W^{-1}]$ at the class $W$ of morphisms whose cone lies in $\calD \subseteq \calC$. Thus we have to check that for such a morphism $f \in W$ the tensor product $f \otimes c$ for any object $c \in \calC$ is again  in $W$. But this is clear since $\calD$ is a tensor ideal and the tensor product is exact, i.e.~the cone of $f \otimes c$ is the tensor product of the cone of $f$ with $c$. Then we invoke Proposition \ref{propHinich} to get a unique symmetric monoidal structure on $\calC/\calD$ with a symmetric monoidal refinement of the functor $\calC \to \calC/\calD$  such that for every symmetric monoidal $\infty$-category $\calE$ the functor
$$
\Fun_\lax(\calC/\calD,\calE) \to \Fun_\lax(\calC,\calE)
$$
is fully faithful with essential image those functors that send $W$ to equivalences in $\calE$. If $\calE$ is stable this functor induces equivalences between the respective full subcategories of exact functors by Theorem \ref{thm:verdier}(i) since exactness can be tested after forgetting the lax symmetric monoidal structures.

For part (ii), we note that the functor
\[
\Fun_\lax^\Ex(\calC,\calE)\to \Fun_\lax^\Ex(\calC/\calD,\calE)\to \Fun_\lax^\Ex(\calC,\calE)
\]
is well-defined, as all functors in the composition
\[
\calC/\calD\to \Ind(\calC)\xto{\Ind(F)} \Ind(\calE)\to \calE\ ,
\]
are naturally lax symmetric monoidal. For the first functor, this follows from Proposition~\ref{prop:verdierquotind}, as $\Ind(\calC/\calD)\to \Ind(\calC)$ is right adjoint to the symmetric monoidal projection $\Ind(\mathcal C)\to \Ind(\mathcal C/\mathcal D)$, and thus lax symmetric monoidal by \cite[Corollary 7.3.2.7]{HA}. To check the criterion of \cite[Proposition 5.2.7.4]{HTT}, it suffices to check it without the lax symmetric monoidal structures as a lax symmetric monoidal natural transformation is an equivalence if and only if the underlying natural transformation is an equivalence. Thus, it follows from Theorem~\ref{thm:verdier}~(iii) that it is a localization functor as claimed.
\end{proof}

Now we apply the Verdier localization in our setup. Thus, let $\calC=\Sp^{BG}$ be the stable $\infty$-category of spectra with $G$-action.

\begin{definition} Let $\Sp^{BG}_\ind\subseteq \Sp^{BG}$ be the stable subcategory generated by spectra of the form $\bigoplus_{g\in G} X$ with permutation $G$-action, where $X\in \Sp$ is a spectrum.
\end{definition}

In other words, all induced spectra are in $\Sp^{BG}_\ind$, and an object of $\Sp^{BG}$ lies in $\Sp^{BG}_\ind$ if and only if it can be built in finitely many steps by taking cones of maps between objects already known to lie in $\Sp^{BG}_\ind$. Note that not any map between induced spectra with $G$-action
\[
\bigoplus_{g\in G} X\to \bigoplus_{g\in G} Y
\]
comes from a map of spectra $X\to Y$, so that one can build many objects in $\Sp^{BG}_\ind$ which are not themselves induced.

We will need the following properties of $\Sp^{BG}_\ind$.

\begin{lemma}\label{lem:propindbg} Let $X\in \Sp^{BG}$.
\begin{altenumerate}
\item If $X\in \Sp^{BG}_\ind$, then the Tate construction $X^{tG}\simeq 0$ vanishes.
\item For all $Y\in \Sp^{BG}_\ind$, the tensor product $X\otimes Y\in \Sp^{BG}_\ind$. In other words, $\Sp^{BG}_\ind\subseteq \Sp^{BG}$ is a $\otimes$-ideal.
\item The natural maps
\[
\colim_{Y\in (\Sp^{BG}_\ind)_{/X}} Y\to X
\]
and
\[
\colim_{Y\in (\Sp^{BG}_\ind)_{/X}} \cofib(Y\to X)^{hG}\to \colim_{Y\in (\Sp^{BG}_\ind)_{/X}} \cofib(Y\to X)^{tG} = X^{tG}
\]
are equivalences.
\end{altenumerate}
\end{lemma}

Note that part (iii) says that any object of $\Sp^{BG}$ is canonically a filtered colimit (cf.~Theorem~\ref{thm:verdier}~(ii)) of objects in $\Sp^{BG}_\ind$. In particular, $\Sp^{BG}_\ind$ cannot be closed under filtered colimits.

\begin{proof} In part (i), note that the full subcategory of $\Sp^{BG}$ on which the Tate construction vanishes is a stable subcategory. Thus, to prove (i), it suffices to prove that the Tate construction vanishes on induced spectra. But for a finite group $G$, one has identifications
\[
(\bigoplus_{g\in G} X)_{hG} = X = (\bigoplus_{g\in G} X)^{hG}\ ,
\]
under which the norm map is the identity.

In part (ii), given $X\in \Sp^{BG}$, the full subcategory of all $Y\in \Sp^{BG}$ for which $X\otimes Y\in \Sp^{BG}_\ind$ is a stable subcategory. Thus, it suffices to show that $X\otimes Y\in \Sp^{BG}_\ind$ if $Y=\bigoplus_{g\in G} Z$ is an induced spectrum with $G$-action. But in that case there is a (non-equivariant) map of spectra $X\otimes Z\to X\otimes Y$ given by inclusion of one summand, which is adjoint to a $G$-equivariant map of spectra
\[
\bigoplus_{g\in G} X\otimes Z\to X\otimes Y\ ,
\]
which is easily seen to be an equivalence. But the left-hand side is an induced spectrum, so $X\otimes Y\in \Sp^{BG}_\ind$.

For part (iii), we check that for all $i\in \Z$, the map
\[
\colim_{Y\in (\Sp^{BG}_\ind)_{/X}} \pi_i Y\to \pi_i X
\]
is an isomorphism; this is enough for the first equivalence, as the colimit is filtered and thus commutes with $\pi_i$. By translation, we can assume that $i=0$. First, we check surjectivity. If $\alpha\in \pi_0 X$, then there is a corresponding map of spectra $\bS\to X$, which gives a $G$-equivariant map $Y:=\bigoplus_{g\in G} \bS\to X$ such that $\alpha$ lies in the image of $\pi_0 Y$. Similarly, one checks injectivity: If $Y\in (\Sp^{BG}_{\ind})_{/X}$ and $\beta\in \ker(\pi_0 Y\to \pi_0 X)$, then there is a corresponding map $\bS\to Y$ whose composite $\bS\to Y\to X$ is zero. Let $\bar{Y} = \cofib(\bigoplus_{g\in G} \bS\to Y)$, which is an object of $\Sp^{BG}_\ind$ which comes with an induced map to $X$, so that there is a factorization $Y\to \bar{Y}\to X$. By construction, the element $\beta$ dies in $\pi_0 \bar{Y}$. As the colimit is filtered, this finishes the proof.

For the second equivalence in (iii), note that for all $Y\in (\Sp^{BG}_\ind)_{/X}$, there are fiber sequences
\[
\cofib(Y\to X)_{hG}\to \cofib(Y\to X)^{hG}\to \cofib(Y\to X)^{tG} = X^{tG}\ .
\]
Passing to the filtered colimit over all $Y$, the left term vanishes as
\[
\colim_{Y\in (\Sp^{BG}_\ind)_{/X}} \cofib(Y\to X)_{hG} = (\colim_{Y\in (\Sp^{BG}_\ind)_{/X}} \cofib(Y\to X))_{hG} = 0
\]
by the first equivalence of (iii).
\end{proof}

Now we can prove Theorem~\ref{thm:tatelaxsymm}.

\begin{proof}[Proof of Theorem~\ref{thm:tatelaxsymm}.] First, we give an argument ignoring set-theoretic issues. In Theorem~\ref{thm:verdiermult}, we take $\calC=\Sp^{BG}$ and $\calD=\Sp^{BG}_\ind$. Thus, by Lemma~\ref{lem:propindbg}~(i) and Theorem~\ref{thm:verdier}~(i), $-^{tG}$ factors over a functor which we still denote $-^{tG}: \calC/\calD\to \Sp$. Note that, by Theorem~\ref{thm:verdiermult}~(i) and Lemma~\ref{lem:propindbg}~(ii), any lax symmetric monoidal structure on $-^{tG}$ with a lax symmetric monoidal transformation $-^{hG}\to -^{tG}$ gives rise to an exact lax symmetric monoidal functor $H: \calC/\calD\to \Sp$ such that the composite of the projection $\calC\to \calC/\calD$ with $H$ receives a lax symmetric monoidal transformation from $-^{hG}$, and conversely.

On the other hand, taking $F=-^{hG}: \calC=\Sp^{BG}\to \calE=\Sp$ in Theorem~\ref{thm:verdiermult}~(ii), there is a universal exact lax symmetric monoidal functor $H: \calC/\calD\to \calE$ with a lax symmetric monoidal transformation from $F$ to the composite of the projection $\calC\to \calC/\calD$ with $H$. It remains to see that the underlying functor of $H$ is given by $-^{tG}$. As the localization functors of Theorem~\ref{thm:verdiermult}~(ii) and Theorem~\ref{thm:verdier}~(iii) are compatible, it suffices to check this without multiplicative structures. In that case, the universal property of $H$ gives a unique natural transformation $H\to -^{tG}$, which we claim to be an equivalence. But this follows from the description of the localization functor in Theorem~\ref{thm:verdier}~(iii), the description of the functor $\calC/\calD\to \Ind(\calC)$ in Theorem~\ref{thm:verdier}~(ii), and the computation of Lemma~\ref{lem:propindbg}~(iii).

This argument does not work as written, as $\Sp^{BG}$ is not a small $\infty$-category. However, we may choose a regular cardinal $\kappa$ so that $-^{hG}$ and $-^{tG}$ are $\kappa$-accessible functors. One can then run the argument for the full subcategory of $\kappa$-compact objects $\Sp^{BG}_\kappa$ in $\Sp^{BG}$, which is also the category of $G$-equivariant objects in the full subcategory of $\kappa$-compect objects $\Sp_\kappa$ in $\Sp$. It is still lax symmetric monoidal, and we get a unique lax symmetric monoidal structure on $-^{tG}: \Sp_\kappa^{BG}\to \Sp$ which is compatible with the structure on $-^{hG}$. One may now pass to $\Ind_\kappa$-categories to get the desired result.
\end{proof}

We will need the following corollary to the uniqueness assertion in Theorem~\ref{thm:tatelaxsymm}.

\begin{corollary}\label{cor:tatelaxsymmgroup} Let $G$ be a finite group, and assume that $G$ is a normal subgroup of a (topological) group $H$. The functor $-^{tG}: \Sp^{BH}\to \Sp^{B(H/G)}$ admits a natural lax symmetric monoidal structure which makes the natural transformation $-^{hG}\to -^{tG}$ of functors $\Sp^{BH}\to \Sp^{B(H/G)}$ lax symmetric monoidal, compatibly with the natural transformation of lax symmetric monoidal functors $-^{hG}\to -^{tG}: \Sp^{BG}\to \Sp$.
\end{corollary}

In particular, we get a lax symmetric monoidal functor
\[
-^{tC_p}: \Sp^{B\T}\to \Sp^{B(\T/C_p)}
\]
refining $-^{tC_p}: \Sp^{BC_p}\to \Sp$, where $\T$ is the circle group.

\begin{proof} We have a map from $B(H/G)$ to the space of Kan complexes equivalent to $BG$ (without base points, which however have not played any role), given by fibers of the projection $BH\to B(H/G)$. Composing with the map $S\mapsto \Sp^S$ and the natural transformation between the functors $-^{hG}$ and $-^{tG}$ to $\Sp$, we get a map from $B(H/G)$ to the $\infty$-category $\mathcal A$ whose objects are pairs of symmetric monoidal $\infty$-categories $\calC$, $\calD$ with a lax symmetric monoidal functor $F_1: \calC\to \calD$ and a functor of the underlying $\infty$-categories $F_2: \calC\to \calD$ and a natural transformation $F_1\to F_2$. We claim that we can lift this uniquely (up to a contractible choice) to a map from $B(H/G)$ to the $\infty$-category $\mathcal B$ whose objects are pairs of symmetric monoidal $\infty$-categories $\calC$, $\calD$ with two lax symmetric monoidal functors $F_1,F_2:\calC\to \calD$ and a lax symmetric monoidal transformation $F_1\to F_2$. Indeed, there is a functor $\mathcal B\to \mathcal A$, and after replacing $\mathcal B$ by an equivalent $\infty$-category, we can assume that it is a categorical fibration. By Theorem~\ref{thm:tatelaxsymm}, we know that the map
\[
B(H/G)\times_{\mathcal A} \mathcal B\to B(H/G)
\]
has contractible fibers. But this is a categorical fibration over a Kan complex, thus a Cartesian fibration by \cite[Proposition 3.3.1.8]{HTT}. As its fibers are contractible, it is a right fibration by \cite[Proposition 2.4.2.4]{HTT}, and thus a trivial Kan fibration by (the dual of) \cite[Lemma 2.1.3.4]{HTT}. In particular, the space of sections is contractible.

Thus, we can lift canonically to a functor $B(H/G)\to \mathcal B$. Taking the limit of this diagram in $\mathcal B$ (which is computed objectwise), we arrive at the desired diagram of lax symmetric monoidal functors $-^{hG}, -^{tG}: \Sp^{BH}\to \Sp^{B(H/G)}$ with a lax symmetric monoidal transformation $-^{hG}\to -^{tG}$. The resulting diagram maps to the corresponding diagram indexed by $\ast\in B(H/G)$, so it is compatible with the natural transformation of lax symmetric monoidal functors $-^{hG}\to -^{tG}: \Sp^{BG}\to \Sp$.
\end{proof}

\section{Farrell--Tate cohomology}\label{sec:farrelltate}

In this section, we briefly mention a generalization of the Tate construction to general groups $G$, including compact and infinite discrete groups. In fact, most of what we do works for a general Kan complex $S$ in place of $BG$.

Recall that Farrell, \cite{Farrelltate}, had generalized Tate cohomology to infinite discrete groups of finite virtual cohomological dimension. He essentially constructed a norm map
\[
(D_{BG}\otimes X)_{hG}\to X^{hG}
\]
for a certain $G$-equivariant object $D_{BG}$, but he worked only on the level of abelian groups. Klein generalized this to spectra, \cite{KleinDualizing}, and gave a universal characterization of the resulting cohomology theory, \cite{KleinAxioms}.

Here, we prove the following general result. It is closely related to Klein's axioms in \cite{KleinAxioms}. We have also been informed by Tobias Barthel that he has obtained similar results in current work in progress. 

\begin{thm}\label{thm:genfarrelltate} Let $S$ be a Kan complex, and consider the $\infty$-category $\Sp^S=\Fun(S,\Sp)$. Let $p: S\to \ast$ be the projection to the point. Then $p^\ast: \Sp\to \Sp^S$ has a left adjoint $p_!: \Sp^S\to \Sp$ given by homology, and a right adjoint $p_\ast: \Sp^S\to \Sp$ given by cohomology.
\begin{altenumerate}
\item The $\infty$-category $\Sp^S$ is a compactly generated presentable stable $\infty$-category. For every $s\in S$, the functor $s_!: \Sp\to \Sp^S$ takes compact objects to compact objects, and for varying $s\in S$, the objects $s_! \bS$ generate $\Sp^S$.
\item There is an initial functor $p_\ast^T: \Sp^S\to \Sp$ with a natural transformation $p_\ast\to p_\ast^T$ with the property that $p_\ast^T$ vanishes on compact objects.
\item The functor $p_\ast^T: \Sp^S\to \Sp$ is the unique functor with a natural transformation $p_\ast\to p_\ast^T$ such that $p_\ast^T$ vanishes on all compact objects, and the  fiber of $p_\ast\to p_\ast^T$ commutes with all colimits.
\item The fiber of $p_\ast\to p_\ast^T$ is given by $X\mapsto p_!(D_S\otimes X)$ for a unique object $D_S\in \Sp^S$. The object $D_S$ is given as follows. Considering $S$ as an $\infty$-category, one has a functor $\Map: S\times S\to \calS$, sending a pair $(s,t)$ of points in $S$ to the space of paths between $s$ and $t$. Then $D_S$ is given by the composite
\[
D_S: S\to \Fun(S,\calS)\xto{\Sigma_+^\infty} \Fun(S,\Sp)=\Sp^S\xto{p_\ast}\Sp\ .
\]
\item The map $p_!(D_S\otimes -)\to p_\ast$ is final in the category of colimit-preserving functors from $\Sp^S$ to $\Sp$ over $p_\ast$, i.e.~it is an assembly map in the sense of Weiss--Williams, \cite{WeissWilliams}.
\item Assume that for all $s\in S$ and $X\in \Sp$, one has $p_\ast^T(s_! X)\simeq 0$. Then there is a unique lax symmetric monoidal structure on $p_\ast^T$ which makes the natural transformation $p_\ast\to p_\ast^T$ lax symmetric monoidal.
\end{altenumerate}
\end{thm}

We note that by \cite[Corollary 10.2]{KleinDualizing}, the condition in (vi) is satisfied if $S=BG$ where $G$ is a compact Lie group. In fact, in the case of groups, the functor $s_!$, for $s: \ast\to BG$ the inclusion of the base point, is given by induction, and the condition is asking that the generalized Tate cohomology vanishes on induced representations.

This theorem suggests the following definition.

\begin{definition} The Spivak--Klein dualizing spectrum of $S$ is the object $D_S\in \Sp^S$ of Theorem~\ref{thm:genfarrelltate}~(iv).
\end{definition}

In classical language, this is a parametrized spectrum over $S$ in the sense of May--Sigurdsson, \cite{MaySigurdsson}, and it has been discussed in this form by Klein, \cite[Section 5]{KleinDualizing2}. By Theorem~\ref{thm:genfarrelltate}~(iv), the fiber over any point $s\in S$ is given by $\lim_{t\in S} \Sigma_+^\infty \Map(s,t)$. Klein observed that this gives a homotopy-theoretic definition of the Spivak fibration for finite CW complexes, \cite{Spivak}. In particular, if $S$ is a closed manifold of dimension $d$, then by Poincar\'e duality, $D_S$ is fiberwise a sphere spectrum shifted into degree $-d$. In the case $S=BG$ for a topological group $G$, the dualizing spectrum has also been described by Klein, \cite{KleinDualizing}. In that case, one has the formula
\[
D_{BG}=(\Sigma_+^\infty G)^{hG}\ ,
\]
regarded as a spectrum with $G$-action via the remaining $G$-action from the other side. This is a spectral version of the $\Gamma$-module
\[
H^\ast(\Gamma,\Z[\Gamma])
\]
that appears in Bieri--Eckmann duality of discrete groups, \cite{BieriEckmann}. The norm map takes the form
\[
((\Sigma_+^\infty G)^{hG}\otimes X)_{hG}\to X^{hG}\ .
\]

If the functor $p_\ast: \Sp^S\to \Sp$ commutes with all colimits, for example if $S$ is a finite CW complex, then by Theorem~\ref{thm:genfarrelltate}~(v), the map $p_!(D_S\otimes X)\to p_\ast X$ is an equivalence for all $X\in \Sp^S$, as $p_\ast$ itself is colimit-preserving. This is a topological version of Bieri--Eckmann duality, \cite{BieriEckmann}, which states that for certain discrete groups $\Gamma$, there are functorial isomorphism
\[
H_{d-i}(\Gamma,D_\Gamma\otimes X)\cong H^i(\Gamma,X)
\]
for the $\Gamma$-module $D_\Gamma=H^d(\Gamma,\Z[\Gamma])$, and some $d$ depending on $\Gamma$. In general, the norm map encodes a ``generalized Poincar\'e duality on $S$'', and the cofiber $p_\ast^T$ of the norm map can be regarded as a ``failure of generalized Poincar\'e duality on $S$''. In the proof, we will ignore all set-theoretic issues. These can be resolved by passing to $\kappa$-compact objects as in the proof of Theorem~\ref{thm:tatelaxsymm}.

\begin{proof}[Proof of Theorem~\ref{thm:genfarrelltate}.] By \cite[5.5.3.6]{HTT} and \cite[Proposition 1.1.3.1]{HA}, $\Sp^S$ is presentable and stable. As $s^\ast$ commutes with all colimits (and in particular filtered colimits), it follows that $s_!$ preserves compact objects. Now a map $f: X\to Y$ of objects in $\Sp^S$ is an equivalence if and only if for all $s\in S$ and $i\in \Z$, the induced map $\pi_i s^\ast X\to \pi_i s^\ast Y$ is an isomorphism. But $\pi_i s^\ast X=\Hom(\bS[i],s^\ast X) = \Hom(s_! \bS[i],X)$, so that the objects $s_! \bS$ are generators.

Now part (ii) follows from Theorem~\ref{thm:verdier}~(iii), applied to $\calC=\Sp^S$ and $\calD=(\Sp^S)_\omega$ the compact objects in $\calC$. By Theorem~\ref{thm:verdier}~(iii), the functor $p_\ast^T$ can be computed as follows. It sends any $X\in \Sp^S$ to
\[
\colim_{Y\in (\Sp^S)_{\omega/X}} p_\ast\cofib(Y\to X)\ .
\]
For part (iii), we claim first that $\fib(p_\ast\to p_\ast^T)$ does indeed commute with all colimits. By definition all functors are exact, so we have to show that the 
fiber commutes with all filtered colimits. For this, assume that $X=\colim_i X_i$ is a filtered colimit. In that case, as all objects of $(\Sp^S)_\omega$ are compact,
\[
(\Sp^S)_{\omega/X} = \colim_i (\Sp^S)_{\omega/X_i}\ .
\]
Commuting colimits using \cite[Proposition 4.2.3.8, Corollary 4.2.3.10]{HTT}, one sees that
\[\begin{aligned}
\fib(p_\ast(X)\to p_\ast^T(X))&=\colim_{Y\in (\Sp^S)_{\omega/X}} p_\ast Y = \colim_i \colim_{Y\in (\Sp^S)_{\omega/X_i}} p_\ast Y\\
& = \colim_i \fib(p_\ast(X_i)\to p_\ast^T(X_i))\ ,
\end{aligned}\]
as desired. For the uniqueness, note that colimit-preserving functors from $\Sp^S$ to $\Sp$ are equivalent to exact functors from $(\Sp^S)_\omega$ to $\Sp$ by \cite[Proposition 5.5.1.9]{HTT}; but we are given an exact functor from $(\Sp^S)_\omega$ to $\Sp$, namely the restriction of $p_\ast$. The same argument proves (v).

For part (iv), note that colimit-preserving functors $\Sp^S\to \Sp$ are equivalent to colimit-preserving functors $\calS^S\to \Sp$ by \cite[Corollary 1.4.4.5]{HA}. These, in turn, are given by functors $S\to \Sp$ by \cite[Theorem 5.1.5.6]{HTT}, i.e.~objects of $\Sp^S$. Unraveling, any colimit-preserving functor $F: \Sp^S\to \Sp$ is given by $F(X)=p_!(D\otimes X)$ for a unique object $D\in \Sp^S$. To identify $D_S$ in our case, we look at the family of compact objects $F: S\to (\Sp^S)_\omega$ sending $s\in S$ to $s_! \bS$, i.e.~we consider $\Delta_! \bS\in \Fun(S\times S,\Sp)$ as a functor $S\to \Fun(S,\Sp)$ for the diagonal $\Delta: S\to S\times S$. The functor $p_\ast^T$ vanishes identically, so
\[
p_!(D_S\otimes F)\simeq p_\ast F
\]
as functors $S\to \Sp$. Unraveling definitions, the left side is just $D_S$, and the right side gives the desired formula.

Finally, in part (vi), note that if $p_\ast^T$ vanishes on all objects of the form $s_! X$, then we can use Theorem~\ref{thm:verdiermult}~(ii) to produce the desired lax symmetric monoidal structure as in the proof of Theorem~\ref{thm:tatelaxsymm}, by applying it to $\calC=\Sp^S$ and $\calD$ the full subcategory generated by $s_! X$ for all $s\in S$, $X\in \Sp$, which is a $\otimes$-ideal.
\end{proof}

\begin{corollary}\label{cor:tates1} Consider the $\infty$-category $\Sp^{B\T}$ of spectra with $\T$-action. There is a natural transformation $\Sigma (-_{h\T})\to -^{h\T}$ which exhibits $\Sigma (-_{h\T})$ as the universal colimit preserving functor mapping to the target. The cofiber $-^{t\T}$ admits a unique lax symmetric monoidal structure making $-^{h\T}\to -^{t\T}$ lax symmetric monoidal.

Moreover, for any $n\geq 1$, consider the functors $-^{hC_n}, -^{tC_n}: \Sp^{B\T}\to \Sp$ which factor over $\Sp^{BC_n}$. There is a unique (lax symmetric monoidal) natural transformation $-^{t\T}\to -^{tC_n}$ making the diagram
\[\xymatrix{
-^{h\T}\ar[r]\ar[d] & -^{t\T}\ar[d] \\
-^{hC_n}\ar[r] & -^{tC_n}
}\]
of (lax symmetric monoidal) functors commute.
\end{corollary}

\begin{proof} The first part follows from Theorem~\ref{thm:genfarrelltate} and \cite[Theorem 10.1, Corollary 10.2]{KleinDualizing}. For the second part, we first check the result without lax symmetric monoidal structures. Then, by the universal property of $-^{t\T}$, it suffices to show that $-^{tC_n}$ vanishes on the spectra $s_! \bS[i]$ where $s: \ast\to B\T$ is the inclusion of a point, and $i\in \Z$. For this, it suffices to show that $s_! \bS[i]$ is compact as an object of $\Sp^{BC_n}$, i.e.~$f^\ast: \Sp^{B\T}\to \Sp^{BC_n}$ preserves compact objects, where $f: BC_n\to B\T$ is the canonical map. For this, it suffices to show that $f_\ast$ commutes with all colimits. But the fibers of $f_\ast$ are given by cohomology of $\T/C_n$, which is a finite CW complex (a circle).

To get the result with lax symmetric monoidal structures, we need to check that $-^{tC_n}$ vanishes on all induced spectra $s_! X$ for $X\in \Sp$. Although one can give more abstract reasons, let us just check this by hand. First, note that $s_!$ commutes with all limits, as its fibers are given by homology of $\T$, which is a finite CW complex; in fact, $s_!$ decreases coconnectivity by at most $1$, and so also the composite $(s_!)_{hC_n}$ commutes with Postnikov limits as in Lemma~\ref{lem:tateconvergence}~(ii). Thus, by taking the limit over all $\tau_{\leq n} X$, we can assume that $X$ is bounded above, say $X$ is coconnective. Now all functors commute with filtered colimits, so one can assume that $X$ is bounded, and then that $X$ is an Eilenberg-MacLane spectrum, in fact $X=H\Z$. It remains to compute $-^{tC_n}$ on $H\Z[S^1]=s_! H\Z$. As the homotopy of $(H\Z[S^1])^{tC_n}$ is a module over the homotopy of $(H\Z)^{tC_n}$ which is a $2$-periodic ring, it suffices to see that $\pi_i (H\Z[S^1])_{hC_n}$ vanishes for $i>1$. But this is given by $\pi_i(H\Z[S^1/C_n])\cong \pi_i H\Z[S^1]$, which is nonzero only in degrees $0$ and $1$.
\end{proof}

We will need the following small computation. It also follows directly from the known computation of $\pi_\ast(H\Z)^{t\T}$ as the ring $\Z((X))$, but we argue directly with our slightly inexplicit definition of $(H\Z)^{t\T}$.

\begin{lemma}\label{lem:tates1cn} Consider $H\Z\in \Sp^{B\T}$ endowed with the trivial $\T$-action. The natural map $(H\Z)^{t\T}\to (H\Z)^{tC_n}$ induces isomorphisms
\[
\pi_i(H\Z)^{t\T}/n\simeq \pi_i(H\Z)^{tC_n}\ .
\]
\end{lemma}

\begin{proof} The result is classical on $\pi_i$ for $i\leq 0$, as it becomes a result about group cohomology. Note that by adjunction there is a natural map $H\Z[S^1]\to H\Z$, and $-^{t\T}$ and $-^{tC_n}$ vanish on $H\Z[S^1]$ (the latter by the proof of the previous corollary). On the other hand, the cone of $H\Z[S^1]\to H\Z$ is given by $H\Z[2]$ with necessarily trivial $\T$-action. Inductively, we can replace $H\Z$ by $H\Z[2i]$ for all $i\geq 0$, and then the result is clear in any given  degree by taking $i$ large enough.
\end{proof}

\chapter{Cyclotomic spectra}\label{ch:cyclotomic}

In this chapter, we define the $\infty$-category of cyclotomic spectra. More precisely, we define two variants: One is the classical definition in terms of genuine equivariant homotopy theory, as an orthogonal $\T$-spectrum $X$ together with commuting $\T/C_p\cong \T$-equivariant equivalences $\Phi_p: \Phi^{C_p} X\simeq X$. The other variant is $\infty$-categorical, and is a $\T$-equivariant object $X$ in the $\infty$-category $\Sp$, together with $\T\cong \T/C_p$-equivariant maps $\varphi_p: X\to X^{tC_p}$. The goal of this chapter is to prove that these notions are equivalent on bounded below objects.

First, we define the $\infty$-category of cyclotomic spectra in Section~\ref{sec:prelim}. Next, we review the genuine equivariant homotopy theory going into the classical definition of cyclotomic spectra and topological Hochschild homology in Sections~\ref{sec:genuine} and~\ref{sec:gencyclotomic}. In Section~\ref{sec:tccomp}, we give a direct proof of Corollary~\ref{cor:formulatc}. We note that a part of the proof, Corollary~\ref{cor:segalinduct}, establishes a generalization of Ravenel's observation, \cite{Ravenelsegalconj}, that the Segal conjecture for $C_p$ implies the Segal conjecture for $C_{p^n}$, cf.~also~\cite{TsalidisTHHdescent}, \cite{BBLNR}. Afterwards, in Sections~\ref{sec:coalgebras} and~\ref{sec:proofequivalence}, we prove Theorem~\ref{thm:intromain}.
 
\section{Cyclotomic spectra and TC}\label{sec:prelim}

In this section we give our new definition of cyclotomic spectra. The following definition gives the objects of an $\infty$-category that will be defined in Definition~\ref{def:cyclotomicinftycat}.

In the following we will use that for a spectrum $X$ with $\T$-action the spectrum $X^{tC_p}$ can be equipped with the residual $\T/{C_p}$-action which is identified with a $\T$-action using the $p$-th power map $\T/C_p\cong\T$.
Moreover $C_{p^\infty} \subseteq \T$ denotes the subgroup of elements of $p$-power torsion which is isomorphic to the Pr\"ufer group  $\Q_p/\Z_p$.\footnote{One has $C_{p^\infty}\cong \Q_p/\Z_p \cong \Q/\Z_{(p)}\cong \Z[\tfrac 1p]/\Z$, but the second author is very confused by the notation $\Z/p^\infty$ that is sometimes used, and suggests to change it to any of the previous alternatives, or to $p^{-\infty}\Z/\Z$.} Again we have a canonical identification $C_{p^\infty}/C_p\cong C_{p^\infty}$ given by the $p$-th power map.

\begin{definition}\label{def:cyclotomic}$ $
\begin{altenumerate}
\item[{\rm (i)}] A cyclotomic spectrum is a spectrum $X$ with $\T$-action together with $\T$-equivariant maps $\varphi_p: X \to X^{tC_p}$ for every prime $p$.
\item[{\rm (ii)}] For a fixed prime $p$, a $p$-cyclotomic spectrum is a spectrum $X$ with $C_{p^\infty}$-action and a $C_{p^\infty}$-equivariant map $\varphi_p: X \to X^{tC_p}$. \end{altenumerate}
\end{definition}

\begin{example} \label{examplesTHH}We give some examples of cyclotomic spectra.
\begin{altenumerate}
\item For every associative and unital ring spectrum $R\in \Alg_{\E_1}(\Sp)$, the topological Hoch\-schild homology $\THH(R)$ is a cyclotomic spectrum, cf.~Section~\ref{sec:thhnaive} below.
\item Consider the sphere spectrum $\bS$ equipped with the trivial $\T$-action. There are canonical maps $\varphi_p: \bS \to \bS^{tC_p}$ given as the composite $\bS\to \bS^{hC_p}\to \bS^{tC_p}$. These maps are $\T\cong \T/C_p$-equivariant: To make them equivariant, we need to lift the map $\bS\to \bS^{tC_p}$ to a map $\bS\to (\bS^{tC_p})^{h(\T/C_p)}$, as the $\T$-action on $\bS$ is trivial. But we have a natural map
\[
\bS \to \bS^{h\T} \simeq (\bS^{hC_p})^{h(\T/C_p)}\to (\bS^{tC_p})^{h(\T/C_p)}\ .
\]
This defines a cyclotomic spectrum which is in fact equivalent to $\THH(\bS)$. Note, cf.~Remark~\ref{remark:segal} below, that it is a consequence of the Segal conjecture that the maps $\varphi_p$ are $p$-completions. We refer to this cyclotomic spectrum as the cyclotomic sphere and denote it also by $\bS$.
\item For every cyclotomic spectrum we get a $p$-cyclotomic spectrum by restriction. In particular we can consider $\THH(R)$ and $\bS$ as $p$-cyclotomic spectra and do not distinguish these notationally.
\end{altenumerate}
\end{example}

\begin{remark}\label{remarktaction} One can define a slightly different notion of $p$-cyclotomic spectrum, as a spectrum with a $\T$-action and a $\T\cong \T/C_p$-equivariant map $\varphi_p: X \to X^{tC_p}$. This is what has been used (in a different language) in the literature before \cite{BlumbergMandell}. We however prefer to restrict the action to a $\Tp$-action since this is sufficient for the definition of $\TC(-,p)$ below and makes the construction more canonical; in particular, it is necessary for the interpretation of $\TC(-,p)$ as the mapping spectrum from $\bS$ in the $\infty$-category of $p$-cyclotomic spectra. 
The functor from the $\infty$-category of $p$-cyclotomic spectra with $\T$-action to the $\infty$-category of $p$-cyclotomic spectra as defined above is fully faithful when restricted to the subcategory of $p$-complete and bounded below objects (so that $X^{tC_p}$ is also $p$-complete by Lemma~\ref{lem:tatecomplete}). However, not every $\Tp$-action on a $p$-complete spectrum extends to a $\T$-action since a $\Tp$-action can act non-trivially on homotopy groups. 
\end{remark}

Let us now make the definition of cyclotomic spectra more precise by defining the relevant $\infty$-categories.

As in Definition~\ref{def:equivobject}, we denote the $\infty$-category of spectra equipped with a $\T$-action by $\Sp^{B\T} = \Fun(B\T,\Sp)$. Note that here, $\T$ is regarded as a topological group, and $B\T\simeq \C\mathrm{P}^\infty$ denotes the corresponding topological classifying space. We warn the reader again that this notion of $G$-equivariant spectrum is different from the notions usually considered in equivariant stable homotopy theory, and we discuss their relation in Section~\ref{sec:genuine} below.

Note that the $\infty$-category of cyclotomic spectra is the $\infty$-category of $X\in \Sp^{B\T}$ together with maps $X\to X^{tC_p}$. This is a special case of the following general definition.

\begin{definition}\label{def:laxeq} Let $\calC$ and $\calD$ be $\infty$-categories, and $F,G: \calC\to \calD$ functors. The \emph{lax equalizer}\footnote{In accordance to classical 2-category theory it would be more precise to call this construction an \emph{inserter}. But since we want to emphasize the relation to the equalizer we go with the current terminology. We thank Emily Riehl for a discussion of this point.} 
 of $F$ and $G$ is the $\infty$-category
\[
\mathrm{LEq}(F,G) = \mathrm{LEq}\left( \xymatrix{\calC \ar^F[r]<2pt> \ar_G[r]<-2pt> &  \calD} \right)
\]
defined as the pullback
\[\xymatrix{
\mathrm{LEq}(F,G) \ar[r] \ar[d] & \calD^{\Delta^1} \ar[d]^{(\ev_0,\ev_1)}\\
\calC \ar[r]^-{(F,G)} &  \calD \times \calD
}\]
of simplicial sets. In particular, objects of $\mathrm{LEq}(F,G)$ are given by pairs $(c,f)$ of an object $c\in \calC$ and a map $f: F(c)\to G(c)$ in $\calD$.
\end{definition}

\begin{proposition}\label{prop:laxeq} Consider a cartesian diagam
\[\xymatrix{
\mathrm{LEq}(F,G) \ar[r] \ar[d] & \calD^{\Delta^1} \ar[d]^{(\ev_0,\ev_1)}\\
\calC \ar[r]^-{(F,G)} &  \calD \times \calD
}\]
as in Definition~\ref{def:laxeq}.
\begin{altenumerate}
\item[{\rm (i)}] The pullback
\[\xymatrix{
\mathrm{LEq}(F,G) \ar[r] \ar[d] & \calD^{\Delta^1} \ar[d]^{(\ev_0,\ev_1)}\\
\calC \ar[r]^-{(F,G)} &  \calD \times \calD
}\]
is a homotopy Cartesian diagram of $\infty$-categories.
\item[{\rm (ii)}] Let $X,Y\in \mathrm{LEq}(F,G)$ be two objects, given by pairs $(c_X,f_X)$, $(c_Y,f_Y)$, where $c_X,c_Y\in \calC$, and $f_X: F(c_X)\to G(c_X)$, $f_Y: F(c_Y)\to G(c_Y)$ are maps in $\calD$. Then the space $\Map_{\mathrm{LEq}(F,G)}(X,Y)$ is given by the equalizer
\[
\Map_{\mathrm{LEq}(F,G)}(X,Y) \simeq \mathrm{Eq}\Big(\xymatrix{\Map_\calC(c_X,c_Y) \ar^-{(f_X)^* G}[rr]<2pt> \ar_-{(f_Y)_* F}[rr]<-2pt> && \Map_{\calD}(F(c_X),G(c_Y))}\Big)\ .
\]
Moreover, a map $f: X\to Y$ in $\mathrm{LEq}(F,G)$ is an equivalence if and only if its image in $\calC$ is an equivalence.
\item[{\rm (iii)}] If $\calC$ and $\calD$ are stable $\infty$-categories and $F$ and $G$ are exact, then $\mathrm{LEq}(F,G)$ is a stable $\infty$-category, and the functor $\mathrm{LEq}(F,G)\to \calC$ is exact.
\item[{\rm (iv)}] If $\calC$ is presentable, $\calD$ is accessible, $F$ is colimit-preserving and $G$ is accessible, then $\mathrm{LEq}(F,G)$ is presentable, and the functor $\mathrm{LEq}(F,G)\to \calC$ is colimit-preserving.
\item[{\rm (v)}]  Assume that $p: K \to \mathrm{LEq}(F,G)$ is a diagram such that the composite diagram $K \to \mathrm{LEq}(F,G) \to \calC$ admits a limit and this limit is preserved by the functor $G: \calC \to \calD$. Then $p$ admits a limit and the functor $\mathrm{LEq}(F,G) \to \calC$ preserves this limit. 
\end{altenumerate}
\end{proposition}

In part (ii), we denote for any $\infty$-category $\calC$ by
\[
\Map_\calC: \calC^\op\times \calC\to \calS
\]
the functor defined in \cite[Section 5.1.3]{HTT}. It is functorial in both variables, but its definition is rather involved. There are more explicit models for the mapping spaces with less functoriality, for example the Kan complex $\Hom^R_\calC(X,Y)$ with $n$-simplices given by the set of $n+1$-simplices $\Delta^{n+1}\to \calC$ whose restriction to $\Delta^{\{0,\ldots,n\}}$ is constant at $X$, and which send the vertex $\Delta^{\{n+1\}}$ to $Y$, cf.~\cite[Section 1.2.2]{HTT}.

In part (iv), recall that a functor between accessible $\infty$-categories is accessible if it commutes with small $\kappa$-filtered colimits for some (large enough) $\kappa$. For the definition of accessible and presentable $\infty$-categories, we refer to \cite[Definition 5.4.2.1, Definition 5.5.0.1]{HTT}.

\begin{proof} For part (i), it is enough to show that the functor $(\ev_0,\ev_1): \calD^{\Delta^1}\to \calD\times \calD$ of $\infty$-categories is a categorical fibration, which follows from \cite[Corollary 2.3.2.5, Corollary 2.4.6.5]{HTT}.

For part (ii), note that it follows from the definition of $\Hom^R$ that there is a pullback diagram of simplicial sets
\[\xymatrix{
\Hom^R_{\mathrm{LEq}(F,G)}(X,Y)\ar[r]\ar[d] & \Hom^R_{\calD^{\Delta^1}}(f_X,f_Y)\ar[d]\\
\Hom^R_\calC(c_X,c_Y)\ar[r] & \Hom^R_\calD(F(c_X),F(c_Y))\times \Hom^R_\calD(G(c_X),G(c_Y))\ .
}\]
This is a homotopy pullback in the Quillen model structure, as the map
\[
\Hom^R_{\calD^{\Delta^1}}(f_X,f_Y)\to \Hom^R_\calD(F(c_X),F(c_Y))\times \Hom^R_\calD(G(c_X),G(c_Y))
\]
is a Kan fibration by \cite[Corollary 2.3.2.5, Lemma 2.4.4.1]{HTT}. Therefore, the diagram
\[\xymatrix{
\Map_{\mathrm{LEq}(F,G)}(X,Y)\ar[r]\ar[d] & \Map_{\calD^{\Delta^1}}(f_X,f_Y)\ar[d]\\
\Map_\calC(c_X,c_Y)\ar[r] & \Map_\calD(F(c_X),F(c_Y))\times \Map_\calD(G(c_X),G(c_Y))
}\]
is a pullback in $\calS$. To finish the identification of mapping spaces in part (ii), it suffices to observe that the right vertical map is the equalizer of
\[\xymatrix{
\Map_\calD(F(c_X),F(c_Y))\times \Map_\calD(G(c_X),G(c_Y))\ar[r]<2pt> \ar[r]<-2pt> & \Map_{\calD}(F(c_X),G(c_Y))\ ,
}\]
which follows by unraveling the definitions. For the final sentence of (ii), note that if $f: X\to Y$ is a map in $\mathrm{LEq}(F,G)$ which becomes an equivalence in $\calC$, then by the formula for the mapping spaces, one sees that
\[
f^\ast: \Map_{\mathrm{LEq}(F,G)}(Y,Z)\to \Map_{\mathrm{LEq}(F,G)}(X,Z)
\]
is an equivalence for all $Z\in \mathrm{LEq}(F,G)$; thus, by the Yoneda lemma, $f$ is an equivalence.

For part (iii), note first that $\calD\times \calD$ and $\calD^{\Delta^1}$ are again stable by \cite[Proposition 1.1.3.1]{HA}. Now the pullback of a diagram of stable $\infty$-categories along exact functors is again stable by \cite[Proposition 1.1.4.2]{HA}.

For part (iv), note first that by \cite[Proposition 5.4.4.3]{HTT}, $\calD\times \calD$ and $\calD^{\Delta^1}$ are again accessible. By \cite[Proposition 5.4.6.6]{HTT}, $\mathrm{LEq}(F,G)$ is accessible. It remains to see that $\mathrm{LEq}(F,G)$ admits all small colimits, and that $\mathrm{LEq}(F,G)\to \calC$ preserves all small colimits.

Let $K$ be some small simplicial set with a map $p: K\to \mathrm{LEq}(F,G)$. First, we check that an extension $p^\triangleright: K^\triangleright\to \mathrm{LEq}(F,G)$ is a colimit of $p$ if the composite $K^\triangleright\to \mathrm{LEq}(F,G)\to \calC$ is a colimit of $K\to \mathrm{LEq}(F,G)\to \calC$. Indeed, this follows from the Yoneda characterization of colimits \cite[Lemma 4.2.4.3]{HTT}, the description of mapping spaces in part (ii), and the assumption that $F$ preserves all small colimits. Thus, it remains to see that we can always extend $p$ to a map $p^\triangleright: K^\triangleright\to \mathrm{LEq}(F,G)$ whose image in $\calC$ is a colimit diagram.

By assumption, the composite $K\to \mathrm{LEq}(F,G)\to \calC$ can be extended to a colimit diagram $a: K^\triangleright\to \calC$. Moreover, $F(a): K^\triangleright\to \calD$ is still a colimit diagram, while $G(a): K^\triangleright\to \calD$ may not be a colimit diagram. However, we have the map $K\to \mathrm{LEq}(F,G)\to \calD^{\Delta^1}$, which is given by a map $K\times \Delta^1\to \calD$ with restriction to $K\times \{0\}$ equal to $F(a)|_K$, and restriction to $K\times \{1\}$ equal to $G(a)|_K$. By the universal property of the colimit $F(a)$, this can be extended to a map $D: K^\triangleright\times \Delta^1\to \calD$ with restriction to $K^\triangleright\times \{0\}$ given by $F(a)$, and restriction to $K^\triangleright\times \{1\}$ given by $G(a)$. Thus, $(a,D)$ defines a map
\[
p^\triangleright: K^\triangleright\to \calC\times_{\calD\times \calD} \calD^{\Delta^1} = \mathrm{LEq}(F,G)
\]
whose image in $\calC$ is a colimit diagram, as desired.

For part (v), use the dual argument to the one given in part (iv).
\end{proof}

Now we can define the $\infty$-categories of ($p$-)cyclotomic spectra. Recall from Definition~\ref{def:tate} the functors
\[
-^{tC_p}: \Sp^{B\T}\to \Sp^{B(\T/C_p)}\simeq \Sp^{B\T}
\]
and
\[
-^{tC_p}: \Sp^{B\Tp}\to \Sp^{B(\Tp/C_p)}\simeq \Sp^{B\Tp}\ .
\]

\begin{definition}\label{def:cyclotomicinftycat}{\ }\begin{altenumerate}
\item The $\infty$-category of cyclotomic spectra is the lax equalizer
\[
\Cycn := \mathrm{LEq}\left( \xymatrix{\Sp^{B\T} \ar[r]<2pt> \ar[r]<-2pt> &  \prod_{p\in \bP} \Sp^{B\T}} \right)
\]
where the two functors have $p$-th components given by the functors $\id: \Sp^{B\T} \to \Sp^{B\T}$ and $-^{tC_p}: \Sp^{B\T} \to \Sp^{B(\T/C_p)}\simeq \Sp^{B\T}$.
\item The $\infty$-category of $p$-cyclotomic spectra is the lax equalizer
 \[
\Cycn_p := \mathrm{LEq}\left( \xymatrix{\Sp^{BC_{p^\infty}} \ar[r]<2pt> \ar[r]<-2pt> & \Sp^{BC_{p^\infty}}} \right)
\]
of the functors $\id: \Sp^{B\Tp} \to \Sp^{B\Tp}$ and $-^{tC_p}: \Sp^{B\Tp} \to \Sp^{B(\Tp/C_p)}\simeq \Sp^{B\Tp}$. 
\end{altenumerate}
\end{definition}

\begin{corollary} The $\infty$-categories $\Cycn$ and $\Cycn_p$ are presentable stable $\infty$-categories. The forgetful functors $\Cycn\to \Sp$, $\Cycn_p\to \Sp$ reflect equivalences, are exact, and preserve all small colimits.
\end{corollary}

\begin{proof} The Tate spectrum functor $X\mapsto X^{tC_p}$ is accessible as it is the cofiber of functors which admit adjoints, cf.~\cite[Proposition 5.4.7.7]{HTT}. Moreover, the forgetful functor $\Sp^{BG}\to \Sp$ preserves all small colimits and reflects equivalences by \cite[Corollary 5.1.2.3]{HTT}; in particular, it is exact. Now all statements follow from Proposition~\ref{prop:laxeq}.
\end{proof}
 
For every pair of objects in a stable $\infty$-category $\calC$, the mapping space refines to a mapping spectrum. In fact, by \cite[Corollary 1.4.2.23]{HA}, for every $X\in \calC^\op$, the left-exact functor
\[
\Map(X,-): \calC\to \calS
\]
from the stable $\infty$-category $\calC$ lifts uniquely to an exact functor
\[
\map(X,-): \calC\to \Sp\ .
\]
Varying $X$, i.e.~by looking at the functor
\[
\calC^\op\to \Fun^\Lex(\calC,\calS)\simeq \Fun^\Ex(\calC,\Sp)\ ,
\]
we get a functor $\map: \calC^\op\times \calC\to \Sp$.

\begin{definition}\label{def:naivetc}$ $\begin{altenumerate}
\item Let $\big(X,(\varphi_p)_{p\in \bP}\big)$ be a cyclotomic spectrum. The integral topological cyclic homology $\TC(X)$ of $X$ is the mapping spectrum $\map_{\Cycn}(\bS, X)\in \Sp$.

\item Let $(X,\varphi_p)$ be a $p$-cyclotomic spectrum. The $p$-typical topological cyclic homology $\TC(X,p)$ is the mapping spectrum $\map_{\Cycn_p}(\bS, X)$.

\item Let $R\in \Alg_{\E_1}(\Sp)$ be an associative ring spectrum. Then $\TC(R) := \TC(\THH(R))$ and $\TC(R,p) := \TC(\THH(R),p)$.
\end{altenumerate}
\end{definition}

In fact, by Proposition~\ref{prop:laxeq}~(ii), it is easy to compute $\TC(X)$ and $\TC(X,p)$.

\begin{proposition}\label{prop:formula}$ $
\begin{altenumerate}
\item Let $\big(X,(\varphi_p)_{p\in \bP}\big)$ be a cyclotomic spectrum. There is a functorial fiber sequence
\[
\TC(X) \to  X^{h\T} \xto{(\varphi_p^{h\T} - \can)_{p \in \bP}} \prod_{p \in \bP}(X^{tC_p})^{h\T}\ ,
\]
where the maps are given by
\[
\varphi_p^{h\T}: X^{h\T}\to (X^{tC_p})^{h\T}
\]
and
\[
\can: X^{h\T} \simeq (X^{hC_p})^{h(\T/C_p)}\simeq (X^{hC_p})^{h\T}\to (X^{tC_p})^{h\T}\ ,
\]
where the middle equivalence comes from the $p$-th power map $\T/C_p\cong \T$.

\item Let $(X,\varphi_p)$ be a $p$-cyclotomic spectrum. There is a functorial fiber sequence 
\[
\TC(X,p) \to X^{h\Tp} \xto{\varphi_p^{h\Tp} - \can} (X^{tC_p})^{h\Tp}\ ,
\]
with notation as in part (i).
\end{altenumerate}
\end{proposition}

\begin{proof} Note that the fiber of a difference map is equivalent to the equalizer of the two maps. By the equivalence $\Fun^\Ex(\Cycn,\Sp)\simeq \Fun^\Lex(\Cycn,\calS)$ (respectively for $\Cycn_p$) via composition with $\Omega^\infty$, it suffices to check the formulas for the mapping space. Thus, the result follows from Proposition~\ref{prop:laxeq}~(ii).
\end{proof}

\section{Equivariant stable homotopy theory}\label{sec:genuine}

In this section we introduce the necessary preliminaries to talk about genuine cyclotomic spectra and the classical definition of $\TC$. For this, we need to recall the definition of genuine equivariant spectra. There are many references, including \cite{MR866482}, \cite{MandellMay}, \cite{HillHopkinsRavenel}. We will refer to the discussion given by Schwede, \cite{SchwedeGenuine}, which has the advantage that the objects considered are the most concrete. Another reference for this model and the results is \cite[Chapter III]{Global}.

Until recently, all discussions of genuine equivariant stable homotopy theory were relying on explicit point-set models of spectra; now, an $\infty$-categorical foundation has been laid out by Barwick, \cite{BarwickMackey}, based on a result of Guillou--May, \cite{GuillouMay}. However, we prefer to stick here with the traditional explicit point-set models. In particular, Schwede uses an explicit symmetric monoidal $1$-category of spectra, the category of orthogonal spectra. In the following, all topological spaces are compactly generated weak Hausdorff spaces.

\begin{definition}[{\cite[Definition 1.1, Section 1]{SchwedeGenuine}}]\label{def:orthspectrum}$ $\begin{altenumerate}
\item[{\rm (1)}] An orthogonal spectrum $X$ is the collection of\vspace{0.05in}
\begin{altenumerate}
\item a pointed space $X_n$ for all $n\geq 0$,
\item a continuous action of the orthogonal group $O(n)$ on $X_n$, preserving the base point, and
\item base-point preserving maps $\sigma_n: X_n\wedge S^1\to X_{n+1}$,\vspace{0.05in}
\end{altenumerate}
subject to the condition that for all $n,m\geq 0$, the iterated structure maps
\[
\sigma_{n+m-1}\circ (\sigma_{n+m-2}\wedge S^1)\circ\ldots \circ (\sigma_n\wedge S^{m-1}): X_n\wedge S^m\to X_{n+m}
\]
are $O(n)\times O(m)\subseteq O(n+m)$-equivariant.\footnote{There is a more conceptual definition of orthogonal spectra as pointed, continuous functors $\mathbf{O} \to \mathrm{Top}_*$ on a topologically enriched category $\mathbf{O}$ whose objects are finite dimensional inner product spaces $V,W$ and whose mapping spaces are the Thom spaces $\mathrm{Th}(\xi)$ of the orthogonal complement bundle $\xi \to L(V,W)$ where $L(V,W)$ is the space of linear isometric embeddings. The space $\mathrm{Th}(\xi)$ is also homeomorphic to a certain subspace of the space of pointed maps $S^V \to S^W$ where $S^V$ and $S^W$ are the one-point compactifications of $V$ and $W$.}

\item[{\rm (2)}] Let $X$ be an orthogonal spectrum and $i\in \Z$. The $i$-th stable homotopy group of $X$ is
\[
\pi_i X = \colim_\sigma \pi_{i+n}(X_n)\ .
\]
\item[{\rm (3)}] A map of orthogonal spectra $f: X\to Y$ is a stable equivalence if for all $i\in \Z$, the induced map $\pi_i f: \pi_i X\to \pi_i Y$ is an isomorphism.
\end{altenumerate}
\end{definition}

Orthogonal spectra are naturally organized into a $1$-category, which we denote $\Sp^O$. An example of an orthogonal spectrum is given by the sphere spectrum $\bS\in \Sp^O$ with $\bS_n = S^n$, where $S^n$ is the one-point compactification of $\mathbb R^n$ (with its standard inner product), pointed at $\infty$, and with its natural $O(n)$-action. The category $\Sp^O$ admits a natural symmetric monoidal structure called the smash product; we refer to \cite[Section 1]{SchwedeGenuine} for the definition. The sphere spectrum is a unit object in this symmetric monoidal structure. Note that for an orthogonal spectrum the maps $X_n\to \Omega X_{n+1}$ adjunct to $\sigma_n$ need not be equivalences, so they would sometimes only be called ``prespectra''. Here, we call orthogonal spectra for which these maps are weak equivalences orthogonal $\Omega$-spectra. It is known that any orthogonal spectrum is stably equivalent to an orthogonal $\Omega$-spectrum. This is actually a special property of orthogonal spectra, and fails for symmetric spectra, where one has to be more careful with the naive definition of stable homotopy groups given above.

Moreover, cf.~\cite{MMSS}, if one inverts the stable equivalences in $N(\Sp^O)$, the resulting $\infty$-category is equivalent to the $\infty$-category $\Sp$ of spectra; under this equivalence the stable homotopy groups of an orthogonal spectrum as defined above correspond to the homotopy groups of the associated spectrum. On general orthogonal spectra, the functor $N(\Sp^O)\to \Sp$ is \emph{not} compatible with the symmetric monoidal structure: The issue is that if $f: X\to Y$ and $f^\prime: X^\prime\to Y^\prime$ are stable equivalences of orthogonal spectra, then in general $f\wedge f^\prime: X\wedge X^\prime\to Y\wedge Y^\prime$ is not a stable equivalence. However, this is true if $X,Y,X^\prime$ and $Y^\prime$ are cofibrant, and the lax symmetric monoidal functor $N(\Sp^O)\to \Sp$ is symmetric monoidal when restricted to the full subcategory of cofibrant orthogonal spectra. We refer to Appendix \ref{app:symmmon}, in particular Theorem~\ref{DKlocalizationsymmon}, for a general discussion of the relation between symmetric monoidal structures on a model category, and the associated $\infty$-category. In particular, Theorem~\ref{DKlocalizationsymmon} shows that the functor $N(\Sp^O)\to \Sp$ is lax symmetric monoidal (without restricting to cofibrant objects).

Let $G$ be a finite group. It is now easy to define a symmetric monoidal $1$-category of orthogonal $G$-spectra. The definition seems different from classical definitions which make reference to deloopings for all representation spheres; the equivalence of the resulting notions is discussed in detail in \cite[Remark 2.7]{SchwedeGenuine}.

\begin{definition}[{\cite[Definition 2.1]{SchwedeGenuine}}] An orthogonal $G$-spectrum is an orthogonal spectrum $X$ with an action of $G$, i.e.~the category $G\Sp^O$ of orthogonal $G$-spectra is given by $\Fun(BG,\Sp^O)$.
\end{definition}

The resulting category inherits a natural symmetric monoidal structure which is given by the smash product of the underlying orthogonal spectra with diagonal $G$-action. However, it is more subtle to define the relevant notion of weak equivalence; namely, the notion of equivalence used in genuine equivariant stable homotopy theory is significantly stronger than a morphism in $G\Sp^O$ inducing a stable equivalence of the underlying orthogonal spectra. Roughly, it asks that certain set-theoretic fixed points are preserved by weak equivalences.

To define the relevant notion of weak equivalence, we make use of a functor from orthogonal $G$-spectra to orthogonal spectra known as the ``geometric fixed points''.

\begin{definition}\label{def:genuineequiv}$ $\begin{altenumerate}
\item Let $X$ be an orthogonal $G$-spectrum, and let $V$ be a representation of $G$ on an $n$-dimensional inner product space over $\mathbb R$. We define a new based space
\[
X(V) = \mathcal L(\mathbb R^n,V)_+\wedge_{O(n)} X_n\ ,
\]
cf.~\cite[Equation (2.2)]{SchwedeGenuine}, which comes equipped with a $G$-action through the diagonal $G$-action. Here, $\mathcal{L}(\mathbb R^n,V)$ denotes the linear isometries between $V$ and $\mathbb R^n$, which has a natural action by $O(n)$ via precomposition making it an $O(n)$-torsor.
\item Let $\rho_G$ denote the regular representation of $G$. Define an orthogonal spectrum $\Phi^G X$, called the geometric fixed points of $X$, whose $n$-th term is given by
\[
(\Phi^G X)_n = X(\mathbb R^n\otimes \rho_G)^G\ ,
\]
cf.~\cite[Section 7.3]{SchwedeGenuine}, where the fixed points are the set-theoretic fixed points.
\item Let $f: X\to Y$ be a map of orthogonal $G$-spectra. Then $f$ is an equivalence if for all subgroups $H\subseteq G$, the map $\Phi^H X\to \Phi^H Y$ is a stable equivalence of orthogonal spectra.
\end{altenumerate}
\end{definition}

By \cite[Theorem 7.11]{SchwedeGenuine}, the definition of equivalence in part (iii) agrees with the notion of $\underline{\pi}_\ast$-isomorphism used in \cite{SchwedeGenuine}.

\begin{proposition}[{\cite[Proposition 7.14]{SchwedeGenuine}}] The functor $\Phi^G: G\Sp^O\to \Sp^O$ has a natural lax symmetric monoidal structure. When $X$ and $Y$ are cofibrant $G$-spectra, the map $\Phi^G X\wedge \Phi^G Y\to \Phi^G(X\wedge Y)$ is a stable equivalence.
\end{proposition}

Now we pass to the corresponding $\infty$-categories. 

\begin{definition} Let $G$ be a finite group.
\begin{altenumerate}
\item The $\infty$-category of genuine $G$-equivariant spectra is the $\infty$-category $G\Sp$ obtained from $N(G\Sp^O)$ by inverting equivalences of orthogonal $G$-spectra. It inherits a natural symmetric monoidal structure compatible with the smash product of cofibrant orthogonal $G$-spectra.
\item Assume that $H$ is a subgroup of $G$. The geometric fixed point functor
\[
\Phi^H: G\Sp\to \Sp
\]
is the symmetric monoidal functor obtained from $\Phi^H: G\Sp^O\to \Sp^O$ by restricting to cofibrant orthogonal $G$-spectra, and inverting equivalences of orthogonal $G$-spectra.
\end{altenumerate}
\end{definition}

We note that it is very important to pay attention to the notational difference between the 1-category $G\Sp^O$ and the $\infty$-category $G\Sp$. Since our arguments and results will mix the two worlds this can otherwise easily lead to confusion. Eventually we will only be interested in the $\infty$-category $G\Sp$ but the most convenient way to construct functors or give proofs is often to work in the model. 

There is another fixed point functor for orthogonal $G$-spectra, taking an orthogonal $G$-spectrum $X$ to the orthogonal spectrum whose $n$-th term is given by the set-theoretic fixed points $X_n^G$. Unfortunately, if $f: X\to Y$ is an equivalence of orthogonal $G$-spectra, the induced map $(X_n^G)_n\to (Y_n^G)_n$ of orthogonal spectra need not be a stable equivalence. However, this is true for orthogonal $G$-$\Omega$-spectra, cf.~\cite[Section 7.1]{SchwedeGenuine} and \cite[Definition 3.18]{SchwedeGenuine} for the definition of orthogonal $G$-$\Omega$-spectra which refers to the deloopings of Definition~\ref{def:genuineequiv}~(i) for general representations of $V$. Thus, one can derive the functor, and get a functor of $\infty$-categories
\[
-^H: G\Sp\to \Sp
\]
for all subgroups $H$ of $G$, usually simply called the fixed point functor. To better linguistically distinguish it from other fixed point functors, we refer to it as the genuine fixed point functor. As the trivial representation embeds naturally into $\rho_G$, there is a natural transformation $-^H\to \Phi^H$ of functors $G\Sp\to \Sp$. Moreover, all functors and natural transformations in sight are lax symmetric monoidal, cf.~\cite[Proposition 7.13, Proposition 7.14]{SchwedeGenuine}.

\begin{remark} It was proved by Guillou--May, \cite{GuillouMay}, that one can describe $G\Sp$ equivalently in terms of the data of the spectra $X^H$ for all subgroups $H\subseteq G$, equipped with the structure of a spectral Mackey functor; i.e., roughly, for any inclusion $H^\prime\subseteq H$, one has a restriction map $X^H\to X^{H^\prime}$ and a norm map $X^{H^\prime}\to X^H$, and if $g\in G$, there is a conjugation equivalence $X^H\simeq X^{g^{-1}Hg}$, satisfying many compatibilities. From this point of view, one can define the $\infty$-category $G\Sp$ without reference to a point-set model for the category of spectra, as done by Barwick, \cite{BarwickMackey}. The advantage of this point of view is that it makes it easier to generalize: For example one can replace $\Sp$ by other stable $\infty$-categories, or the group by a profinite group.
\end{remark}

Note that any equivalence of orthogonal $G$-spectra is in particular an equivalence of the underlying orthogonal spectra (as $\Phi^{\{e\}} X$ is the underlying orthogonal spectrum). It follows that there is a natural functor
\[
G\Sp\to \Sp^{BG}
\]
from genuine $G$-equivariant spectra to spectra with $G$-action. In particular, there is yet another fixed point functor, namely the homotopy fixed points
\[
-^{hH}: G\Sp\to \Sp^{BG}\to \Sp\ .
\]
As the set-theoretic fixed points map naturally to the homotopy fixed points, one has a natural transformation of lax symmetric monoidal functors $-^H\to -^{hH}$ from $G\Sp$ to $\Sp$. However, this map does not factor over the geometric fixed points, and in fact the geometric fixed points have no direct relation to the homotopy fixed points.

We need the following result about the functor $G\Sp\to \Sp^{BG}$. Recall that we are now in an $\infty$-categorical setting and the next results are purely formulated in terms of these $\infty$-categories. The proofs will of course use their presentation through relative categories such as $\Sp^O$.

\begin{thm}\label{thm:borelcompletion} The functor $G\Sp\to \Sp^{BG}$ admits a fully faithful right adjoint $B_G: \Sp^{BG}\to G\Sp$. The essential image of $B_G$ is the full subcategory of all $X\in G\Sp$ for which the natural map $X^H\to X^{hH}$ is an equivalence for all subgroups $H\subseteq G$; we refer to these objects as Borel-complete.
\end{thm}

In other words, the $\infty$-category $\Sp^{BG}$ of spectra with $G$-action can be regarded as a full subcategory $G\Sp_B$ of the $\infty$-category $G\Sp$ consisting of the Borel-complete genuine $G$-equivariant spectra. Under the equivalence $G\Sp_B\simeq \Sp^{BG}$, the functors $-^H$ and $-^{hH}$ match for all subgroups $H$ of $G$.

\begin{proof} Fix a (compactly generated weak Hausdorff) contractible space $EG$ with free $G$-action. Consider the functor of 1-categories $L: G\Sp^O\to G\Sp^O$ sending an orthogonal $G$-spectrum $X=(X_n)_n$ to $\Map(EG,X_n)_n$, where $\Map(EG,X_n)$ is pointed at the map sending $EG$ to the base point of $X_n$, cf.~\cite[Example 5.2]{SchwedeGenuine}. As explained there, this defines a functor $G\Sp^O\to G\Sp^O$, and if $X$ is an orthogonal $G$-$\Omega$-spectrum, then $L(X)$ is again an orthogonal $G$-$\Omega$-spectrum. It follows from the usual definition
\[
X^{hH} = \Map(EH,X)^H
\]
of the homotopy fixed points (and the possibility of choosing $EH=EG$ for all subgroups $H$ of $G$) that
\[
L(X)^H = X^{hH} = L(X)^{hH}
\]
for all subgroups $H$ of $G$. In particular, using \cite[Theorem 7.11]{SchwedeGenuine}, if $f: X\to Y$ is a map of orthogonal $G$-$\Omega$-spectra which induces a stable equivalence of the underlying orthogonal spectra, for example if $f$ is an equivalence of orthogonal $G$-spectra, then $L(f)$ is an equivalence of orthogonal $G$-spectra. Moreover, $L(X)$ is always Borel-complete.

There is a natural transformation $\mathrm{id}\to L$ of functors $G\Sp^O\to G\Sp^O$, as there are natural maps $X_n\to \Map(EG,X_n)$ (compatibly with all structure maps) given by the constant maps. After restricting to orthogonal $G$-$\Omega$-spectra and inverting weak equivalences, we get a functor of $\infty$-categories $L: G\Sp\to G\Sp$ with a natural transformation $\mathrm{id}\to L$, which satisfies the criterion of \cite[Proposition 5.2.7.4]{HTT}, as follows from the previous discussion. Moreover, the image can be characterized as the full subcategory $G\Sp_B\subseteq G\Sp$ of Borel-complete objects.

As $G\Sp\to \Sp^{BG}$ factors over $L$, it remains to prove that the functor $G\Sp_B\to \Sp^{BG}$ is an equivalence of $\infty$-categories. Note that the functor $N(G\Sp^O)\xto{L} G\Sp_B$ inverts all morphisms which are equivalences of the underlying spectrum; it follows from \cite[Theorem 9.2]{MMSS} and \cite[Proposition 1.3.4.25]{HA} that inverting these in $N(G\Sp^O)$ gives $\Sp^{BG}$, so one gets a natural functor $B_G: \Sp^{BG}\to G\Sp_B$. It follows from the construction that $B_G$ and $G\Sp_B\to \Sp^{BG}$ are inverse equivalences.
\end{proof}

\begin{corollary}\label{cor:borelcompletionlaxsymm} There is a natural lax symmetric monoidal structure on the right adjoint $B_G: \Sp^{BG}\to G\Sp$, and a natural refinement of the adjunction map $\mathrm{id}\to B_G$ of functors $G\Sp\to G\Sp$ to a lax symmetric monoidal transformation.
\end{corollary}

\begin{proof} It follows from \cite[Corollary 7.3.2.7]{HA} that a right adjoint to a symmetric monoidal functor is naturally lax symmetric monoidal, with the adjunction being a lax symmetric monoidal transformation. In the relevant case of a localization, this also follows from \cite[Proposition 2.2.1.9]{HA}.
\end{proof}

To go on, we need to define a residual $G/H$-action on $X^H$ and $\Phi^H X$ in case $H$ is normal in $G$. In the case of $X^H$, this is easy: Restricted to orthogonal $G$-$\Omega$-spectra, the functor $G\Sp^O\to \Sp^O$ is given by $(X_n)_n\mapsto (X_n^H)_n$, which has an evident $G/H$-action. This functor takes orthogonal $G$-$\Omega$-spectra to orthogonal $G/H$-$\Omega$-spectra, and is compatible with composition, i.e.~$(X^H)^{G/H} = X^G$. In particular, it maps equivalences to equivalences, and we get an induced (lax symmetric monoidal) functor $G\Sp\to (G/H)\Sp$ of $\infty$-categories.

The situation is more subtle in the case of $\Phi^H X$. The issue is that the $n$-th space $X(\mathbb R^n\otimes \rho_H)^H$ has no evident $G/H$-action, as $G$ does not act on $\rho_H$. To repair this, one effectively needs to replace $\rho_H$ by a representation on which $G$-acts, and in fact in the definition of geometric fixed points, it only really matters that $\rho_H^H= \mathbb R$, and that $\rho_H$ contains all representations of $H$. This replacement is best done by choosing a complete $G$-universe in the sense of the following definition.

\begin{definition} A complete $G$-universe is a representation $\mathcal U$ of $G$ on a countably dimensional inner product $\mathbb R$-vector space that is a direct sum of countably many copies of each irreducible representation of $G$.
\end{definition}

Note that if $\mathcal U$ is a complete $G$-universe and $H\subseteq G$ is a subgroup, then it is also a complete $H$-universe. Moreover, if $H$ is normal in $G$, then the subrepresentation $\mathcal U^H$ of $H$-fixed vectors is a complete $G/H$-universe. Thus, if one has fixed a complete $G$-universe, one gets corresponding compatible complete $G^\prime$-universes for all subquotients of $G$, and in the following we will always fix a complete $G$-universe for the biggest group $G$ around, and endow all other groups with their induced complete universes.
In the following for a given universe $\mathcal U$ we write $V \in \mathcal U$ if $V$ is a finite dimensional $G$-subrepresentation of $\mathcal U$. These naturally form a poset.

\begin{definition}\label{def:geomfix} Let $G$ be a finite group with a complete $G$-universe $\mathcal U$, and let $H\subseteq G$ be a normal subgroup. The geometric fixed point functor
\[
\Phi^H_{\mathcal U} : G\Sp^O\to (G/H)\Sp^O
\]
is given by sending $X\in G\Sp^O$ to the orthogonal $G/H$-spectrum $\Phi^H_\mathcal U X$ whose $n$-th term is given by
\[
\hocolim_{V\in \mathcal U, V^H=0} X(\mathbb R^n\oplus V)^H\ ,
\]
where $O(n)$ acts on $\mathbb R^n$, and the $G$-action factors over a $G/H$-action. The structure maps are the evident maps.
\end{definition}

For this definition we use the explicit model given by the Bousfield-Kan formula for the homotopy colimit taken in compactly generated weak Hausdorff topological spaces. We review the Bousfield-Kan formula in Appendix \ref{app:colim}. 

Now we compare this model of the geometric fixed points to the definition $\Phi^H$.

\begin{lemma}
For every $G$-spectrum $X$ there is a zig-zag of stable equivalences of orthogonal spectra $\Phi^G X \to \widetilde \Phi^G_{\mathcal{U}}X \leftarrow \Phi^G_{\mathcal{U}} X$ which is natural in $X$.\footnote{We thank Stefan Schwede for a helpful discussion of this point.}
\end{lemma}
\begin{proof}
We first define
$$
\widetilde \Phi^G_{\mathcal U} X := \hocolim_{V\in \mathcal U, V^H=0} X(\R^n \otimes \rho_G \oplus V)^G.
$$
Now there  is a map
$$
\Phi^G X = X(\R^n \otimes \rho_G \oplus 0)^G \to  \hocolim_{V\in \mathcal U, V^H=0}X(\R^n \otimes \rho_G \oplus V)^G
$$
induced from the map into the homotopy colimit and another map
$$
 \Phi^G_{\mathcal{U}} X = \hocolim_{V\in \mathcal U, V^H=0}X(\R^n  \oplus V)^G \to  \hocolim_{V\in \mathcal U, V^H=0} X(\R^n \otimes \rho_G \oplus V)^G.
$$
which is induced from the standard inclusion $\R \to (\rho_g)^G$ which sends  $1$  to the unit vector $\frac{1}{\sqrt{|G|}} \sum_{g \in G} g \in \R[G] = \rho_G$. 
Both of these maps induce maps of orthogonal spectra as we let $n$ vary. 

Now we have to check that these maps are stable equivalences. This can be checked on homotopy groups, where it is well known (note that the $\hocolim$'s become an actual colimit upon taking stable homotopy groups by the results of Appendix \ref{app:colim}).
\end{proof}

We need the following proposition.

\begin{proposition}\label{prop:compgeomfix} Let $G$ be a finite group with complete $G$-universe $\mathcal U$, and let $H\subseteq H^\prime\subseteq G$ be normal subgroups. There is a natural transformation $\Phi^{H^\prime/H}_{\mathcal{U}^H}\circ \Phi^H_{\mathcal{U}}\to \Phi^{H^\prime}_{\mathcal{U}}$ of functors $G\Sp^O\to (G/H^\prime)\Sp^O$. For each $X\in G\Sp^O$, the map
\[
\Phi^{H^\prime/H}_{\mathcal{U}^H}(\Phi^H_{\mathcal{U}}(X))\to \Phi^{H^\prime}_{\mathcal{U}}(X)
\]
is an equivalence of orthogonal $(G/H^\prime)$-spectra.
\end{proposition}

\begin{proof} Let us compute $\Phi^{H^\prime/H}_{\mathcal{U}^H}(\Phi^H_{\mathcal{U}}(X))$. It is given by the spectrum whose $n$-th term is
\[
\hocolim_{W\in\mathcal{U}^H,W^{H^\prime}=0} \Big(L(\mathbb R^{d_W},W)_+\wedge_{O(d_W)} \hocolim_{V\in\mathcal{U},V^H=0} \big(L(\mathbb R^{d_V})_+\wedge_{O(d_V)} X_{n+d_W+d_V}\big)^H\Big)^{H/H^\prime}\ .
\]
Here, $d_V$ and $d_W$ denote the dimensions of $V$ and $W$. Note that as $H$ acts trivially on $L(\mathbb R^{d_W},W)$, one can pull in this smash product; also, taking fixed points commutes with all $\hocolim$'s by Lemma \ref{fixedcommutes}. Thus, a simple rewriting gives
\[
\hocolim_{W\in \mathcal{U}^H,W^{H^\prime}=0} \hocolim_{V\in\mathcal{U},V^H=0} \Big(L(\mathbb R^{d_W+d_V},W\oplus V)_+\wedge_{O(d_W+d_V)} X_{n+d_W+d_V}\Big)^{H^\prime}\ .
\]
Combining the $\hocolim$'s into a single index category gives a map to the $\hocolim$ over all pairs $(W,V)$. But note that such pairs $(W,V)$ are equivalent to $U\in \mathcal{U}$ with $U^{H^\prime}=0$ via $U=W\oplus V$; indeed, one can recover $W=U^H$, and $V$ as the orthogonal complement of $W$. Thus, we get a map to
\[
\hocolim_{U\in\mathcal{U},U^{H^\prime}=0} \big(L(\mathbb R^{d_U},U)_+\wedge_{O(d_U)} X_{n+d_u}\big)^{H^\prime}\ ,
\]
which is precisely the $n$-th term of $\Phi^{H^\prime} X$. We leave it to the reader to verify that this is compatible with all extra structure.

It is clear that this defines a stable equivalence of the underlying orthogonal spectra, as it is even a levelwise equivalence. In particular, one sees that $\Phi^H_{\mathcal{U}}$ takes equivalences of orthogonal $G$-spectra to equivalences of orthogonal $G/H$-spectra. The compatibility of the geometric fixed point functor with composition then implies that the natural transformation $\Phi^{H^\prime/H}_{\mathcal{U}^H}(\Phi^H_{\mathcal{U}}(X))\to \Phi^{H^\prime}_{\mathcal{U}}(X)$ is an equivalence of orthogonal $G/H$-spectra.
\end{proof}

In particular, we get functors $\Phi^H: G\Sp\to (G/H)\Sp$ of $\infty$-categories which are compatible with composition; we will ignore the choice of $\mathcal{U}$ after passing to $\infty$-categories, as e.g.~by Proposition~\ref{prop:rightadjointgeometric} below, this functor is canonically independent of the choice. The natural transformation $-^H\to \Phi^H$ lifts to a natural transformation of functors $G\Sp\to (G/H)\Sp$, by restricting to orthogonal $G$-$\Omega$-spectra. One can compute the fiber of the map $-^H\to \Phi^H$ in terms of data for subgroups of $H$, cf.~\cite[Proposition 7.6]{SchwedeGenuine}. For example, there is the following result of Hesselholt--Madsen, \cite[Proposition 2.1]{HesselholtMadsen}.

\begin{proposition}\label{prop:hesselholtmadsen} Let $G$ be a cyclic group of $p$-power order, and $H=C_p\subseteq G$ the subgroup of order $p$. For $X\in G\Sp$ there is a natural fiber sequence
\[
X_{hG}\xto{N} X^G\to (\Phi^{C_p} X)^{G/C_p}
\]
of spectra, where the second map is the map induced on genuine $G/C_p$-fixed points by the natural transformation $-^{C_p}\to \Phi^{C_p}$ of lax symmetric monoidal functors $G\Sp\to (G/C_p)\Sp$.$\hfill \Box$
\end{proposition}

Note that if one applies this diagram to the map $X\to B_G(X)$ from $X$ to its Borel completion, one gets a commutative diagram
\[\xymatrix{
X_{hG}\ar^N[r]\ar@{=}[d] & X^G\ar[d]\ar[r] & (\Phi^{C_p} X)^{G/C_p}\ar[d]\\
X_{hG}\ar[r] & X^{hG}\ar[r] & X^{tG}\ ,
}\]
where both rows are fiber sequences, and the functors in the right square are all lax symmetric monoidal, and the transformations are lax symmetric monoidal.\footnote{For the identification of the different norm maps $X_{hG}\to X^{hG}$, one can for example use Theorem~\ref{thm:genfarrelltate}.} In particular, one sees that the functor $-^{tG}: \Sp^{BG}\to \Sp$ has a natural lax symmetric monoidal structure making the transformation $-^{hG}\to -^{tG}$ of functors $\Sp^{BG}\to \Sp$ lax symmetric monoidal. Indeed, $-^{tG}: \Sp^{BG}\to \Sp$ can be written as the composite 
$$\Sp^{BG}\xto{B_G} G\Sp\xto{(\Phi^{C_p})^{G/C_p}} \Sp$$
 of lax symmetric monoidal functors.
We note that this is the classical proof that $-^{tG}: \Sp^{BG}\to \Sp$ has a natural lax symmetric monoidal structure for cyclic groups $G$ of prime power order.\footnote{For general groups $G$ one can write $-^{tG}$ as the composition 
$Sp^{BG}\xto{B_G} G\Sp \xto{\otimes \widetilde{EG} } G\Sp\xto{-^G} \Sp$
where  $\widetilde{EG}$ is the pointed $G$-space obtained as the cofiber of the map $EG_+ \to S^0$ (see in the  proof of Proposition \ref{prop:rightadjointgeometric} for a similar construction). All three functors are lax symmetric monoidal, the middle one since it is a smashing localization of $G\Sp$.}
By Theorem~\ref{thm:tatelaxsymm}, we see that this lax symmetric monoidal structure agrees with the one constructed there.

As another application of the relation between genuine and geometric fixed points, one has the following proposition.

\begin{proposition}\label{prop:rightadjointgeometric} Let $G$ be a finite group, and $H\subseteq G$ a normal subgroup. The functor $\Phi^H: G\Sp\to (G/H)\Sp$ has a fully faithful right adjoint $R_H: (G/H)\Sp\to G\Sp$. The essential image of $R_H$ is the full subcategory $G\Sp_{\geq H}\subseteq G\Sp$ of all $X\in G\Sp$ such that $X^N\simeq 0$ for all subgroups $N\subseteq G$ that do not contain $H$. On $G\Sp_{\geq H}$, the natural transformation $-^H\to \Phi^H: G\Sp_{\geq H}\to (G/H)\Sp$ is an equivalence.
\end{proposition}

Note that by \cite[Theorem 7.11]{SchwedeGenuine}, the condition on the essential image can be replaced by the condition $\Phi^N X\simeq 0$ for all subgroups $N\subseteq G$ that do not contain $H$. Also, the proposition implies as in Corollary~\ref{cor:borelcompletionlaxsymm} that $R_H$ has a natural lax symmetric monoidal structure. Moreover, one sees that one can redefine $\Phi^H: G\Sp\to (G/H)\Sp$ equivalently as the composition of the localization $G\Sp\to G\Sp_{\geq H}$ and $-^H$, which shows that $\Phi^H$ is independent of all choices.

\begin{proof} Choose a (compactly generated weak Hausdorff) space $E\mathcal P$ with $G$-action which has the property that $E\mathcal P^H$ is empty, but for all subgroups $N$ of $G$ that do not contain $H$, $E\mathcal P^N$ is contractible. Let $\widetilde{E\mathcal P}$ be the based $G$-space given as the homotopy cofiber of the map $E\mathcal{P}_+\to S^0$ mapping $E\mathcal P$ to the non-basepoint of $S^0$. Then $\widetilde{E\mathcal P}^N$ is contractible for all subgroups $N$ of $G$ that do not contain $H$, while $\widetilde{E\mathcal P}^N$ is given by $S^0$ if $N$ contains $H$.

Consider the functor $L: G\Sp^O\to G\Sp^O$ sending $X$ to $\widetilde{E\mathcal P}\wedge X$. By \cite[Proposition 5.4]{SchwedeGenuine}, this functor preserves equivalences of orthogonal $G$-spectra. Moreover, the natural map $S^0\to \widetilde{E\mathcal P}$ induces a natural transformation $\mathrm{id}\to L$. It follows from \cite[Proposition 7.6]{SchwedeGenuine} that for all $X\in G\Sp^O$, one has $\Phi^N L(X)\simeq 0$ if $N$ does not contain $H$, while $\Phi^N X\to \Phi^N L(X)$ is an equivalence if $N$ contains $H$. This implies that the induced functor $L: G\Sp\to G\Sp$ with its natural transformation $\mathrm{id}\to L$ satisfies the hypothesis of \cite[Proposition 5.2.7.4]{HTT}, and the essential image of $L$ is given by $G\Sp_{\geq H}$.

As $\Phi^H: G\Sp\to (G/H)\Sp$ factors over $L$, it remains to see that the functor $\Phi^H: G\Sp_{\geq H}\to (G/H)\Sp$ is an equivalence. Note that there is a functor $R^\prime: (G/H)\Sp\to G\Sp$ by letting $G$ act through its quotient $G/H$, and the composite $(G/H)\Sp\to G\Sp\xto{\Phi^H} (G/H)\Sp$ is equivalent to the identity as follows from the definition. Consider the functor $L\circ R^\prime: (G/H)\Sp\to G\Sp_{\geq H}$. We will construct natural equivalences between both composites of $L\circ R^\prime$ and $\Phi^H$ and the identity on the respective $\infty$-category. For the composite $\Phi^H\circ L\circ R^\prime\simeq \Phi^H\circ R^\prime$, we have already done this. For the other composite, we note that there is a natural transformation
\[
L\circ R^\prime\circ \Phi^H\simeq L\circ R^\prime\circ -^H\leftarrow L\simeq \mathrm{id}: G\Sp_{\geq H}\to G\Sp_{\geq H}
\]
as there is a natural transformation $R^\prime\circ -^H\to \mathrm{id}$ on orthogonal $G$-$\Omega$-spectra (given by inclusion of fixed points). To check whether this is a natural equivalence, we can apply $\Phi^H$, as $\Phi^H: G\Sp_{\geq H}\to (G/H)\Sp$ reflects equivalences; this reduces us to the assertion about the other composite that we have already proved.\end{proof}

In order to prepare for the definition of genuine cyclotomic spectra in the next section, we need to introduce variants of the previous definition when $G$ is no longer required to be finite. We will need the following two cases.

\begin{definition}\label{defprueferequivariant}$ $\begin{altenumerate}
\item The $\infty$-category $\Tp\Sp$ of genuine $\Tp$-equivariant spectra is the limit of the $\infty$-categories $C_{p^n}\Sp$ for varying $C_{p^n}\subseteq \Tp$, along the forgetful functors $C_{p^n}\Sp\to C_{p^{n-1}}\Sp$.\footnote{This $\infty$-category is equivalent to the underlying $\infty$-category of the model category studied by Degrijse, Hausmann, L\"uck, Patchkoria and Schwede in \cite{proper}. We thank Irakli Patchkoria for sharing his idea to use the Pr\"ufer group instead of the full circle group $\T$ long before this project was started.}
\item The category $\T\Sp^O$ of orthogonal $\T$-spectra is the category of orthogonal spectra with continuous $\T$-action. 
Let $\F$ be the set of finite subgroups $C_n\subseteq \T$. A map $f: X\to Y$ of orthogonal $\T$-spectra is an $\F$-equivalence if the induced map of orthogonal $C_n$-spectra is an equivalence for all finite subgroups $C_n\subseteq \T$. The $\infty$-category $\T\Sp_\F$ of $\F$-genuine $\T$-equivariant spectra is obtained by inverting the $\F$-equivalences in $N(\T\Sp^O)$.
\end{altenumerate}
\end{definition}

In both cases, the functors $-^H$ and $\Phi^H$ make sense for all finite subgroups $H$ of $\Tp$ respectively $\T$, and they satisfy the same properties as before. An obvious variant of Theorem~\ref{thm:borelcompletion} and Proposition~\ref{prop:rightadjointgeometric} holds for $\Tp\Sp$ and $\T\Sp_\F$. In the case of $\Tp\Sp$, this follows formally by passing to the limit. In the case of $\T\Sp_\F$, one has to repeat the arguments.

\begin{corollary}\label{geometricformula}
For the functors 
$R_{C_p}: \Tp\Sp \to \Tp\Sp$ and $R_{C_p}: \T\Sp_\F \to \T\Sp_F$ we have the formula
$$
(R_{C_p}X)^H \simeq \begin{cases}
X^{H/C_p} & \text{if } C_p \subseteq H \\
0 & \text{otherwise}
\end{cases}
$$
\end{corollary}

\begin{proof}
The formulas only use finite fixed points, thus we can use Proposition \ref{prop:rightadjointgeometric}. We immediately get that $(R_{C_p}X)^H \simeq 0$ if $C_p \not\subseteq H$. If $C_p \subseteq H$ then we get
$$
(R_{C_p}X)^H \simeq ((R_{C_p}X)^{C_p})^{H/{C_p}} \simeq (\Phi^{C_p} R_{C_p}X)^{H/C_p} \simeq X^{H/C_p}\ ,
$$
where the last equivalence follows since $R_{C_p}$ is fully faithful, i.e.~the counit of the adjunction is an equivalence.
\end{proof}

\section{Genuine cyclotomic spectra}\label{sec:gencyclotomic}

In this section we give the definition of the $\infty$-category of genuine cyclotomic spectra and genuine $p$-cyclotomic spectra. We start with a discussion of genuine $p$-cyclotomic spectra.

\begin{definition}\label{def:genuinepcyclo} A genuine $p$-cyclotomic spectrum is a genuine $\Tp$-spectrum $X$ together with an equivalence $\Phi_p: \Phi^{C_p} X\xto{\simeq} X$ in $\Tp\Sp$, where
\[
\Phi^{C_p} X\in (\Tp/C_p)\Sp\simeq \Tp\Sp
\]
via the $p$-th power map $\Tp/C_p\cong \Tp$. The $\infty$-category of genuine $p$-cyclotomic spectra is the equalizer
\[
\Cyc_p = \mathrm{Eq}\left( \xymatrix{\Tp\Sp\ar^{\mathrm{id}}[r]<2pt> \ar_{\Phi^{C_p}}[r]<-2pt> & \Tp\Sp} \right).
\]
\end{definition}

For every genuine $p$-cyclotomic spectrum $(X,\Phi_p)$ there is an associated $p$-cyclo\-to\-mic spectrum in the sense of Definition~\ref{def:cyclotomicinftycat}~(ii). Indeed, there is a functor $\Tp\Sp\to \Sp^{B\Tp}$, and one can compose the inverse map $X\to \Phi^{C_p} X$ with the Borel completion $\Phi^{C_p} X\to \Phi^{C_p} B_{\Tp}(X)$; the corresponding map of underlying spectra with $\Tp$-action is a $\Tp$-equivariant map $X\to X^{tC_p}$, where the $\Tp$-action on the right is the residual $\Tp/C_p$-action via the $p$-th power map $\Tp/C_p\cong \Tp$.

\begin{proposition}\label{propforgetp}
The assignment described above defines a functor 
$$
\Cyc_p \to \Cycn_p.
$$
\end{proposition}

\begin{proof} As the equalizer is a full subcategory of the lax equalizer, it suffices to construct a functor
\[
\mathrm{LEq}\left( \xymatrix{\Tp\Sp\ar^{\mathrm{id}}[r]<2pt> \ar_{\Phi^{C_p}}[r]<-2pt> & \Tp\Sp} \right)\to \mathrm{LEq}\left( \xymatrix{\Sp^{B\Tp}\ar^{\mathrm{id}}[r]<2pt> \ar_{-^{tC_p}}[r]<-2pt> & \Sp^{B\Tp}} \right)\ .
\]
But there is the natural functor $\Tp\Sp\to \Sp^{B\Tp}$ which commutes with the first functor $\mathrm{id}$ in the lax equalizer; for the second functor, there is a natural transformation between the composite 
$$\Tp\Sp\xto{\Phi^{C_p}}\Tp\Sp\to \Sp^{B\Tp}$$ 
and the composite 
$$\Tp\Sp\to \Sp^{B\Tp}\xto{-^{tC_p}} \Sp^{B\Tp},$$
by passing to the underlying spectrum in the natural transformation $\Phi^{C_p}\to \Phi^{C_p} B_{\Tp}$. For the identification of $-^{tC_p}$ with the underlying spectrum of $\Phi^{C_p} B_{\Tp}$, cf.~Proposition~\ref{prop:hesselholtmadsen}.

In this situation, one always gets an induced functor of lax equalizers, by looking at the definition (Definition \ref{def:laxeq}).
\end{proof}

Next, we  introduce the category of genuine cyclotomic spectra as described in \cite[Section 2]{HesselholtMadsen}. For the definition, we need to fix a complete $\T$-universe $\mathcal U$; more precisely, we fix
\[
\mathcal U = \bigoplus_{k\in \Z, i\geq 1} \mathbb C_{k,i}\ ,
\]
where $\T$ acts on $\mathbb C_{k,i}$ via the $k$-th power of the embedding $\T\hookrightarrow \mathbb C^\times$. For this universe, we have an identification
\[
\mathcal U^{C_n} = \bigoplus_{k\in n\Z, i\geq 1} \mathbb C_{k,i}\ ,
\]
which is a representation of $\T/C_n$. Now there is a natural isomorphism\[
\mathcal U^{C_n}\xto{\cong} \mathcal U
\]
which sends the summand $\mathbb C_{k,i} \subseteq \mathcal U^{C_n}$ for $k \in n\Z$ and $i \geq 1$ identically to the summand $\mathbb C_{k/n,i} \subseteq  \mathcal U$. This isomorphism is
equivariant for the $\T/C_n\cong \T$-action given by the $n$-th power. We get functors
\[
\Phi^{C_n}_{\mathcal U}: \T\Sp^O\to (\T/C_n)\Sp^O\simeq \T\Sp^O
\]
such that there are natural coherent $\F$-equivalences
\[
\Phi^{C_m}_{\mathcal{U}}\circ \Phi^{C_n}_{\mathcal{U}}\to \Phi^{C_{mn}}_{\mathcal{U}}\ ,
\]
as in Proposition~\ref{prop:compgeomfix}. By inverting $\F$-equivalences, one gets an action of the multiplicative monoid $\mathbb N_{>0}$ of positive integers on the $\infty$-category $\T\Sp_\F$.

\begin{definition}\label{def:genuinecyclo}
The $\infty$-category of genuine cyclotomic spectra is given by the homotopy fixed points of $\mathbb{N}_{>0}$ on the $\infty$-category of $\F$-genuine $\T$-equivariant spectra,
\[
\Cyc = \left(\T\Sp_\F\right)^{h\mathbb{N}_{>0}}\ .
\]
\end{definition}

Roughly, objects of $\Cyc$ are objects $X\in \T\Sp_\F$ together with equivalences $\Phi_n: X\simeq \Phi^{C_n} X$ for all $n\geq 1$ which are homotopy-coherently commutative.

\begin{proposition}\label{propforget} There is a natural functor
\[
\Cyc\to \Cycn = \Sp^{B\T}\times_{\prod_{p\in \bP}\Sp^{B\Tp}} \prod_{p\in \bP} \Cycn_p\ .
\]
\end{proposition}

\begin{proof} For any $p$, we have a natural functor $\Cyc\to \Cyc_p$ given by the natural functor $\T\Sp_\F\to \Tp\Sp$, remembering only $\Phi_p$. The induced functor $\Cyc\to \prod_{p\in \bP} \Sp^{B\Tp}$ lifts to a functor $\Cyc\to \Sp^{B\T}$ by looking at the underlying spectrum with $\T$-action of $X\in \T\Sp_\F$.
\end{proof}

\begin{remark} The functor $\Cyc\to \Cycn$ a priori seems to lose a lot of information: For example, it entirely forgets all nontrivial fixed point spectra and all coherence isomorphisms between the different $\Phi_n$.
\end{remark}

In \cite[Definition 2.2]{HesselholtMadsen}, the definition of (genuine) cyclotomic spectra is different in that one takes the homotopy fixed points of $\mathbb N_{>0}$ on the category $\T\Sp^O$, i.e.~before inverting weak equivalences (or rather lax fixed points where one requires the structure maps to be weak equivalences). Let us call these objects orthogonal cyclotomic spectra.\begin{definition}\label{def:orthogonalcyclo} An orthogonal cyclotomic spectrum is an object $X\in \T\Sp^O$ together with $\F$-equivalences of orthogonal $\T$-spectra $\Phi_n: \Phi^{C_n}_{\mathcal{U}} X\to X$ in $\T\Sp^O\simeq (\T/C_n)\Sp^O$ for all $n\geq 1$, such that for all $m,n\geq 1$, the diagram
\[\xymatrix{
\Phi^{C_n}_{\mathcal{U}}(\Phi^{C_m}_{\mathcal{U}} X)\ar^\simeq[r]\ar^{\Phi^{C_n}_{\mathcal{U}}(\Phi_m)}[d] & \Phi^{C_{mn}}_{\mathcal{U}} X\ar^{\Phi_{mn}}[d] \\
\Phi^{C_n}_{\mathcal{U}} X\ar^{\Phi_n}[r] & X
}\]
commutes. A map $f: X\to Y$ of orthogonal cyclotomic spectra is an equivalence if it is an $\F$-equivalence of the underlying object in $\T\Sp^O$. Let $\CycO$ denote the category of orthogonal cyclotomic spectra.
\end{definition}

It may appear more natural to take the morphisms $\Phi^{C_n}_{\mathcal{U}} X\to X$ in the other direction, but the point-set definition of topological Hochschild homology as an orthogonal cyclotomic spectrum actually requires this direction, cf.~Section~\ref{sec:thhorth} below. We note that Hesselholt--Madsen ask that $\Phi_n$ be an equivalence of orthogonal $\T$-spectra (not merely an $\F$-equivalence), i.e.~it also induces an equivalence on the genuine $\T$-fixed points. Also, they ask that $X$ be a $\T$-$\Omega$-spectrum, where we allow $X$ to be merely a (pre)spectrum. Our definition follows the conventions of Blumberg--Mandell, \cite[Definition 4.7]{BlumbergMandell}.\footnote{Their commutative diagram in \cite[Definition 4.7]{BlumbergMandell} looks different from ours, and does not seem to ask for a relation between $\Phi_{mn}$ and $\Phi_m$, $\Phi_n$; we believe ours is the correct one, following \cite[Definition 2.2]{HesselholtMadsen}.} In \cite[Section 5]{BlumbergMandell}, Blumberg--Mandell construct a ``model-$\ast$-category'' structure on the category of orthogonal cyclotomic spectra. It is unfortunately not a model category, as the geometric fixed point functor does not commute with all colimits; this is however only a feature of the point-set model. It follows from a result of Barwick--Kan, \cite{BarwickKan}, that if one inverts weak equivalences in the model-$\ast$-category of orthogonal cyclotomic spectra, one gets the $\infty$-category $\Cyc$:

\begin{thm}\label{thm:orthocyclo} The morphism $N(\CycO)\to \Cyc$ is the universal functor of $\infty$-categories inverting the equivalences of orthogonal cyclotomic spectra.
\end{thm}

\begin{proof} This follows from \cite[Lemma 3.24]{BarwickGlasman}.
\end{proof}

One can also make a similar discussion for $p$-cyclotomic spectra; however, in this case, our Definition~\ref{def:cyclotomicinftycat}~(ii) differs from previous definitions in that we only require an action of the subgroup $\Tp$ of $\T$. If one changes our definition to a $\T$-action, one could compare Definition~\ref{def:genuinepcyclo} with the definitions of $p$-cyclotomic spectra in \cite{HesselholtMadsen} and \cite{BlumbergMandell} as in Theorem~\ref{thm:orthocyclo}.

Finally, we can state our main theorem. It uses spectra $\TC^{\mathrm{gen}}(X,p)$ for a genuine $p$-cyclotomic spectrum $X$ and $\TC^{\mathrm{gen}}(X)$ for a genuine cyclotomic spectrum $X$, whose definition we recall in Section~\ref{sec:tccomp} below. Recall that by $\TC(X)$ resp.~$\TC(X,p)$, we denote the functors defined in Definition~\ref{def:naivetc}.

\begin{thm}\label{thm:main}$ $
\begin{altenumerate}
\item Let $X\in \Cyc$ be a genuine cyclotomic spectrum whose underlying spectrum is bounded below. Then there is an equivalence of spectra $\TC^\mathrm{gen}(X)\simeq \TC(X)$.
\item Let $X \in \Cyc_p$ be a genuine $p$-cyclotomic spectrum whose underlying spectrum is bounded below. Then there is an equivalence of spectra $\TC^\mathrm{gen}(X,p) \simeq \TC(X,p)$.
\end{altenumerate}

Moreover, the forgetful functors $\Cyc \to \Cycn$ and $\Cyc_p \to \Cycn_p$ are equivalences of $\infty$-categories when restricted to the respective subcategories of bounded below spectra.
\end{thm}

\section{Equivalence of TC}\label{sec:tccomp}

In this section, we prove parts (i) and (ii) of Theorem~\ref{thm:main}. In other words, if $X$ is a genuine cyclotomic spectrum whose underlying spectrum is bounded below, then
\[
\TC^\mathrm{gen}(X)\simeq \TC(X)\ .
\]
We recall the  definition of $\TC^\mathrm{gen}(X)$ in Definition \ref{def:TCpgen}  and diagram \eqref{tcgen} below.
The proof relies on some consequences of the Tate orbit lemma (Lemma \ref{lem:tateorbit}), which we record first. Here and in the following, we identify $C_{p^n}/C_{p^m}\cong C_{p^{n-m}}$, so for example if $X$ is a spectrum with $C_{p^n}$-action, then $X^{hC_{p^m}}$ is considered as a spectrum with $C_{p^{n-m}}$-action via the identification $C_{p^n}/C_{p^m}\cong C_{p^{n-m}}$.

\begin{lemma}\label{lemtet} Let $X \in \Sp^{BC_{p^n}}$ be a spectrum with $C_{p^n}$-action that is bounded below. Then the canonical morphism 
$X^{tC_{p^n}}  \to \left(X^{tC_p}\right)^{hC_{p^{n-1}}}$ is an equivalence.
\end{lemma}

\begin{proof} First, by applying the Tate orbit lemma to $X_{hC_{p^{n-2}}}$, we see that for $n \geq 2$ the norm morphism
$$
N: X_{hC_{p^n}} \to (X_{hC_{p^{n-1}}})^{hC_p}
$$ 
is an equivalence. Then by induction the map
\[
N: X_{hC_{p^n}}\to (X_{hC_p})^{hC_{p^{n-1}}}
\]
is an equivalence. This norm morphism fits into a diagram
\[
\xymatrix{
X_{hC_{p^n}} \ar[d] \ar[r] & X^{hC_{p^n}}\ar[r] \ar[d] & X^{tC_{p^n}}\ar[d]\\
(X_{hC_p})^{hC_{p^{n-1}}} \ar[r]& (X^{hC_p})^{hC_{p^{n-1}}}\ar[r] & (X^{tC_p})^{hC_{p^{n-1}}} \\
}\]
which commutes since all maps in the left square are norm maps, and the right vertical map is defined as the cofiber. Now since the left and middle vertical maps are equivalences and the rows are fiber sequences it follows that the right vertical morphism is an equivalence as well.
\end{proof}

In the following lemma, we use the Tate construction for $\T$ defined in Corollary~\ref{cor:tates1}.

\begin{lemma}\label{lemtate2}
If $X$ is a bounded below spectrum with $\T$-action then $\left(X^{tC_p}\right)^{h\T}$ is $p$-complete and the canonical morphism  $X^{t\T} \to \left(X^{tC_p}\right)^{h\T}$ exhibits it as the $p$-completion of $X^{t\T}$.
\end{lemma}

\begin{remark} This lemma leads to a further simplification of the fiber sequence
\[
\TC(X)\to X^{h\T}\xto{(\varphi_p^{h\T}-\can)_{p \in \mathbb{P}}} \prod_p (X^{tC_p})^{h\T}\ ,
\]
as one can identify the final term with the profinite completion of $X^{t\T}$ in case $X$ is a bounded below cyclotomic spectrum.
\end{remark}

\begin{proof} The canonical morphism $(X^{tC_p})^{h\T} \to (X^{tC_p})^{h\Tp}\simeq \invlim (X^{tC_p})^{hC_{p^n}}$ is an equivalence since both sides are $p$-complete by Lemma~\ref{lem:tatecomplete} and $\Tp\to \T$ is a $p$-adic equivalence. Thus, by Lemma~\ref{lemtet}, in the commutative diagram
$$
\xymatrix{
X^{t\T} \ar[r]\ar[d] & \left(X^{tC_p}\right)^{h\T} \ar[d] \\
\invlim(X^{tC_{p^n}}) \ar[r]&  \invlim (X^{tC_p})^{hC_{p^n}}\ ,
}
$$
all corners except possibly for $X^{t\T}$ are equivalent. Thus, it remains to prove that the canonical map $X^{t\T} \to \invlim(X^{tC_{p^n}})$ is a $p$-completion.

By Lemma~\ref{lem:tateconvergence}, we can assume that $X$ is bounded, so that by a filtration we can assume that $X=HM$ is an Eilenberg-MacLane spectrum. Note that in this case, the $\T$-action on $M$ is necessarily trivial. We can also assume that $M$ is torsion free by passing to a $2$-term resolution by torsion free groups. We find that
$$
\pi_i(HM^{t\T}) = \begin{cases} 
M & \text{$i$ even } \\
0 & \text{$i$ odd}
\end{cases}
$$
and
$$
\pi_i(HM^{tC_{p^n}}) = \begin{cases} 
M/{p^n} & \text{$i$ even } \\
0 & \text{$i$ odd.}
\end{cases}
$$
The maps in the limit diagram are given by the obvious projections $M\to M/p^n$ by Lemma~\ref{lem:tates1cn}, so the result follows.
\end{proof}

Now let $X$ be a genuine $p$-cyclotomic spectrum in the sense of Definition~\ref{def:genuinepcyclo}. Let us recall the definition of $\TC^\mathrm{gen}(X,p)$ by B\"okstedt--Hsiang--Madsen, \cite{BHM}. First, $X$ has genuine $C_{p^n}$-fixed points $X^{C_{p^n}}$ for all $n\geq 0$, and there are maps $F: X^{C_{p^n}}\to X^{C_{p^{n-1}}}$ for $n\geq 1$ which are the inclusion of fixed points. Moreover, for all $n\geq 1$ there are maps $R: X^{C_{p^n}}\to X^{C_{p^{n-1}}}$, and the maps $R$ and $F$ commute (coherently). The maps $R: X^{C_{p^n}}\to X^{C_{p^{n-1}}}$ arise as the composition of the map $X^{C_{p^n}}\to (\Phi^{C_p} X)^{C_{p^{n-1}}}$ that exists for any genuine $C_{p^n}$-equivariant spectrum, and the equivalence $(\Phi^{C_p} X)^{C_{p^{n-1}}}\simeq X^{C_{p^{n-1}}}$ which comes from the genuine cyclotomic structure on $X$.

These structures determine $\TC^\mathrm{gen}(X,p)$ as follows.

\begin{definition}\label{def:TCpgen} Let $X$ be a genuine $p$-cyclotomic spectrum. Define $\TR(X,p) = \invlim_R X^{C_{p^n}}$, which has an action of $F$, and 
\vspace{-1ex}
\begin{align*}
\TC^\mathrm{gen}(X,p) & := \mathrm{Eq}\left( \xymatrix{\TR(X,p) \ar^-{\mathrm{id}}[r]<2pt> \ar_-F[r]<-2pt> & \TR(X,p) }\right)\\ 
& \simeq \invlim_R \ \mathrm{Eq}\left( \xymatrix{ X^{C_{p^n}} \ar^-{R}[r]<2pt> \ar_-F[r]<-2pt> & X^{C_{p^{n-1}}} }\right)\ .
\end{align*}

\end{definition}
To compare this with our definition, we need the following lemma, which follows directly from (the discussion following) Proposition~\ref{prop:hesselholtmadsen}.

\begin{lemma}\label{lemmapullback}
Let $X$ be a genuine $C_{p^n}$-equivariant spectrum. There is a natural pullback diagram of spectra
\[\xymatrix{
 X^{C_{p^n}}\ar^-R[r]\ar[d] & (\Phi^{C_p} X)^{C_{p^{n-1}}}\ar[d] \\
 X^{hC_{p^n}} \ar[r]&  X^{tC_{p^n}}
}\]
for all $n\geq 1$.$\hfill \Box$
\end{lemma}

Note that while the upper spectra depend on the genuine $C_{p^n}$-equivariant structure of $X$, the lower spectra only depend on the underlying spectrum with $C_{p^n}$-action.

\begin{proposition}\label{square}
Let $X$ be a genuine $C_{p^n}$-equivariant spectrum. Assume that the underlying spectrum is bounded below. Then for every $n \geq 1$ there exists a canonical pullback square
\[\xymatrix{
 X^{C_{p^{n}}}\ar[r]\ar[d] & (\Phi^{C_p} X)^{C_{p^{n-1}}}\ar[d] \\
 X^{hC_{p^{n}}} \ar[r]&  (X^{tC_p})^{hC_{p^{n-1}}}
}\]
\end{proposition}

\begin{proof} Combine Lemma~\ref{lemtet} with Lemma~\ref{lemmapullback}.
\end{proof}

\begin{corollary}\label{cor:longpullback} Let $X$ be a genuine $C_{p^n}$-equivariant spectrum. Assume that the spectra 
$$
X, \Phi^{C_p}X, \Phi^{C_{p^2}} X, \ldots, \Phi^{C_{p^{n-1}}}X
$$
are all bounded below. Then we have a diagram
$$
\xymatrix@=1em{
 X^{C_{p^{n}}}\ar[rrrr] \ar[dddd] & & & & \Phi^{C_{p^n}} X \ar[d]  \\
 & & & \left(\Phi^{C_{p^{n-1}}}X\right)^{hC_p}  \ar[r] \ar[d] & \left(\Phi^{C_{p^{n-1}}}X\right)^{tC_p} \\
 & & \left(\Phi^{C_{p^2}} X\right)^{hC_{p^{n-2}}} \ar[d] \ar[r]& \ldots    \\
 & \left(\Phi^{C_p}X\right)^{hC_{p^{n-1}}} \ar[d] \ar[r] & {\left(\left(\Phi^{C_p}X\right)^{tC_p}\right)^{hC_{p^{n-2}}}}  \\
 X^{hC_{p^n}} \ar[r] & \left(X^{tC_p}\right)^{hC_{p^{n-1}}} 
}
$$
which exhibits $X^{C_{p^n}}$ as a limit of the right side (i.e.~an iterated pullback).
\end{corollary}

\begin{proof} Use induction on $n$ and Proposition~\ref{square}.
\end{proof}

\begin{remark} One can use the last corollary to give a description of the full subcategory of $C_{p^n}\Sp$ consisting of those genuine $C_{p^n}$-equivariant spectra all of whose geometric fixed points are bounded below. An object then explicitly consists of a sequence 
$$
X, \Phi^{C_p} X,  \ldots, \Phi^{C_{p^n}} X
$$
of bounded below spectra with $C_{p^n}/C_{p^i}$-action together with $C_{p^n}/C_{p^i}$-equivariant morphisms 
$$\Phi^{C_{p^i}} X \to \left(\Phi^{C_{p^{i-1}}} X\right)^{tC_p} $$
for all $1\leq i\leq n$. 
One can prove these equivalences for example similarly to our proof of Theorem \ref{thmpequivalence} in Sections \ref{sec:proofequivalence} and \ref{sec:coalgebras}. 
Vice versa one can also deduce our Theorem \ref{thmpequivalence} from these equivalences. 

Note that this description is quite different from the description as Mackey functors, which is in terms of the genuine fixed points; a general translation appears in work of Glasman, \cite{glasman2015stratified}. One can also go a step further and give a description of the category of all genuine $C_{p^n}$-spectra or $\T$-spectra (without connectivity hypothesis) along these lines, but the resulting description is more complicated and less tractable as it contains a lot of coherence data. This approach has recently appeared in work of Ayala--Mazel-Gee--Rozenblyum, \cite{AMGR}.
\end{remark}

Let us note the following corollary. This generalizes a result of Ravenel, \cite{Ravenelsegalconj}, that proves the Segal conjecture for $C_{p^n}$ by reduction to $C_p$. This had been further generalized to $\THH$ by Tsalidis, \cite{TsalidisTHHdescent}, and then by B\"okstedt--Bruner--Lun\o e-Nielsen--Rognes, \cite[Theorem 2.5]{BBLNR}, who proved the following theorem under a finite type hypothesis.

\begin{corollary}\label{cor:segalinduct} Let $X$ be a genuine $C_{p^n}$-equivariant spectrum, and assume that for any $Y\in \{X, \Phi^{C_p} X,  \ldots, \Phi^{C_{p^{n-1}}} X\}$, the spectrum $Y$ is bounded below, and the map
\[
(Y^{C_p})^\wedge_p\to (Y^{hC_p})^\wedge_p
\]
induces an isomorphism on $\pi_i$ for all $i\geq k$. Then the map $(X^{C_{p^n}})^\wedge_p\to (X^{hC_{p^n}})^\wedge_p$ induces an isomorphism on $\pi_i$ for all $i\geq k$.
\end{corollary}

As in \cite[Theorem 2.5]{BBLNR}, this result also holds true if one instead measures connectivity after taking function spectra from some $W$ in the localizing ideal of spectra generated by $\bS[\tfrac 1p]/\bS$. Note that for cyclotomic spectra, one needs to check the hypothesis only for $Y=X$.

\begin{proof} Note that for $1\leq i\leq n$, one has $\Phi^{C_p} (\Phi^{C_{p^{i-1}}} X) = \Phi^{C_{p^i}} X$. Thus, using Lemma~\ref{lemmapullback} for the genuine $C_p$-equivariant spectrum $\Phi^{C_{p^{i-1}}} X$, we see that the assumption is equivalent to the condition that
\[
\Phi^{C_{p^i}} X\to (\Phi^{C_{p^{i-1}}} X)^{tC_p}
\]
induces an isomorphism on $\pi_i$ for all $i\geq k$ after $p$-completion. Now we use that in Corollary~\ref{cor:longpullback}, all the short vertical maps induce an isomorphism on $\pi_i$ for all $i\geq k$ after $p$-completion. It follows that the long left vertical map does the same, as desired.
\end{proof}

\begin{thm}\label{computeTCp}
Let $X$ be a genuine $p$-cyclotomic spectrum such that the underlying spectrum is bounded below. Then there is a canonical fiber sequence
$$
\TC^\gen(X,p) \to X^{h\Tp} \xto{\varphi_p^{h\Tp} - \can} (X^{tC_p})^{h\Tp} 
$$
In particular we get an equivalence $\TC^\gen(X,p) \simeq \TC(X,p)$.
\end{thm}

\begin{proof} By Proposition~\ref{square} and the equivalence $X \simeq \Phi^{C_p} X$, we inductively get an equivalence
$$
X^{C_{p^n}} \simeq  X^{hC_{p^n}} \times_{ (X^{tC_p})^{hC_{p^{n-1}}} } X^{hC_{p^{n-1}}} \times \ldots \times_{X^{tC_p}} X\ ,
$$
cf.~also Corollary~\ref{cor:longpullback}. Here the projection to the right is always the tautological one and the projection to the left is $\varphi_p$ (or more precisely the induced map on homotopy fixed points). Under this equivalence the map $R: X^{C_{p^n}} \to X^{C_{p^{n-1}}}$ corresponds to forgetting the first factor and $F: X^{C_{p^n}} \to X^{C_{p^{n-1}}}$ corresponds to forgetting the last factor followed by the maps $X^{hC_{p^k}} \to X^{hC_{p^{k-1}}}$ applied factorwise.

Now we have to compute the fiber of  $R - F: X^{C_{p^n}} \to X^{C_{p^{n-1}}}$. To that end we consider the following square
$$
\xymatrix{
X^{hC_{p^n}} \times \ldots \times X \ar[r]^{R' - F'} \ar[d]^{\varphi_p - \can}&  X^{hC_{p^{n-1}}} \times \ldots \times X \ar[d]^{\varphi_p - \can}\\
(X^{tC_p})^{hC_{p^{n-1}}} \times \ldots \times X^{tC_p} \ar[r]^{R'' - F''}  &   (X^{tC_p})^{hC_{p^{n-2}}} \times \ldots \times X^{tC_p}\ .
}
$$
Here, $R'$ forgets the first factor, $F'$ forgets the last factor and projects $X^{hC_{p^k}}\to X^{hC_{p^{k-1}}}$, $R''$ forgets the first factor, and $F''$ forgets the last factor and projects $(X^{tC_p})^{hC_{p^k}}\to (X^{tC_p})^{hC_{p^{k-1}}}$. Also $\varphi_p$ denotes the various maps $X^{hC_{p^k}}\to (X^{tC_p})^{hC_{p^k}}$, and $\can$ the various maps $X^{hC_{p^k}}\to (X^{tC_p})^{hC_{p^{k-1}}}$. It is easy to see that the diagram commutes (since the lower right is a product this can be checked for every factor separately), and by design the vertical fibers are $X^{C_{p^n}}$ and $X^{C_{p^{n-1}}}$ with the induced map given by $R - F$. 

Now we compute the fibers horizontally. We get $X^{hC_{p^n}}$ via the diagonal embedding for the upper line and analogously $(X^{tC_p})^{hC_{p^{n-1}}}$ via the diagonal embedding for the lower line. Therefore the fiber of $R - F: X^{C_{p^n}} \to  X^{C_{p^{n-1}}}$ is equivalent to the fiber of $\varphi_p - \can: X^{hC_{p^n}} \to (X^{tC_p})^{hC_{p^{n-1}}}$. Finally we take the  limit over $n$ to get the desired result.
\end{proof}

Recall from \cite[Lemma 6.4.3.2]{Local} Goodwillie's definition of the integral topological cyclic homology for a genuine cyclotomic spectrum $X$ in the sense of Definition~\ref{def:genuinecyclo}. It is defined\footnote{This is not the definition initially given by Goodwillie, but it is equivalent to the corrected version of Goodwillie's definition as we learned from B.~Dundas and we take it as a definition here.} by the pullback square
\begin{equation}\label{tcgen}
\xymatrix{
\TC^\gen(X) \ar[r] \ar[d] & X^{h\T} \ar[d]\\
\prod_{p \in \bP} \TC^{\gen}(X,p)^\wedge_p\ar[r] & \prod_{p \in \bP} (X^\wedge_p)^{h\T}.
}
\end{equation}
Note that the homotopy fixed points $(X^\wedge_p)^{h \T}$ in the lower right corner are equivalent to $(X^\wedge_p)^{h\Tp}$ since the inclusion $\Tp \to \T$ is a $p$-adic equivalence. Taking homotopy fixed points commutes with $p$-completion since the mod-$p$ Moore spectrum $M(\Z/p) = \bS/p$ is a finite spectrum and fixed points of a $p$-complete spectrum stay $p$-complete. It follows that the right vertical map is a profinite completion, and therefore also the left vertical map is a profinite completion.

\begin{thm}\label{computeTC}
Let $X$ be a genuine cyclotomic spectrum such that the underlying spectrum is bounded below. Then there is a canonical fiber sequence
$$
\TC^\gen(X)\to X^{h\T} \xto{(\varphi_p^{h\T} - \can)_{p \in \bP}} \prod_{p \in \bP}(X^{tC_p})^{h\T}\ .
$$
In particular we get an equivalence $\TC^\gen(X) \simeq \TC(X)$.
\end{thm}

\begin{proof} We consider the commutative diagram
$$
\xymatrix{
\TC^\gen(X) \ar[r] \ar[d] & X^{h\T} \ar[d] \\
\prod_{p \in \bP} \TC^{\gen}(X,p)^\wedge_p\ar[r] & \prod_{p \in \bP} (X^\wedge_p)^{h\T}\ar[r] &  \prod_{p \in \bP}((X^\wedge_p)^{tC_p})^{h\Tp}\ .
}
$$
Here, the lower line is the product over all $p$ of the $p$-completions of the fiber sequences from Theorem~\ref{computeTCp}, and the square is the  pullback defining $\TC^\gen(X)$. Now by Lemma~\ref{lem:tatecomplete} and using that $\Tp\to \T$ is a $p$-adic equivalence, the natural maps
\[
(X^{tC_p})^{h\T}\to ((X^\wedge_p)^{tC_p})^{h\Tp}
\]
are equivalences for all primes $p$. Thus, we can fill the diagram to a commutative diagram
\[\xymatrix{
\TC^\gen(X) \ar[r] \ar[d] & X^{h\T} \ar[d] \ar[r] & \prod_{p\in \bP} (X^{tC_p})^{h\T}\ar^\simeq[d] \\
\prod_{p \in \bP} \TC^{\gen}(X,p)^\wedge_p\ar[r] & \prod_{p \in \bP} (X^\wedge_p)^{h\T}\ar[r] &  \prod_{p \in \bP}((X^\wedge_p)^{tC_p})^{h\Tp}\ .
}\]
Now the result follows formally as the lower line is a fiber sequence, the left square is a  pullback and the right map is an equivalence.\end{proof}

\section{Coalgebras for endofunctors}\label{sec:coalgebras}

\newcommand{\Prl}{\mathcal{P}\mathrm{r}^\mathrm{L}}

In this section we investigate the relation between the categories of coalgebras and fixed points for endofunctors. This will be applied in the proof of the  equivalence between ``genuine'' and ``naive'' cyclotomic spectra that will be given in the next section (Theorem \ref{thmpequivalence} and Theorem \ref{theoremintegralequivalence}), where the key step is to upgrade the lax map $X\to X^{tC_p}$ to an equivalence $Y\simeq \Phi^{C_p} Y$ of a related object $Y$.

\begin{definition} Let $\calC$ be an $\infty$-category and $F: \calC \to \calC$ be an endofunctor. Then an $F$-coalgebra is given by an object $X \in \calC$ together with a morphism $X \to FX$. A fixpoint of $F$ is an object $X \in \calC$ together with an equivalence $X \xto{\sim} FX$.
\end{definition}

We define the $\infty$-category $\CoAlg_F(\calC)$ as the lax equalizer 
$$
\CoAlg_F(\calC) =  \mathrm{LEq}\big( \xymatrix{\calC \ar[r]<2pt>^\id \ar[r]<-2pt>_F & \calC} \big)
$$
of $\infty$-categories, cf.~Definition~\ref{def:laxeq} for the notion of lax equalizers. If the $\infty$-category $\calC$ is clear from the context we will only write $\CoAlg_F$. The $\infty$-category $\Fix_F = \Fix_F(\calC)$ is the full subcategory $\Fix_F \subseteq \CoAlg_F$ spanned by the fixed points. 
The example that is relevant for us is that $\calC$ is the $\infty$-category $\Tp\Sp$ of genuine $\Tp$-spectra  and $F$ is the endofunctor $\Phi^{C_p}$ as discussed in Definition \ref{defprueferequivariant}. The reader should have this example in mind but we prefer to give the discussion abstractly. 

If $\calC$ is presentable and $F$ is accessible then it follows from Proposition \ref{prop:laxeq} that $\CoAlg_F$ is presentable and that the forgetful functor $\CoAlg_F \to \calC$ preserves all colimits. 
Now assume that the functor $F$ preserves all colimits, i.e.~admits a right adjoint $G$. Then $\Fix_F$ is an equalizer in the $\infty$-category $\Prl$ of presentable $\infty$-categories and colimit preserving functors (see \cite[Proposition 5.5.3.13]{HTT} for a discussion of limits in $\Prl$). As such it is itself presentable and the forgetful functor $\Fix_F \to \calC$ preserves all colimits. As a consequence, the inclusion $\iota: \Fix_F \subseteq \CoAlg_F$ preserves all colimits, thus by the adjoint functor theorem \cite[Corollary 5.5.2.9]{HTT} the functor $\iota$ admits a right adjoint $R_\iota: \CoAlg_F \to \Fix_F$. We do not know of an explicit way of describing the colocalization $\iota R_\iota$ in general, but we will give a formula under some extra hypothesis now.

\begin{construction}\label{construction}
There is an endofunctor
$\bar F$ given on objects by
$$
\bar{F}: \CoAlg_F \to \CoAlg_F \qquad (X \xto{\varphi} FX) \mapsto (FX \xto{F\varphi} F^2X)
$$
which comes with a natural transformation $\mu: \id \to \bar{F}$. As a functor $\bar{F}$ is induced from the map of diagrams of $\infty$-categories 
$$
\big(\xymatrix{\calC \ar[r]<2pt>^\id \ar[r]<-2pt>_F & \calC}\big) \xto{(F, F)} 
\big(\xymatrix{\calC \ar[r]<2pt>^\id \ar[r]<-2pt>_F & \calC}\big)
$$
using the functoriality of the lax equalizer. Similarly the transformation $\mu$ can be described as a functor
$$
\mathrm{LEq}\big(\xymatrix{\calC \ar[r]<2pt>^\id \ar[r]<-2pt>_F & \calC}\big)  \to \mathrm{LEq}\big(\xymatrix{\calC \ar[r]<2pt>^\id \ar[r]<-2pt>_F & \calC}\big)^{\Delta^1}\cong \calC^{\Delta^1} \times_{\calC^{\Delta^1} \times \calC^{\Delta^1}} \calC^{\Delta^1 \times \Delta^1} 
$$
which is given in components as the projection $\mathrm{LEq}\big(\xymatrix{\calC \ar[r]<2pt>^\id \ar[r]<-2pt>_F & \calC}\big) \to \calC^{\Delta^1}$ and the map
$$
C \times_{C \times C} C^{\Delta^1} \to C^{\Delta^1 \times \Delta^1} \cong C^{\Delta^2} \times_{C^{\Delta^1}} C^{\Delta^2}
$$
which sends $X \to FX$ to the 2-simplices
$$
\xymatrix{
X \ar[r]\ar[d] & FX \\
FX \ar@{=}[ru] & 
} \qquad \text{and} 
\xymatrix{
  & FX \ar[d]\ar@{=}[ld] \\
FX \ar[r] & F^2X  
.}
$$
The last map, which we described on vertices refines to a map of simplicial sets since it is  a composition of degeneracy maps and the map $F: \calC \to \calC$.
\end{construction}

Recall the standing assumption that the underlying $\infty$-category $\calC$ is presentable and that $F$ preserves all colimits.
By the adjoint functor theorem the functor $\bar{F}$  from Construction \ref{construction} admits a right adjoint 
$R_{\bar F}: \CoAlg_F \to \CoAlg_F$. We will describe $R_{\bar F}$ more explicitly below (Lemma \ref{lem:rightad}). There is a transformation $\nu: R_{\bar F} \to \id$ obtained as the adjoint of the transformation
$\mu$ from Construction \ref{construction}. 

\begin{proposition}\label{proplimit}
The endofunctor $\iota R_\iota: \CoAlg_F \to \CoAlg_F$ is given by the limit of the following diagram of endofunctors
$$
\ldots \to R_{\bar F}^3 \xto{R^2_{\bar F} \nu} R^2_{\bar F} \xto{R_{\bar F} \nu} R_{\bar F} \xto{\nu} \id \ .
$$
\end{proposition}

\begin{proof} As a first step we use that the inclusion $\iota: \Fix_F \subseteq \CoAlg_F$ not only admits a right adjoint, but also a left adjoint $L_\iota: \CoAlg_F \to \Fix_F$ which can be described concretely as follows: the composition of $L_\iota$ with the inclusion $\iota: \Fix_F \to \CoAlg_F$ is given by the colimit of functors
$$
\bar F_\infty := \dirlim\big( \id \xto{\mu} \bar F \xto{\bar F\mu} \bar F^2 \xto{\bar F^2\mu} \bar F^3 \to \ldots \quad \big) .
$$

To see this we have to show that $\bar F_\infty: \CoAlg_F \to \CoAlg_F$ is a localization (more precisely the transformation $\id \to\bar  F_\infty$ coming from the colimit structure maps) and that the local objects for this localization are precisely given by the fixed-point coalgebras. We first note that by definition of $\bar F$ the subcategory $\Fix_F \subseteq \CoAlg_F$ can be described as those coalgebras $X$ for which the transformation $\mu_X: X \to \bar F(X)$ is an equivalence. Since $\bar F$ preserves colimits it follows immediately that for a given coalgebra $X$ the coalgebra $\bar{F}_\infty X$ is a fixed point. We also make the observation that for a given fixed point $X \in \Fix_F$ the canonical map $X \to \bar F_\infty X$ is an equivalence. If follows now that $\bar F_\infty$ applied twice is equivalent to $\bar F_\infty$ which shows that $\bar F_\infty$ is a localization (using \cite[Proposition 5.2.7.4]{HTT}) and also that the local objects are precisely the fixed points. 

Now the functor $\iota L_\iota: \CoAlg_F \to \CoAlg_F$ is left adjoint to the functor $\iota R_\iota: \CoAlg_F \to \CoAlg_F$. Thus the formula for $\bar F_\infty$ as a colimit implies immediately the claim as the right adjoint to a colimit is the limit of the respective right adjoints. 
\end{proof}

Note that the limit in Proposition \ref{proplimit} is taken in the $\infty$-category of endofunctors of $F$-coalgebras which makes it potentially complicated to understand since limits of coalgebras can be hard to analyze. However if for a given coalgebra $X \to FX$ the endofunctor $F: \calC \to \calC$ preserves the limit of underlying objects of the diagram
$$
\ldots \to R_{\bar F}^3X \to R^2_{\bar F}X \to R_{\bar F}X \to X\ ,
$$
then by Proposition \ref{prop:laxeq}(v) the limit of coalgebras is given by the underlying limit. This will be the case in our application and the reader should keep this case in mind. 

Now we give a formula for the functor $R_{\bar F}: \CoAlg_F \to \CoAlg_F$, which was by definition the right adjoint of the functor $\bar F: \CoAlg_F \to \CoAlg_F$ from Construction \ref{construction}. To simplify the treatment we make some extra assumptions on $F$. Recall that by assumption $F: \calC \to \calC$ admits a right adjoint. Denote this right adjoint by $R_F: \calC \to \calC$. We assume that the counit of the adjunction $FR_{F} \to \id_\calC$ be a natural equivalence. This is  equivalent to $R_F$ being fully faithful. We will denote by $\eta: \id_\calC \to R_F F$ the unit of the adjunction. We moreover assume that $F$ preserves pullbacks.

\begin{lemma}\label{lem:rightad} The functor $R_{\bar F}: \CoAlg_F \to \CoAlg_F$ takes a coalgebra $X \xto{\varphi} FX$ to the coalgebra $R_{\bar{F}}(X\xto{\varphi} FX) = Y\xto{\varphi} FY$ whose underlying object $Y$ is the pullback
$$
\xymatrix{
Y \ar[r]\ar[d] & R_F X \ar[d]^-{R_F\varphi}\\
X \ar[r]^-{\eta_X} & R_F FX.
}
$$
in $\calC$, and where the coalgebra structure $\varphi: Y\to FY$ is given by the left vertical map under the identification $FY \xto{\sim} FR_F X \xto{\sim} X$ induced by applying $F$ to the upper horizontal map in the diagram and the counit of the adjunction between $F$ and $R_F$.

Moreover, the counit $\bar FR_{\bar F}\to \id$ is an equivalence.
\end{lemma}

 \begin{proof} We will check the universal mapping property using the explicit description of mapping spaces in $\infty$-categories of coalgebras given in Proposition \ref{prop:laxeq}(ii). We get that maps from an arbitrary coalgebra $Z \to FZ$ into the coalgebra $Y \to FY\simeq X$ are given by the equalizer
\begin{align*}
\xymatrix{
\mathrm{Eq} \big( \Map_\calC(Z,Y)  \ar[r]<2pt> \ar[r]<-2pt> & \Map_\calC(Z, X)\big)
} 
\end{align*}
We use the definition of $R_{\bar F}X$ as a pullback to identify the first term with
$$
\Map_\calC(Z,R_F X) \times_{\Map_\calC(Z, R_F FX)} \Map_\calC(Z,X). 
$$
Under this identification the two maps to $\Map_{\calC}(Z,X)$ are given as follows: the first map is given by the projection to the second factor. The second map
 is given by the projection to the first factor followed by the map
 $$
 \Map_{\calC}(Z,R_FX) \xto{F}\Map_{\calC}(FZ, FR_FX) \xto{\sim} \Map_{\calC}(FZ, X)\to \Map_{\calC}(Z,X)  
 $$
 where the middle equivalence is induced by the counit of the adjunction, and the final map is precomposition with $Z\to FZ$. As a result we find that
 \begin{align*}
 &\xymatrix{
\mathrm{Eq} \big( \Map_\calC(Z,Y)  \ar[r]<2pt> \ar[r]<-2pt> & \Map_\calC(Z, X)\big)} \\
&\simeq
\xymatrix{ 
\mathrm{Eq} \big( \Map_\calC(Z,R_FX)  \ar[r]<2pt> \ar[r]<-2pt> & \Map_\calC(Z, R_F FX)\big) } \\
& \simeq 
\xymatrix{
\mathrm{Eq} \big( \Map_\calC(FZ,X)  \ar[r]<2pt> \ar[r]<-2pt> & \Map_\calC(FZ, FX)\big) }  \\
& \simeq \ \Map_{\CoAlg_F}(\bar F(Z), X)\ .
\end{align*} 
This shows the universal property, finishing the identification of $R_{\bar F}(X)$. The explicit description implies that the counit $\bar FR_{\bar F}(X)\to X$ is an equivalence, as on underlying objects, it is given by the equivalence $FR_F X\simeq X$.
\end{proof}
 
\begin{corollary}\label{corrightad} Under the same assumptions as for Lemma \ref{lem:rightad}, the underlying object of the $k$-fold iteration $R_{\bar F}^k X$ is equivalent to
$$ 
R_F^k X \times_{R_F^k FX} R_F^{k-1} X \times_{R_F^{k-1} FX}  \ldots \times_{R_F FX} X
$$
where the maps to the right are induced by the coalgebra map $X \to FX$ and the maps to the left are induced by the unit $\eta$. In this description the map $R_{\bar F}^k X \to R^{k-1}_{\bar F} X$ can be described as forgetting the first factor.
\end{corollary}

\begin{proof} This follows by induction on $n$ using that $R_F$ as a right adjoint preserves pullbacks and $F$ does so by assumption.
\end{proof}
 
Now let us specialize the general discussion given here to the case of interest. We let $\calC$ be the $\infty$-category $\Tp\Sp$ of genuine $\Tp$-spectra  and $F$ be the endofunctor $\Phi^{C_p}$ as discussed in Definition \ref{defprueferequivariant}. Then the $\infty$-category of genuine $p$-cyclotomic spectra is given by the fixed points of $\Phi^{C_p}$ (see Definition \ref{def:genuinepcyclo}). In this case all the assumptions made above are satisfied: $\Tp\Sp$ is presentable as a limit of presentable $\infty$-categories along adjoint functors, the functor $\Phi^{C_p}$ preserves all colimits, and its right adjoint is fully faithful, as follows from Proposition~\ref{prop:rightadjointgeometric} and the discussion after Definition~\ref{defprueferequivariant}.

\begin{thm}\label{thmunderlying}
The inclusion $\iota: \Cyc_p \subseteq \CoAlg_{\Phi^{C_p}}(\Tp\Sp)$ admits a right adjoint $R_\iota$ such that the counit $\iota R_\iota \to \id$  of the adjunction induces an equivalence of underlying non-equivariant spectra.
\end{thm}

\begin{proof}
Let us follow the notation of Proposition  \ref{prop:rightadjointgeometric} and denote the right adjoint of $\Phi^{C_p}$ by $R_{C_p}: \Tp\Sp \to \Tp\Sp$. Then the right adjoint of $\bar\Phi^{C_p}$ will be denoted by $R_{\bar C_p}: \CoAlg_{\Phi^{C_p}} \to \CoAlg_{\Phi^{C_p}}$. 

We have that $(R_{C_p}X)^{C_{p^m}} \simeq X^{C_{p^{m-1}}}$ for $m \geq 1$ and $(R_{C_p}X)^{\{e\}} \simeq 0$ as shown in Corollary \ref{geometricformula}. If we use the formula given in Corollary \ref{corrightad} then we obtain that for $(X,\varphi: X\to FX)\in \CoAlg_{\Phi^{C_p}}$, the underlying object of $R_{\bar C_p}^k X$ has genuine fixed points given by
\[
(R_{\bar C_p}^k X)^{C_{p^m}} \simeq  X^{C_{p^{m-k}}} \times_{(\Phi^{C_p} X)^{C_{p^{m-k}}}} X^{C_{p^{m-k+1}}}  \times \ldots
  \times_{{(\Phi^{C_p} X)^{C_{p^{m-1}}}}} X^{C_{p^m}}\ .
\]
We see from this formula that as soon as $k$ gets bigger than $m$ this becomes independent of $k$. Thus the limit of the fixed points 
\begin{equation} \label{limifixedpoints}
\ldots \to (R_{\bar C_p}^3X)^{C_{p^m}} \xto{R_{\bar C_p}^2 \nu} (R_{\bar C_p}^2 X)^{C_{p^m}} \xto{R_{\bar C_p} \nu} (R_{\bar C_p}X)^{C_{p^m}} \xto{\nu} X^{C_{p^m}} .
\end{equation}
is eventually constant. 

We see from Lemma \ref{lem:rightad} that $\bar{\Phi}^{C_p} R_{\bar C_p} \simeq \id$ as endofunctors of $\CoAlg_{\Phi^{C_p}}(\Tp\Sp)$. Thus we get that $\bar{\Phi}^{C_p}$ applied to the tower $\ldots\to R_{\bar C_{p}^2}X \to R_{\bar C_p}X \to X$ is given by the tower
\[
\ldots\to R_{\bar C_{p}^2}X \to R_{\bar C_p}X \to X \to \bar{\Phi}^{C_p} X
\]
which is also eventually constant on all fixed points. In particular $\bar{\Phi}^{C_p}$ commutes with the limit of the tower. Thus by Proposition \ref{prop:laxeq}(v) the limit of coalgebras is computed as the limit of underlying objects. Then by Proposition \ref{proplimit} we get that the limit of the sequence \eqref{limifixedpoints} are the $C_{p^m}$-fixed points of the spectrum $R_\iota X $. Specializing to $m=0$ we obtain that the canonical map $(\iota R_\iota X)^{\{e\}} \to X^{\{e\}}$ is an equivalence, as desired.
\end{proof}

Now we  study a slightly more complicated situation that will be relevant for global cyclotomic spectra. Therefore let $F_1$ and $F_2$ be two commuting endomorphisms of an $\infty$-category $\calC$.  By this we mean that they come with a chosen equivalence $F_1 \circ F_2 \simeq F_2 \circ F_1$ or equivalently they define an action of the monoid $(\N \times \N,+)$ on the $\infty$-category $\calC$ where $(1,0)$ acts by $F_1$ and $(0,1)$ acts by $F_2$. Since actions of monoids on $\infty$-categories can always be strictified (e.g.~using \cite[Proposition 4.2.4.4]{HTT}) we can assume without loss of generality that $F_1$ and $F_2$ commute strictly as endomorphisms of simplicial sets.  We will make this assumption to simplify the discussion.

Then $F_2$ induces an endofunctor $\CoAlg_{F_1}(\calC) \to \CoAlg_{F_1}(\calC)$ which sends the coalgebra $X \to F_1X$ to the coalgebra $F_2X \to F_2F_1X = F_1F_2 X$ (defined as  in  Construction  \ref{construction}). This endofunctor restricts to an endofunctor of $\Fix_{F_1}(\calC) \subseteq \CoAlg_{F_1}(\calC)$. Our goal is to study the $\infty$-category 
$$\CoAlg_{F_1,F_2}(\calC) := \CoAlg_{F_2}(\CoAlg_{F_1}(\calC))\ .$$
In fact we will directly consider the situation of countably many commuting endofunctors $(F_1,F_2,\ldots)$. This can again be described as an action of the monoid $(\mathbb{N}_{>0},\cdot) \cong (\oplus_{i = 1}^\infty \mathbb{N},+)$ which we assume to be strict for simplicity. Then we define inductively
\begin{align*}
\CoAlg_{F_1,\ldots,F_n}(\calC) &:= \CoAlg_{F_n}(\CoAlg_{F_1,\ldots,F_{n-1}})(\calC) \\
\CoAlg_{(F_i)_{i \in \mathbb{N}}}(\calC) &:= \invlim \CoAlg_{F_1,\ldots,F_n}(\calC) \ .
\end{align*}
\begin{remark}
Informally an object in the $\infty$-category $\CoAlg_{(F_i)_{i \in \mathbb{N}}}(\calC)$ consists of
\begin{itemize}
\item an object $X \in \calC$; 
\item morphisms $X \to F_i(X)$ for all $i \in \mathbb{N}$; 
\item a homotopy between the two resulting maps $X \to F_i F_j(X)$ for all pairs of distinct $i, j \in \mathbb{N}$;
\item a 2-simplex between the three 1-simplices in the mapping space
\[
X \to F_iF_jF_k(X)
\]
for all triples of distinct $i,j,k \in \mathbb{N}$;
\item $\ldots$
\end{itemize}
We will not need a description of this type and therefore will not give a precise formulation and proof of this fact.
\end{remark}

If we assume that $\calC$ is presentable and all $F_i$ are accessible then the $\infty$-category $\CoAlg_{(F_i)_{i \in \mathbb{N}}}(\calC)$ is presentable. To see this we first inductively show that all $\infty$-categories $\CoAlg_{F_1,\ldots,F_n}(\calC)$ are presentable by Proposition~\ref{prop:laxeq}(iv). Then we conclude that  $\CoAlg_{(F_i)_{i \in \mathbb{N}}}(\calC)$ is presentable as a limit of presentable $\infty$-categories along colimit preserving functors. Moreover this shows that forgetful functor to $\CoAlg_{(F_i)_{i \in \mathbb{N}}}(\calC) \to \calC$ preserves all colimits.

\begin{lemma}\label{fixpointszero}
Let $\calC$ be a presentable $\infty$-category with commuting endofunctors $F_i$. Moreover assume that all the functors $F_i$ are accessible and  preserve the terminal object of $\calC$. 
\begin{altenumerate}
\item Assume that for an object $X \in \calC$ the objects $F_i(F_j(X))$ are terminal for all distinct $i,j\in \mathbb N$. Then isomorphism classes of $(F_i)_{i \in \mathbb{N}}$-coalgebra structures on $X$ are in bijection with isomorphism classes of families of morphisms $\varphi_i: X \to F_i(X)$ for all $i$, with no further dependencies between different $\varphi_i$'s.
\item In the situation of (i), assume that $X$ has a coalgebra structure giving rise to morphisms $\varphi_i: X\to F_i(X)$ for all $i$. Then for any coalgebra $Y \in  \CoAlg_{(F_i)_{i \in \mathbb{N}}}(\calC)$ the mapping space $\Map_{\CoAlg_{(F_i)_{i \in \mathbb{N}}}(\calC)}(Y,X)$ is equivalent to the equalizer
\[
\mathrm{Eq}\big(\xymatrix{\Map_\calC(Y,X) \ar[r]<2pt> \ar[r]<-2pt> &  \prod_{i = 1}^\infty  \Map_{\calC}(Y,F_iX)}\big)\ .
\]
\item Assume that for all distinct $i,j\in \mathbb N$ and all $X\in \calC$, the object $F_i(F_j(X))$ is terminal. Then 
$\CoAlg_{(F_i)_{i \in \mathbb{N}}}(\calC)$ is equivalent to the lax equalizer
$$
\mathrm{LEq}\big( \xymatrix{\calC \ar[rr]<2pt>^-{(\id,\id,\ldots)} \ar[rr]<-2pt>_-{(F_1,F_2,\ldots)} && \prod_{i = 1}^\infty \calC} \big)\ .
$$
\end{altenumerate}
\end{lemma}

\begin{proof} We prove the parts (i) and (ii) by induction on the number of commuting endomorphisms (or rather the obvious analogue of the first two statements for a finite number of endofunctors). Then part (iii) follows from parts (i) and (ii) and Proposition~\ref{prop:laxeq}.

Thus, fix $X\in \calC$ such that $F_i(F_j(X))$ is terminal for all distinct $i,j\in\mathbb N$. For a single endomorphism, the result follows from Proposition~\ref{prop:laxeq}~(ii). Thus assume that parts (i) and (ii) have been shown for endomorphisms $F_1,\ldots,F_n$ and we are given an additional endomorphism $F_{n+1}$. We know by induction that a refinement of $X$ to an object of $\CoAlg_{(F_1,\ldots,F_n)}(\calC)$ is given by choosing maps $X \to F_i(X)$ for $i \leq n$. Now a refinement to an object of $\CoAlg_{(F_1,\ldots,F_n,F_{n+1})}(\calC) \simeq \CoAlg_{F_{n+1}}(\CoAlg_{(F_1,\ldots,F_n)}(\calC))$ consists of a map $X \to F_{n+1}(X)$ of objects in $\CoAlg_{(F_1,\ldots,F_n)}(\calC)$. We observe that $F_{n+1}(X)$ also satisfies the hypothesis of (i). Invoking part (ii) of the inductive hypothesis and the fact that $F_i(F_{n+1}(X))$ is terminal for $i\leq n$, we deduce that a map $X\to F_{n+1}(X)$ in $\CoAlg_{(F_1,\ldots,F_n)}(\calC)$ is equivalent to a map $X \to F_{n+1}(X)$ in $\calC$. This shows part (i).

To prove part (ii), we use the formula for mapping spaces in a lax equalizer given in Proposition \ref{prop:laxeq}:
the space of maps for any coalgebra
\[
Y\in \CoAlg_{F_{n+1}}(\CoAlg_{(F_1,\ldots,F_n)}(\calC))
\]
to $X$ is given by the equalizer
\[
\mathrm{Eq}\Big(\xymatrix{\Map_{\CoAlg_{(F_1,\ldots,F_n)}(\calC)}(Y,X) \ar[r]<2pt> \ar[r]<-2pt> &   \Map_{\CoAlg_{(F_1,\ldots,F_n)}(\calC)}(Y,F_{n+1}(X))}\Big)\ .
\]
By part (ii) of the inductive hypothesis the first term is given by the equalizer
\[
\mathrm{Eq}\Big(\xymatrix{\Map_\calC(Y,X) \ar[r]<2pt> \ar[r]<-2pt> &  \prod_{i = 1}^n  \Map_{\calC}(Y,F_i(X))}\Big)
\]
and similarly the second term by
\[
\mathrm{Eq}\Big(\xymatrix{\Map_\calC(Y,F_{n+1}(X)) \ar[r]<2pt> \ar[r]<-2pt> &  \prod_{i = 1}^n  \Map_{\calC}(Y,F_i(F_{n+1}(X)))}\Big) \simeq \Map_\calC(Y,F_{n+1}(X))
\]
where we have used that $F_i(F_{n+1}(X))$ is terminal for all $i \leq n$. As a result we get that the mapping space in question is given by an iterated equalizer which is equivalent to 
\[
\mathrm{Eq}\Big(\xymatrix{\Map_\calC(Y,X) \ar[r]<2pt> \ar[r]<-2pt> &  \prod_{i = 1}^{n+1}  \Map_{\calC}(Y,F_i(X))}\Big)\ ,
\]
which finishes part (ii) of the induction.

This induction shows parts (i) and (ii) for a finite number of endomorphisms. Now the claim for a countable number follows by passing to the limit and noting that in a limit of $\infty$-categories the mapping spaces are given by limits as well.
\end{proof}

Now we define 
\[
\Fix_{(F_i)_{i \in \mathbb{N}}}(\calC)  \subseteq \CoAlg_{(F_i)_{i \in \mathbb{N}}}(\calC)
\]
to be the full subcategory consisting of the objects for which all the morphisms $X \to F_i(X)$ are equivalences.
 
\begin{lemma}\label{lemequiv}
Let $\calC$ be an $\infty$-category with commuting endofunctors  $(F_1,F_2,\ldots)$. Then we have an equivalence
$$\Fix_{(F_i)_{i \in \mathbb{N}}}(\calC) \simeq \calC^{h\mathbb{N}_{> 0}}$$
\end{lemma}
\begin{proof}
We have that 
$$\calC^{h\mathbb{N}_{> 0}} \simeq \calC^{h\mathbb{N}^\infty} \simeq \invlim\big( \calC^{h\N^k}\big)$$ 
and also that $\calC^{h\mathbb{N}^{k+1}} \simeq \big(\calC^{h\mathbb{N}^k}\big)^{h\mathbb{N}}$. Using this decomposition and induction we can reduce the statement to showing that for a single endomorphism $\Fix_{F}(\calC) \simeq \calC^{h\mathbb{N}}$. But this follows since the homotopy fixed points for $\N$ are the same as the equalizer of the morphism $\calC \to \calC$ given by acting with $1 \in \mathbb{N}$ and the identity. This can for example be seen using that $\N$ is the free $A_\infty$-space on one generator.
\end{proof}
 
Now we assume that all the commuting endofunctors $F_i: \calC \to \calC$ are colimit-preserving and that $\calC$ is presentable. Then the $\infty$-category 
$$
\Fix_{(F_i)_{i \in \mathbb{N}}}(\calC)  \subseteq \CoAlg_{(F_i)_{i \in \mathbb{N}}}(\calC)
$$
is presentable since it is a limit of presentable $\infty$-categories along left adjoint funtors by Proposition \ref{lemequiv}. Moreover the inclusion into $\CoAlg_{(F_i)_{i \in \mathbb{N}}}$ preserves all colimits. Thus it admits a right adjoint. Our goal is to study this right adjoint. The idea is to factor it into inclusions
$$
\iota_n: \Fix_{F_1,\ldots,F_n}(\CoAlg_{F_{n+1},F_{n+2},\ldots}) \subseteq
\Fix_{F_1,\ldots,F_{n-1}}(\CoAlg_{F_n,F_{n+1},\ldots}) 
$$
and understand the individual right adjoints $R_{\iota_n}$. To this end we note that we have an endofunctor 
$$
\bar F_n: \Fix_{F_1,\ldots,F_{n-1}}(\CoAlg_{F_{n},F_{n+1},\ldots}) \to \Fix_{F_1,\ldots,F_{n-1}}(\CoAlg_{F_{n},F_{n+1},\ldots}) 
$$
which is on underlying objects given by $X \mapsto F_n X$. This follows using Construction \ref{construction} and the isomorphism
of simplicial sets 
$$
\Fix_{F_1,\ldots,F_{n-1}}(\CoAlg_{F_{n},F_{n+1},\ldots}) \simeq \CoAlg_{F_n}(\Fix_{F_1,\ldots,F_{n-1}}(\CoAlg_{F_{n+1},F_{n+2},\ldots})).
$$
Invoking the results given at the begining of the section we also see that $\bar F_n$ admits a right adjoint $R_{\bar F_n}$ which comes with a natural transformation $\nu: R_{\bar F_n} \to \id$ and that the following is true:
\begin{lemma}\label{lemma_multilimit}
The endofunctor $\iota_n R_{\iota_n}$ is given by the limit of the following diagram of endofunctors
$$
\ldots \to R_{\bar F_n}^3 \xto{R^2_{\bar F_n} \nu} R^2_{\bar F_n} \xto{R_{\bar F_n} \nu} R_{\bar F_n} \xto{\nu} \id \ .
$$ \qed
\end{lemma}

Now as a last step it remains to give a formula for the right adjoint $R_{\bar F_n}$. To get such a formula we will make very strong assumptions (with the geometric fixed point functor $\Phi^{C_p}$ in mind). So we will assume that for all $n$, the functor $F_n: \calC  \to \calC$ admits a fully faithful right adjoint $R_{F_n}$
such that for all $i \neq n$ the canonical morphism $F_i R_{F_n} \to R_{F_n} F_i$ adjoint to $F_n F_i R_{F_n} = F_i F_n R_{F_n}\simeq F_i$ is an equivalence. Moreover, we assume that all the functors $F_i$ commute with pullbacks.

\begin{lemma}\label{lemformularcommuting} Under these assumptions the underlying object of $R_{\bar F_n}^k X$ for a given object $X \in \Fix_{F_1,\ldots,F_{n-1}}(\CoAlg_{F_{n},F_{n+1},\ldots}) $ is equivalent to 
$$ 
R_{F_n}^k X \times_{R_{F_n}^k FX} R_{F_n}^{k-1} X \times_{R_{F_{n}}^{k-1} FX}  \ldots \times_{R_{F_n} FX} X
$$
where the maps to the right are induced by the coalgebra map $X \to F_nX$ and the maps to the left are induced by the unit $\eta$ of the adjunction between  $F_n$ and $R_{F_n}$. In this description the map $R_{\bar F_n}^k X \to R^{k-1}_{\bar F_n} X$ can be described as forgetting the first factor.
\end{lemma}
\begin{proof}
We deduce this from Corollary \ref{corrightad}. To this end consider the $\infty$-category 
$$
\calD :=  \Fix_{F_1,\ldots,F_{n-1}}(\CoAlg_{F_{n+1},F_{n+2},\ldots}).
$$
Then as discussed above the $\infty$-category in question is $\CoAlg_{F_n^\calD}(\calD)$ where $F_n^\calD: \calD \to \calD$ is the canonical extension of the functor $F_n$ to $\calD$. We claim that there is a similar canonical extension $R_n^\calD: \calD \to \calD$ of $R_n$; this follows from the assumption that $R_n$ commutes with all the $F_i$ for $i \neq n$. With a similar argument one can extend the unit and counit of the adjunction. This way we deduce that the functor $R_n^\calD$ is right adjoint to $F_n^\calD$ as endofunctors of $\calD$. Moreover $R_n^\calD$ is also fully faithful since this is equivalent to the fact that the counit of the adjunction is an equivalence.

Now we can apply Corollary \ref{corrightad} to compute the structure of the right adjoint $R_{\bar F_n}$ to 
$\bar F_n: \CoAlg_{F_n}(\calD) \to \CoAlg_{F_n}(\calD)$. This then gives the formula for the underlying functor as well since the forgetful functor $\calD \to \calC$  preserves all pullbacks as a consequence of the fact that all $F_i$ do.
\end{proof}

For the application, we let $\calC$ be the $\infty$-category $\T\Sp_\F$ and consider the commuting endofunctors $F_n := \Phi^{C_{p_n}}$ where 
$p_1, p_2,\ldots $ is the list of primes. These functors commute up to coherent equivalence by the discussion before Definition \ref{def:genuinecyclo}. Without loss of generality we can assume that they strictly commute to be able to apply the preceding discussion.  We first have to verify that the assumptions 
of Lemma \ref{lemformularcommuting}  are satisfied.
 
\begin{lemma}
The canonical morphism $\Phi^{C_p} R_{C_q} \to R_{C_q}\Phi^{C_p}$ of endofunctors $\T\Sp_\F \to \T\Sp_\F$ is an equivalence for all primes $p \neq q$. 
\end{lemma}

\begin{proof}
 We  show that the morphism
 $$
 (\Phi^{C_p} R_{C_q}X)^H \to (R_{C_q}\Phi^{C_p}X)^H
 $$
 is an equivalence for all objects $X \in \T\Sp_\F$ and all  finite subgroups $H \subseteq \T$.
 
We distinguish two cases. Assume first that $C_q \not \subseteq H$. Then we immediately get that $(R_{C_q}\Phi^{C_p}X)^H \simeq 0$.
Next, we have to understand $(\Phi^{C_p} R_{C_q}X)^H$. There is a transformation of lax symmetric monoidal functors ${-}^{C_p} \to \Phi^{C_p}$.

Thus $(\Phi^{C_p} R_{C_q}X)^H$ is a module over $(R_{C_q}\mathbb S)^{\tilde H}$ where $\tilde H$ is the preimage of $H$ under the $p$-th power morphism $\T \to \T$. Since $p$ and $q$ are different and $H$ does not contain $C_q$ it follows that also $\tilde{H}$ does not contain $C_q$. Thus we get that $(R_{C_q}\mathbb S)^{\tilde H} \simeq 0$. Therefore also $(\Phi^{C_p} R_{C_q}X)^H \simeq 0$ which finishes the argument.

Note that in particular, it follows that $\Phi^{C_p} R_{C_q} X$ lies in the essential image of $R_{C_q}$. By Proposition \ref{prop:rightadjointgeometric}, it suffices to show that the natural map $\Phi^{C_p} R_{C_q} X \to R_{C_q}\Phi^{C_p} X$ is an equivalence after applying $\Phi^{C_q}$. But
\[
\Phi^{C_q}R_{C_q}\Phi^{C_p} \simeq \Phi^{C_p} \simeq \Phi^{C_p}\Phi^{C_q}R_{C_q} \simeq \Phi^{C_q}\Phi^{C_p}R_{C_q}\,
\]
as desired.
\end{proof}

\begin{thm}\label{theoremunderlyingglobal}
The inclusion 
$$
\iota: \Cyc \simeq  \Fix_{(\Phi^{C_p})_{p \in \mathbb{P}}}(\T\Sp_\F)   \subseteq  \CoAlg_{(\Phi^{C_p})_{p \in \mathbb{P}}}(\T\Sp_\F) 
$$
admits a right adjoint $R_\iota$ such that the counit $\iota R_\iota \to \id$  of the adjunction induces an equivalence on underlying spectra.
\end{thm}

\begin{proof} We let $\calC$ be the $\infty$-category $\T\Sp_\F$ and consider the commuting endofunctors $F_n := \Phi^{C_{p_n}}$ where $p_1, p_2,\ldots$ is the list of primes. We first  prove that the inclusion\[
\iota_n: \calC_{n+1} :=  \Fix_{F_1,\ldots,F_n}(\CoAlg_{F_{n+1},F_{n+2},\ldots}) \subseteq
\Fix_{F_1,\ldots,F_{n-1}}(\CoAlg_{F_{n},F_{n+1},\ldots}) := \calC_{n}
\]
admits a right adjoint $R_{\iota_n}$ such that the counit of the adjunction induces an equivalence on underlying spectra. We use the notation of Lemma \ref{lemformularcommuting} and invoke  Corollary \ref{geometricformula}  to deduce (similarly to the proof of Theorem \ref{thmunderlying}) that we have 
 \begin{align*}
(R_{\bar C_p}^k X)^{C_{p^m} \times C_r} \simeq 
 X^{C_{p^{m-k}}\times C_r} \times_{(\Phi^{C_p} X)^{C_{p^{m-k}}\times C_r}}  \times \ldots
  \times_{{(\Phi^{C_p} X)^{C_{p^{m-1}}\times C_r}}} X^{{C_{p^m}}\times C_r}
  \end{align*}
for $r$ coprime to $p$. For fixed value of $m$ and $r$ this term stabilizes in $k$. The same is true if we apply geometric fixed points to this sequence by an isotropy separation argument (or by a direct calulation that the resulting tower is equivalent to the initial tower as in the proof of Theorem \ref{thmunderlying}). Therefore we deduce that the limit of the diagram of objects in $\T\Sp_\F$
$$
\ldots \to R_{\bar F_n}^3 \xto{R^2_{\bar F_n} \nu} R^2_{\bar F_n} \xto{R_{\bar F_n} \nu} R_{\bar F_n} \xto{\nu} \id \ .
$$
commutes with all endofunctors $\Phi^{C_p}$. Thus it is the underlying object of the limit in the $\infty$-category 
$\Fix_{F_1,\ldots,F_{n-1}}(\CoAlg_{F_{n},F_{n+1},\ldots})$. Therefore Lemma \ref{lemma_multilimit} implies that this is the underlying object of the right adjoint $R_{\iota_n}$ which we try to understand. Then specializing the formula above to $r = 0$ and $m = 0$
we obtain that the counit of the adjunction between $\bar F_n$ and $R_{\bar F_n}$ induces an equivalence $(\iota_n R_{\iota_n}X)^{\{e\}} \simeq X^{\{e\}}$as desired.  

Now the right adjoint of the full inclusion
$$\iota : \invlim(\calC_n) = \Fix_{(\Phi^{C_p})_{p \in \mathbb{P}}}(\T\Sp_\F)   \subseteq  \CoAlg_{(\Phi^{C_p})_{p \in \mathbb{P}}}(\T\Sp_\F) = \calC_1 $$
is the colimit of the right adjoints 
$$
\calC_1 \xto{R_{\iota_1}} \calC_2 \xto{R_{\iota_2}} \calC_3 \xto{R_{\iota_3}} \ldots\ .
$$
For each finite subgroup $C_m\subseteq \T$, the genuine $C_m$-fixed points in this limit stabilize; namely, they are constant after passing all primes dividing $m$. This again implies that the limit is preserved by all $\Phi^{C_p}$, and so the limit can be calculated on the underlying objects in $\T\Sp_\F$, which gives the claim of the Theorem.
\end{proof}

\section{Equivalence of $\infty$-categories of cyclotomic spectra}\label{sec:proofequivalence}

We start this section by proving that  for a fixed prime $p$, our $\infty$-category of $p$-cyclotomic spectra is equivalent to the $\infty$-category of genuine $p$-cyclotomic spectra when restricted to bounded below spectra. 
The ``forgetful'' functor 
\begin{equation}\label{functor}
\Cyc_p\to \Cycn_p
\end{equation}
from Proposition~\ref{propforgetp} restricts to a functor
of bounded below objects on both sides. 
By the adjoint functor theorem (and the results about presentability and colimits of the last section) the functor \eqref{functor} admits a right adjoint functor. We will try to understand the right adjoint well enough to prove that this functor induces an equivalence of subcategories of bounded below cyclotomic spectra.

To this end we factor the forgetful functor \eqref{functor} as

$$
\xymatrix{
\Cyc_p = \Fix_{\Phi^{C_p}}(\Tp\Sp) \ar[r]^-{\iota} & \CoAlg_{\Phi^{C_p}}(\Tp\Sp) \ar[d]^{U} \\
& \CoAlg_{(-)^{tC_p}}(\Sp^{B\Tp}) = \Cycn_p.
}
$$
Here  
the first functor $\iota$ is the inclusion and 
the second functor $U$ takes the underlying naive spectrum. See also the construction \ref{propforgetp} of the forgetful functor. 
Both functors in this diagram admit right adjoints by the adjoint functor theorem, the right adjoint $R$ of the first functor was discussed in the last section and the right adjoint $B$ of the second functor $U$ will be discussed now.

Recall the notion of Borel complete spectra and the Borel completion from Theorem \ref{thm:borelcompletion}. The following lemma is the non-formal input in our proof and a consequence of the Tate orbit lemma (Lemma \ref{lem:tateorbit}).

\begin{lemma}\label{lemborel}
Let $X$ be a Borel complete object in $\Tp\Sp$ whose underlying spectrum is bounded below. Then the object $\Phi^{C_p}X \in \Tp\Sp$ is also Borel complete. In particular the canonical map
$$
\Phi^{C_p}B_{\Tp}Y \xto{\simeq} B_{\Tp} (Y^{tC_p})
$$
is an equivalence 
for every bounded below spectrum $Y$ with $\Tp$-action.
\end{lemma}

\begin{proof}
As $X$ is bounded below we get from Proposition~\ref{square} a pullback square
\[\xymatrix{
 X^{C_{p^{n}}}\ar[r]\ar[d] & (\Phi^{C_p} X)^{C_{p^{n-1}}}\ar[d] \\
 X^{hC_{p^{n}}} \ar[r]&  (X^{tC_p})^{hC_{p^{n-1}}}
}\]
where we want to remind the reader of the discussion following Proposition \ref{prop:hesselholtmadsen} for the definition of the right hand map. 
To show that $\Phi^{C_p}X$ is Borel complete we have to verify that the right hand map is an equivalence. This now follows since the left hand side is an equivalence by assumption.
\end{proof}

\begin{lemma}\label{lemborel3}
Let $(X,\varphi_p) \in \Cycn_{p}$ be a bounded below cyclotomic spectum. Then the right adjoint $B: \Cycn_p \to \CoAlg_{\Phi^{C_p}}(\Tp\Sp)$  of the forgetful functor applied to $(X,\varphi_p)$ is given by the Borel complete spectrum $B_{\Tp}X \in \Tp\Sp$ with the coalgebra structure map 
$$
\B_{\Tp}\varphi_p\colon B_{\Tp}X \to B_{\Tp} (X^{tC_p}) \simeq \Phi^{C_p}(B_{\Tp}X).
$$
The counit $UB(X,\varphi) \to (X,\varphi)$ is an equivalence.
\end{lemma}

\begin{proof}
We have to check the universal mapping property. Thus let $(Y, \Phi_p)$ be a $\Phi^{C_p}$-coalgebra. Then the mapping space in the $\infty$-category $\CoAlg_{\Phi^{C_p}}(\Tp\Sp)$ from $(Y, \Phi_p)$ to $(B_{\Tp}X, B_\Tp \varphi_p)$ is given by the equalizer of the diagram
\[
\xymatrix{\Map_{\Tp\Sp}(Y,B_{\Tp}X) \ar^-{(\Phi_p)_*\Phi^{C_p}}[rr]<2pt> \ar_-{(B_{\Tp} \varphi_p)_*}[rr]<-2pt> && \Map_{\Tp\Sp}(Y,B_{\Tp}X^{tC_p})}
\]
as we can see from Proposition \ref{prop:laxeq}(ii). But since the Borel functor $B_{\Tp}$ is right adjoint to the forgetful functor $\Tp\Sp \to \Sp^{B\Tp}$, this equalizer can be rewritten as the equalizer of
\[
\xymatrix{\Map_{\Sp^{B\Tp}}(Y,X) \ar^-{(\varphi_{Y,p})_*(-^{tC_p})}[rr]<2pt> \ar_-{(\varphi_p)_*}[rr]<-2pt> && \Map_{\Sp^{B\Tp}}(Y,X^{tC_p})}\ ,
\]
where we have written $(Y, \varphi_{Y,p})$ for the corresponding naive $p$-cyclotomic spectrum. This last equalizer is again by Proposition \ref{prop:laxeq}(ii) the mapping space between the cyclotomic spectra $(Y, \varphi_{Y,p})$ and $(X,\varphi_p)$ in the $\infty$-category $\Cycn_p$. This shows the claim.
\end{proof}

\begin{thm}\label{thmpequivalence}
The forgetful functor 
$\Cyc_p \to \Cycn_p$  induces an equivalence between the subcategories of those objects whose underlying non-equivariant spectra are bounded below.
\end{thm}
\begin{proof}
We have to show that the composition 
$$
U \circ \iota: \Cyc_p \to \Cycn_p
$$
is an equivalence of $\infty$-categories when restricted to full subcategories of bounded below objects. 
As we have argued before both functors $U$ as well as $\iota$ have right adjoints $B$ and $R$, thus the composition has a right adjoint. We  show that the unit and the counit of the adjunction are equivalences. First we observe that the functor $U \circ \iota$ reflects equivalences. This follows since equivalences of genuine $p$-cyclotomic spectra can be detected on underlying spectra, since they can be detected on geometric fixed points.

Thus, it is sufficient to check that the counit of the adjunction is an equivalence, i.e.~the map
$$
U\iota RB(X) \to X  
$$
is an equivalence of spectra for every bounded below spectrum $X \in \Cycn_p$. This follows from Lemma \ref{lemborel3} and Theorem \ref{thmunderlying}.
\end{proof}

\begin{remark} With the same proof one also gets an equivalence between genuine and naive $\infty$-categories of $p$-cyclotomic spectra which support a $\T$-action (as opposed to a $\Tp$-action), see Remark \ref{remarktaction}. This variant of the genuine $p$-cyclotomic $\infty$-category is in fact equivalent to the underlying $\infty$-category of the model-$\ast$-category of $p$-cyclotomic spectra considered by Blumberg-Mandell, as shown by Barwick-Glasman \cite{BarwickGlasman}  (cf.~Theorem \ref{thm:orthocyclo} above). We have decided for the slightly different $\infty$-category of $p$-cyclotomic spectra since the $\Tp$-action is sufficient to get $\TC(-,p)$ and this avoids the completion issues that show up in the work of Blumberg-Mandell \cite{BlumbergMandell}.
\end{remark}

\begin{remark}
With the same methods (but much easier since Lemma \ref{lemborel} is almost a  tautology in this situation) we also get an unstable statement, namely that a genuine  $\Tp$-space $X$ with an equivalence of $\Tp$-spaces $X \xto{\sim} X^{C_p}$ is essentially the same as a naive $\Tp$-space together with a $\Tp$-equivariant map $X \to X^{hC_p}$.
\end{remark}

Now we  prove the global analogue of the statement above, namely that  our $\infty$-category of  cyclotomic spectra is equivalent to the $\infty$-category of genuine cyclotomic spectra when restricted to bounded below spectra. Similar to the $p$-primary case we consider the subcategories of $\Cyc$ and $\Cycn$ 
of objects whose underlying spectra are bounded below. The ``forgetful'' functor 
\begin{equation}\label{functor2}
\Cyc \to \Cycn
\end{equation}
from Proposition~\ref{propforget} takes bounded below objects to bounded below objects. 
By the adjoint functor theorem and the results of the last section the forgetful functor \eqref{functor2} admits a right adjoint functor. 
We factor the functor \eqref{functor2} as
$$
\xymatrix{
\Cyc = \Fix_{(\Phi^{C_p})_{p \in \mathbb{P}}}(\T\Sp_\F) \ar[r]^-{\iota} & \CoAlg_{(\Phi^{C_p})_{p \in \mathbb{P}}}(\T\Sp_\F) \ar[d]^{U} \\
& \CoAlg_{(-^{tC_p})_{p \in \mathbb{P}}}(\Sp^{B\T}) \simeq \Cycn.
}
$$
Here  
the first functor $\iota$ is the inclusion and 
the second functor $U$ takes the underlying naive spectrum. The equivalence $ \CoAlg_{(-^{tC_p})_{p \in \mathbb{P}}}(\Sp^{B\T}) \simeq \Cycn$ follows from Lemma \ref{fixpointszero}. We now want to understand the right adjoint functor of the functor $U$.

\begin{lemma}\label{Borel}
Let $X$ be a Borel complete object in $\T\Sp_\F$ whose underlying spectrum is bounded below. Then for every prime $p$ the spectrum $\Phi^{C_p}X \in \T\Sp_\F$ is also Borel complete. 
\end{lemma}

\begin{proof} We know that the underlying spectrum with $\T$-action of $\Phi^{C_p}X$ is given by $X^{tC_p} \in \Sp^{B\T}$. As in Theorem \ref{thm:borelcompletion} we consider the associated Borel complete spectrum  $B(X^{tC_p}) \in \T\Sp_\F$. There is a morphism $\Phi^{C_p}X \to B(X^{tC_p})$ as the unit of the adjunction. We need to show that this map is an equivalence, i.e.~that it is an equivalence on all geometric fixed points for finite subgroups $H \subseteq \T$. For $H$ a cyclic $p$-group this has already been done in Lemma \ref{lemborel}. Thus we can restrict attention to subgroups $H$ that have $q$-torsion for some prime $q \neq p$. In this case it follows from the next lemma that $\Phi^H\Phi^{C_p}X \simeq \Phi^{\tilde{H}} X \simeq 0$ where $\tilde H = \{ h \in \T \mid h^p \in H\}$.  

If $X$ is an algebra object, it follows from  Corollary \ref{cor:borelcompletionlaxsymm} that the map $\Phi^{C_p}X \to B(X^{tC_p})$ is a map of algebras. Since geometric fixed points are lax symmetric monoidal we get that also the map
$$
\Phi^H\Phi^{C_p}X \to \Phi^H B(X^{tC_p})
$$
is a map of algebras. Since the source is zero this also implies that the target is zero. Now since every Borel complete spectrum $X$ is a module over the Borel complete sphere it follows that for all Borel complete spectra $X$  the spectrum $\Phi^H B(X^{tC_p})$ is zero and thus the claim.
\end{proof}

\begin{lemma}\label{vanishingBorel}
Let $X$ be a Borel complete $G$-spectrum for some finite group $G$ which is not a $p$-group for some prime $p$. Then $\Phi^G(X) \simeq 0$. 
\end{lemma}

\begin{proof} As $X$ is a module over the Borel complete sphere spectrum, we can assume that $X$ is the Borel complete sphere spectrum. In this case, $\Phi^G X$ is an $\E_\infty$-algebra, and it suffices to see that $1=0\in \pi_0 \Phi^G X$. There is a map of $\E_\infty$-algebras $\mathbb S^{hG} = X^G\to \Phi^G X$, and there are norm maps $X^H\to X^G$ for all proper subgroups $H\subsetneq G$ whose composite $X^H\to X^G\to \Phi^G X$ is homotopic to $0$. Let $I\subseteq \pi_0 \mathbb S^{hG}$ be the ideal generated by the images of the norm maps $\pi_0 X^H\to \pi_0 X^G$. It suffices to see that $I$ contains a unit. For this, note that there is a natural surjective map
\[
\pi_0 \mathbb S^{hG}\to \pi_0 \mathbb Z^{hG} = \mathbb Z
\]
whose kernel lies in the Jacobson radical. More precisely, we can write
\[
\mathbb S^{hG} = \lim_n (\tau_{\leq n} \mathbb S)^{hG}\ ,
\]
and the map
\[
\pi_0 (\tau_{\leq n} \mathbb S)^{hG}\to \mathbb Z
\]
is surjective with nilpotent and finite kernel. By finiteness of all $\pi_1 (\tau_{\leq n} \mathbb S)^{hG}$, we get
\[
\pi_0 \mathbb S^{hG} = \lim_n \pi_0 (\tau_{\leq n} \mathbb S)^{hG}\ .
\]
Now if an element $\alpha\in \pi_0\mathbb S^{hG}$ maps to a unit in $\mathbb Z$, then it maps to a unit in all $\pi_0 (\tau_{\leq n} \mathbb S)^{hG}$, and therefore is a unit in $\pi_0 \mathbb S^{hG}$.

Now note that the transfer maps $X^H=\mathbb S^{hH}\to X^G =\mathbb S^{hG}$ sit in commutative diagrams
\[\xymatrix{
\pi_0 \mathbb S^{hH}\ar[r]\ar[d] & \pi_0\mathbb S^{hG}\ar[d]\\
\mathbb Z\ar[r]^{[G:H]} & \mathbb Z\ .
}\]
If $G$ is not a $p$-group, then (by the existence of $p$-Sylow subgroups) the ideal of $\mathbb Z$ generated by $[G:H]$ for proper subgroups $H$ of $G$ is given by $\mathbb Z$. By the above, this implies that $I$ contains a unit, as desired.
\end{proof}

Recall that the functor 
$$
U:  \CoAlg_{(\Phi^{C_p})_{p \in \mathbb{P}}}(\T\Sp_\F) \to  \CoAlg_{(-^{tC_p})_{p \in \mathbb{P}}}(\Sp^{B\T}) \simeq \Cycn
$$
has a right adjoint. The next lemma is analogous to Lemma \ref{lemborel3}.

\begin{lemma}\label{lemborel2} Let $X \in \Cycn$ be a bounded below cyclotomic spectrum. Then the right adjoint $B: \Cycn\to \CoAlg_{(\Phi^{C_p})_{p \in \mathbb{P}}}(\T\Sp_\F)$ applied to $X$ has underlying object given by the Borel complete spectrum $B_{\T}X \in \T\Sp_\F$, and the counit $UB(X,\varphi) \to (X,\varphi)$ is an equivalence.
\end{lemma}

\begin{proof} Consider the Borel complete spectrum $B_\T X$ with the maps\[
B_\T X \to B_\T(X^{tC_p}) \simeq \Phi^{C_p}B_\T X\ ,
\]
where we have used Lemma \ref{Borel}. We also note that we have $\Phi^{C_p} \Phi^{C_q}B_\T X \simeq 0$ for distinct primes $p,q$ by Lemma~\ref{vanishingBorel}. Thus by Lemma~\ref{fixpointszero}~(i) this equips $B_\T X$ with the structure of an object in $\CoAlg_{(\Phi^{C_p})_{p \in \mathbb{P}}}(\T\Sp_\F)$. Moreover for any other object $Y$ in $\CoAlg_{(\Phi^{C_p})_{p \in \mathbb{P}}}(\T\Sp_\F)$ we get by Lemma~\ref{fixpointszero}~(ii) the equivalence of mapping spaces
\begin{align*}
&\Map_{\CoAlg_{(\Phi^{C_p})_{p \in \mathbb{P}}}(\T\Sp_\F)}(Y,B_\T X) \\
& \simeq
\mathrm{Eq}\Big(\xymatrix{\Map_{\T\Sp_\F}(Y,B_\T X) \ar[r]<2pt> \ar[r]<-2pt> &  \prod_{p\in \bP}  \Map_{\T\Sp_\F}(Y,B_\T(X^{tC_p}))}\Big) \\
& \simeq \mathrm{Eq}\Big(\xymatrix{\Map_{\Sp^{B\T}}(UY,X) \ar[r]<2pt> \ar[r]<-2pt> &  \prod_{p\in \bP}  \Map_{\Sp^{B\T}}(UY,X^{tC_p})}\Big) \\
& \simeq \Map_{\Cycn}(UY,X). 
\end{align*}
\end{proof}

\begin{thm}\label{theoremintegralequivalence}
The forgetful functor $\Cyc \to \Cycn$  induces an equivalence between the subcategories of those objects whose underlying non-equivariant spectra are bounded below.
\end{thm}

\begin{proof}
As in the the proof of Theorem \ref{thmpequivalence} this follows from Theorem \ref{theoremunderlyingglobal} and  Lemma \ref{lemborel2}.
\end{proof}

\begin{remark}\label{integralTC}
From Theorem \ref{theoremintegralequivalence} and the definition $\TC(X) = \map_{\Cycn}(\bS, X)$ we deduce that $\TC(X) \simeq \map_{\Cyc}(\bS, X)$ for a bounded below genuine cyclotomic spectrum $X$. We also know that this is equivalent to Goodwillie's integral TC by Theorem \ref{computeTC}. Since the bounded below part of the $\infty$-category $\Cyc$ is equivalent to the bounded below part of the $\infty$-category underlying the model-$\ast$-category of Blumberg and Mandell by Theorem \ref{thm:orthocyclo} we deduce that the mapping spectrum in their category is also equivalent to Goodwillie's integral TC. Blumberg and Mandell have only shown this equivalence after $p$-completion and not integrally. In this sense our result refines their result. It would be interesting to see a proof of this fact in the language of \cite{BlumbergMandell}. We have been informed that such a discussion will be given in forthcoming work of Calvin Woo. 
\end{remark}

\chapter{Topological Hochschild Homology}\label{ch:thh}

In this chapter, we discuss the construction of topological Hochschild homology as a cyclotomic spectrum.

We start by giving our new construction in Sections~\ref{sec:tatediag},~\ref{sec:thhnaive} and~\ref{sec:tatediagfunc}. More precisely, in Section~\ref{sec:tatediag} we introduce the Tate diagonal of Theorem~\ref{thm:introtatediag}, which we also prove in that section. In Section~\ref{sec:thhnaive}, it is explained how this gives the construction of the cyclotomic structure maps, once one has a finer functoriality of the Tate diagonal, in particular that it is lax symmetric monoidal. These finer functorialities are obtained in Section~\ref{sec:tatediagfunc} as an application of strong uniqueness results from \cite{Nik}.

Afterwards, we recall the classical definition of $\THH$ in terms of genuine equivariant homotopy theory. This uses critically the B\"okstedt construction, which is analyzed in detail in Section~\ref{sec:boekstedt}, where we collect various results from the literature, in particular that it is equivalent to the usual smash product, and that it interacts well with the geometric fixed points functor. In Section~\ref{sec:thhorth}, this is used to give the classical construction of $\THH(A)$ as an orthogonal cyclotomic spectrum. In the final Section~\ref{sec:comparison}, we prove that our new construction is equivalent to the classical construction.

\section{The Tate diagonal}\label{sec:tatediag}

Our new definition of the cyclotomic structure on topological Hochschild homology needs a certain construction that we call the Tate diagonal, and which we define in this section. Recall that there is no diagonal map in the $\infty$-category of spectra, i.e.~for a general spectrum $X$ there is no map $X \to X \otimes X$ which is symmetric, i.e.~factors through the homotopy fixed points $(X \otimes X)^{hC_2} \to X \otimes X$. Of course, if $X$ is a suspension spectrum $\Sigma^\infty_+Y$ for a space $Y$, then there is such a map induced by the diagonal map of spaces $\Delta: Y \to Y \times Y$.

We  give a substitute for this map in the stable situation, which we call the Tate diagonal. This has been considered at various places in the literature before, cf.~e.g.~\cite[Section 10]{Klein}, \cite{LNR}, \cite{Heuts}. As an input for the construction we need the following result, variants of which have been first observed (in a slightlty different form and language since Tate spectra were not invented) by Jones and Wegmann \cite{MR720798}, see also the treatment of May in \cite[Chapter II.3]{MR836132} specifically Theorem 3.4. A modern reference in the precise form that we need is {\cite[Proposition 2.2.3]{DAG13}}.

\begin{proposition}\label{prop:tatediagonaladditive} Let $p$ be a prime. The functor $T_p: \Sp\to \Sp$ taking a spectrum $X\in \Sp$ to
\[
(X\otimes\ldots\otimes X)^{tC_p}
\]
is exact, where $X\otimes\ldots\otimes X$ denotes the $p$-fold self tensor product with the $C_p$-action given by cyclic permutation of the factors.
\end{proposition}

\begin{proof} Let us first check the weaker statement that $T_p$ preserves sums. Thus we compute
\begin{align*} 
T_p(X_0 \oplus X_1)  & \simeq   \left(\bigoplus_{ (i_1,\ldots,i_p) \in   \{0,1\}^p }  X_{i_1} \otimes \ldots \otimes X_{i_p} \right)^{tC_p} \\
&\simeq  T_p(X_0)\oplus T_p(X_1)\oplus \bigoplus_{[i_1,\ldots,i_p]} \left(\bigoplus_{(i_1,\ldots,i_p)\in [i_1,\ldots,i_p]} X_{i_1} \otimes \ldots \otimes X_{i_p}\right)^{tC_p}
\end{align*}
where in the second sum $[i_1,\ldots,i_p]$ runs through a set of representatives of orbits of the cyclic $C_p$-action on the set $S = \{0,1\}^p \setminus \{(0,\ldots,0),(1,\ldots,1)\}$. As $p$ is prime, these orbits are all isomorphic to $C_p$. Thus, each summand
\[
\bigoplus_{(i_1,\ldots,i_p)\in [i_1,\ldots,i_p]} X_{i_1} \otimes \ldots \otimes X_{i_p}
\]
is a $C_p$-spectrum which is induced up from the trivial subgroup $\{1\}\subseteq C_p$. But on induced $C_p$-spectra the Tate construction vanishes. Thus the projection to the first two summands is an equivalence from $T_p(X_0\oplus X_1)$ to $T_p(X_0) \oplus T_p(X_1)$.

In general, it suffices to check that $T_p$ commutes with extensions. Now, if $X_0\to \tilde{X}\to X_1$ is any fiber sequence (i.e., exact triangle) in $\Sp$, then one gets a corresponding filtration of $\tilde{X}\otimes\ldots\otimes \tilde{X}$ whose successive filtration steps are given by $X_0\otimes\ldots\otimes X_0$,
\[
\bigoplus_{(i_1,\ldots,i_p)\in [i_1,\ldots,i_p]} X_{i_1} \otimes \ldots \otimes X_{i_p}
\]
for varying $[i_1,\ldots,i_p]$ in $S/C_p$ (ordered by their sum $i_1+\ldots+i_p$), and $X_1\otimes\ldots\otimes X_1$. As $-^{tC_p}$ is an exact operation and kills all intermediate steps, we see that
\[
(X_0\otimes\ldots\otimes X_0)^{tC_p}\to (\tilde{X}\otimes\ldots\otimes\tilde{X})^{tC_p}\to (X_1\otimes\ldots\otimes X_1)^{tC_p}
\]
is again a fiber sequence, as desired.
\end{proof}

\begin{proposition}\label{prop:spaceoftrafos} Consider the full subcategory $\Fun^\Ex(\Sp,\Sp)\subseteq \Fun(\Sp,\Sp)$ of exact functors. Let $\mathrm{id}_{\Sp}\in \Fun^\Ex(\Sp,\Sp)$ denote the identity functor.

For any $F\in \Fun^\Ex(\Sp,\Sp)$, evaluation at the sphere spectrum $\bS\in \Sp$ induces an equivalence
\[
\Map_{\Fun^\Ex(\Sp,\Sp)}(\mathrm{id}_{\Sp},F)\to \Hom_{\Sp}(\bS,F(\bS))=\Omega^\infty F(\bS)\ .
\]
\end{proposition}

Here, $\Omega^\infty: \Sp\to \calS$ denotes the usual functor from spectra to spaces.

\begin{proof} As in~\cite[Proposition 6.3]{Nik}, this is a consequence of~\cite[Corollary 1.4.2.23]{HA} and the Yoneda lemma. Namely, by~\cite[Corollary 1.4.2.23]{HA}, $\Fun^\Ex(\Sp,\Sp)$ is equivalent to $\Fun^\Lex(\Sp,\calS)$ via composition with $\Omega^\infty$. Here $\Fun^\Lex(\Sp,\calS) \subseteq \Fun(\Sp,\calS) $ is the full subcategory of left exact, i.e.~finite limit preserving, functors. But maps in the latter from a corepresentable functor are computed by the Yoneda lemma.
\end{proof}

\begin{corollary}\label{cor:spaceoftatediag} The space of natural transformations from $\mathrm{id}_{\Sp}$ to $T_p$ as functors $\Sp\to \Sp$ is equivalent to $\Omega^\infty T_p(\bS)=\Hom_{\Sp}(\bS,T_p(\bS))$.$\hfill \Box$
\end{corollary}

\begin{definition}\label{def:tatediag} The Tate diagonal is the natural transformation
\[
\Delta_p: \mathrm{id}_{\Sp}\to T_p\ : X\to (X\otimes \ldots \otimes X)^{tC_p}
\]
of endofunctors of $\Sp$ which under the equivalence of Proposition~\ref{prop:spaceoftrafos} corresponds to the map
\[
\bS\to T_p(\bS) = \bS^{tC_p}
\]
which is the composition $\bS\to \bS^{hC_p}\to \bS^{tC_p}$.
\end{definition}

\begin{remark} The spectrum $(X \otimes \ldots \otimes X)^{tC_p}$ and the map $X\to (X\otimes\ldots\otimes X)^{tC_p}$ play a crucial role in the paper \cite{LNR} by Lun\o e-Nielsen and Rognes and are called the \emph{topological Singer construction} there. Also, the construction of the map $X \to (X \otimes \ldots \otimes X)^{tC_p}$ is essentially equivalent to the construction of the Hill-Hopkins-Ravenel norm for $C_p$, see \cite[Appendix A.4]{HillHopkinsRavenel},  since the fixed points of the norm $N_e^{C_p}X$ and in fact the full genuine $C_p$-equivariant spectrum can be recovered from the isotropy separation square
$$
\xymatrix{
\big(N_{e}^{C_p}X\big)^{C_p} \ar[r]\ar[d] & X \ar[d]^{\Delta_p}\\
(X \otimes \ldots \otimes X)^{hC_p} \ar[r] &  (X\otimes \ldots \otimes X)^{tC_p}
}
$$
see e.g.~\cite[Theorem 6.24]{MR3570153} or \cite[Example 3.26]{glasman2015stratified}.
\end{remark}

\begin{remark}\label{remark:segal} The Segal conjecture, proved in this version by Lin for $p=2$ \cite{Lin} and Gunawardena for odd $p$ \cite{gunawardena}, implies that the map
\[
\Delta_p(\bS): \bS\to T_p(\bS)=\bS^{tC_p}
\]
realizes $\bS^{tC_p}$ as the $p$-adic completion of $\bS$. In particular, $\bS^{tC_p}$ is connective, which is in stark contrast with $H\Z^{tC_p}$. Contemplating the resulting spectral sequence
\[
\widehat{H}^i(C_p,\pi_{-j} \bS)\Rightarrow \pi_{-i-j} \bS^\wedge_p
\]
implies that there is an infinite number of differentials relating the stable homotopy groups of spheres in an intricate pattern. Note that in fact, this is one base case of the Segal conjecture, which for $p=2$ predates the general conjecture, cf.~\cite{MR0339178}, and to which the general case is reduced in \cite{Carlsson}.
\end{remark}

Following up on the previous remark, there is the following generalization of the Segal conjecture for $C_p$.

\begin{thm}\label{thm:generalsegal} Let $X\in\Sp$ be bounded below. Then the map
\[
\Delta_p: X\to (X\otimes \ldots \otimes X)^{tC_p}
\]
exhibits $(X\otimes \ldots \otimes X)^{tC_p}$ as the $p$-completion of $X$.
\end{thm}

If $X=\bS$ is the sphere spectrum, this is the Segal conjecture (for the group $C_p$). In \cite{LNR}, Theorem~\ref{thm:generalsegal} is proved if all homotopy groups of $X$ are finitely generated. Here, we simply observe that actually, the general version of the theorem follows essentially formally from the base case $X=H\Fp$. The base case $X=H\Fp$ is a difficult calculation with an Adams type spectral sequence, and we cannot offer new insight into this part. The case $X = H\mathbb{F}_2$ had also been treated by Jones and Wegmann \cite[Theorems 3.1 and 3.2]{MR720798}.

Note that the statement is false if we drop the bounded below condition. For example $(\KU \otimes \KU)^{tC_2}$ is easily seen to be rational, e.g.~using that it is a module over the $C_2$-Tate spectrum of the $\KU$-local sphere with trivial action. The latter is rational since in the $K(1)$-local setting all Tate spectra vanish, see Greenlees-Sadofsky \cite{MR1400199} or Hovey-Sadofsky \cite{MR1343699}.

It would be interesting to find a corresponding generalization of the Segal conjecture for any finite group $G$ in place of $C_p$. In the case of $G=C_{p^n}$, this should follow from Theorem~\ref{thm:generalsegal} and Corollary~\ref{cor:segalinduct}.

\begin{proof} By shifting, we may assume that $X$ is connective. We claim that both sides are the limit of the values at $\tau_{\leq n} X$. For the left side, this is clear, and for the right side, we have to prove that\[
(X\otimes \ldots\otimes X)_{hC_p}\to \lim\nolimits_n (\tau_{\leq n} X\otimes\ldots\otimes \tau_{\leq n} X)_{hC_p}
\]
and
\[
(X\otimes \ldots\otimes X)^{hC_p}\to \lim\nolimits_n (\tau_{\leq n} X\otimes\ldots\otimes \tau_{\leq n} X)^{hC_p}
\]
are equivalences. For the first, we note that the map
\[
X\otimes\ldots\otimes X\to \tau_{\leq n}X\otimes \ldots\otimes \tau_{\leq n} X
\]
is $n$-connected (as $X$ is connective), and thus so is the map on homotopy orbits; passing to the limit as $n\to \infty$ gives the result. For the second, it suffices to show that
\[
X\otimes \ldots\otimes X\to\lim\nolimits_n \tau_{\leq n}X\otimes\ldots\otimes \tau_{\leq n} X
\]
is an equivalence, as homotopy fixed points commute with limits. But as noted before, the map
\[
X\otimes \ldots\otimes X\to \tau_{\leq n}X\otimes\ldots\otimes \tau_{\leq n} X
\]
is $n$-connected, so in the limit as $n\to\infty$, one gets an equivalence.

Therefore, we may assume that $X$ is bounded, and then (as both sides are exact) that $X=HM$ is an Eilenberg-MacLane spectrum, where we can assume that $M$ is $p$-torsion free. By Lemma~\ref{lem:tatecomplete}, the right side is $p$-complete. Therefore, we need to show that $\Delta_p$ is a $p$-adic equivalence, i.e.~an equivalence after smashing with $\bS/p$. Equivalently, by exactness again, we need to prove that $\Delta_p$ is an equivalence for $X=H(M/p)$.

Now $V=M/p$ is a (possibly infinite) direct sum of copies of $\Fp$. Unfortunately, it is not clear that the functor $X\mapsto (X\otimes \ldots \otimes X)^{tC_p}$ commutes with infinite direct sums; the issue is that homotopy fixed points do not commute with infinite direct sums if the spectra are not bounded above. We argue as follows. For any $n$, consider the functor
\[
T_p^{(n)}: X\mapsto (\tau_{\leq n} (X\otimes \ldots \otimes X))^{tC_p}\ .\]
Then by Lemma~\ref{lem:tateconvergence}, the natural transformation $T_p\to \lim\nolimits_n T_p^{(n)}$ is an equivalence. Moreover, for all $n$, $T_p^{(n)}$ commutes with infinite direct sums.

Now we know (by~\cite{LNR}) that
\[
H\Fp\to T_p(H\Fp) = \lim\nolimits_n T_p^{(n)}(H\Fp)
\]
is an equivalence. We claim that for all $i\in \Z$ and $n$, $\pi_i T_p^{(n)}(H\Fp)$ is finite. Indeed,
\[
\tau_{\leq n} (H\Fp\otimes \ldots \otimes H\Fp)
\]
has only finitely many nonzero homotopy groups, each of which is finite. Thus, by looking at the Tate spectral sequence
\[
\widehat{H}^i(C_p,\pi_{-j} \tau_{\leq n}(H\Fp\otimes \ldots\otimes H\Fp))\Rightarrow \pi_{-i-j} (\tau_{\leq n}(H\Fp\otimes\ldots H\Fp))^{tC_p}\ ,
\]
we see that there are only finitely many possible contributions to each degree.

Therefore, by Mittag-Leffler, the maps
\[
\pi_i H\Fp\to \pi_i T_p(H\Fp)\to \lim\nolimits_n \pi_i T_p^{(n)}(H\Fp)
\]
are isomorphisms for all $i\in \Z$. Now we observe the following.

\begin{lemma} Let $S_n$, $n\geq 0$, be a sequence of finite sets, and assume that $S_\infty = \lim\nolimits_n S_n$ is finite. Then the map $(S_\infty)_{n\geq 0}\to (S_n)_{n\geq 0}$ is a pro-isomorphism.
\end{lemma}

\begin{proof} As $S_\infty$ is finite, the map $S_\infty\to S_n$ is injective for all sufficiently large $n$. It remains to see that given any large $n$, the map $S_m\to S_n$ factors over $S_\infty\subseteq S_n$ if $m$ is large enough. For this consider the system $(S_m\times_{S_n} (S_n\setminus S_\infty))_m$. This is a projective system of finite sets with empty inverse limit. Thus, by compactness, one of these finite sets has to be empty.
\end{proof}

Thus, we know that for all $i\in \Z$, the map
\[
\pi_i H\Fp\to (\pi_i T_p^{(n)}(H\Fp))_n
\]
is a pro-isomorphism. This statement passes to infinite direct sums, so we see that for any $\Fp$-vector space $V$, the map
\[
\pi_i HV\to (\pi_i T_p^{(n)}(HV))_n
\]
(which is a direct sum of the map above)
is a pro-isomorphism. In particular, the right-hand side is pro-constant, so that
\[
\pi_i T_p(HV)\simeq \lim\nolimits_n \pi_i T_p^{(n)}(HV)\simeq \pi_i HV\ ,\]
as desired.
\end{proof}

\begin{remark} In fact, the existence of the Tate diagonal is a special feature of spectra, and does not exist in classical algebra. One could try to repeat the construction of the Tate diagonal in the $\infty$-derived category $\calD(\Z)$ of $\Z$-modules. The Tate construction makes sense in $\calD(\Z)$ as in every other stable $\infty$-category which admits orbits and fixed points. The construction of the Tate diagonal however makes use of a universal property of the stable $\infty$-category of spectra; the analogues of Proposition~\ref{prop:spaceoftrafos} and \cite[Corollary 6.9(1)]{Nik} fail in $\calD(\Z)$. To get a valid analogue, one would need to require that the endofunctors are $H\Z$-linear, which is extra structure. The point is that the analogue of the functor $T_p$ (from Proposition \ref{prop:tatediagonaladditive} above) in $\calD(\Z)$ is not an $H\Z$-linear functor.
\end{remark}

In fact, we can prove the following no-go theorem.

\begin{thm}\label{thmnogo} Every natural transformation $C \to (C \otimes_\Z \ldots \otimes_\Z C)^{tC_p}$ of functors $\calD(\Z) \to \calD(\Z)$ induces the zero map in homology $$H_*(C) \to H_*((C \otimes_\Z \ldots \otimes_\Z C)^{tC_p}) = \hat H^{-*}(C_p;  C \otimes_\Z \ldots \otimes_\Z C). $$ In particular there is no lax symmetric monoidal transformation $C \to (C \otimes_\Z \ldots \otimes_\Z C)^{tC_p}$.
\end{thm}

\begin{proof} Let $H:\calD(\Z)\to \Sp$ be the Eilenberg-MacLane spectrum functor. Under this functor, the tensor product of chain complexes corresponds to the tensor product of spectra over $H\Z$. We deduce that the forgetful functor $H$ preserves all limits and colimits, and admits a lax symmetric monoidal structure.

We first claim that for every transformation  $\Delta_p^\Z: C \to (C \otimes_\Z \ldots \otimes_\Z C)^{tC_p}$ of chain complexes, the underlying transformation of spectra factors as
\begin{equation}\label{comtriangle}
\xymatrix{
&&  (HC \otimes \ldots \otimes HC)^{tC_p} \ar[d]^\can \\
HC \ar[rr]^-{\qquad H(\Delta_p^\Z)}\ar@{-->}[rru]^-{a\cdot \Delta_p} && H(C \otimes_\Z\ldots \otimes_\Z C)^{tC_p}.
}
\end{equation}
for some $a\in \Z$. Here $\can$ is the canonical lax symmetric monoidal structure map. To see this claim we observe as before that the functor $\calD(\Z) \to \Sp$ given by $C \mapsto  H\left( (C \otimes_\Z \ldots \otimes_\Z C)^{tC_p} \right)$ is exact. The functor $H: \calD(\Z) \to \Sp$ is corepresented as a stable functor by the unit $\Z[0]\in \calD(\Z)$, so by (the version for $\calD(\Z)$ of) Proposition~\ref{prop:spaceoftrafos}, cf.~\cite[Proposition 6.3]{Nik}, we conclude that the set of natural transformations $HC\to  H(C \otimes_\Z \ldots \otimes_\Z C)^{tC_p}$ is in bijection with
$$
\pi_0 H(\Z\otimes_\Z \ldots \otimes_\Z \Z)^{tC_p} \cong \bF_p.
$$
Now the map $HC \xto{\Delta_p} (HC \otimes \ldots \otimes HC)^{tC_p} \xto{\can} H(C \otimes_\Z \ldots \otimes_\Z C)^{tC_p}$ induced from the Tate diagonal induces on $\pi_0$ for $C = H\Z$ the canonical projection $\Z \to \Fp$. Thus for a given $\Delta_p^\Z$ we can choose $a\in \Z$ (well-defined modulo $p$) so that the diagram \eqref{comtriangle} commutes (up to homotopy).

Now consider the $\E_\infty$-algebra $\Z[0]\in \calD(\Z)$. Then $H\Z$ is an $\E_\infty$-ring spectrum and we have a commutative diagram of spectra\[\xymatrix{
H\Z\ar^{\!\!\!\!\!\!\!\!\!\!\!\!a\Delta_p}[rr]\ar@{=}[d] && (H\Z \otimes \ldots \otimes H\Z)^{tC_p} \ar[d]^\can \ar[rr]^-{m_{H\Z}^{tC_p}} && H\Z^{tC_p} \ar@{=}[d] \\
H\Z\ar^{\!\!\!\!\!\!\!\!\!\!\!\!H(\Delta_p^\Z)}[rr] && H (\Z \otimes _\Z\ldots \otimes_\Z \Z)^{tC_p}\ar[rr]^-{H(m_\Z^{tC_p})} &&  H\Z^{tC_p}\ .
}\]
The commutativity of the diagram now implies that the upper horizontal map $H\Z\to H\Z^{tC_p}$ of spectra is $H\Z$-linear, as the lower line comes from a map in $\calD(\Z)$. Note that every map $\Z\to \Z^{tC_p}$ in $\calD(\Z)$ factors over $\tau_{\leq 0} \Z^{tC_p}$. However, the upper horizontal map is $a$ times the Frobenius $\varphi_{H\Z}$ that will be studied in Section~\ref{sec:tatefrob} below. In particular it follows from Theorem~\ref{frobhz} that $a$ times the map $\varphi_{H\Z}: H\Z\to H\Z^{tC_p}$ does not factor over $\tau_{\leq 0} H\Z^{tC_p}$ unless $a\equiv 0\mod p$, so we can take $a=0$. 
\end{proof}

Note that the proof shows a bit more than is stated, namley that every transformation $C \to (C \otimes_\Z \ldots \otimes_\Z C)^{tC_p}$ is trivial (as a natural transformation) on underlying spectra. Now every map  $C \to D$ in $\calD(\Z)$ such that the underlying map of spectra is nullhomotopic is already nullhomotopic itself. To see this we can write $C \simeq \oplus H_i C[i]$ and $D \simeq \oplus H_iD[i]$ and thereby reduce to the case of chain complexes concentrated in a single degree. Then there are only two cases: either $C$ and $D$ are concentrated in the same degree, in which case we can detect non-triviality already on homology (and thus for the spectrum map on homotopy). If $D$ is one degree higher than $C$ then the map $f: C \to D$ is zero, precisely if the fiber $\fib(f) \to C$ admits a section. Since $\fib(f)$ is by the long exact sequence concentrated in the same degree as $C$ this can also be tested on homology. The forgetful functor from chain complexes to spectra preserves fiber sequences so this can also be seen on homotopy groups of the underlying spectrum.

Thus every transformation $C \to (C \otimes_\Z \ldots \otimes_\Z C)^{tC_p}$ as in Theorem \ref{thmnogo} is pointwise the zero transformation. We do however not know if this implies also that every transformation itself is zero as a transformation of functors.

\section{The construction of THH}\label{sec:thhnaive}

In this section, we give the direct $\infty$-categorical construction of topological Hoch\-schild homology as a cyclotomic spectrum in our sense. We will do this for associative ring spectra. One can more generally define topological Hochschild homology for spectrally enriched $\infty$-categories and variants of the constructions here give a cyclotomic structure in this generality. Also we note that one can give a more geometric approach to these structures using the language of factorization homology, which has the advantage that it generalizes to higher dimensions. However we prefer to stick to the more concrete and combinatorial description here. \\

Let $A$ be an associative algebra in the symmetric monoidal $\infty$-category $\Sp$, in the sense of the following definition.

\begin{definition}[{\cite[Definition 4.1.1.6]{HA}}]\label{def:assalg} The $\infty$-category $\Alg_{\E_1}(\Sp)$ of associative algebras, or $\E_1$-algebras, in the symmetric monoidal $\infty$-category $\Sp$ is given by the $\infty$-category of operad maps $A^\otimes: N(\Ass^\otimes)\to \Sp^\otimes$. Equivalently, it is the $\infty$-category of functors
\[
N(\Ass^\otimes)\to \Sp^\otimes
\]
over $N(\Fin_\ast)$ that carry inert maps to inert maps.
\end{definition}

We refer to Appendix~\ref{app:symmmon} for a discussion of inert maps and to Appendix~\ref{app:cyclic} for a discussion of $N\Ass^\otimes$. The final condition essentially says that for all $n\geq 1$, the image of the unique object $\langle n\rangle_\Ass$ in $\Ass^\otimes$ mapping to $\langle n\rangle\in \Fin_\ast$ is given by $(A,\ldots,A)\in \Sp^n$, where $A\in \Sp$ is the image of $\langle 1\rangle_\Ass$; this condition is spelled out more precisely in the next proposition.

The operad $\Ass^\otimes$ is somewhat complicated. One can give an equivalent and simpler definition of associative algebras.

\begin{proposition}\label{prop:assalgcomp} Consider the natural functor $\Delta^\op\to \Ass^\otimes$ from Appendix~\ref{app:cyclic}. Then restriction along this functor defines an equivalence between $\Alg_{\E_1}(\Sp)$ and the $\infty$-category of functors $A^\otimes: N(\Delta^\op)\to \Sp^\otimes$ making the diagram
\[\xymatrix{
N(\Delta^\op)\ar[r]^{A^\otimes}\ar[dr] & \Sp^\otimes\ar[d]\\
& N(\Fin_\ast)
}\]
commute, and satisfying the `Segal'-condition.\end{proposition}

By the `Segal' condition we mean the following. Let $A=A^\otimes([1])\in \Sp^\otimes_{\langle 1\rangle} = \Sp$. Consider the object $A^\otimes([n])\in \Sp^\otimes_{\langle n\rangle}\simeq \Sp^n$, which we may identify with a sequence $A_1,\ldots,A_n\in \calC$ of spectra. Then the $n$ maps $[1]\to [n]$ sending $[1]=\{0,1\}$ to $\{i-1,i\}\subseteq [n]$ for $i=1,\ldots,n$ induce maps $\rho^i: \langle n\rangle\to \langle 1\rangle$ which map all elements except for $i\in \langle n\rangle$ to the base point. The induced functor $\rho^i_!: \Sp^n=\Sp^\otimes_{\langle n\rangle}\to \Sp^\otimes_{\langle 1\rangle} = \Sp$ takes the tuple $(A_1,\ldots,A_n)$ to $A_i$ by definition. In particular, the map $A^\otimes([n])\to A^\otimes([1])=A$ induces a map $A_i = \rho^i_! A^\otimes([n])\to A$. The condition is that this map is always an equivalence.

\begin{proof} This is essentially \cite[Proposition 4.1.2.15]{HA}, cf.~\cite[Definition 2.1.2.7]{HA} for the definition of an algebra object over an operad. One needs to check that the condition imposed implies that the functor carries inert morphisms to inert morphisms in general, which is an easy verification (using that $\Sp$ is symmetric monoidal).
\end{proof}

Now, given $A\in \Alg_{\E_1}(\Sp)$, we can form the cyclic spectrum
\[
N(\Lambda^\op)\xto{V^\circ} N(\Ass^\otimes_\act)\xto{A^\otimes} \Sp^\otimes_\act\xto{\otimes} \Sp\ ,
\]
where $\otimes$ is the symmetric monoidal functor taking a sequence $(X_1,\ldots,X_n)$ of spectra to $X_1\otimes\ldots\otimes X_n$, as discussed after Proposition~\ref{prop:symenvelope} above and $V^\circ$ is the functor discussed in Proposition~\ref{prop:lambdaass} of Appendix~\ref{app:cyclic}. Note that this cyclic spectrum is just making precise the diagram
\[\xymatrix{
 \cdots \ar[r]<1.5pt>\ar[r]<-1.5pt>\ar[r]<4.5pt>\ar[r]<-4.5pt> & A\otimes A\otimes A\ar@(ul,ur)^{C_3} \ar[r]<3pt>\ar[r]\ar[r]<-3pt>  & A\otimes A\ar@(ul,ur)^{C_2} \ar[r]<1.5pt>\ar[r]<-1.5pt> & A\ .
}\]

\begin{definition}\label{def_bar_infty} For an $\E_1$-ring spectrum $A$ 
we let $\THH(A)\in \Sp^{B\T}$ be the geometric realization\footnote{For the notion of geometric realization of cyclic objects in this $\infty$-categorical setting see Proposition~\ref{prop:geomreal}.} of the cyclic spectrum
\[
N(\Lambda^\op) {\xto{V^\circ}} N(\Ass^\otimes_\act)\xto{A^\otimes} \Sp^\otimes_\act\xto{\otimes} \Sp\ .
\]
\end{definition}

Now we  give the construction of the cyclotomic structure on $\THH(A)$, i.e.~$\T/C_p\cong \T$-equivariant maps
\[
\varphi_p: \THH(A)\to \THH(A)^{tC_p}
\]
for all primes $p$. The idea is to extend the Tate diagonal to a map of cyclic spectra
\[\xymatrix{
\cdots \ar[r]<1.2ex> \ar[r]<-1.2ex>\ar[r]<0.4ex> \ar[r]<-0.4ex> & A^{\otimes 3} \ar[d]^{\Delta_p} \ar@(ul,ur)^{C_3}\ar[r]<0.8ex>\ar[r]<0ex>\ar[r]<-0.8ex> \ar[r]& A^{\otimes 2} \ar[d]^{\Delta_p}\ar@(ul,ur)^{C_2}
\ar[r]<0.4ex> \ar[r]<-0.4ex>& A \ar[d]^{\Delta_p} \\
\cdots \ar[r]<1.2ex> \ar[r]<-1.2ex>\ar[r]<0.4ex> \ar[r]<-0.4ex> & \big(A^{\otimes 3p}\big)^{tC_p} \ar@(dl,dr)_{C_3}\ar[r]<0.8ex>\ar[r]<0ex>\ar[r]<-0.8ex> \ar[r]& \big(A^{\otimes 2p}\big)^{tC_p}\ar@(dl,dr)_{C_2}
\ar[r]<0.4ex> \ar[r]<-0.4ex>& \big(A^{\otimes p}\big)^{tC_p}
}\]
where the lower cyclic object is a variant of the subdivision of the upper one.

Let us first describe this second cyclic object more precisely. Let $\Free_{C_p}$ be the category of finite free $C_p$-sets $S$, which has a natural map to the category of finite sets via $S\mapsto \overline{S}=S/C_p$. We start by looking at the simplicial subdivision, corresponding to pullback along the functor $\sd_p: \Lambda_p^\op\to \Lambda^\op: [n]_{\Lambda_p}\mapsto [np]_\Lambda$. This gives a functor
\[
N(\Lambda_p^\op)\to N(\Free_{C_p})\times_{N(\Fin)} N(\Ass^\otimes_\act)\xto{A^\otimes} N(\Free_{C_p})\times_{N(\Fin)} \Sp^\otimes_\act\ .
\]
Note that the target category consists of pairs $(S,(X_{\overline{s}})_{\overline{s}\in \overline{S}=S/C_p})$ of a finite free $C_p$-set and spectra $X_{\overline{s}}$ parametrized by $\overline{s}\in \overline{S}=S/C_p$. We need to compose this with the functor which takes
\[
(S,(X_{\overline{s}})_{\overline{s}\in \overline{S}=S/C_p})\in N(\Free_{C_p})\times_{N(\Fin)} \Sp^\otimes_\act
\]
to $\bigotimes_{s\in S} X_{\overline{s}}\in \Sp^{BC_p}$. This is the composite of the functor
\[
N(\Free_{C_p})\times_{N(\Fin)} \Sp^\otimes_\act\to (\Sp^\otimes_\act)^{BC_p}\ :\ (S,(X_{\overline{s}})_{\overline{s}\in \overline{S}})\mapsto (S,(X_{\overline{s}})_{s\in S})
\]
from Proposition~\ref{prop:funnyfunctor} below and $\otimes: (\Sp^\otimes_\act)^{BC_p}\to \Sp^{BC_p}$.

Now we have the functor
\[
N(\Lambda_p^\op)\to N(\Free_{C_p})\times_{N(\Fin)} \Sp^\otimes_\act\to (\Sp^\otimes_\act)^{BC_p}\xto{\otimes} \Sp^{BC_p}\xto{-^{tC_p}} \Sp\ ,
\]
which is $BC_p$-equivariant (using that $-^{tC_p}$ is $BC_p$-equivariant, by Theorem~\ref{thm:genfarrelltate}). Thus, it factors over a functor
\[
N(\Lambda^\op)=N(\Lambda_p^\op)/BC_p\to \Sp\ ,
\]
which gives the desired construction of the cyclic spectrum
\begin{equation}\label{subdiivision}
\xymatrix{
\cdots \ar[r]<1.2ex> \ar[r]<-1.2ex>\ar[r]<0.4ex> \ar[r]<-0.4ex> & \big(A^{\otimes 3p}\big)^{tC_p} \ar@(dl,dr)_{C_3}\ar[r]<0.8ex>\ar[r]<0ex>\ar[r]<-0.8ex> \ar[r]& \big(A^{\otimes 2p}\big)^{tC_p}\ar@(dl,dr)_{C_2}
\ar[r]<0.4ex> \ar[r]<-0.4ex>& \big(A^{\otimes p}\big)^{tC_p}\ .
}
\end{equation}

The geometric realization of this cyclic spectrum is by definition (see Proposition \ref{prop:geomreal} in Appendix \ref{app:cyclic}) the colimit of its restriction to the paracyclic category $N\Lambda_\infty$. Commuting $-^{tC_p}$ and the colimit, the geometric realization of the cyclic object \eqref{subdiivision} maps $\T$-equivariantly to $\THH(A)^{tC_p}$ where we have identified the geometric realization of the subdivision of a cylic object with the geometric realization of the initial cyclic object using the usual equivalence, see Proposition \ref{subdivisionrealization} and Proposition~\ref{prop:commutetaterealization} in Appendix \ref{app:cyclic}. It remains to construct the map
\[\xymatrix{
\cdots \ar[r]<1.2ex> \ar[r]<-1.2ex>\ar[r]<0.4ex> \ar[r]<-0.4ex> & A^{\otimes 3} \ar[d]^{\Delta_p} \ar@(ul,ur)^{C_3}\ar[r]<0.8ex>\ar[r]<0ex>\ar[r]<-0.8ex> \ar[r]& A^{\otimes 2} \ar[d]^{\Delta_p}\ar@(ul,ur)^{C_2}
\ar[r]<0.4ex> \ar[r]<-0.4ex>& A \ar[d]^{\Delta_p} \\
\cdots \ar[r]<1.2ex> \ar[r]<-1.2ex>\ar[r]<0.4ex> \ar[r]<-0.4ex> & \big(A^{\otimes 3p}\big)^{tC_p} \ar@(dl,dr)_{C_3}\ar[r]<0.8ex>\ar[r]<0ex>\ar[r]<-0.8ex> \ar[r]& \big(A^{\otimes 2p}\big)^{tC_p}\ar@(dl,dr)_{C_2}
\ar[r]<0.4ex> \ar[r]<-0.4ex>& \big(A^{\otimes p}\big)^{tC_p}
}
\]
of cyclic spectra. Equivalently, we have to construct a natural transformation of $BC_p$-equivariant functors from
\[
N(\Lambda_p^\op)\to N(\Free_{C_p})\times_{N(\Fin)} \Sp^\otimes_\act\to \Sp^\otimes_\act\xto{\otimes}\Sp
\]
to
\[
N(\Lambda_p^\op)\to N(\Free_{C_p})\times_{N(\Fin)} \Sp^\otimes_\act\to (\Sp^\otimes_\act)^{BC_p}\xto{\otimes} \Sp^{BC_p}\xto{-^{tC_p}} \Sp\ .
\]
This follows from Corollary~\ref{cor:tildetptrafo} in the next section, which gives a natural $BC_p$-equivariant natural transformation from the functor
\[
I: N(\Free_{C_p})\times_{N(\Fin)} \Sp^\otimes_\act\to \Sp^\otimes_\act\xto{\otimes}\Sp
\]
given by
\[
(S,(X_{\overline{s}\in \overline{S}=S/C_p}))\mapsto \bigotimes_{\overline{s}\in \overline{S}} X_{\overline{s}}
\]
to the functor
\[
\tilde{T}_p: N(\Free_{C_p})\times_{N(\Fin)} \Sp^\otimes_\act\to (\Sp^\otimes_\act)^{BC_p}\xto{\otimes} \Sp^{BC_p}\xto{-^{tC_p}} \Sp
\]
given by
\[
(S,(X_{\overline{s}\in \overline{S}=S/C_p}))\mapsto (\bigotimes_{s\in S} X_{\overline{s}})^{tC_p}\ .
\]
Composing this natural transformation with the functor
\[
N(\Lambda_p^\op)\to N(\Free_{C_p})\times_{N(\Fin)} \Sp^\otimes_\act
\]
induced by $A$, this finishes our construction of the cyclotomic structure on $\THH(A)$, modulo the required functoriality of the Tate diagonal that will be established in the next section.

\section{Functoriality of the Tate diagonal}\label{sec:tatediagfunc}

As explained in the last section, we need to equip the Tate diagonal $\Delta_p$ with stronger functoriality, and we also need a version in many variables. We start by noting that by Theorem~\ref{thm:tatelaxsymm}, the functor
\[
T_p:\Sp\to \Sp:X\mapsto (X\otimes\ldots\otimes X)^{tC_p}
\]
acquires a canonical lax symmetric monoidal structure. We can now also make the natural transformation $\Delta_p: \mathrm{id}_{\Sp}\to T_p$ into a lax symmetric monoidal transformation.

\begin{proposition}\label{prop:tatediaglaxsymm} There is a unique lax symmetric monoidal transformation
\[
\Delta_p: \mathrm{id}_{\Sp}\to T_p\ .
\]
The underlying transformation of functors is given by the transformation from Definition~\ref{def:tatediag}.
\end{proposition}

\begin{proof} This follows from \cite[Corollary 6.9(1)]{Nik}, which states that the identity is initial among exact lax symmetric monoidal functors from $\Sp$ to $\Sp$.
\end{proof}

Now, for the construction of the cyclotomic structure maps, we need to construct a natural transformation between the two $BC_p$-equivariant functors
\[
N(\Free_{C_p})\times_{N(\Fin)} \Sp^\otimes_\act\to \Sp\ ,
\]
which are given by
\[
I: (S,(X_{\overline{s}})_{\overline{s}\in \overline{S}})\mapsto \bigotimes_{\overline{s}\in \overline{S}} X_{\overline{s}}
\]
and
\[
\tilde{T}_p: (S,(X_{\overline{s}})_{\overline{s}\in \overline{S}})\mapsto (\bigotimes_{s\in S} X_s)^{tC_p}\ ,
\]
respectively. If one trivializes a free $C_p$-set to $S=\overline{S}\times C_p$, the transformation is given by
\[
\bigotimes_{\overline{s}\in \overline{S}} X_{\overline{s}}\to \bigotimes_{\overline{s}\in \overline{S}} (X_{\overline{s}}\otimes\ldots\otimes X_{\overline{s}})^{tC_p}\to (\bigotimes_{s\in S} X_{\overline{s}})^{tC_p}
\]
where the first map is the tensor product of $\Delta_p$ for all $\overline{s}\in \overline{S}$, and the second map uses that $-^{tC_p}$ is lax symmetric monoidal. One needs to see that this map does not depend on the chosen trivialization of $S$.

In fact, to get this natural transformation $I\to \tilde{T}_p$, we will do something stronger. Namely, we will make both functors lax symmetric monoidal, and then show that there is a unique lax symmetric monoidal transformation. For this, we have to recall a few results about lax symmetric monoidal functors.

First, cf.~Appendix~\ref{app:symmmon}, we recall that for any symmetric monoidal $\infty$-category $\calC$ (more generally, for any $\infty$-operad $\calO$), one can form a new symmetric monoidal $\infty$-category $\calC^\otimes_\act$ whose underlying $\infty$-category is given by
\[
\calC^\otimes_\act = \calC^\otimes\times_{N(\Fin_\ast)} N(\Fin)\ ,
\]
where $\Fin$ is the category of finite (possibly empty) sets, and the functor $\Fin\to \Fin_\ast$ adds a disjoint base point. The fiber of $\calC^\otimes_\act$ over a finite set $I\in \Fin$ is given by $\calC^I$. Here, the symmetric monoidal structure on $\calC^\otimes_\act$ is given by a ``disjoint union'' type operation; it takes $(X_i)_{i\in I}\in \calC^I$ and $(X_j)_{j\in J}\in \calC^J$ to $(X_k)_{k\in I\sqcup J}\in \calC^{I\sqcup J}$. There is a natural lax symmetric monoidal functor $\calC\to \calC^\otimes_\act$ whose underlying functor is the inclusion into the fiber of $\calC^\otimes_\act$ over the one-element set.

\begin{proposition}\label{prop:symenvelope} Let $\calC$ and $\calD$ be symmetric monoidal $\infty$-categories. Restriction along $\calC\subseteq \calC^\otimes_\act$ is an equivalence of $\infty$-categories
\[
\Fun_\otimes(\calC^\otimes_\act,\calD)\simeq \Fun_\lax(\calC,\calD)\ .
\]
\end{proposition}

\begin{proof} This is a special case of \cite[Proposition 2.2.4.9]{HA}.
\end{proof}

Note in particular that the identity $\calC\to \calC$ corresponds to a functor $\otimes: \calC^\otimes_\act\to \calC$, which is informally given by sending a list $(X_1,\ldots,X_n)$ of objects in $\calC$ to $X_1\otimes\ldots\otimes X_n$. Another way to construct it is by noting that $\calC^\otimes_\act\to N(\Fin)$ is a coCartesian fibration, and $N(\Fin)$ has a final object. This implies that the colimit of the corresponding functor $N(\Fin)\to \Cat_\infty$ is given by the fiber $\calC=\calC^\otimes_{\langle 1\rangle}$ over the final object $\langle 1\rangle\in N(\Fin)$, and there is a natural functor from $\calC^\otimes_\act$ to the colimit $\calC$, cf.~\cite[Corollary 3.3.4.3]{HTT}. Moreover, the functor $\otimes: \calC^\otimes_\act \to \calC$ can also be characterised as the left adjoint to the inclusion $\calC \subseteq \calC^\otimes_\act$, see Remark~\ref{remarkadjoint} below.

In the situation of Proposition~\ref{prop:symenvelope}, it will be useful to understand lax symmetric monoidal functors from $\calC^\otimes_\act$ as well. For this, we use the following statement.

\begin{lemma}\label{lem:laxsymmrightadj} Let $\calC$ and $\calD$ be symmetric monoidal $\infty$-categories. Then the canonical inclusion
\[
\Fun_\otimes(\calC^\otimes_\act, \calD)\subseteq \Fun_\lax(\calC^\otimes_\act, \calD)
\]
admits a right adjoint $R$. The composition $\Fun_\lax(\calC^\otimes_\act, \calD) \xto{R} \Fun_\otimes(\calC^\otimes_\act, \calD)\xto{\sim} \Fun_\lax(\calC,\calD)$ is given by restriction along the lax symmetric monoidal functor $\calC\to \calC^\otimes_\act$.
\end{lemma}

\begin{proof} We will use the concrete proof of \cite[Proposition 2.2.4.9]{HA} in which it is shown that the composition
$$
\Fun_\otimes(\calC^\otimes_\act, \calD) \subseteq \Fun_\lax(\calC^\otimes_\act, \calD) \xto{i^*} \Fun_\lax(\calC,\calD)
$$
is an equivalence of $\infty$-categories. To do this let us write $q: (\calC^\otimes_\act)^\otimes \to N\Fin_\ast$ for the symmetric monoidal $\infty$-category whose underlying $\infty$-category is $\calC^\otimes_\act$. Then Lurie proves the following two facts:
\begin{altenumerate}
\item
Every lax symmetric monoidal functor $\calC \to \calD$ which is considered as an object of $\Fun_{N(\Fin_\ast)}(\calC^\otimes, \calD^\otimes)$ admits a relative left Kan extension to a functor in $\Fun_{N(\Fin_\ast)}( (\calC^\otimes_\act)^\otimes,\calD^\otimes)$. 
\item
A functor in $\Fun_{N(\Fin_\ast)}( (\calC^\otimes_\act)^\otimes,\calD^\otimes)$ is symmetric monoidal precisely if it is the relative left Kan extension of its restriction to $\calC^\otimes \subseteq (\calC^\otimes_\act)^\otimes$.
\end{altenumerate}
This finishes Lurie's argument using \cite[Proposition 4.3.2.15]{HTT}. In particular we see that the inclusion $\Fun_\otimes(\calC^\otimes_\act,\calD) \subseteq \Fun_\lax(\calC^\otimes_\act,\calD)$ is equivalent to the relative left Kan extension functor
$$
i_!: \Fun_\lax(\calC, \calD) \to  \Fun_\lax(\calC^\otimes_\act,\calD).
$$
This functor is left adjoint to the restriction functor
\[
i^\ast: \Fun_\lax(\calC^\otimes_\act,\calD) \to \Fun_\lax(\calC, \calD)
\]
as an argument similar to \cite[Proposition 4.3.2.17]{HTT} shows: the composition $i^*i_!$ is equivalent to the identity giving a candidate for the unit of the adjunction. Then one uses \cite[Lemma 4.3.2.12]{HTT} to verify that the induced map on mapping spaces is an equivalence to get the adjunction property. 
\end{proof}

\begin{remark}\label{remarkadjoint} For a symmetric monoidal $\infty$-category $\calC$ the full, lax symmetric monoidal inclusion $\calC \subseteq \calC^\otimes_\act$ admits a symmetric monoidal left adjoint $L$ as we will show now. From this one can deduce that in the situation of Lemma \ref{lem:laxsymmrightadj} the right adjoint
$$
 R: \Fun_\lax(\calC^\otimes_\act, \calD) \to \Fun_\otimes(\calC^\otimes_\act, \calD) \simeq \Fun_\lax(\calC,\calD)
$$
admits a further right adjoint given by precomposition with $L$.

Let us start by proving that the underlying functor admits a left adjoint. To this end assume we are given an object $\bar c \in \calC^\otimes_\act \subseteq \calC^\otimes$. The object lies over some finite pointed set (equivalent to) $\langle n \rangle$. There is a unique active morphism $\langle n \rangle \to \langle 1 \rangle$. Now choose a coCartesian lift $f: \bar c \to c$ in $\calC^\otimes$ covering this morphism in $N\Fin_\ast$. Then the morphism $f$ lies in $\calC^\otimes_\act$ and the object $c$ lies in $\calC \subseteq \calC^\otimes_\act$.  Moreover $f$ is initial among morphisms in $\calC^\otimes_\act$ from $\bar c$ to an object in $\calC \subseteq \calC^\otimes_\act$. This follows since every such morphism has to cover the active morphism $\langle n \rangle \to \langle 1 \rangle$ and from the defining property of coCartesian lifts. As a result we find that $c$ is the reflection of $\bar c$ into $\calC \subseteq \calC^\otimes_\act$. Since this reflection exists for every $\bar c \in \calC^\otimes$ the inclusion $\calC \subseteq \calC^\otimes_\act$ admits a left adjoint $L$. On a more informal level the object $\bar c$ is given by a list $c_1,\ldots,c_n$ of objects of $\calC$. Then $c = L(\bar c)$ is given by the tensor product $c_1 \otimes \ldots \otimes c_n$. 

Now we show that the left adjoint $L$ is a symmetric monoidal localization. To this end we have to verify the assumptions of \cite[Proposition 2.2.1.9]{HA} given in \cite[Definition 2.2.1.6 and Example 2.2.1.7]{HA}. Thus we have to show that for every morphism $f: \bar c \to \bar d$ in $\calC^\otimes$ such that $Lf$ is an equivalence in $\calC$ and every object $\bar e$ the morphism $L(f \oplus \bar e)$ is an equivalence in $\calC$ where $\oplus$ is the tensor product in $\calC^\otimes_\act$. Unwinding the definitions this amounts to the following: we have that $\bar c$ is given by a list $c_1,\ldots,c_n$, $\bar d$ by a list $d_1,\ldots,d_m$ and $\bar e$ by a list $e_1,\ldots,e_r$. Then the induced morphism $L(f)$ is by assumption an equivalence $c_1 \otimes \ldots \otimes c_n \to d_1 \otimes \ldots \otimes d_m$. But then clearly also the morphism $L(f \otimes \bar e): c_1 \otimes\ldots \otimes c_n \otimes e_1 \otimes \ldots \otimes e_r \to d_1 \otimes \ldots \otimes d_m \otimes e_1 \otimes \ldots \otimes e_r$ is an equivalence in $\calC$.

We remark that the left adjoint $L \in \Fun_\otimes(\calC^\otimes_\act, \calC)$ corresponds under the equivalence $\Fun_\otimes(\calC^\otimes_\act, \calC) \xto{\sim} \Fun_\lax(\calC,\calC)$ to the identity functor.
\end{remark}

\begin{remark}\label{rem:operadversion} In fact, Proposition~\ref{prop:symenvelope} and Lemma~\ref{lem:laxsymmrightadj} hold true if $\calC$ (or rather $\calC^\otimes$) is replaced by any $\infty$-operad $\calO^\otimes$. Let us briefly record the statements, as we will need them later. As before, $\calD$ is a symmetric monoidal $\infty$-category. Let $\calO^\otimes$ be an $\infty$-operad with symmetric monoidal envelope $\calO^\otimes_\act = \calO^\otimes\times_{N(\Fin_\ast)} N(\Fin)$. There is a natural map of $\infty$-operads $\calO^\otimes\to (\calO^\otimes_\act)^\otimes$. Restriction along this functor defines an equivalence
\[
\Fun_\otimes(\calO^\otimes_\act,\calD)\simeq \Alg_\calO(\calD)\ ,
\]
where the right-hand side denotes the $\infty$-category of $\calO$-algebras in $\calD$; equivalently, of $\infty$-operad maps $\calO^\otimes\to \calD^\otimes$. Moreover, the full inclusion
\[
\Fun_\otimes(\calO^\otimes_\act,\calD)\to \Fun_\lax(\calO^\otimes_\act,\calD)
\]
admits a right adjoint given by the composition $\Fun_\lax(\calO^\otimes_\act,\calD)\to \Alg_\calO(\calD)\simeq \Fun_\otimes(\calO^\otimes_\act,\calD)$, where the first functor is restriction along $\calO^\otimes\to (\calO^\otimes_\act)^\otimes$.
\end{remark}

Recall that we want to construct a natural transformation between two $BC_p$-equivariant functors
\[
N(\Free_{C_p})\times_{N(\Fin)} \Sp^\otimes_\act\to \Sp\ ,
\]
which are given by
\[
I: (S,(X_{\overline{s}})_{\overline{s}\in \overline{S}})\mapsto \bigotimes_{\overline{s}\in \overline{S}} X_{\overline{s}}
\]
and
\[
\tilde{T}_p: (S,(X_{\overline{s}})_{\overline{s}\in \overline{S}})\mapsto (\bigotimes_{s\in S} X_s)^{tC_p}\ ,
\]
respectively. More precisely, the first functor
\[
I: N(\Free_{C_p})\times_{N(\Fin)} \Sp^\otimes_\act\to \Sp
\]
is the symmetric monoidal functor given by the composition of the projection to $\Sp^\otimes_\act$ and $\otimes: \Sp^\otimes_\act\to \Sp$.

For the second functor, we first have to construct a symmetric monoidal functor
\[
N(\Free_{C_p})\times_{N(\Fin)} \Sp^\otimes_\act\to (\Sp^\otimes_\act)^{BC_p}\ :\ (S,(X_{\overline{s}})_{\overline{s}\in \overline{S}})\mapsto(S,(X_{\overline{s}})_{s\in S})\ .
\]

\begin{proposition}\label{prop:funnyfunctor} For any symmetric monoidal $\infty$-category $\calC$ and any integer $p\geq 1$, there is a natural symmetric monoidal functor
\[
N(\Free_{C_p})\times_{N(\Fin)} \calC^\otimes_\act\to (\calC^\otimes_\act)^{BC_p}\ :\ (S,(X_{\overline{s}})_{\overline{s}\in \overline{S}})\mapsto (S,(X_{\overline{s}})_{s\in S})\ .
\]
\end{proposition}

\begin{proof} We observe that as in the proof of Lemma~\ref{lem:initialtech} below,
\[
\Fun_\otimes(N(\Free_{C_p})\times_{N(\Fin)} \calC^\otimes_\act,\calD) = \Fun(N(\Tor_{C_p}),\Fun_\lax(\calC,\calD))
\]
for any symmetric monoidal $\infty$-category $\calD$ (here applied to $\calD=(\calC^\otimes_\act)^{BC_p}$), where $\Tor_{C_p}$ is the category of $C_p$-torsors. This can be constructed as the composite of the lax symmetric monoidal embedding $\calC\to \calC^\otimes_\act$ and a functor
\[
N(\Tor_{C_p})\to \Fun_\otimes(\calE,\calE^{BC_p})=\Fun(BC_p,\Fun_\otimes(\calE,\calE))
\]
which exists for any symmetric monoidal $\infty$-category $\calE$ (here, applied to $\calE=\Sp^\otimes_\act$). For this, note that $\Fun_\otimes(\calE,\calE)$ is itself symmetric monoidal (using the Day convolution), and so the identity lifts to a unique symmetric monoidal functor from the symmetric monoidal envelope $N(\Fin)^\simeq$ of the trivial category. Restricting the functor $N(\Fin)^\simeq\to \Fun_\otimes(\calE,\calE)$ to $C_p$-torsors gives the desired symmetric monoidal functor
\[
N(\Tor_{C_p})\to \Fun(BC_p,\Fun_\otimes(\calE,\calE))\ .
\]
\end{proof}

Composing the resulting symmetric monoidal functor
\[
N(\Free_{C_p})\times_{N(\Fin)} \Sp^\otimes_\act\to (\Sp^\otimes_\act)^{BC_p}
\]
with the lax symmetric monoidal functors $\otimes: (\Sp^\otimes_\act)^{BC_p}\to \Sp^{BC_p}$ and $-^{tC_p}: \Sp^{BC_p}\to \Sp$, we get the desired lax symmetric monoidal functor
\[
\tilde{T}_p: N(\Free_{C_p})\times_{N(\Fin)} \Sp^\otimes_\act\to \Sp\ .
\]
In fact, the lax symmetric monoidal functors $I$ and $\tilde{T}_p$ are $BC_p$-equivariant for the natural action on the source (acting on the set $S$), and the trivial action on the target. This is clear for $I$, and for $\tilde{T}_p$, the only critical step is the functor $-^{tC_p}: \Sp^{BC_p}\to \Sp$, where it follows from the uniqueness results in Theorem~\ref{thm:genfarrelltate}.

Now consider the $\infty$-category of all lax symmetric monoidal functors
\[
\Fun_\lax(N(\Free_{C_p}) \times_{N(\Fin)}  \Sp^\otimes_\act,\Sp)\ .
\]
We call a functor $F: N(\Free_{C_p}) \times_{N(\Fin)} \Sp^\otimes_\act \to \Sp$ partially exact if for every $C_p$-torsor $S$ the induced functor $F(S,-): \Sp \to \Sp$, obtained by restriction to
\[
\{S\}\times_{\{\ast\}} \Sp\subseteq N(\Free_{C_p})\times_{N(\Fin)} \Sp^\otimes_\act\ ,
\]
is exact. This leads to the following lemma, which will also be critical to comparing our new construction with the old construction of $\THH$.

\begin{lemma}\label{lem:initialtech} The functor $I$ is initial among all lax symmetric monoidal functors
\[
N(\Free_{C_p})\times_{N(\Fin)} \Sp^\otimes_\act\to \Sp
\]
which are partially exact.
\end{lemma}

\begin{proof} By Lemma~\ref{lem:laxsymmrightadj}, or rather the $\infty$-operad version of Remark~\ref{rem:operadversion}, it is enough to prove the similar assertion in the $\infty$-category of symmetric monoidal and partially exact functors. Now we claim that restriction along
\[
N(\Tor_{C_p})\times \Sp\subseteq N(\Free_{C_p}) \times_{N(\Fin)} \Sp^\otimes_\act
\]
induces an equivalence of $\infty$-categories
\[
\Fun_\otimes(N(\Free_{C_p}) \times_{N(\Fin)} \Sp^\otimes_\act,\Sp)\simeq \Fun(N(\Tor_{C_p}),\Fun_\lax(\Sp,\Sp))\ ,
\]
where $\Tor_{C_p}$ is the category of $C_p$-torsors. As partially exact functors correspond to $\Fun(N(\Tor_{C_p}),\Fun^\Ex_\lax(\Sp,\Sp))$ on the right, the result will then follow from \cite[Corollary 6.9(1)]{Nik}.

By \cite[Theorem 2.4.3.18]{HA} the $\infty$-category $\Fun(N(\Tor_{C_p}),\Fun_{\lax}(\Sp,\Sp))$ is equivalent to the $\infty$-category of $\infty$-operad maps
from $N(\Tor_{C_p})^\sqcup \times_{N(\Fin_*)} \Sp^\otimes$ to $\Sp^\otimes$. Here $N(\Tor_{C_p})^\sqcup$ is the coCartesian $\infty$-operad associated to $N(\Tor_{C_p})$, cf.~\cite[Section 2.4.3]{HA}. Now, as in Proposition~\ref{prop:symenvelope}, the $\infty$-category of operad maps from any $\infty$-operad $\calO^\otimes$ to $\Sp^\otimes$ is equivalent to the $\infty$-category of symmetric monoidal functors from the symmetric monoidal envelope $\calO^\otimes_\act$ to $\Sp$, cf.~\cite[Proposition 2.2.4.9]{HA}. But the envelope of $N(\Tor_{C_p})^\sqcup \times_{N(\Fin_*)} \Sp^\otimes$ is
\[
(N(\Tor_{C_p})^\sqcup \times_{N(\Fin_*)} \Sp^\otimes)\times_{N(\Fin_\ast)} N(\Fin)\simeq N(\Free_{C_p})\times_{N(\Fin)} \Sp^\otimes_\act\ ,
\]
using the equivalence $N(\Tor_{C_p})^\sqcup\times_{N(\Fin_\ast)} N(\Fin) \simeq N(\Free_{C_p})$. Unraveling, we have proved the desired assertion.
\end{proof}

The following corollary is immediate, and finishes our construction of $\THH$ as a cyclotomic spectrum.

\begin{corollary}\label{cor:tildetptrafo}
The $BC_p$-equivariant lax symmetric monoidal functor 
\[
\tilde{T}_p: N(\Free_{C_p}) \times_{N(\Fin)}  \Sp^\otimes_{\act} \to \Sp\ :\ (S, (X_{\overline{s}})_{\overline{s}\in \overline{S}}) \mapsto \Big(\bigotimes_{s\in S} X_{\overline{s}}\Big)^{tC_p}
\]
receives an essentially unique $BC_p$-equivariant lax symmetric monoidal transformation from
\[
I: N(\Free_{C_p}) \times_{N(\Fin)}  \Sp^\otimes_{\act} \to \Sp\ :\ (S, (X_{\overline{s}})_{\overline{s}\in \overline{S}}) \mapsto \bigotimes_{\overline{s}\in \overline{S}} X_{\overline{s}}\ .
\]$\hfill \Box$
\end{corollary}

\section{The B\"okstedt construction}\label{sec:boekstedt}

The classical definition of topological Hochschild homology as an orthogonal cyclotomic spectrum relies on a specific construction of the smash product of orthogonal spectra, known as the B\"okstedt construction. In fact, the B\"okstedt construction can also be applied to any finite number of orthogonal spectra, and already the case of a single orthogonal spectrum is interesting, where it leads to a fibrant replacement functor.

B\"okstedt's construction is based on the following category.

\begin{definition} Let $\mathcal I$ be the category of finite (possibly empty) sets and injective maps.
\end{definition}

If $A$ is any orthogonal spectrum, one has a natural functor from $\mathcal I$ to pointed spaces, taking a finite set $I\cong \{1,\ldots,i\}$, $i\geq 0$, to $A_i$; here, we make use of the action of the symmetric group $\Sigma_i\subseteq O(i)$ on $A_i$ to show that this is as a functor independent of the choice of isomorphism $I\cong \{1,\ldots,i\}$. More precisely we send $I$ to the value $A(\R^I)$ where $\R^I$ is equipped with the standard inner product and we use the fact that we can canonically evaluate $A$ on every inner product space as in Definition \ref{def:genuineequiv}.


Note that $\mathcal I$ is not filtered. However, one has B\"okstedt's important ``approximation lemma''. Note that as $\mathcal I$ has an initial object, the statement of the following lemma is equivalent to the similar statement in pointed spaces.

\begin{lemma}[{\cite[Lemma 2.2.3]{Local}}]\label{lem:approximation} Let $F: \mathcal I^n\to \calS$ be a map to the $\infty$-category of spaces, $x\in \mathcal I^n$ an object, and $\mathcal I^n_x\subseteq \mathcal I^n$ the full subcategory of objects supporting a map from $x$. Assume that $F$ sends all maps in $\mathcal I^n_x$ to $k$-connected maps of spaces. Then the map
\[
F(x)\to \colim_{\mathcal I^n} F
\]
is $k$-connected.$\hfill \Box$
\end{lemma}

Intuitively, this says that $\colim_{\mathcal I^n} F$ behaves like a filtered colimit if the maps become sufficiently connective.

The advantage of $\mathcal I$ over the category corresponding to the ordered set of positive integers is that $\mathcal I$ has a natural symmetric monoidal structure, given by the disjoint union. This will be critical in the following.

In the following, as in Section~\ref{sec:genuine} above, we use specific strictly functorial models of homotopy colimits taken in (pointed) compactly generated weak Hausdorff topological spaces, as defined  in Appendix \ref{app:colim}  and write $\hocolim$ to denote these functors.

\begin{definition}\label{def:boekstedt} Consider the symmetric monoidal category $\Sp^O$ of orthogonal spectra. B\"okstedt's construction is the lax symmetric monoidal functor
\[
B: (\Sp^O)^\otimes_\act\to \Sp^O
\]
that takes a finite family of orthogonal spectra $(X_i)_{i\in I}$ to the orthogonal spectrum whose $n$-th term is
\[
\hocolim_{(I_i)_{i\in I}\in \mathcal{I}^I} \Map_\ast(\bigwedge_{i\in I} S^{I_i},S^n\wedge \bigwedge_{i\in I} (X_i)_{I_i})
\]
with $O(n)$-action on $S^n$, and the natural structure maps. 
\end{definition}

Here $\Map_\ast$ denotes the space of based maps. Note that this can indeed be promoted to a functor out of $(\Sp^O)^\otimes_\act$, i.e.~if one has a map $f: I\to J$ of finite sets, $X_i,Y_j\in \Sp^O$ for $i\in I$, $j\in J$, and maps $\bigwedge_{i\in f^{-1}(j)} X_i\to Y_j$ of orthogonal spectra, one gets induced maps
\[\begin{aligned}
&\hocolim_{(I_i)_{i\in I}\in \mathcal{I}^I} \Map_\ast(\bigwedge_{i\in I} S^{I_i},S^n\wedge \bigwedge_{i\in I} (X_i)_{I_i})\\
&\to \hocolim_{(I_j)_{j\in J}\in \mathcal{I}^J} \Map_\ast(\bigwedge_{j\in J} S^{I_j},S^n\wedge \bigwedge_{j\in J} (Y_j)_{I_j})
\end{aligned}\]
which on the index categories is given by sending $(I_i)_{i\in I}$ to $(I_j=\bigsqcup_{i\in f^{-1}(j)} I_i)_{j\in J}$. Indeed, under this map of index categories, the spheres
\[
\bigwedge_{i\in I} S^{I_i} = \bigwedge_{j\in J} S^{I_j}
\]
agree (up to canonical isomorphism), and there is a natural map
\[
\bigwedge_{i\in I} (X_i)_{I_i}\to \bigwedge_{j\in J} (Y_j)_{I_j}
\]
by definition of the smash product of orthogonal spectra, cf.~\cite[Section 1]{SchwedeGenuine}. Moreover, $B$ is indeed lax symmetric monoidal: Given orthogonal spectra $(X_i)_{i\in I}$ and $(Y_j)_{j\in J}$, as well as integers $m$, $n$, one needs to produce $O(m)\times O(n)$-equivariant maps from the smash product of
\[
\hocolim_{(I_i)_{i\in I}\in \mathcal{I}^I} \Map_\ast(\bigwedge_{i\in I} S^{I_i},S^m\wedge \bigwedge_{i\in I} (X_i)_{I_i})
\]
and
\[
\hocolim_{(I_j)_{j\in J}\in \mathcal{I}^J} \Map_\ast(\bigwedge_{j\in J} S^{I_j},S^n\wedge \bigwedge_{j\in J} (Y_j)_{I_j})
\]
to
\[
\hocolim_{(I_k)_{k\in I\sqcup J}\in \mathcal{I}^{I\sqcup J}} \Map_\ast(\bigwedge_{k\in I\sqcup J} S^{I_k},S^m\wedge S^n\wedge\bigwedge_{k\in I} (X_k)_{I_k}\wedge\bigwedge_{k\in J} (Y_k)_{I_k})\ ,
\]
compatible with the structure maps. The desired transformation is given by smashing the maps.

Now we need the following result of Shipley, which follows directly from \cite[Corollary 4.2.4]{Shipley}.\footnote{Shipley works with simplicial symmetric spectra, but every orthogonal spectrum gives rise to a symmetric spectrum by taking the singular complex and restriction of symmetry groups. Note that in orthogonal spectra, it is not necessary to apply the level-fibrant replacement functor $L^\prime$.} Here we say that a map $(X_i)_{i \in I} \to (Y_j)_{j \in J}$ in $(\Sp^O)^\otimes_\act$ is a stable equivalence if it lies over an equivalence $f: I \to J$ in $\Fin$ and the maps $X_i \to Y_{f(i)}$ are stable equivalences for all $i \in I$.

\begin{thm} The functor $B$ takes stable equivalences to stable equivalences.
\end{thm}

\begin{proof} Apart from the comparison between orthogonal and symmetric groups we note that Shipley works with simplicial sets instead of topological spaces. As a result the homotopy colimits in Definition \ref{def:boekstedt} has in her setting always automatically the correct homotopy type. We claim that this is also true in our setting. To see this we take all the spaces 
\[
\Map_\ast(\bigwedge_{i\in I} S^{I_i},S^n\wedge \bigwedge_{i\in I} (X_i)_{I_i})
\]
for varying $n$ together, so that we obtain a diagram $\mathcal{I}^I \to \Sp^O$ over which we take the homotopy colimit. The result is the same orthogonal spectrum as in Definition \ref{def:boekstedt} since  Bousfield-Kan type homotopy colimits are computed levelwise. The advantage is that the combined  homotopy colimit has the correct stable homotopy type by Proposition \ref{orthogonalBK} of Appendix \ref{app:colim}. This then lets us apply \cite[Corollary 4.2.4]{Shipley} to finish the proof. 
\end{proof}

Now we use Theorem \ref{DKlocalizationsymmon} from Appendix \ref{app:symmmon}  to see that the symmetric monoidal $\infty$-category $\Sp^\otimes_\act$ is the Dwyer-Kan localization of $(\Sp^O)^\otimes_\act$ at the stable equivalences as described above. In particular, if we compose $B$ with the functor $N\Sp^O\to \Sp$, then the results of Appendix~\ref{app:symmmon} and the preceding theorem imply that it factors over a lax symmetric monoidal functor
\[
\overline{B}: \Sp^\otimes_\act\to \Sp\ .
\]

\begin{thm}\label{thm:identboekstedt} There is a unique (up to contractible choice) transformation of lax symmetric monoidal functors from $\otimes: \Sp^\otimes_\act\to \Sp$ to $\overline{B}: \Sp^\otimes_\act\to \Sp$. This natural transformation is an equivalence.
\end{thm}

\begin{proof} Shipley shows in \cite[Proposition 4.2.3, Corollary 4.2.4]{Shipley} that the functor $\overline{B}: \Sp^\otimes_\act\to \Sp$ is not only lax symmetric monoidal, but actually symmetric monoidal. Note that this is only true on the level of $\infty$-categories and not on the model. Moreover the composite functor
$$
\Sp \subseteq \Sp^\otimes_\act \xto{\overline{B}} \Sp
$$
is equivalent to the identity by \cite[Theorem 3.1.6]{Shipley}. Since this functor is exact and lax symmetric monoidal it follows from \cite[Corollary 6.9(1)]{Nik} that there is no choice involved in the equivalence. By Proposition~\ref{prop:symenvelope} every such functor is equivalent to $\otimes$.
\end{proof}

In other words, we have verified that the B\"okstedt construction is a model for the $\infty$-categorical tensor product of spectra.

As mentioned in the introduction, the main advantage of the B\"okstedt construction over the usual smash product of orthogonal spectra is its behaviour with respect to geometric fixed points. More precisely, for any orthogonal spectrum $X$, and integer $p\geq 1$,\footnote{Although we write $p$ here, it is not yet required to be a prime.} there is a natural map
\[
\Phi^{C_p} B(X,\ldots,X)\to B(X)
\]
of orthogonal spectra, which is moreover a stable equivalence.

In fact, there is a many variables version of this statement. Namely, for any finite free $C_p$-set $S$ with $\overline{S} = S/C_p$, consider a family of orthogonal spectra $(X_{\overline{s}})_{\overline{s}\in \overline{S}}$. Then there is a natural $C_p$-action on the orthogonal spectrum $B((X_{\overline{s}})_{s\in S})$ by acting on the index set $S$, and a natural map
\[
\Phi^{C_p} B((X_{\overline{s}})_{s\in S})\to B((X_{\overline{s}})_{\overline{s}\in \overline{S}})\ ,
\]
which is a stable equivalence.

Let us formalize this statement, including all functorialities. As above, let $\Free_{C_p}$ denote the category of finite free $C_p$-sets. Then our source category is
\[
\Free_{C_p}\times_{\Fin} (\Sp^O)^\otimes_\act\ .
\]
Here, objects are pairs of a finite free $C_p$-set $S$ with quotient $\overline{S}=S/C_p$, together with orthogonal spectra $X_{\overline{s}}$ indexed by $\overline{s}\in \overline{S}$.
We consider two functors
\[
\Free_{C_p}\times_{\Fin} (\Sp^O)^\otimes_\act\to \Sp^O\ .
\]
The first is given by projection to $(\Sp^O)^\otimes_\act$ and composition with $B: (\Sp^O)^\otimes_\act\to \Sp^O$. We denote this simply by
\[
B: \Free_{C_p}\times_{\Fin} (\Sp^O)^\otimes_\act\to \Sp^O : (S,(X_{\overline{s}})_{\overline{s}\in \overline{S}})\mapsto B((X_{\overline{s}})_{\overline{s}\in \overline{S}})\ .
\]
The other functor first applies the functor
\[
\Free_{C_p}\times_{\Fin} (\Sp^O)^\otimes_\act\to ((\Sp^O)^\otimes_\act)^{BC_p} : (S,(X_{\overline{s}})_{\overline{s}\in \overline{S}})\mapsto (S,(X_{\overline{s}})_{s\in S})\ ,
\]
composes with the B\"okstedt construction
\[
B: ((\Sp^O)^\otimes_\act)^{BC_p}\to (\Sp^O)^{BC_p} = C_p\Sp^O\ ,
\]
and then applies
\[
\Phi^{C_p}_{\mathcal U}: C_p\Sp^O\to \Sp^O
\]
for some fixed choice of complete $C_p$-universe $\mathcal U$. We denote this second functor by
\[
B_p: \Free_{C_p}\times_{\Fin} (\Sp^O)^\otimes_\act\to \Sp^O : (S,(X_{\overline{s}})_{\overline{s}\in \overline{S}})\mapsto \Phi^{C_p}_{\mathcal U} B((X_{\overline{s}})_{s\in S})\ .
\]

\begin{construction}\label{cons:BptoB} There is a natural transformation
\[
B_p\to B : \Phi^{C_p}_{\mathcal U} B((X_{\overline{s}})_{s\in S})\to B((X_{\overline{s}})_{\overline{s}\in \overline{S}})
\]
of functors
\[
\Free_{C_p}\times_{\Fin} (\Sp^O)^\otimes_\act\to \Sp^O\ .
\]
Indeed, note that by definition the $n$-th space $(\Phi^{C_p}_{\mathcal U} B((X_{\overline{s}})_{s\in S}))_n$ of the left-hand side is given by
\[
\hocolim_{V\in \mathcal U\mid V^{C_p}=0}\left(\hocolim_{(I_s)\in \mathcal I^S} \Map_\ast(\bigwedge_{s\in S} S^{I_s},S^n\wedge S^V\wedge \bigwedge_{s\in S} X_{\overline{s},I_s})\right)^{C_p}\ .
\]
For any $V\in \mathcal U$ with $V^{C_p}=0$, we have to construct a map
\[\begin{aligned}
&\left(\hocolim_{(I_s)\in \mathcal I^S} \Map_\ast(\bigwedge_{s\in S} S^{I_s},S^n\wedge S^V\wedge \bigwedge_{s\in S} X_{\overline{s},I_s})\right)^{C_p}\\
&\to \hocolim_{(I_{\overline{s}})\in \mathcal I^{\overline{S}}} \Map_\ast(\bigwedge_{\overline{s}\in \overline{S}} S^{I_{\overline{s}}},S^n\wedge \bigwedge_{\overline{s}\in \overline{S}} X_{\overline{s},I_{\overline{s}}})\ .
\end{aligned}\]
By construction of $\hocolim$, the $C_p$-fixed points are given by the $C_p$-fixed points over the $\hocolim$ over $C_p$-fixed points in the index category. Note that the $C_p$-fixed points of $\mathcal I^S$ are given by $\mathcal I^{\overline{S}}$. Thus, we have to construct a map
\[
\hocolim_{(I_{\overline{s}})\in \mathcal I^{\overline{S}}} \Map_\ast(\bigwedge_{s\in S} S^{I_{\overline{s}}},S^n\wedge S^V\wedge \bigwedge_{s\in S} X_{\overline{s},I_{\overline{s}}})^{C_p}\to \hocolim_{(I_{\overline{s}})\in \mathcal I^{\overline{S}}} \Map_\ast(\bigwedge_{\overline{s}\in \overline{S}} S^{I_{\overline{s}}},S^n\wedge \bigwedge_{\overline{s}\in \overline{S}} X_{\overline{s},I_{\overline{s}}})\ .
\]
But any $C_p$-equivariant map induces a map between $C_p$-fixed points. Moreover,
\[
(\bigwedge_{s\in S} S^{I_{\overline{s}}})^{C_p} = \bigwedge_{\overline{s}\in \overline{S}} S^{\overline{I}_s}
\]
and
\[
(S^n\wedge S^V\wedge \bigwedge_{s\in S} X_{\overline{s},I_{\overline{s}}})^{C_p} = S^n\wedge \bigwedge_{\overline{s}\in \overline{S}} X_{\overline{s},I_{\overline{s}}}\ ,
\]
using that in general $(A\wedge B)^{C_p} = A^{C_p}\wedge B^{C_p}$, $(A\wedge\ldots\wedge A)^{C_p} = A$ for $p$ factors of $A$, and $(S^V)^{C_p}$ is a $0$-sphere, as $V^{C_p}=0$. Thus, restricting a map to the $C_p$-fixed points gives the desired map. It is easy to see that it is functorial.\end{construction}

Finally, we can state the main theorem about the B\"okstedt construction.

\begin{thm}\label{thm:boekstedtgeomfixpoints} For any
\[
(S,(X_{\overline{s}})_{\overline{s}\in \overline{S}})\in \Free_{C_p}\times_{\Fin} (\Sp^O)^\otimes_\act\ ,
\]
the natural map
\[
B_p(S,(X_{\overline{s}})_{\overline{s}\in \overline{S}})\to B(S,(X_{\overline{s}})_{\overline{s}\in \overline{S}}): \Phi^{C_p}_{\mathcal U} B((X_{\overline{s}})_{s\in S})\to B((X_{\overline{s}})_{\overline{s}\in \overline{S}})
\]
is a stable equivalence.
\end{thm}

\begin{proof} This is essentially due to Hesselholt and Madsen, \cite[Proposition 2.5]{HesselholtMadsen}. Let us sketch the argument. First, one reduces to the case that all $X_{\overline{s}}$ are bounded below in the sense that there is some integer $m$ such that $\pi_i X_{\overline{s},n}=0$ for $i<n-m$. Indeed, any orthogonal spectrum $X$ can be written as a filtered colimit of orthogonal spectra $X^{(m)}$ which are bounded below such that the maps $X^{(m)}_i\to X_i$ are homeomorphisms for $i\leq m$. Indeed, one can consider a truncated version of orthogonal spectra keeping track only of the spaces $X_i$ for $i\leq m$; then the restriction functor has a left adjoint, whose composite with the restriction functor is the identity, and which takes values in bounded below orthogonal spectra.\footnote{In the language of orthogonal spectra as pointed continuous functors $\mathbf O\to \Top_\ast$ from the category of finite dimensional inner product spaces with Thom spaces as mapping spaces, these functors correspond to restriction to the full subcategory $\mathbf O_{\leq m}$ of spaces of dimension $\leq m$, and the left Kan extension from $\mathbf O_{\leq m}$ to $\mathbf O$.} Now the statement for general $X$ follows by passage to the colimit of the statement for the $X^{(m)}$'s which are bounded below.

Now we have to estimate the connectivity of the map
\[
\Map_\ast\Big(\bigwedge_{s\in S} S^{I_{\overline{s}}},S^n\wedge S^V\wedge \bigwedge_{s\in S} X_{\overline{s},I_{\overline{s}}}\Big)^{C_p}\to \Map_\ast\Big(\bigwedge_{\overline{s}\in \overline{S}} S^{I_{\overline{s}}},S^n\wedge \bigwedge_{\overline{s}\in \overline{S}} X_{\overline{s},I_{\overline{s}}}\Big)
\]
as $V\in \mathcal U$ with $V^{C_p}=0$ gets larger and $(I_{\overline{s}})\in \mathcal I^{\overline{S}}$ gets larger. Note that the right-hand side can also be written as
\[
\Map_\ast\Big(\bigwedge_{\overline{s}\in \overline{S}} S^{I_{\overline{s}}},S^n\wedge S^V\wedge \bigwedge_{s\in S} X_{\overline{s},I_{\overline{s}}}\Big)^{C_p}\ .
\]
This shows that the homotopy fiber of the first map can be identified with
\[
\Map_\ast\Big(\bigwedge_{s\in S} S^{I_{\overline{s}}}/\bigwedge_{\overline{s}\in \overline{S}} S^{I_{\overline{s}}},S^n\wedge S^V\wedge \bigwedge_{s\in S} X_{\overline{s},I_{\overline{s}}}\Big)^{C_p}\ .
\]
Now, as in the proof of \cite[Proposition 2.5]{HesselholtMadsen}, we use that the connectivity of $\Map_\ast(A,B)^{C_p}$ is at least the minimum of $\mathrm{conn}(B^H)-\dim(A^H)$, where $H$ runs over all subgroups of $C_p$, and $\mathrm{conn}$ denotes the connectivity.

If $H=C_p$, then
\[
\Big(\bigwedge_{s\in S} S^{I_{\overline{s}}}/\bigwedge_{\overline{s}\in \overline{S}} S^{I_{\overline{s}}}\Big)^H
\]
is a point, while the connectivity of
\[
\Big(S^n\wedge S^V\wedge \bigwedge_{s\in S} X_{\overline{s},I_{\overline{s}}}\Big)^H = S^n\wedge\bigwedge_{\overline{s}\in \overline{S}} X_{\overline{s},I_{\overline{s}}}
\]
is at least $n + \sum_{\overline{s}\in \overline{S}} (i_{\overline{s}}-m)$, where $i_{\overline{s}} = |I_{\overline{s}}|$. Here, we have used our assumption that all $X_{\overline{s}}$ are bounded below. This term goes to $\infty$ with $(I_{\overline{s}})\in \mathcal I^{\overline{S}}$. 

If $H\subsetneq C_p$, then for fixed $(I_{\overline{s}})\in \mathcal I^{\overline{S}}$, the term
\[
\dim \Big(\bigwedge_{s\in S} S^{I_{\overline{s}}}/\bigwedge_{\overline{s}\in \overline{S}} S^{I_{\overline{s}}}\Big)^H
\]
is bounded, while the connectivity of
\[
\Big(S^n\wedge S^V\wedge \bigwedge_{s\in S} X_{\overline{s},I_{\overline{s}}}\Big)^H
\]
goes to $\infty$ with $V\in \mathcal U$, as it is at least the connectivity of $(S^V)^H = S^{V^H}$, which gets large. More precisely, let $\rho_{C_p}$ denote the regular representation of $C_p$. Then
\[
\dim\Big(\bigwedge_{s\in S} S^{I_{\overline{s}}}/\bigwedge_{\overline{s}\in \overline{S}} S^{I_{\overline{s}}}\Big)^H = \dim \rho_{C_p}^H \sum_{\overline{s}\in \overline{S}} i_{\overline{s}}\ ,
\]
and the connectivity of 
\[
\Big(S^n\wedge S^V\wedge \bigwedge_{s\in S} X_{\overline{s},I_{\overline{s}}}\Big)^H
\]
is at least
\[
n+\dim V^H + \dim \rho_{C_p}^H \sum_{\overline{s}\in \overline{S}} (i_{\overline{s}} - m)\ .
\]
In this case, the difference is given by
\[
n + \dim V^H - \dim \rho_{C_p}^H |\overline{S}| m\ .
\]
Let
\[
n_V=\min_{H\subsetneq V} (n + \dim V^H - \dim \rho_{C_p}^H |\overline{S}| m)
\]
denote the minimum of these numbers for varying $H\subsetneq V$. Note that $n_V\to\infty$ as $V$ gets large.

Choosing $(I_{\overline{s}})\in \mathcal I^{\overline{S}}$ in the approximation lemma~\ref{lem:approximation} such that
\[
n + \sum_{\overline{s}\in \overline{S}} (i_{\overline{s}}-m)\geq n_V\ ,
\]
we see that the connectivity of the map
\[
\hocolim_{(I_{\overline{s}})\in \mathcal I^{\overline{S}}} \Map_\ast(\bigwedge_{s\in S} S^{I_{\overline{s}}},S^n\wedge S^V\wedge \bigwedge_{s\in S} X_{\overline{s},I_{\overline{s}}})^{C_p}\to \hocolim_{(I_{\overline{s}})\in \mathcal I^{\overline{S}}} \Map_\ast(\bigwedge_{\overline{s}\in \overline{S}} S^{I_{\overline{s}}},S^n\wedge \bigwedge_{\overline{s}\in \overline{S}} X_{\overline{s},I_{\overline{s}}})
\]
is at least $n_V$. Now as $V\in \mathcal U$ with $V^{C_p}=0$ gets large, one has $n_V\to\infty$, so we get an equivalence in the homotopy colimit over $V$.
\end{proof}

\section{The classical construction of THH}\label{sec:thhorth}

In this section, we recall the construction of Hesselholt--Madsen of topological Hochschild homology as an orthogonal cyclotomic spectrum, \cite[Section 2.4]{HesselholtMadsen}.

In order to be (essentially) in the setup of Hesselholt--Madsen, we take as input an orthogonal ring spectrum $R$. Recall, cf.~\cite[Section 1]{SchwedeGenuine}, that an orthogonal ring spectrum is defined as an assocative algebra object in the symmetric monoidal $1$-category $\Sp^O$ of orthogonal spectra.

Note that in orthogonal spectra, one can write down the cyclic bar construction
\[\xymatrix{
 \cdots \ar[r]<4.5pt>\ar[r]<1.5pt>\ar[r]<-4.5pt>\ar[r]<-1.5pt> & R \wedge R \wedge R  \ar[r]<3pt>\ar[r]\ar[r]<-3pt>  & R\wedge R \ar[r]<1.5pt>\ar[r]<-1.5pt> & R\ ,
}\]
and a construction of $\THH$ as a cyclotomic spectrum based on this has recently been given in \cite{sixauthors}. The classical construction proceeds with a slightly different cyclic object, which is essentially a fibrant replacement of this diagram. This has the advantage that one can read off the genuine fixed points, relevant to the definition of $\TC$, directly.

\begin{definition}\label{def:thhorth} Let $R$ be an orthogonal ring spectrum.
\begin{altenumerate}
\item For any pointed space $X$, define a pointed space with $\T$-action $\THH(R;X)$ as the geometric realization of the cyclic space\footnote{We refer to Appendix~\ref{app:cyclic} for our conventions concerning Connes' category $\Lambda$ and the geometric realization of cyclic spaces.} sending $[k]_\Lambda\in \Lambda^\op$, $k\geq 1$, to
\[
B(R,\ldots,R) = \hocolim_{(I_1,\ldots,I_k)\in \mathcal{I}^k} \Map_\ast(S^{I_1}\wedge \ldots\wedge S^{I_k},X\wedge R_{I_1}\wedge\ldots\wedge R_{I_k})\ ,
\]
where $\Map_\ast$ is the based mapping space.
\item The orthogonal $\T$-spectrum $\THH(R)$ has the $m$-th term given by $\THH(R;S^m)$ with the $O(m)$-action induced by the action on $S^m$. The structure maps $\sigma_m$ are the natural maps
\[
\THH(R;S^m)\wedge S^1\to \THH(R;S^m\wedge S^1)=\THH(R;S^{m+1})\ .
\]
\end{altenumerate}
\end{definition}

Note that in the language of Appendix~\ref{app:cyclic}, the cyclic spectrum $\THH_\bullet(R)$ which realizes to $\THH(R)$ is given by the composition of the natural map
\[
\Lambda^\op\to \Ass^\otimes_\act
\]
with the map
\[
\Ass^\otimes_\act\to (\Sp^O)^\otimes_\act
\]
defining the orthogonal ring spectrum $R$, and the B\"okstedt construction
\[
B: (\Sp^O)^\otimes_\act\to \Sp^O\ .
\]
We now  verify that the `naive' geometric realization of this cyclic spectrum has the `correct' homotopy type. 
According to Corollary \ref{prop:geomrealcompspectra} we have to verify that $\THH_\bullet(R;X)$ is proper, i.e. that all the latching maps are levelwise h-cofibrations. 
\begin{lemma}\label{lemproper}\cite[Lemma 6.17]{MR3558224}
Assume that the orthogonal ring spectrum $R$ is levelwise well pointed and that the unit in the zeroth space $S^0 \to R_0$ is an h-cofibration. Then the cyclic spaces 
$$
\THH_\bullet(R;X) : \Lambda^\op \to \Top
$$
are proper for all $X$. 
\end{lemma}
\begin{proof}
In the cited reference it is shown that the degeneracies of the associated simplicial objects are h-cofibrations. But this implies by  \cite[Corollary 2.4]{Lewis} that it is proper.
\end{proof}

\begin{remark}
In fact, if we are just interested in the correct homotopy type of the orthogonal spectrum $\THH(R)$ one can get away with slightly less assumptions on the orthogonal ring spectrum $R$ since we really only need that the maps from the latching objects to the layers of the underlying simplicial diagram of $\THH_\bullet(R) \in \Fun(\Lambda^\op, \Sp^O)$ are h-cofibrations of orthogonal spectra (see the discussion in Appendix \ref{app:colim} around Proposition \ref{orthogonalBK}). For this we do not need that $R$ is levelwise well-pointed but only that the unit $\bS \to R$ is an h-cofibration of orthogonal spectra. 
Since we do not need this generality and the conditions on $R$ given above are in practice almost always satisfied we will not prove this slightly more general result here. 
\end{remark}
We have to define the cyclotomic structure maps. For this, we fix the complete $\T$-universe
\[
\mathcal U = \bigoplus_{k\in \mathbb Z, i\geq 1} \mathbb C_{k,i}
\]
as in the definition of orthogonal cyclotomic spectra, Definition~\ref{def:orthogonalcyclo}. Here $\T$ acts on $\mathbb C_{k,i}$ via the $k$-th power of the embedding $\T\hookrightarrow \mathbb C^\times$.
We note that with this choice of $\mathcal U$, the functor
\[
\Phi^{C_p}_{\mathcal U}: C_p\Sp^O\to \Sp^O\ ,
\]
for any $p\geq 1$, actually takes values in $\T\Sp^O$.\footnote{The reader may be surprised by this extra $\T$-action. It is a feature of this point-set model, and the action is homotopically trivial. However, for the point-set construction of the cyclotomic structure, it is critical.} Indeed, in the formula
\[
\Phi^{C_p}_{\mathcal U}(X)_n = \hocolim_{V\in \mathcal U\mid V^{C_p}=0} X(\mathbb R^n\oplus V)^{C_p}\ ,
\]
there is a remaining $\T$-action on $V$. The resulting functor
\[
\Phi^{C_p}_{\mathcal U}: C_p\Sp^O\to \T\Sp^O
\]
is $BC_p$-equivariant, where $BC_p$ acts on the category $C_p\Sp^O$ by noting that any object has a $C_p$-action (which commutes with itself), and similarly in $\T\Sp^O$, any object has a $C_p\subseteq \T$-action. The $BC_p$-equivariance means that the two $C_p$-actions on
\[
\hocolim_{V\in \mathcal U\mid V^{C_p}=0} X(\mathbb R^n\oplus V)^{C_p}\ ,
\]
one on $X$, and one on $V$ (via restriction along $C_p\subseteq \T$), agree. But this is clear, as we are looking at the $C_p$-fixed points under the diagonal action.
Now let us recall the definition of
\[
\Phi^{C_p}_{\mathcal U} \THH(R)\in \T\Sp^O\ .
\]
Its $n$-th space is given by
\[
(\Phi^{C_p}_{\mathcal U} \THH(R))_n = \hocolim_{V\in \mathcal U\mid V^{C_p}=0} \THH\big(R;S^n\wedge S^V\big)^{C_p}
\]
where the $C_p$-action on $\THH\big(R;S^n\wedge S^V\big)$ is given by the diagonal action combining the action on the representation $V$ and the action of $C_p \subseteq \T$ on $\THH$.
We give a reinterpretation of this formula as follows, using the functor $B_p$ from Construction~\ref{cons:BptoB}. Note that by simplicial subdivision, $\THH(R)$ is also the geometric realization of the $\Lambda_p^\op$-orthogonal spectrum given by restriction along the functor
\[
\Lambda_p^\op\to \Lambda^\op: [n]_{\Lambda_p}\mapsto [pn]_{\Lambda}
\]
of the cyclic orthogonal spectrum defining $\THH(R)$ (see Proposition \ref{subdivisionrealization} in Appendix \ref{app:cyclic}). This $\Lambda_p^\op$-orthogonal spectrum is equivalently given by the composition of
\[
\Lambda_p^\op\to \Free_{C_p}\times_{\Fin} \Ass^\otimes_\act
\]
with
\[
\Free_{C_p}\times_{\Fin} \Ass^\otimes_\act\to \Free_{C_p}\times_{\Fin} (\Sp^O)^\otimes_\act\ ,
\]
the functor
\[
\Free_{C_p}\times_{\Fin} (\Sp^O)^\otimes_\act\to ((\Sp^O)^\otimes_\act)^{BC_p}\ ,
\]
and the B\"okstedt construction
\[
B: ((\Sp^O)^\otimes_\act)^{BC_p}\to (\Sp^O)^{BC_p} = C_p\Sp^O\ .
\]
This discussion leads to the following proposition.

\begin{proposition}\label{prop:compgeomfixpointsthh} There is a natural map of orthogonal spectra from  $\Phi^{C_p}_{\mathcal U} \THH(R)$ to the geometric realization of the $\Lambda_p^\op$-orthogonal spectrum given as the composition of
\[
\Lambda_p^\op\to \Free_{C_p}\times_{\Fin} \Ass^\otimes_\act\to \Free_{C_p}\times_{\Fin} (\Sp^O)^\otimes_\act
\]
with the functor
\[
B_p: \Free_{C_p}\times_{\Fin} (\Sp^O)^\otimes_\act\to \Sp^O : (S,(X_{\overline{s}})_{\overline{s}\in \overline{S}})\mapsto \Phi^{C_p}_{\mathcal U} B((X_{\overline{s}})_{s\in S})\ .
\]
This map is a homeomorphism of orthogonal spectra.
\end{proposition}

\begin{proof} This follows from the preceding by passing to the homotopy colimit over all $V$. More precisely we consider the canonical map
$$
| \Phi^{C_p}_{\mathcal U}(\sd_p^*\THH_\bullet(R))| \to  \Phi^{C_p}_{\mathcal U} |\sd_p^* \THH_\bullet(R)|  \cong \Phi^{C_p}_{\mathcal U} |\THH_\bullet(R)| \
$$
which exists by definition of the functor $\Phi^{C_p}_{\mathcal U}$. Geometric realization commutes with taking fixed points (Lemma \ref{fixedcommutes}). Since it also commutes with taking geometric realization and shifts we conclude that this map is a homeomorphism. \end{proof}

The problem with this formula is that a priori, we lose track of the $\T$-action on $\Phi^{C_p}_{\mathcal U}\THH(R)$. Namely, the $\T$-action above is not related to the circle action coming from the $\Lambda_p^\op$-structure. This is because $\T$ also implicitly acts on the variable $V\in \mathcal U$. However, the $\T$-action can be recovered by using the refinement of $\Phi^{C_p}_{\mathcal U}$ to a functor
\[
\Phi^{C_p}_{\mathcal U}: C_p\Sp^O\to \T\Sp^O\ .
\]
More precisely, using this refinement, we get a similar refinement
\[
B_p^\T: \Free_{C_p}\times_{\Fin} (\Sp^O)^\otimes_\act\to \T\Sp^O\ .
\]
This functor is again $BC_p$-equivariant, as $\Phi^{C_p}_{\mathcal U}$ is. In particular, the $BC_p$-equivariant map
\[
\Lambda_p^\op\to \Free_{C_p}\times_{\Fin} \Ass^\otimes_\act\to \Free_{C_p}\times_{\Fin} (\Sp^O)^\otimes_\act\xto{B_p^\T} \T\Sp^O
\]
gives rise to an orthogonal spectrum with a continuous $(\T\times \T)/C_p$-action. Restricting to the diagonal $\T/C_p\cong\T$-action, we get an orthogonal spectrum with $\T$-action, cf.~also Proposition~\ref{prop:geomrealtactiontop}. The following proposition follows by unraveling the definitions.

\begin{proposition}\label{prop:compgeomfixpointsthhgen} 
The $\T$-equivariant orthogonal spectrum $\Phi^{C_p}_{\mathcal U} \THH(R)\in \T\Sp^O$ has a natural map to the object of $\T\Sp^O$ corresponding to the $BC_p$-equivariant functor
\[
\Lambda_p^\op \to \Free_{C_p}\times_{\Fin} \Ass^\otimes_\act\to \Free_{C_p}\times_{\Fin} (\Sp^O)^\otimes_\act\xto{B_p^\T} \T\Sp^O
\]
under the functor of Proposition~\ref{prop:geomrealtactiontop}. This map is a homeomorphism of $\T$-orthogonal spectra.
\end{proposition}

\begin{proof} This follows with the same proof as Proposition \ref{prop:compgeomfixpointsthh}.
\end{proof}

In this picture, one can also recover $\THH(R)\in \T\Sp^O$ itself.

\begin{proposition} The $\T$-equivariant orthogonal spectrum $\THH(R)\in \T\Sp^O$ is given by the object of $\T\Sp^O$ corresponding to the $BC_p$-equivariant functor
\[
\Lambda_p^\op \to \Free_{C_p}\times_{\Fin} \Ass^\otimes_\act\to \Free_{C_p}\times_{\Fin} (\Sp^O)^\otimes_\act\xto{B} \Sp^O\subseteq \T\Sp^O\ ,
\]
where $B$ is as in Construction~\ref{cons:BptoB}, and in the final inclusion, we equip any orthogonal spectrum with the trivial $\T$-action.$\hfill \Box$
\end{proposition}

Finally, Construction~\ref{cons:BptoB} produces a natural transformation of $BC_p$-equivariant functors from
\[
\Lambda_p^\op \to \Free_{C_p}\times_{\Fin} \Ass^\otimes_\act\to \Free_{C_p}\times_{\Fin} (\Sp^O)^\otimes_\act\xto{B_p^\T} \T\Sp^O
\]
to
\[
\Lambda_p^\op \to \Free_{C_p}\times_{\Fin} \Ass^\otimes_\act\to \Free_{C_p}\times_{\Fin} (\Sp^O)^\otimes_\act\xto{B} \Sp^O\subseteq \T\Sp^O\ .
\]
Applying Proposition~\ref{prop:geomrealtactiontop} and the maps constructed above, we get the desired map
\[
\Phi_p: \Phi^{C_p}_{\mathcal U} \THH(R)\to \THH(R)
\]
of objects of $\T\Sp^O$. By Proposition~\ref{prop:compgeomfixpointsthhgen} and Theorem~\ref{thm:boekstedtgeomfixpoints}, the maps $\Phi_p$ are equivalences of the underlying orthogonal spectra. Moreover, it is easy to see that the maps $\Phi_p$ sit in commutative diagrams as in Definition~\ref{def:orthogonalcyclo}. Using for all integers $p,q\geq 1$ the resulting commutative diagram
\[\xymatrix{
\Phi^{C_q}_{\mathcal U} \Phi^{C_p}_{\mathcal U} \THH(R)\ar[r]^\simeq\ar[d]^{\Phi^{C_q}_{\mathcal U}(\Phi_p)} & \Phi^{C_{pq}}_{\mathcal U} \THH(R)\ar[d]^{\Phi_{pq}}\\
\Phi^{C_q}_{\mathcal U} \THH(R)\ar[r]^{\Phi_q} & \THH(R)\ ,
}\]
one sees that also $\Phi^{C_q}_{\mathcal U}(\Phi_p)$ is an equivalence of underlying orthogonal spectra. This implies that $\Phi_p$ is an $\F$-equivalence, as desired.

\section{Comparison}\label{sec:comparison}

In this section, we compare the constructions of Section~\ref{sec:thhnaive} and Section~\ref{sec:thhorth}. A similar comparison between the B\"okstedt model for $\THH$ and the cyclic bar construction model for $\THH$ (in the model category of \cite{sixauthors}) will also appear in \cite{TCcomp}.

We start with an associative orthogonal ring spectrum $R\in \Alg(\Sp^O)$, regarded as a functor
\[
R^\otimes: \Ass^\otimes\to (\Sp^O)^\otimes
\]
of operads. Let
\[
A^\otimes: N(\Ass^\otimes)\to \Sp^\otimes
\]
denote the associated $\E_1$-algebra $A\in \Alg_{\E_1}(\Sp)$.

First, we have to compare the cyclic spectra computing $\THH$. Recall that in Section~\ref{sec:thhnaive}, it is given by the composite
\[
N(\Lambda^\op)\to N(\Ass^\otimes_\act)\xto{A^\otimes} \Sp^\otimes_\act\xto{\otimes}\Sp\ ,
\]
and in Section~\ref{sec:thhorth}, it is given by the composite
\[
\Lambda^\op\to \Ass^\otimes_\act\xto{R^\otimes} (\Sp^O)^\otimes_\act\xto{B} \Sp^O\ ,
\]
where $B$ denotes the B\"okstedt construction. Here, the comparison of these cyclic objects is an immediate consequence of Theorem~\ref{thm:identboekstedt}. To compare the geometric realizations we combine Lemma \ref{lemproper} and Corollary \ref{prop:geomrealcompspectra} to obtain the following result.

\begin{thm}\label{thmcomparison}
Let $R$ be an orthogonal ring spectrum that is levelwise well-pointed and such that the zeroth component of the unit $S^0 \to R_0$ is an h-cofibration. Then there is a natural (naive) $\T$-equivariant equivalence between the classical construction of $\THH(R)$ using B\"okstedt's construction as in Definition \ref{def:thhorth} and the $\infty$-categorical version $\THH(A)$ using the cyclic bar construction as in Definition \ref{def_bar_infty} of the associated $\E_1$-ring spectrum $A$. 
More precisely there is a commutative square of lax symmetric monoidal functors
\[
\xymatrix{
N\Alg(\Sp^O)_{\mathrm{well}} \ar[rr]^-{\THH}\ar[d] && N\T\Sp^O \ar[d] \\
\Alg_{\E_1}(\Sp) \ar[rr]^-{\THH} && \Sp^{B\T}
}
\]
in which the vertical functors are Dwyer-Kan localizations and
\[
\Alg(\Sp^O)_{\mathrm{well}} \subseteq \Alg(\Sp^O)
\]
denotes the subcategory of all orthogonal ring spectra that satisfy the above well-pointedness condition.$\hfill \Box$
\end{thm}

It remains to identify the Frobenius maps, so fix a prime $p$. In Section~\ref{sec:thhorth}, the Frobenius map $\THH(R)\to \THH(R)^{tC_p}$ is the composite of the inverse of the equivalence $\Phi^{C_p}_{\mathcal U}\THH(R)\to \THH(R)$ and the natural map $\Phi^{C_p}_{\mathcal U} \THH(R)\to \THH(R)^{tC_p}$. Recall that for the construction of the map $\Phi^{C_p}_{\mathcal U} \THH(R)\to \THH(R)$, the source $\Phi^{C_p}_{\mathcal U} \THH(R)$ was modelled by the $BC_p$-equivariant functor
\[
\Lambda_p^\op\to \Free_{C_p}\times_{\Fin} \Ass^\otimes_\act\xto{R^\otimes} \Free_{C_p}\times_{\Fin} (\Sp^O)^\otimes_\act\xto{B_p^\T} \T\Sp^O\ ,
\]
pictorially
\begin{equation}\label{eq:model1}\xymatrix{
\cdots \ar[r]<1.2ex> \ar[r]<-1.2ex>\ar[r]<0.4ex> \ar[r]<-0.4ex> & \Phi^{C_p}_{\mathcal U} B(\underbrace{R,\ldots,R}_{3p}) \ar@(ul,ur)^{(C_{3p}\times \T)/C_p}\ar[r]<0.8ex>\ar[r]<0ex>\ar[r]<-0.8ex> \ar[r]& \Phi^{C_p}_{\mathcal U} B(\underbrace{R,\ldots,R}_{2p})\ar@(ul,ur)^{(C_{2p}\times \T)/C_p}
\ar[r]<0.4ex> \ar[r]<-0.4ex>& \Phi^{C_p}_{\mathcal U} B(\underbrace{R,\ldots,R}_p)\ar@(ul,ur)^{(C_p\times \T)/C_p}\ ;
}\end{equation}
here and in the following, we lie a little bit, as there are actually $p$ times as many arrows as are drawn (as we are dealing with a $\Lambda_p^\op$-object), which are however permuted by the $BC_p$-equivariance.

Moreover, the target $\THH(R)$ is modelled by the $BC_p$-equivariant functor
\[
\Lambda_p^\op\to \Free_{C_p}\times_{\Fin} \Ass^\otimes_\act\xto{R^\otimes} \Free_{C_p}\times_{\Fin} (\Sp^O)^\otimes_\act\xto{B} \Sp^O\ ,
\]
pictorially
\begin{equation}\label{eq:model2}\xymatrix{
\cdots \ar[r]<1.2ex> \ar[r]<-1.2ex>\ar[r]<0.4ex> \ar[r]<-0.4ex> & B(R,R,R) \ar@(ul,ur)^{C_{3p}/C_p}\ar[r]<0.8ex>\ar[r]<0ex>\ar[r]<-0.8ex> \ar[r]& B(R,R)\ar@(ul,ur)^{C_{2p}/C_p}
\ar[r]<0.4ex> \ar[r]<-0.4ex>& B(R)\ar@(ul,ur)^{C_p/C_p}\ .
}\end{equation}
The natural transformation then comes from the natural transformation $B_p^\T\to B$ of $BC_p$-equivariant functors
\[
\Free_{C_p}\times_{\Fin} (\Sp^O)^\otimes_\act\to \T\Sp^O
\]
in Construction~\ref{cons:BptoB} (which is $\T$-equivariant).

Similarly, in Section~\ref{sec:thhnaive}, the Frobenius map $\THH(A)\to \THH(A)^{tC_p}$ was constructed by modelling $\THH(A)$ by the $BC_p$-equivariant functor
\[
N(\Lambda_p^\op)\to N(\Free_{C_p})\times_{N(\Fin)} N(\Ass^\otimes_\act)\xto{A^\otimes} N(\Free_{C_p})\times_{N(\Fin)} \Sp^\otimes_\act\xto{I} \Sp\ ,
\]
pictorially
\begin{equation}\label{eq:model3}\xymatrix{
\cdots \ar[r]<1.2ex> \ar[r]<-1.2ex>\ar[r]<0.4ex> \ar[r]<-0.4ex> & A\otimes A\otimes A \ar@(ul,ur)^{C_{3p}/C_p}\ar[r]<0.8ex>\ar[r]<0ex>\ar[r]<-0.8ex> \ar[r]& A\otimes A\ar@(ul,ur)^{C_{2p}/C_p}
\ar[r]<0.4ex> \ar[r]<-0.4ex>& A\ar@(ul,ur)^{C_p/C_p}\ .
}\end{equation}
Moreover, $\THH(A)^{tC_p}$ admits a natural map from the realization (in the sense of Proposition~\ref{prop:geomrealtaction}) of the $BC_p$-equivariant functor
\[
N(\Lambda_p^\op)\to N(\Free_{C_p})\times_{N(\Fin)} N(\Ass^\otimes_\act)\xto{A^\otimes} N(\Free_{C_p})\times_{N(\Fin)} \Sp^\otimes_\act\xto{\tilde{T}_p}\Sp\ ,
\]
pictorially
\begin{equation}\label{eq:model4}\xymatrix{
\cdots \ar[r]<1.2ex> \ar[r]<-1.2ex>\ar[r]<0.4ex> \ar[r]<-0.4ex> & \big(A^{\otimes 3p}\big)^{tC_p} \ar@(ul,ur)^{C_{3p}/C_p}\ar[r]<0.8ex>\ar[r]<0ex>\ar[r]<-0.8ex> \ar[r]& \big(A^{\otimes 2p}\big)^{tC_p}\ar@(ul,ur)^{C_{2p}/C_p}
\ar[r]<0.4ex> \ar[r]<-0.4ex>& \big(A^{\otimes p}\big)^{tC_p}\ar@(ul,ur)^{C_p/C_p}\ .
}\end{equation}
Now the natural transformation comes from the natural transformation $I\to \tilde{T}_p$ of $BC_p$-equivariant functors
\[
N(\Free_{C_p})\times_{N(\Fin)} \Sp^\otimes_\act\to \Sp
\]
given by Corollary~\ref{cor:tildetptrafo}.

We  compare all different constructions by invoking Lemma~\ref{lem:initialtech}. For this, we need to rewrite everything in terms of lax symmetric monoidal functors on $N(\Free_{C_p})\times_{N(\Fin)} \Sp^\otimes_\act$. Unfortunately, one of the functors, namely $B_p^\T$, is not lax symmetric monoidal, because our version $\Phi^{C_p}_{\mathcal U}: C_p\Sp^O\to \T\Sp^O$ of the geometric fixed points functor is not lax symmetric monoidal.

For this reason, we need to recall a few things about different models of the geometric fixed points functor $\Phi^{C_p}: C_p\Sp\to \Sp$. First, as a consequence of Proposition~\ref{prop:rightadjointgeometric}, we can upgrade the functoriality of the geometric fixed points functor $\Phi^{C_p}: C_p\Sp\to \Sp$.

\begin{construction}\label{cons:geomfixpointsbcpequivlaxsymm} There is a natural $BC_p$-equivariant symmetric monoidal functor
\[
\Phi^{C_p}_\infty: C_p\Sp\to \Sp\ .
\]
\end{construction}

\begin{proof} We define $\Phi^{C_p}_\infty$ as the composite of the $BC_p$-equivariant smashing localization $C_p\Sp\to C_p\Sp_{\geq C_p}$ and the $BC_p$-equivariant lax symmetric monoidal functor $-^{C_p}: C_p\Sp_{\geq C_p}\subseteq C_p\Sp\to \Sp$.
\end{proof}

We need to compare this with the models $\Phi^{C_p}$ and $\Phi^{C_p}_{\mathcal U}$ used above. Note that $\Phi^{C_p}$ is lax symmetric monoidal but not $BC_p$-equivariant, while $\Phi^{C_p}_{\mathcal U}: C_p\Sp^O\to \T\Sp^O$ is $BC_p$-equivariant, but not lax symmetric monoidal. Here, $\mathcal U$ denotes our fixed $\T$-universe, as usual. First, we define a related functor that is both $BC_p$-equivariant and lax symmetric monoidal.

\begin{construction}\label{cons:geomfixpointsbcpequivlaxsymmmodel} 
We construct a $BC_p$-equivariant lax symmetric monoidal functor
\[
\Phi^{C_p}_{\mathcal U,\lax}: C_p\Sp^O\to \T\Sp^O
\]
as follows. It takes a $C_p$-equivariant orthogonal spectrum $X$ to the orthogonal spectrum $\Phi^{C_p}_{\mathcal U,\lax}(X)$ whose $n$-th space is given by
\[
\Phi^{C_p}_{\mathcal U,\lax}(X)_n = \hocolim_{I\in \mathcal I,(V_i)\in \mathcal U^I|V_i^{C_p}=0} X(\mathbb R^n\oplus \bigoplus_{i\in I} V_i)^{C_p}\ ,
\]
where $\mathcal I$ denotes B\"okstedt's category of finite sets with injective maps. The structure maps are the evident maps, and there is a $\T$-action given by acting on the $V_i$'s diagonally. The functor is $BC_p$-equivariant, as the $C_p$-action on $X$ agrees with the $C_p\subseteq \T$-action through the $V_i$'s on the fixed points for the diagonal action. The lax symmetric monoidal structure is induced by the map
\[
(I,(V_i)\in \mathcal U^I),(J,(W_j)\in \mathcal U^J)\mapsto (I\sqcup J,(V_i,W_j)\in \mathcal U^{I\sqcup J})
\]
on index categories, and the natural maps
\[
X(\mathbb R^n\oplus\bigoplus_{i\in I} V_i)^{C_p}\wedge Y(\mathbb R^m\oplus\bigoplus_{j\in J} W_j)^{C_p}\to (X\wedge Y)(\mathbb R^{n+m}\oplus \bigoplus_{i\in I}\oplus \bigoplus_{j\in J} W_j)^{C_p}
\]
on spaces.
\end{construction}

\begin{construction} There is a natural $BC_p$-equivariant transformation from
\[
\Phi^{C_p}_{\mathcal U}: C_p\Sp^O\to \T\Sp^O
\]
to
\[
\Phi^{C_p}_{\mathcal U,\lax}: C_p\Sp^O\to \T\Sp^O\ .
\]
Indeed, this comes from the inclusion of the terms given by a fixed one-element set $I\in \mathcal I$.

Moreover, this natural transformation induces stable equivalences $\Phi^{C_p}_{\mathcal U}(X)\to \Phi^{C_p}_{\mathcal U,\lax}(X)$ of the underlying orthogonal spectra for all $X\in C_p\Sp^O$. Indeed, the homotopy colimit in the definition of $\Phi^{C_p}_{\mathcal U,\lax}(X)$ is over homotopy equivalences, so this follows from Lemma~\ref{lem:approximation}.
\end{construction}

\begin{construction} There is a natural lax symmetric monoidal transformation from
\[
\Phi^{C_p}: C_p\Sp^O\to \Sp^O\subseteq \T\Sp^O : (X_n)_n\mapsto (X(\rho_{C_p}\otimes \mathbb R^n)^{C_p})_n
\]
to
\[
\Phi^{C_p}_{\mathcal U,\lax}: C_p\Sp^O\to \T\Sp^O\ .
\]
Indeed, this comes from the inclusion of the term given by $I=\{1,\ldots,n\}\in \mathcal I$ and $V_1=\ldots=V_n=\rho_{C_p}/\mathbb R$ into the homotopy colimit.

Moreover, this natural transformation induces stable equivalences $\Phi^{C_p}(X)\to \Phi^{C_p}_{\mathcal U,\lax}(X)$ of the underlying orthogonal spectra for all $X\in C_p\Sp^O$.
\end{construction}

Finally, we can state the desired compatibility between the point-set constructions and the abstract definition.

\begin{proposition}\label{prop:geomfixpointscompatible} The diagram
\[\xymatrix{
C_p\Sp^O\ar[r]^{\Phi^{C_p}_{\mathcal U,\lax}}\ar[d] & \T\Sp^O\ar[dr] \\
C_p\Sp\ar[r]^{\Phi^{C_p}_\infty} & \Sp\ar[r] & \Sp^{B\T}
}\]
of $BC_p$-equivariant lax symmetric monoidal functors commutes.
\end{proposition}

We note that the proof is strictly speaking a construction.

\begin{proof} First, we note that by the results of Appendix~\ref{app:symmmon}, the composite
\[
C_p\Sp^O\xto{\Phi^{C_p}_{\mathcal U,\lax}} \T\Sp^O\to \Sp^{B\T}
\]
factors uniquely over a $BC_p$-equivariant lax symmetric monoidal functor\[
F: C_p\Sp\to \Sp^{B\T}\ .
\]
This functor has the property that it factors over the localization $C_p\Sp_{\geq C_p}$ (as the underlying functor is equivalent to the usual $\Phi^{C_p}$). Thus, it remains to compare the two $BC_p$-equivariant lax symmetric monoidal functors
\[
\Phi^{C_p}_\infty = -^{C_p}, F:C_p\Sp_{\geq C_p}\to \Sp^{B\T}\ .
\]
But there is a natural $BC_p$-equivariant lax symmetric monoidal transformation
\[
-^{C_p}\to \Phi^{C_p}_{\mathcal U,\lax}: C_p\Sp^O_\Omega\to \T\Sp^O
\]
by inclusion of $X(\mathbb R^n)^{C_p}=(X(\mathbb R^n)\wedge S^{\bigoplus_{i\in I} V_i})^{C_p}$ into $X(\mathbb R^n\oplus\bigoplus_{i\in I} V_i)^{C_p}$; here $C_p\Sp^O_\Omega\subseteq C_p\Sp^O$ denotes the $C_p$-$\Omega$-spectra. Passing to the associated $BC_p$-equivariant lax symmetric monoidal transformation $C_p\Sp\to \Sp^{B\T}$ and restricting to $C_p\Sp_{\geq C_p}$, we get the desired result.
\end{proof}

Now we get a number of related $BC_p$-equivariant lax symmetric monoidal functors.
\begin{altenumerate}
\item The functor
\[
B_{p,\lax}^\T: \Free_{C_p}\times_{\Fin} (\Sp^O)^\otimes_\act\to \T\Sp^O\ :\ (S,(X_{\overline{s}})_{\overline{s}\in \overline{S}})\mapsto \Phi^{C_p}_{\mathcal U,\lax} B((X_{\overline{s}})_{s\in S})\ .
\]
\item The functor
\[
B_{p,\infty}: \Free_{C_p}\times_{\Fin} (\Sp^O)^\otimes_\act\to \Sp\ :\ (S,(X_{\overline{s}})_{\overline{s}\in \overline{S}})\mapsto \Phi^{C_p}_\infty B((X_{\overline{s}})_{s\in S})\ .
\]
\item The functor
\[
B: \Free_{C_p}\times_{\Fin} (\Sp^O)^\otimes_\act\to \Sp^O\ : (S,(X_{\overline{s}})_{\overline{s}\in \overline{S}})\mapsto B((X_{\overline{s}})_{\overline{s}\in \overline{S}})\ .
\]
\item The functor
\[
I: N(\Free_{C_p})\times_{N(\Fin)} \Sp^\otimes_\act\to \Sp\ :\ (S,(X_{\overline{s}})_{\overline{s}\in \overline{S}})\mapsto \bigotimes_{\overline{s}\in \overline{S}} X_{\overline{s}}\ .
\]
\item The functor
\[
\tilde{T}_p: N(\Free_{C_p})\times_{N(\Fin)} \Sp^\otimes_\act\to \Sp\ :\ (S,(X_{\overline{s}})_{\overline{s}\in \overline{S}})\mapsto \Big(\bigotimes_{s\in S} X_{\overline{s}}\Big)^{tC_p}\ .
\]
\end{altenumerate}

Moreover, they are related by a number of $BC_p$-equivariant lax symmetric monoidal natural transformations.
\begin{altenumerate}
\item[{\rm (T1)}] A natural transformation $B_{p,\lax}^\T\to B$ (after composing $B$ with $\Sp^O\to \T\Sp^O$). This comes from the obvious adaptation of Construction~\ref{cons:BptoB}.
\item[{\rm (T2)}] A natural transformation $B_{p,\lax}^\T\to B_{p,\infty}$ (after composing $B_{p,\lax}^\T$ with $\T\Sp^O\to \Sp^{B\T}$ and $B_{p,\infty}$ with $\Sp\to \Sp^{B\T}$). This comes from Proposition~\ref{prop:geomfixpointscompatible}.
\item[{\rm (T3)}] A natural transformation $I\to B$ (after composing $B$ with $\Sp^O\to \Sp$). This comes from Theorem~\ref{thm:identboekstedt}.
\item[{\rm (T4)}] A natural transformation $I\to \tilde{T}_p$ from Corollary~\ref{cor:tildetptrafo}.
\item[{\rm (T5)}] A natural transformation $B_{p,\infty}\to \tilde{T}_p$. This comes from Theorem~\ref{thm:identboekstedt} and the natural $BC_p$-equivariant lax symmetric monoidal transformation
\[
\Phi^{C_p}_\infty\to -^{tC_p}: C_p\Sp\to \Sp
\]
which arises by writing $-^{tC_p}$ as the composite of the ($BC_p$-equivariant, lax symmetric monoidal) Borel completion functor $C_p\Sp\to C_p\Sp_B$ and $\Phi^{C_p}_\infty$. Here, we use that the Borel completion functor is a localization, cf.~Theorem~\ref{thm:borelcompletion}, and thus automatically $BC_p$-equivariant and lax symmetric monoidal.
\end{altenumerate}

The main theorem relating all these functors and natural transformations is now given by the following.

\begin{thm}\label{thm:maincomp} The compositions of $B_{p,\lax}^\T$, $B_{p,\infty}$, $B$, $I$ and $\tilde{T}_p$ with the respective ``forgetful'' $BC_p$-equivariant lax symmetric monoidal functor to $\Sp^{B\T}$ factor uniquely over $BC_p$-equivariant lax symmetric monoidal and partially exact functors
\[
N(\Free_{C_p})\times_{N(\Fin)} \Sp^\otimes_\act\to \Sp^{B\T}\ .
\]
Moreover, (the image of) $I$ is initial in the $\infty$-category of $BC_p$-equivariant lax symmetric monoidal and partially exact functors $N(\Free_{C_p})\times_{N(\Fin)} \Sp^\otimes_\act\to \Sp^{B\T}$.
\end{thm}

\begin{proof} The factorization follows from the results of Appendix~\ref{app:symmmon}. The initiality of $I$ follows from Lemma~\ref{lem:initialtech}.
\end{proof}

Let us quickly explain how this leads to the desired comparison. We want to show that the natural diagram
\[\xymatrix{
\THH(R) & \Phi^{C_p}_{\mathcal U} \THH(R)\ar[d]\ar[l]_{\Phi_p}\\
\THH(A)\ar[u]_\simeq\ar[r]^{\varphi_p} & \THH(A)^{tC_p}
}\]
commutes. Here, as explained above, all objects arise from certain $BC_p$-equivariant functors $N(\Lambda_p^\op)\to \Sp^{B\T}$ by the mechanism of Proposition~\ref{prop:geomrealtaction} (except $\THH(A)^{tC_p}$, which however admits a natural map from such an object, over which everything else factors). These $BC_p$-equivariant functors $N(\Lambda_p^\op)\to \Sp^{B\T}$ are obtained by composing the $BC_p$-equivariant functor
\[
N(\Lambda_p^\op)\to N(\Free_{C_p})\times_{N(\Fin)} N(\Ass^\otimes_\act)\xto{A^\otimes} N(\Free_{C_p})\times_{N(\Fin)} \Sp^\otimes_\act
\]
with $B$ for $\THH(R)$, with $B_p^\T$ for $\Phi^{C_p}_{\mathcal U} \THH(R)$, with $I$ for $\THH(A)$, and with $\tilde{T}_p$ for (the variant of) $\THH(A)^{tC_p}$. In this translation, the upper horizontal map comes from the map $B_p^\T\to B$ which is the composite of $B_p^\T\to B_{p,\lax}^\T$ and $B_{p,\lax}^\T\to B$ from (T1) above. The left vertical map comes from the map $I\to B$ in (T3) above. The right vertical map comes from the composition of $B_p^\T\to B_{p,\lax}^\T$, $B_{p,\lax}^\T\to B_{p,\infty}$ in (T2) and $B_{p,\infty}\to \tilde{T}_p$ in (T5) above. Finally, the lower horizontal map comes from the map $I\to \tilde{T}_p$ in (T4) above. By Theorem \ref{thm:maincomp}, the resulting diagram
\[\xymatrix{
 & B_{p,\lax}^\T\ar[dl]_{(T1)}\ar[dr]^{(T2)}\\
B & B_p^\T\ar[l]\ar[u] & B_{p,\infty}\ar[d]^{(T5)}\\
I\ar[u]^{(T3)}_\simeq\ar[rr]^{(T4)} & & \tilde{T}_p
}\]
of $BC_p$-equivariant functors commutes. In fact, the outer part of the diagram commutes as $BC_p$-equivariant lax symmetric monoidal functors as $I$ is initial by Theorem~\ref{thm:maincomp}, and the small upper triangle commutes $BC_p$-equivariantly by construction (on the point-set model). Thus we have finally proven the following comparison.

\begin{corollary}
Let $R$ be an orthogonal ring spectrum which is levelwise well-pointed and such that the unit $S^0 \to R_0$ is an h-cofibration. Then under the equivalence of Theorem \ref{thmcomparison} the two constructions of cyclotomic structure maps $\THH(A) \to \THH(A)^{tC_p}$, where $A$ is the object of $\Alg_{\E_1}(\Sp)$ associated with the 1-categorical object $R \in \Alg(\Sp^O)$, are equivalent, functorially in $R$.$\hfill \Box$
\end{corollary}

\chapter{Examples}
In this final chapter, we discuss some examples from the point of view of this paper, as well as some complements.

We start in Section~\ref{sec:tatefrob} by defining a certain Frobenius-type map $R\to R^{tC_p}$ defined on any $\E_\infty$-ring spectrum $R$, which in the case of classical rings recovers the usual Frobenius on $\pi_0$, and the Steenrod operations on higher homotopy groups. This relies on the symmetric monoidal properties of the Tate diagonal. In Section~\ref{sec:commutative} we give a more direct construction and characterization of the cyclotomic structure on $\THH(A)$ for an $\E_\infty$-ring $A$, in terms of this $\E_\infty$-Frobenius map. In Section~\ref{sec:loopspaces}, we discuss the case of loop spaces, and finally in Section~\ref{sec:charp}, we discuss rings of characteristic $p$. In particular, we give a simple and complete formula for the $\E_\infty$-algebra in cyclotomic spectra $\THH(H\bF_p)$.

\section{The Tate valued Frobenius}\label{sec:tatefrob}

In this section we define a variant of the Frobenius homomorphism for an $\E_\infty$-ring spectrum $R$. In the next section, we explain the relation to the Frobenius on $\THH$.

\begin{definition}\label{deffrobenius} Let $R$ be an $\E_\infty$-ring spectrum and $p$ a prime. The Tate valued Frobenius of $R$ is the $\E_\infty$-ring map defined as the composition
\[
\varphi_R: R \xto{\Delta_p} (R \otimes \ldots \otimes R)^{tC_p} \xto{m^{tC_p}} R^{tC_p}
\]
where $m: R \otimes\ldots \otimes R \to R$ is the multiplication map, which is a $C_p$-equivariant map of $\E_\infty$-ring spectra when $R$ is equipped with the trivial action.
\end{definition}

\begin{example}
\begin{altenumerate}
\item Assume that $R=HA$ for an ordinary commutative ring $A$. In that case, the Frobenius homomorphism is a map $HA\to HA^{tC_p}$. In degree $0$, this recovers the usual Frobenius homomorphism
\[
A\to A/p=\widehat{H}^0(C_p,A)=\pi_0 HA^{tC_p}: a \mapsto a^p\ .
\]
We will see below that on higher homotopy groups, the Frobenius recovers the Steenrod operations.
\item If $R$ is a $p$-complete spectrum that is compact in the category of $p$-complete spectra, then as a consequence of the Segal conjecture the canonical map $R \to R^{hC_p} \to  R^{tC_p}$ induced by pullback along $BC_p \to \pt$ and the canonical map is an equivalence. Thus, in this case, the Tate valued Frobenius can  be considered as an $\E_\infty$-endomorphism of $R$. 

For example if $R=\bS_p^\wedge$ is the $p$-complete sphere spectrum, the Frobenius endomorphism is given by the identity map.
To make this more interesting, consider for every $n \geq 1$ the unique $\E_\infty$-ring spectrum $\bS_{W(\bF_{p^n})}$ with the following properties: it is $p$-complete, the map $\bS^\wedge_p \to \bS_{W(\bF_{p^n})}$ is \'etale (in the sense of \cite[Definition 7.5.0.4]{HA}) and it induces on $\pi_0$ the unique ring homomorphism $\Z_p=W(\bF_p) \to W(\bF_{p^n})$. We get a Frobenius endomorphism
\[
\bS_{W(\bF_{p^n})} \to \bS_{W(\bF_{p^n})} \ .
\]
We claim that it is  the unique $\E_\infty$-map that induces on $\pi_0$ the Frobenius of the Witt vectors $W(\bF_{p^n})$. To see this we compare it by naturality of the Tate valued Frobenius to its $0$-th Postnikov section. This is the Eilenberg MacLane spectrum $HW(\bF_{p^n})$ whose Frobenius is described in part (1).
\end{altenumerate}
\end{example}

Our goal is to describe the effect of the Tate valued Frobenius in terms of power operations. For this, we analyze some further properties of the Tate diagonal.

\begin{lemma}\label{lemfact} There is a unique lax symmetric monoidal factorization of the Tate diagonal
\[\xymatrix{
&& \left((X \otimes\ldots \otimes X)^{tC_p}\right)^{h \Fp^\times} \ar[d] \\
X \ar[rru] \ar[rr]^{\Delta_p} && (X \otimes\ldots \otimes X)^{tC_p} 
}\]
where $\bF_p^\times \cong \Aut(C_p) \cong C_{p-1}$ acts on $C_p$ (and is equal to the Weyl group of $C_p$ inside of $\Sigma_p$), and the vertical map is the inclusion of (homotopy) fixed points. For a suspension spectrum $X = \Sigma^\infty_+ Y$ there is a unique lax symmetric monoidal factorization
\[\xymatrix{
&& (\Sigma^\infty_+ Y \otimes\ldots \otimes \Sigma^\infty_+ Y)^{h\Sigma_p} \ar[d] \\
\Sigma^\infty_+ Y \ar[rru]^-{\Sigma^\infty_+ \Delta} \ar[rr]^{\Delta_p} && (\Sigma^\infty_+ Y \otimes\ldots \otimes \Sigma^\infty_+ Y)^{tC_p} 
}\]
\end{lemma}

\begin{proof} For the first assertion we construct the map 
$$
X \to \left((X \otimes\ldots \otimes X)^{tC_p}\right)^{h \Fp^\times}
$$ 
as in Proposition~\ref{prop:tatediaglaxsymm}. To this end we only need to observe that the functor 
$$
\Sp \to \Sp : \qquad X \mapsto \left((X \otimes\ldots\otimes X)^{tC_p}\right)^{h\Fp^\times}
$$
is exact. This follows from the fact that the functor $X \mapsto (X \otimes\ldots \otimes X)^{tC_p}$ is exact and that taking homotopy fixed points is exact as well. The commutativity of the diagram follows from the uniqueness part of Proposition~\ref{prop:tatediaglaxsymm}.

For the second assertion we first note that it is clear that there is a map $\Sigma^\infty_+ Y \to (\Sigma^\infty_+ Y \otimes\ldots \otimes \Sigma^\infty_+ Y)^{h\Sigma_p} $ induced from the space level diagonal. Thus we have to compare two lax symmetric monoidal transformations between functors $\calS \to \Sp$, and the result follows from \cite[Corollary 6.9 (4)]{Nik}, which says that the suspension spectrum functor $\Sigma^\infty_+$ is initial among all lax symmetric monoidal functors $\calS \to \Sp$.
\end{proof}

\begin{corollary}\label{refinement} The Tate valued Frobenius factors as an $\E_\infty$-map as a composition $R\to (R^{tC_p})^{h\Fp^\times} \to R^{tC_p}$.$\hfill \Box$
\end{corollary}

Note that the order $p-1$ of $\Fp^\times$ is invertible in $R^{tC_p}$. Therefore the homotopy fixed points are purely algebraic, i.e.~the homotopy groups of the fixed points are the fixed points of the homotopy groups. We will also refer to this refinement as the Tate valued Frobenius and we hope that it will be clear which one we mean.

\begin{corollary}\label{groupring} Let $M$ be an $\E_\infty$-monoid in $\calS$, so that the suspension spectrum $R =\bS[M]$ is an $\E_\infty$-ring spectrum. The Frobenius of $R$ admits a canonical factorization
\[
R \to (R^{hC_p})^{h\Fp^\times} \to R^{tC_p}\ .
\]
\end{corollary}
\begin{proof}
Lemma \ref{lemfact}. 
\end{proof}
\begin{remark} In fact Lemma \ref{lemfact} shows that the map factors not only through 
$$(R^{hC_p})^{h\Fp^\times} \simeq R^{h(C_p \rtimes \Fp^\times)}$$
but through $R^{h\Sigma_p}$. This is not much more general, since the canonical map $R^{h\Sigma_p} \to R^{h(C_p \rtimes \Fp^\times)}$ is a $p$-local equivalence. This relies on a transfer argument using the fact $C_p \subseteq \Sigma_p$ is a $p$-Sylow subgroup, $\Fp^\times$ is its Weyl group (equivalently $C_p \rtimes \Fp^\times$ is the normaliser of $C_p$ in $\Sigma_p$) and that $C_p$ intersects any nontrivial conjugate trivially.
\end{remark}

\begin{lemma}\label{lem:liftingTate} The lax symmetric monoidal transformation
\[
\Omega^\infty(\Delta_p): \Omega^\infty \to \Omega^\infty T_p^{h\Fp^\times}\]
of underlying spaces factors naturally as
\[
\Omega^\infty X \xto{\Delta_{\Omega^\infty X}} ((\Omega^\infty X)^{\times p})^{h\Sigma_p}\to \Omega^\infty(X^{\otimes p})^{h\Sigma_p} \to \left(\Omega^\infty(X^{\otimes p})^{hC_p}\right)^{h\Fp^\times} \xto{\can}  \Omega^\infty T_p(X)^{h\Fp^\times}\ .
\]
\end{lemma}

\begin{proof} In the statement of the Lemma we have two potentially different lax symmetric monoidal transformations from $\Omega^\infty:  \Sp \to \calS$ to the functor $ \Omega^\infty T_p^{h\Fp^\times}$. 
 But the functor $\Omega^\infty$ is initial among lax symmetric monoidal functors $\Sp \to \calS$ which is shown in \cite[Corollary 6.9 (2)]{Nik}. Thus the two transformations are canonically equivalent. 
\end{proof}

\begin{corollary} Let $X$ be a bounded below $p$-complete spectrum. For $i\geq 1$, there are canonical isomorphisms
$$
\pi_i\big((X \otimes \ldots \otimes X)^{hC_p}\big) \cong \pi_i(X) \oplus \pi_i\big((X \otimes \ldots \otimes X)_{hC_p}\big)
$$
of abelian groups, induced by the norm
\[
\Nm: (X \otimes \ldots \otimes X)_{hC_p}\to (X \otimes \ldots \otimes X)^{hC_p}
\]
and the space level diagonal
\[
\Omega^\infty X\to \Omega^\infty (X\otimes \ldots\otimes X)^{hC_p}\ .
\]
Moreover, there is a short exact sequence
\[
0\to \pi_0\big((X\otimes \ldots \otimes X)_{hC_p}\big)\to \pi_0\big((X \otimes \ldots \otimes X)^{hC_p}\big)\to \pi_0(X) \to 0\ .
\]
\end{corollary}

\begin{proof} Consider the fiber sequence
$$
(X \otimes \ldots \otimes X)_{hC_p} \to (X \otimes \ldots \otimes X)^{hC_p} \to (X \otimes \ldots \otimes X)^{tC_p}
$$
which gives rise to a long exact sequence in homotopy groups. The final term $ (X \otimes \ldots \otimes X)^{tC_p}$ receives the Tate diagonal $\Delta_p$ from $X$ which is an equivalence by Theorem \ref{thm:generalsegal}. By Lemma \ref{lem:liftingTate} the map $\Delta_p$ lifts after applying $\Omega^\infty$ through the middle term of the fiber sequence. In particular we get that on the level of homotopy groups, we have a canonical splitting of the map
$$\pi_i\big(( X \otimes \ldots \otimes X)^{hC_p}\big) \to  \pi_i\big((X \otimes \ldots \otimes X)^{tC_p}\big) = \pi_i(X)$$
for $i\geq 0$. If $i\geq 1$ then the splitting is a group homomorphism as the higher homotopy groups of a space are groups. 
\end{proof}

\begin{corollary} For an $\E_\infty$-ring spectrum $R$, the Frobenius morphism factors on underlying multiplicative $\E_\infty$-monoids as 
\[
\Omega^\infty R \xto{P_p} (\Omega^\infty R)^{h\Sigma_p} \to (\Omega^\infty R^{tC_p})^{h\Fp^\times}.
\]
where $P_p$ is the composition of the diagonal and the multiplication map.$\hfill \Box$
\end{corollary}

The last statement shows that the Frobenius is expressible in terms of the power operations. We refer to \cite{barthel2015character} for a discussion of power operations and the references there. The power operation is the map 
$$
P_p: R^0(X) \to R^0(X\times B\Sigma_p)
$$
for a space $X$ which is induced by postcomposition with the map denoted $P_p$ above. For a precise comparison to a more standard formulation of power operations in terms of extended powers see Lemma \ref{poweridentification} below.

\begin{corollary}\label{cor:tatepoweroperation} For an $\E_\infty$-ring spectrum $R$ and a space $X$, postcomposition with the Tate valued  Frobenius  is given on homotopy classes of maps by the composition
\[
\big[\Sigma^\infty_+X,R\big]=R^0(X) \xto{P_p} R^0(X\times B\Sigma_p) = \big[\Sigma^\infty_+X,R^{h\Sigma_p}\big] \xto{\can} \big[\Sigma^\infty_+X,(R^{tC_p})^{h\Fp^\times}\big]
\]
where $P_p$ is the usual power operation for $R$ and and the last map is induced by postcomposition with the forgetful map $R^{h\Sigma_p} \to (R^{hC_p})^{h\Fp^\times}$ and the canonical map $(R^{hC_p})^{h\Fp^\times} \to (R^{tC_p})^{h\Fp^\times}$.$\hfill \Box$
\end{corollary}

\begin{remark}
Note that the power operation $P_p$ is multiplicative but not additive. The way this is usually corrected is by taking the quotient of the ring $R^0(B\Sigma_p)$ by the transfer ideal $I \subseteq R^0(X \times B\Sigma_p)$. The transfer ideal is the image of the transfer $R^0(X) \to R^0(X \times B\Sigma_p)$.
The Tate construction can be considered as homotopically corrected version of this quotient. In fact there is a canonical map of rings $R^0(X \times B\Sigma_p)/I\to [\Sigma^\infty_+ X, R^{tC_p}]$ which however in practice is far from being an isomorphism. By Corollary \ref{cor:tatepoweroperation} the obvious diagram of maps of rings 
\[
\xymatrix{
R^0(X) \ar[r]\ar[rd] & R^0(X \times B\Sigma_p)/I \ar[d] \\
&  [\Sigma^\infty_+ X, R^{tC_p}]
}
\]
then commutes.  
\end{remark}

We  discuss two examples of the Tate valued Frobenius, namely $\KU$ and $H\bF_p$. First let $\KU$ be the periodic complex $K$-theory spectrum with homotopy groups $\pi_*\KU \cong \Z[\beta^{\pm 1}]$ with $|\beta| = 2$. As in Example~\ref{counterexamples}~(iii), we find that
\[
\pi_*\KU^{tC_p} \cong \pi_*\KU((t))/ ((t+1)^p-1) \cong \pi_* \KU \otimes \Q_p(\zeta_p)
\]
where $\zeta_p$ is a $p$-th root of unity. The action of $\mathbb{F}_p^\times = \mathrm{Gal}(\Q_p(\zeta_p)/ \Q_p)$ is given by the Galois action. Therefore we get 
\[
\pi_*(\KU^{tC_p})^{h\Fp^\times} \cong \pi_* \KU \otimes \Q_p(\zeta_p)^{\Fp^\times}=\pi_*\KU \otimes \Q_p
\]
Altogether we find that the natural constant map $\KU\to (\KU^{hC_p})^{h\Fp^\times}\xto{\can} (\KU^{tC_p})^{h\Fp^\times}$ (where the first map is pullback along $B(C_p \rtimes \Fp^\times) \to \pt$) extends to an equivalence
\begin{equation}\label{identification}
\KU^\wedge_p[1/p]\xto{\simeq} (\KU^{tC_p})^{h\Fp^\times}\ .
\end{equation}
More generally, for every space $X$ we look at the mapping spectrum $\map(X,\KU) = \KU^X$ and see with the same argument that the constant map induces an equivalence
\begin{equation*}
 \left(\KU^X\right)^\wedge_p[1/p] \simeq \left(\left(\KU^X\right)^{tC_p}\right)^{h\Fp^\times}\ .
 \end{equation*}
 In particular for a compact space $X$\footnote{By compact we mean a compact object in the $\infty$-category $\calS$ of spaces, i.e. weakly equivalent to a retract of a finite CW complex.} we find that 
 $\pi_0$ is given by $\KU^0(X) \otimes \Q_p$.
\begin{proposition}\label{propAdamsOp}
\begin{altenumerate} $ $
\item
For every  compact space $X$ the Tate valued Frobenius (or rather its refinement as in Corollary \ref{refinement}) of $\KU^X$ is given on $\pi_{0}$ by 
\begin{align*}
\KU^0(X) &\to \KU^0(X) \otimes \Q_p \\
V &\mapsto \psi^p(V)
\end{align*}
where 
$\psi^p: \KU^0(X) \to \KU^0(X)$ is the $p$-th Adams operation and $\psi^p(V)$ is considered as an element in $\KU^0(X) \otimes \Q_p$ under the canonical inclusion from $\KU^0(X)$.
\item
Under the identification \eqref{identification}, the Frobenius
\[
\varphi_p: \KU \to \KU^\wedge_p[1/p]
\]
agrees with the $p$-th stable Adams operation\footnote{Note that the endomorphism $\psi^p: \KU^0(X) \to \KU^0(X)$ does not refine to a stable cohomology operation, hence not to a map of spectra. But after inverting $p$ it does. One direct way to construct $\psi^p: \KU \to \KU[1/p]$ as a map of $\E_\infty$-ring spectra is to use Snaith's theorem $\KU = \bS[\C P^\infty][\beta^{-1}]$ and then induce it from the $p$-th power map $\C P^\infty \to \C P^\infty$.} 
\[
\psi_p: \KU \to \KU[1/p]
\]
composed with the natural ``completion'' map $\KU[1/p]\to \KU^\wedge_p[1/p]$. This composition sends the Bott element $\beta$ to $p \beta$ (recall that $\beta  = L - 1$ for the Hopf bundle $L = \calO_{\mathbb{P}^1}(-1)$ and that $\psi_p(L)  = L^p$ for every line bundle).
\end{altenumerate}
\end{proposition}
\begin{proof}

For the first part we follow closely Atiyah's paper \cite{Atiyah}. 
Since $X$ is compact we have 
$$
\pi_0\left(\left(\KU^X\right)^{hC_p}\right)^{h\Fp^\times} = \KU^0\big(X \times B(C_p \rtimes \Fp^\times)\big) = \KU^0(X) \otimes \KU^0(BC_p)^{\Fp^\times} 
$$
and 
$$
\pi_0\left(\left(\KU^X\right)^{tC_p}\right)^{h\Fp^\times} = \KU^0(X) \otimes \Q_p.
$$
According to Corollary \ref{cor:tatepoweroperation} we have to understand the map
$$
\KU^0(X) \xto{P_p}  \KU^0(X) \otimes \KU^0(BC_p)^{\Fp^\times} \xto{\can} \KU^0(X) \otimes \Q_p .
$$
Now $\KU^0(BC_p)^{\Fp^\times} = \Z \oplus \Z_p$ where $\Z$ is the summand corresponding to the trivial representation and  $\Z_p$ is the augmentation ideal generated by $[N]-p$ where $N$ is the regular representation of $C_p$ and $p$ is the trivial representation of rank $p$. In the computation above, the power series generator $t$ is chosen to corresponds to the element $N - p$. As a result we see that under the canonical map $\KU^0(BC_p)^{\Fp^\times} \to \Q_p$ the summand $\Z$ goes to $\Z \subseteq \Q_p$ and $N - p$ goes to $-p$ (since $N$ is an induced representation and thus lies in the image of the transfer map).  

The map
$$
\KU^0(X) \to  \KU^0(X) \otimes \KU^0(BC_p)^{\Fp^\times} 
$$
is induced by taking a vector bundle $V$ over $X$ to the tensor power $V^{\otimes p}$ considered as a $C_p$-equivariant vector bundle 
over $X$ which induces a vector bundle over $X \times BC_p$. Atiyah calls this operation the cyclic power operation and computes its effect on a general $K$-theory class $V \in \KU^0(X)$ in \cite[Equation (2.6), page 180]{Atiyah}:
$$
V \mapsto (V^p, \theta^p(V))
$$
where $V^p$ is the $p$-fold product in the ring $KU^0(X)$ and  $\theta^p: \KU^0(X) \to \KU^0(X)$ is a certain (unstable) cohomology operation of $K$-theory. The most important property of $\theta^p$ for us is that, in Atiyah's conventions, the $p$-th Adams operation $\psi^p$ is given by $\psi^p(V) = V^p - p \theta^p(V)$. Thus if we send the element $(V^p, \theta^p(V)) \in \KU^0(X) \otimes \KU^0(BC_p)^{\Fp^\times}$ to $\KU^0(X) \otimes \Q_p$ then we get $\psi^p(V)$. This shows part (i) of the Proposition.

For part (ii) we first note that for a compact space $X$ the Frobenius $\KU^X \to (\KU^X)^{tC_p}$ is given by mapping $X$ to the Frobenius $\KU \to \KU^{tC_p}$. Thus if we want to understand the map $\KU \to \KU^{tC_p}$ as a transformation of cohomology theories, we can use part (i). In particular for $X = S^2$ and the Bott element $\beta = L -1 \in KU^0(S^2)$ with $L \to S^2$ the Hopf bundle we find that 
$$\varphi_p(\beta) = \psi^p(\beta) = L^p - 1 = p(L-1) = p \beta \in KU^0(S^2).$$
Thus using part (i) we conclude that the map $\varphi_p: \KU \to \KU^\wedge_p[1/p]$ sends $\beta$ to $p \beta$ and that it agrees on $\beta$ with $\psi^p$. But since the target is rational we can rationalize the source, and $\KU_\Q$ is as a rational $\E_\infty$-spectrum freely generated by the Bott class. Thus an $\E_\infty$-map out of $\KU_\Q$ is uniquely determined by its effect on $\beta$. This shows part (ii).
\end{proof}

\begin{remark}
If one works, following Atiyah, with genuine equivariant $K$-theory then the statement of Proposition \ref{propAdamsOp} becomes a bit cleaner since $\Q_p$ is replaced by $\Z[1/p]$. More precisely for a genuine $C_p$-equivariant $\E_\infty$-ring spectrum $R$ (a term which we have not and will not define in this paper) the Tate valued Frobenius refines to a map $R \to \Phi^{C_p} R \xto{\can} R^{tC_p}$. This first map is usually an uncompleted version of the Tate valued Frobenius.  
\end{remark}

\begin{remark}
There is another argument to prove Proposition \ref{propAdamsOp} which is somewhat simpler but less explicit. Let us give the argument for $p=2$, the odd case being similar. One looks at the space level diagonal of the space $K(\Z,2)$ and composes it with the addition to get a map
$$
K(\Z,2) \to (K(\Z,2) \times K(\Z,2))^{hC_2} \to K(\Z,2)^{hC_2} = K(\Z,2) \sqcup K(\Z,2).
$$
This map factors as 
$$
K(\Z,2) \xto{\cdot 2} K(\Z,2) \xto{\mathrm{triv}} K(\Z,2)^{hC_2}
$$
where triv is trivial inclusion (i.e. pullback along $BC_2 \to \pt$). 
But there is a map of $\E_\infty$-spaces 
$
K(\Z,2) \to \Omega^\infty \KU
$
where the multiplicative $\E_\infty$-structure of $\KU$ is taken. As a result we get that we have a commutative diagram
$$
\xymatrix{
K(\Z,2) \ar[rr]^i\ar[d]^{\cdot 2} && \Omega^\infty \KU\ar[d]^{P_2} \\
K(\Z,2) \ar[rr]^-{\mathrm{triv}\circ i}  && (\Omega^\infty \KU)^{hC_2}
}
$$
where triv now denotes the trivial inclusion for $\KU$. Classes in $\KU^0(X)$ factor through $K(\Z,2)$ precisely if they are represented by line bundles. Thus the commutative diagram tells us that the power operation and the Tate-valued Frobenius are given on classes of line bundles by taking the line bundle to the second power. But then the splitting principle implies that it has to be the Adams operation.
\end{remark}

Now we consider the Eilenberg--Mac-Lane spectrum $H\Fp$. The Tate valued Frobenius is a map
$$\varphi_p: H\Fp\to  (H\Fp^{tC_p})^{h\Fp^\times}.$$
Recall that if $p$ is odd, then
\[
\pi_{-\ast} H\Fp^{tC_p} = \widehat{H}^\ast(C_p,\Fp) = \Fp((t))\otimes \Lambda(e)
\]
with $|t|=-2$ and $|e|=-1$ in cohomological grading; if $p=2$, then
\[
\pi_{-\ast} H\Fp^{tC_p} = \widehat{H}^\ast(C_p,\Fp) = \Fp((s))
\]
with $|s|=-1$. Here we write $\Fp((t))$ for the graded ring $\Fp[t^\pm]$ since both are isomorphic and this convention makes things work out more cleanly later. 
We note that $H\Fp^{tC_p}$ is an $H\Fp$-module spectrum and as such there is a splitting
\[
H\Fp^{tC_p} = \prod_{n\in \Z} H\Fp[n] = \bigoplus_{n\in \Z} H\Fp[n]\ .
\]
The action of $\Fp^\times$ on $\widehat{H}^1(C_p,\Fp)=H^1(C_p,\Fp)=\Hom(C_p,\Fp)=\Fp\cdot t^{-1} e$ is given by the inverse of the natural multiplication action, and the action of $\Fp^\times$ on $\widehat{H}^{-1}=\ker(\Nm: \Fp\to \Fp)=\Fp\cdot e$ is trivial. As it is a ring action, it follows that in general the action on $t^i$ is given by the $i$-th power of $\Fp^\times$. Therefore, we see that
\begin{align*}
(H\mathbb F_2)^{tC_2} &= \prod_{n \in \Z}H\mathbb{F}_2[n]\ ,\\
(H\Fp^{tC_p})^{h\Fp^\times} &= \prod_{n\in \Z} \big(H\Fp[n(p-1)] \oplus H\Fp[n(p-1)+1]\big)\ . 
\end{align*}
where the second holds for all odd $p$. Similarly we have
\begin{align*}
(H\mathbb \Z)^{tC_2} & = \prod_{n \in \Z}H\mathbb{F}_2[2n]\ , \\
(H\Z^{tC_p})^{\Fp^\times} &= \prod_{n\in \Z} H\Fp[n(p-1)] 
\end{align*}
where again the second holds for odd $p$. The following theorem says that the effect of the Frobenius morphism on higher homotopy groups is given by the Steenrod operations. We thank Jacob Lurie for pointing out this result to us. For $p=2$ this has in effect already been shown by Jones and Wegmann, see Lemma 5.4 and the formulas in \cite{MR720798} and for odd $p$ in Chapter 1 of \cite{MR866482}. We thank John Rognes for making us aware of this fact. 

\begin{thm}\label{frobhz}
The Frobenius endomorphism $H\Fp \to (H\Fp^{tC_p})^{h\Fp^\times}$ is given as a map of spectra by the product over all Steenrod squares
\[
Sq^i: H\mathbb F_2\to H\mathbb F_2[i]\ ,\ i\geq 0\ ,
\]
for $p=2$ and by the product over all Steenrod powers
\[
P^i: H\Fp\to H\Fp[i(p-1)]\ ,\ i\geq 0
\]
and
\[
\beta P^i: H\Fp\to H\Fp[i(p-1)+1]
\]
for $p$ odd, where $\beta$ denotes the Bockstein.

The Frobenius endomorphism $H\Z\to (H\Z^{tC_p})^{h\Fp^\times}$ is given by the product over all $Sq^{2i} r$ for $p=2$ and by the product over all $P^i r$ for odd $p$, where $r: H\Z \to H\Fp$ denotes the reduction map. 
\end{thm}

\begin{proof} In Proposition \ref{Steenrod} below, we will describe the Frobenius map for $H\bF_p$ as a map of cohomology theories. Since there are no phantoms between $H\Fp$-modules this shows the claim for $\Fp$. The claim for the integers follows from the naturality of the Frobenius map under the reduction map $H\Z \to H\Fp$.
\end{proof}

We want to understand the effect of
\[
\varphi_p: H\Fp\to (H\Fp^{tC_p})^{h\Fp^\times}
\]
as a transformation of multiplicative cohomology theories. To this end note that by the above, for every space $X$ 
the cohomology theory $(H\Fp^{tC_p})^{h\Fp^\times}$ evaluated on $X$ gives the graded ring
$$
H^\ast(X; \Fp)((t^{p-1}))\otimes \Lambda (e) \qquad \qquad |t^{p-1}| = -2(p-1), |e| = -1.
$$ 

\begin{proposition}\label{Steenrod}  As a map of cohomology theories evaluated on a space $X$, 
the Tate-valued Frobenius $\varphi_p: H\Fp\to  (H\Fp^{tC_p})^{h\Fp^\times}$ is given for $p=2$ by 
\begin{align*}
H^{*}(X; \mathbb{F}_2) &\to H^{*}(X; \mathbb{F}_2)((s)) \\
x &\mapsto \sum_{i = 0}^{\infty} Sq^i(x) s^i
\end{align*}
and for $p$ odd by
\begin{align*}
H^*(X;\Fp) &\to H^*(X; \Fp)((t^{p-1}))\otimes \Lambda (e) \\
 x &\mapsto \sum_{i = 0}^{\infty} P^i(x) t^{(p-1)i} +  \sum_{i = 0}^{\infty}\beta P^i(x)  t^{(p-1)i} e\ .
 \end{align*}
\end{proposition}

In order to prove this result we have to employ power operations in positive degrees and prove an analogue of  Corollary \ref{cor:tatepoweroperation} for those. Let us start by explaining a conceptual framework for power operations in positive degree that we have learned from Charles Rezk. Let $R$ be an $\E_\infty$-ring spectrum. Then we form the underlying space of the $n$-fold suspension $\Omega^\infty (\Sigma^n R)$ and take its space level diagonal 
$$
\Omega^{\infty} \Sigma^nR \to \left( \left(\Omega^{\infty} \Sigma^nR\right)^{\times p} \right)^{h\Sigma_p}.
$$
Since $\Omega^\infty$ is lax symmetric monoidal and limit preserving there is a canonical map
$$
 \left( \left(\Omega^{\infty} \Sigma^nR\right)^{\times p} \right)^{h\Sigma_p} \to   \Omega^{\infty}\left((\Sigma^nR)^{\otimes p}  \right)^{h\Sigma_p}
$$
Now $(\Sigma^n R)^{\otimes p}$ is equivalent to $\Sigma^{np}R^{\otimes p}$ as a spectrum, but $\Sigma_p$ acts on the suspension coordinates. In fact as a spectrum with $\Sigma_p$-action this is equivalent to the tensor product of $R^{\otimes p}$ with the representation sphere\footnote{The representation sphere $S^V$ for a real representation  $V$ is the one point compactification of $V$. It has a canonical basepoint, namely the added point $\infty \in S^V$.} $S^{n \cdot \nat}$ where $n \cdot \nat$ denotes the $n$-fold sum of the natural representation $\nat$ of $\Sigma_p$ on $\mathbb{R}^p$. We will denote the tensor product with the suspension spectrum $\Sigma^\infty S^{n \cdot \nat}$ of this representation sphere by $\Sigma^{n \cdot \nat}$ so that we find
$$
 \Omega^{\infty}\left((\Sigma^nR)^{\otimes p}  \right)^{h\Sigma_p} \simeq \Omega^\infty ( \Sigma^{n \cdot \nat} R^{\otimes p})^{h \Sigma_p}.
$$

Now we can postcompose with the multiplication of $R$, which is a $\Sigma_p$-equivariant map $R^{\otimes p} \to R$ to obtain a map
\begin{equation}\label{mappower}
P_p^n\colon \Omega^{\infty} \Sigma^nR \to \Omega^\infty ( \Sigma^{n \cdot \nat} R)^{h \Sigma_p}
\end{equation}
where $\Sigma_p$ acts trivially on $R$.
We compare this map with the $n$-fold suspension of the Tate valued Frobenius 
$$
\Sigma^n \varphi_R\colon \Sigma^n R \to \Sigma^n (R^{tC_p})^{h\Fp^\times}.
$$
We first compare the targets, which is done in the next lemma.

\begin{lemma}\label{leminvertible}
For every $\E_\infty$-ring spectrum $R$ and every $n \in \Z$ there is an equivalence  $\Sigma^n (R^{tC_p})^{h\Fp^\times} \xto{\simeq}  ((\Sigma^{n \cdot \nat} R)^{t C_p})^{h\Fp^\times}$.
\end{lemma}
\begin{proof}
Consider the functor 
$$
\widetilde{T_p}: \Mod_R \to \Mod_R \qquad M \mapsto  \big((M \otimes_R \ldots \otimes_R M)^{tC_p}\big)^{h\Fp^\times} \ .
$$ 
Then we have that 
$$((\Sigma^{n \cdot \nat} R)^{t C_p})^{h\Fp^\times} \simeq \widetilde{T_p}(\Sigma^n R)\ .$$
The functor $\widetilde{T}_p$ is exact, by the same argument as in Proposition \ref{prop:tatediagonaladditive}.
Thus there is an equivalence
$$
\Sigma^n \widetilde{T}_p(R) \to \widetilde{T}_p(\Sigma^n R).
$$
for every $n \in \Z$.
\end{proof}

\begin{lemma}\label{lempoweretc}
 For an $\E_\infty$-ring spectrum $R$, there is a commutative diagram 
\[
\xymatrix{
\Omega^\infty\Sigma^n R \ar[rr]^-{\Omega^\infty\Sigma^n \varphi_R} \ar[d]_-{P_p^n}&&  \Omega^\infty \Sigma^n (R^{tC_p})^{h\Fp^\times} \ar[d]^{\simeq} \\
\Omega^\infty ( \Sigma^{n \cdot \nat} R)^{h \Sigma_p} \ar[rr]^{\can} && \Omega^\infty ((\Sigma^{n \cdot \nat} R)^{t C_p})^{h\Fp^\times}
}
\]
of spaces.
\end{lemma}
 \begin{proof}
 The $n$-fold supension of the Tate valued Frobenius is by definition obtained from the $n$-fold supension of the Tate valued diagonal 
 $$\Sigma^n \Delta_R: \Sigma^n R \to \Sigma^n ((R^{\otimes p})^{tC_p})^{h\Fp^\times}$$ 
 by postcomposition with the $p$-fold multiplication of $R$ (here our notation slightly differs from the notation for the Tate valued diagonal used before). By the exactness of the source and target this is the same as the diagonal of the $n$-fold suspension, which is a morphism
  $$\Delta_{\Sigma^n R}: \Sigma^n R \to (( \Sigma^{n \cdot \nat}  R^{\otimes p})^{tC_p})^{h\Fp^\times}$$
  and the targets of the two maps are identified by the canonical equivalence 
$$ \Sigma^n ((R^{\otimes p})^{tC_p})^{h\Fp^\times} \xto{\simeq} (( \Sigma^{n \cdot \nat}  R^{\otimes p})^{tC_p})^{h\Fp^\times}.$$
This equivalence fits by construction into a commutative square
$$
\xymatrix{
\Sigma^n ((R^{\otimes p})^{tC_p})^{h\Fp^\times} \ar[d]^m \ar[r]^\simeq & (( \Sigma^{n \cdot \nat}  R^{\otimes p})^{tC_p})^{h\Fp^\times}\ar[d]^m \\
 \Sigma^n (R^{tC_p})^{h\Fp^\times} \ar[r]^{\simeq} & ((\Sigma^{n \cdot \nat} R)^{t C_p})^{h\Fp^\times}
}
$$
where  the vertical maps are the multiplications of $R$ and the lower horizontal map is the map of Lemma \ref{leminvertible}.
Now we can use Lemma \ref{lem:liftingTate} to see that the space level diagonal of $\Omega^\infty(\Sigma^n R)$ factors the Tate diagonal of $\Sigma^n R$.
Thus we get that 
$$\can \circ P_n^k\colon \Omega^\infty\Sigma^n R  \to  \Omega^\infty ((\Sigma^{n \cdot \nat} R)^{t C_p})^{h\Fp^\times}$$ 
is equivalent to $\Omega^\infty$ of the  composition 
$$
m \circ \Delta_{\Sigma^n R}\colon \Sigma^n R \to ((\Sigma^{n \cdot \nat} R)^{t C_p})^{h\Fp^\times} \ .
$$
By what we have said above the last map is canonically equivalent to the clockwise composition in the diagram of the lemma.
 \end{proof}
 
 Now we claim that evaluation of the map 
 $
 P_p^n\colon \Omega^{\infty} \Sigma^nR \to \Omega^\infty ( \Sigma^{n \cdot \nat} R)^{h \Sigma_p}
 $ 
 on a space $X$ is the classical power operation in positive degrees. To see this we give a slightly different description of homotopy classes of maps $[\Sigma^\infty_+ X, (\Sigma^{n \cdot \nat} R)^{h \Sigma_p}]$ in terms of twisted $R$-cohomology of $X \times B\Sigma_p$. 

First if $V \to X$ is a real vector bundle over a topological space\footnote{Here, as usual, we implicitly consider the topological space $X$ as an object in the $\infty$-category of spaces. Then a rank $n$ vector bundle over $X$ really is a map $X \to BO(n)$. We use the total space terminology $V \to X$ just as a placeholder and to make contact to classical terminology.} we will write $X^{-V}$ for the Thom spectrum of the negative of this bundle. For the definition of this spectrum see the discussion in the proof of Lemma \ref{lemThom} below. Then for a spectrum $R$ we define the $V$-twisted $R$-cohomology by $R^V(X) :=[X^{-V}, R]$. For $V$ the trivial rank $n$-bundle this recovers degree $n$-cohomology of $R$. 

\begin{lemma}\label{lemThom}
There is an isomorphism
$$
R^{n \cdot \nat}(X \times B\Sigma_p) \cong \big[\Sigma^\infty_+ X, (\Sigma^{n \cdot \nat} R)^{h \Sigma_p}\big]
$$
where $n\cdot  \nat$ is considered as a vector bundle over $B \Sigma_p \times X$ by pulling back the bundle associated to the representation over $B\Sigma_p$ along the projection. 
\end{lemma}
\begin{proof}
To see this we recall a possible definition of $X^{-V}$ following \cite{MR3286898,MR3252967}. First one considers the spherical fibration $S^V \to X$ associated to $V$ by taking the fiberwise one-point compactifications. Taking fiberwise suspension spectra (where the basepoints at infinity are taken into account) yields a parametrized spectrum $\bS^V$ over $X$\footnote{We think of parametrized spectra over a space $X$ as functors $ X \to \Sp$. As an $\infty$-category $\Sp^X := \Fun(X, \Sp)$.} and we can take the tensor inverse $\bS^{-V}$ in parametrized spectra over $X$. Now for the terminal map $p: X \to \pt$ there is a pullback morphism $p^*: \Sp  \to \Sp^X$ which has a left adjoint denoted $p_!: \Sp^{X} \to \Sp$. Then $X^{-V} := p_!(\bS^{-V})$. To see that this agrees with the usual definition of Thom spectra we note the following facts:
\begin{enumerate}
\item
the usual Thom space $X^V$ of a vector bundle $V$ over $X$ is just $p_!(S^V)$ where $p_!$ now denotes the unstable version, i.e. the left adjoint to the functor $p^*: \calS_\ast \to (\calS_{/X})_\ast$. This left adjoint just identifies the section at infinity of $S^V$, which is the way the Thom space is built.
\item
The usual definition of the associated Thom spectrum (i.e the Thom spectrum of the positive bundle $V$) is  the suspension spectrum $\Sigma^\infty X^V$ where the basepoint is taken into account (which we will abusively also denote by $X^V$ in general but we do not drop the $\Sigma^\infty$ for the moment). This is also the same as $p_! \bS^V$ since the stable and unstable version of $p_!$ commute with the suspension functors, more precisely the diagram 
$$
\xymatrix{
(\calS_{/X})_\ast \ar[r]^{p_!} \ar[d]_{\Sigma^\infty_X}& \calS_\ast \ar[d]^{\Sigma^\infty}  \\
\Sp^X\ar[r]^{p_!} & \Sp
} 
$$
in $\Cat_\infty$ commutes. This can be seen by passing to the diagram of right adjoints.
\item The classically defined Thom spectrum $\Sigma^\infty X^{V \oplus \R^n}$ is equivalent to the $n$-fold suspension of the Thom spectrum $\Sigma^\infty X^V$. The same is true for $p_{!} \bS^{V \oplus \R^n}$ since $S^{V \oplus \R^n} \simeq \Sigma^n S^V$ in $\Sp^X$ and since $p_!$ is left adjoint. 
\item
Finally the spectrum $X^{-V}$ would be classically defined by finding a vector bundle $W$ over $X$ such that $V \oplus W\cong \R^n$ (i.e. that $-V = W -\R^n$) and then setting $X^{-V} := \Sigma^{-n} \Sigma^{\infty}X^W$. Such a $W$ exists for $X$ a finite CW complex but not in general, thus one has to take the colimit over finite approximations of $X$. But we first assume that $W$ exists. Then the spectrum $X^{-V}$ is characterized by the property that suspending $n$ times we get $\Sigma^\infty X^W$. The same is true for $p_! \bS^{-V}$ since we have
$$
\Sigma^n p_! \bS^{-V} \simeq p_!( \Sigma^n \bS^{-V}) \simeq p_!(\bS^{-V + \R^n}) \simeq p_!\bS^W.
$$
As a last step we note that for a colimit $X \simeq \colim X_i$ in $\calS$ we find that  both descriptions of Thom spectra are compatible with colimits.
\end{enumerate}

Now we try to compute $R^V(X) = [X^{-V}, R]$ for some space $X$. Since $X^{-V} = p_! \bS^{-V}$ we find that by adjunction such homotopy classes are the same as maps in the homotopy category $\Sp^X$ from $\bS^{-V}$ to $p^*R$ which is the constant parametrized spectrum over $X$ with value $R$. But then we find
$$
\map (\bS^{-V}, p^*R)= \map(\bS^0, \bS^V \otimes_X p^*R) = p_*(\bS^V \otimes_X p^*R)
$$
where $p_*: \Sp^X \to \Sp$ is the right adjoint to $p^*$ and $\otimes_X$ is the pointwise tensor product of parametrized spectra over $X$. Now we specialize to the case of interest:
\begin{align*}
R^{n \cdot \nat}(X \times B\Sigma_p) &= [(X \times B\Sigma_p)^{-n \cdot \nat}, R] \\
& \cong [\Sigma^\infty_+ X \otimes B\Sigma^{-n \cdot \nat}, R] \\
&\cong [\Sigma^\infty_+ X, \map ( B\Sigma^{-n \cdot \nat} p^*R)] \\
& \cong [\Sigma^\infty_+X, p_*(\bS^{n \cdot \nat} \otimes_{B\Sigma_p} p^*R)] 
\end{align*}
where $p: B\Sigma_p \to \pt$. But by definition parametrized spectra over $B\Sigma_p$ are the same as spectra with $\Sigma_p$-action and under this identification $p_*$ corresponds to taking homotopy fixed points, so that $p_*(\bS^{n \cdot \nat} \otimes_{B\Sigma_p} p^*R)$ is equivalent to $(\Sigma^{n \cdot \nat} R)^{h \Sigma_p}$. 
\end{proof}

Now we can describe the $p$-th power operation in degree $n$ following the formulation that we learned from Rezk as the map 
$$
R^{n}(X) \to R^{n \cdot \nat}(X)
$$
described as follows: represent a degree $n$-cohomology class on $X$ by a map $X^{-n} \to R$ where $n$ is short for the trivial rank $n$-bundle over $X$ (i.e. $X^{-n}$ is just the $n$-fold desuspension of $\Sigma^\infty_+ X$). Take the $p$-th extended power
\begin{equation}\label{powermap}
\big((X^{-n})^{\otimes p}\big)_{h\Sigma_p} \to (R^{\otimes p})_{h\Sigma_p} 
\end{equation}
The source is equivalent to the Thom spectrum $(D_pX)^{-W}$ where $D_pX$ is the space $(X \times \ldots \times X)_{h\Sigma_p}$ and $W$ is the vector bundle obtained from the $p$-fold external sum of the trivial rank $n$-bundle with itself (with its permutation action). As a bundle $W$ is pulled back from the bundle $n \cdot \nat$ over the space $B\Sigma_p$ along the canonical map $D_p X \to B\Sigma_p$. Thus if we pull this bundle further back along the diagonal $X \times B\Sigma_p \to D_p X$ we get the bundle also denoted $n \cdot \nat$, but now considered as a bundle over $X \times B\Sigma_p$ (following the notation from before). Therefore we obtain a canonical map
$$
(X \times B\Sigma_p)^{-n \cdot \nat} \to (D_pX)^{-W}.
$$
Now we take the map in \eqref{powermap} and precompose with the map just described and postcompose with the multiplication map $(R^{\otimes p})_{h\Sigma_p} \to R$ to obtain an element in $R^{n \cdot \nat}(X)$. This describes the power operation $P^n_p: R^{n}(X) \to R^{n \cdot \nat}(X)$.\footnote{Note that we could have more generally started with a class in $R^V(X)$ for some non-trivial bundle $V \to X$ and would have obtained a class in $R^{V \cdot \nat}(X)$ but we will not need this extra generality here.}
The clash of notation between this power operation with the space level map $P_p^n$ in \eqref{mappower} is justified by the next result.
\begin{lemma}\label{poweridentification}
The power operation $R^n(X) \to R^{n \cdot \nat}(X)$ (as just described) is given under the identifications $R^n(X) \cong [X, \Sigma^n R]$ and $R^{n \cdot \nat}(X) \cong [X, (\Sigma^{n \cdot \nat} R)^{h \Sigma_p}]$ by postcomposition with the map $P_p^n$ as described before Lemma \ref{lempoweretc}.
\end{lemma}
\begin{proof}
Consider the following commutative diagram
$$
\xymatrix{
[X, \Omega^\infty \Sigma^n R] \ar[r]^\simeq \ar[d]^{-^{\times p}} & [\Sigma^\infty_+ X, \Sigma^n R] \ar[d] \ar[d]^{-^{\otimes p}}\ar[r]^\simeq &  [X^{-n}, R] \ar[d]^{-^{\otimes p}}\\
\big[X^{\times p}, (\Omega^\infty \Sigma^n R)^{\times p}\big]^{\Sigma_p} \ar[r]\ar[d]^{\Delta^*} &  \big[(\Sigma^\infty_+ X)^{\otimes p} , (\Sigma^n R)^{\otimes p}\big]^{\Sigma_p} \ar[r]^-\simeq \ar[d]^{\Delta^*} & 
\big[(X^{-n})^{\otimes p}, R^{\otimes p}\big]^{\Sigma_p} \ar[d]^{\Delta^*}\\
\big[X, (\Omega^\infty \Sigma^n R)^{\times p}\big]^{\Sigma_p} \ar[r]\ar[d] &  \big[\Sigma^\infty_+ X , (\Sigma^nR)^{\otimes p}\big]^{\Sigma_p} \ar[r]^\simeq\ar[d]^{m_*} & \big[X^{-n \cdot \nat},R^{\otimes p}\big]^{\Sigma_p} \ar[d]^{m_*}  \\ 
[X, \Omega^\infty( \Sigma^{n \cdot \nat}) R]^{\Sigma_p} \ar[r]^\simeq &  [\Sigma^\infty_+ X , \Sigma^{n \cdot \nat}R]^{\Sigma_p} \ar[r]^\simeq & [X^{-n \cdot \nat}, R]^{\Sigma_p}
}
$$
where we denote by $[-,-]^{\Sigma_p}$ homotopy classes of $\Sigma_p$-equivariant maps between spaces resp. spectra with $\Sigma_p$-action (not to be confused with $\Sigma_p$-fixed points of the action on the set of non-equivariant homotopy classes). If there is no obvious $\Sigma_p$ action on an object, for example on $X$ or $R$ then it is considered as equipped with the trivial action. 

Now we note that the composition of the first two maps in the  left vertical column is equivalent to postcomposition with the space valued diagonal
$$
[X, \Omega^\infty \Sigma^n R]  \to \big[X, (\Omega^\infty \Sigma^n R)^{\times p}\big]^{\Sigma_p}.
$$
Since $\Sigma_p$ acts trivially on the source the target is also isomorphic to the group $\big[X, ((\Omega^\infty \Sigma^n R)^{\times p})^{h\Sigma_p} \big]$. Under this isomorphism the map is given by postcomposition with the refined diagonal $ \Omega^\infty \Sigma^n R \to  \big((\Omega^\infty \Sigma^n)^{\times p}\big)^{\Sigma_p}$. Thus by definition of $P_p^n$ we see that the left vertical composition in the big diagram is equivalent to postcomposition with $P_p^n$.  

For the right column in the big diagram we note that the composition of the last two maps
 is equivalent to the composition
$$
\big[(X^{-n})^{\otimes p}, R^{\otimes p}\big]^{\Sigma_p}  \to \big[(X^{-n})^{\otimes p}_{h\Sigma_p}, R^{\otimes p}_{h\Sigma_p}\big] \xto{(\Delta^*, m_*)} [(X^{-n\cdot \nat})_{h\Sigma_p}, R]
$$
which is straigtforward to check. Since $(X^{-n\cdot \nat})_{h\Sigma_p} \simeq (X \times B\Sigma_p)^{-n\cdot \nat}$ we see that the composition in the right column of the big diagram is given by the power operation $P_p^k$. Together this shows the claim.  
\end{proof}

For the next statement we will abusively denote the composition
$$
(\Sigma^{n \cdot \nat} R)^{h\Sigma_p} \to {((\Sigma^{n \cdot \nat} R)^{hC_p})^{h\Fp^\times}}  \xto{\can} ((\Sigma^{n \cdot \nat} R)^{tC_p})^{h\Fp^\times} \xto{\simeq} 
\Sigma^n (R^{tC_p})^{h\Fp^\times} 
$$
by `can' as well.  Here the first map forgets fixed point information and the last map is the inverse of the equivalence of Lemma \ref{leminvertible}.

\begin{thm}\label{thm:tatepoweroperationgeneral} For an $\E_\infty$-ring spectrum $R$ and every space $X$, postcomposition with the $n$-th suspension of the Tate valued  Frobenius $\Sigma^n R \to \Sigma^n (R^{tC_p})^{h\Fp^\times}$ is given by the composition
\[
R^n(X) \xto{P^n_p} R^{n \cdot \nat}(B\Sigma_p \times X) \cong [\Sigma^\infty_+X,(\Sigma^{n \cdot \nat} R)^{h\Sigma_p}] \xto{\can} [\Sigma^\infty_+X,\Sigma^n (R^{tC_p})^{h\Fp^\times}]
\]
where $P^n_p$ is the degree $n$ power operation for $R$ as discussed above.
\end{thm}

\begin{proof}
This is an immediate consequence of Lemma \ref{lempoweretc} and Lemma \ref{poweridentification}.
\end{proof}

It is easy to see that for varying degrees the power operation $P_p^*: R^*(X) \to R^{* \cdot \nat}(B\Sigma_p \times X) $ is a map of graded rings using the usual monoidal properties of Thom spectra. Thus the last corollary is most effectively used by considering it as a factorization of the Frobenius as a map of graded rings. Unfortunately, it is sometimes not so easy to identify the graded ring $R^{* \cdot \nat}(B\Sigma_p \times X)$. However in the case of ordinary cohomology with $\Fp$ coefficients this is possible as we will see now. 

\begin{proof}[Proof of Proposition \ref{Steenrod}]
The result follows from Theorem \ref{thm:tatepoweroperationgeneral} and the fact that the Power operations of $H\Fp$ are given by the Steenrod operations, see e.g.~\cite{MR0145525} or 
\cite[Chapter VIII.2]{MR836132}. But let us be a bit more careful and identify the groups and the maps that show up in Theorem~\ref{thm:tatepoweroperationgeneral} for $H\Fp$.

The definition of the Steenrod operations in \cite[Chapter VIII.2]{MR836132} is the following. One considers the Power operation
\[
[\Sigma^\infty_+ X,\Sigma^n H\Fp]\to [\Sigma^\infty_+ X,(\Sigma^{n\cdot \nat} H\Fp)^{h\Sigma_p}]
\]
and composes with the projection
\[
[\Sigma^\infty_+ X,(\Sigma^{n\cdot \nat} H\Fp)^{h\Sigma_p}]\to [\Sigma^\infty_+ X,(\Sigma^{n\cdot \nat} H\Fp)^{hC_p}]\ .
\]
If $p$ is odd, then the natural representation $\nat$ of $C_p$ on $\R^p$ is oriented (with the usual orientation of $\R^p$ declaring $e_1, \ldots, e_p$ to be a positive basis) . Thus we have a Thom isomorphism $\Sigma^{n \cdot \nat} H\Fp \simeq \Sigma^{np} H\Fp$, which is an equivalence in $\Sp^{BC_p}$. Similarly, if $p=2$, then any vector bundle has a Thom isomorphism with $H\mathbb F_2$-coefficients, so again we find a Thom isomorphism $\Sigma^{n\cdot \nat} H\Fp\simeq \Sigma^{np} H\Fp$. Consequently, we get a map
\[
[\Sigma^\infty_+ X,\Sigma^n H\Fp]\to [\Sigma^\infty_+ X,(\Sigma^{n\cdot \nat} H\Fp)^{hC_p}]\cong [\Sigma^\infty_+ X,(\Sigma^{np} H\Fp)^{hC_p}]\ .
\]
Identifying $H\Fp^{hC_p}=\prod_{m\leq 0} H\Fp[m]$, this gives a map
\[
H^n(X,\Fp)\to \prod_{i\geq 0} H^{pn-i}(X,\Fp)
\]
whose components are by definition the Steenrod operations. In particular, composing further with the map $(\Sigma^{n\cdot \nat} H\Fp)^{hC_p}\to (\Sigma^{n\cdot \nat} H\Fp)^{tC_p}\cong \Sigma^n H\Fp^{tC_p}$ which is an isomorphism in cohomological degrees $\geq -np$, we get the desired description of $\varphi_p$ by Theorem~\ref{thm:tatepoweroperationgeneral}.
 \end{proof}

\section{Commutative algebras}\label{sec:commutative}

In this section we  given an easier description of the cyclotomic structure on $\THH(A)$ for an $\E_\infty$-ring spectrum $A$. First we show that if $A$ is an $\E_\infty$-ring spectrum, then $\THH(A)$ is again an $\E_\infty$-ring spectrum. This follows from the following construction.
\begin{construction} The $\infty$-categories $\Cycn$, $\Cycn_p$ of ($p$-)cyclotomic spectra have a natural symmetric monoidal structure. This follows from the following two observations:
\begin{altenumerate}
\item If $\calC$ is a symmetric monoidal $\infty$-category, then $\calC^S = \Fun(S,\calC)$ is naturally a symmetric monoidal $\infty$-category for any simplicial set $S$. Indeed, one can define the total space as $\Fun(S,\calC)^\otimes = \Fun(S,\calC^\otimes)\times_{\Fun(S,N(\Fin_\ast))} N(\Fin_\ast)$, and it is easy to verify that this defines a symmetric monoidal $\infty$-category, since exponentials of coCartesian fibrations are again coCartesian \cite[Corollary 3.2.2.12]{HTT}.
\item If $\calC$ and $\calD$ are symmetric monoidal $\infty$-categories, $F: \calC\to \calD$ is a symmetric monoidal functor and $G: \calC\to \calD$ is a lax symmetric monoidal functor, then $\mathrm{LEq}(F,G)$ has a natural structure as a symmetric monoidal $\infty$-category. Indeed, one can define the total space as the pullback
$$
\xymatrix{
\mathrm{LEq}(F,G)^\otimes \ar[r] \ar[d]& (\calD^\otimes)_{\id}^{\Delta^1}\ar[d]   \\
\calC^\otimes \ar[r]^{(F^\otimes, G^\otimes)} & \calD^\otimes \times \calD^\otimes
}
$$
where $(\calD^\otimes)_{\id}^{\Delta^1}\subseteq (\calD^\otimes)^{\Delta^1}$ is the full subcategory consisting of those morphisms which lie over identities in $N\Fin_*^{\Delta^1}$ and where $F^\otimes,G^\otimes: \calC^\otimes\to \calD^\otimes$
are the functors on total spaces. In particular 
$\mathrm{LEq}(F,G)^\otimes$ is a full subcategory of $\mathrm{LEq}
(F^\otimes,G^\otimes)$, and for any symmetric monoidal $\infty$-category $\calE$, giving a symmetric monoidal (resp.~lax symmetric monoidal) functor $\calE\to \mathrm{LEq}(F,G)$ is equivalent to giving a symmetric monoidal (resp.~lax symmetric monoidal) functor $H: \calE\to \calC$ together with a lax symmetric monoidal transformation $F\circ H\to G\circ H$.
\end{altenumerate} 

Moreover, $\Alg_{\E_1}(\Sp)$ is naturally a symmetric monoidal $\infty$-category. It follows from the discussion of lax symmetric monoidal structures on all intervening objects in Sections~\ref{sec:tatediag}, \ref{sec:tatediagfunc} and \ref{sec:thhnaive} that the functor
\[
\THH: \Alg_{\E_1}(\Sp)\to \Cycn
\]
is naturally a lax symmetric monoidal functor. In fact, it is symmetric monoidal as this can be checked on the underlying spectrum, where one uses that geometric realizations (or more generally sifted colimits) commute with tensor products in a presentably symmetric monoidal $\infty$-category. In particular, $\THH$ maps $\E_\infty$-algebras to $\E_\infty$-algebras. Moreover, we recall that
\[
\Alg_{\E_\infty}(\Sp) = \Alg_{\E_\infty}(\Alg_{\E_1}(\Sp))\ ,
\]
so indeed $\THH(A)$ is an $\E_\infty$-algebra in cyclotomic spectra if $A$ is an $\E_\infty$-algebra. Concretely, this means that $\THH(A)$ is a $\T$-equivariant $\E_\infty$-algebra in spectra together with a $\T/C_p\cong \T$-equivariant map $\varphi_p: \THH(A)\to \THH(A)^{tC_p}$ of $\E_\infty$-algebras.

Moreover, we note that the inclusion of the bottom cells in the cyclic objects define a commutative diagram
\[\xymatrix{
A\ar[r]\ar[d]^{\Delta_p} & \THH(A)\ar[d]^{\varphi_p}\\
(A\otimes\ldots\otimes A)^{tC_p}\ar[r] & \THH(A)^{tC_p}\ .
}\]
In fact, one has a lax symmetric monoidal functor from $\Alg_{\E_1}(\Sp)$ to the symmetric monoidal $\infty$-category of such diagrams. In particular, if $A$ is an $\E_\infty$-algebra, then all maps are $\E_\infty$-ring maps.
\end{construction}

Recall the following fact:

\begin{proposition}[McClure-Schw\"anzl-Vogt \cite{MR1473888}]\label{thhcommutative}
For an $\E_\infty$-ring spectrum $A$, the map $A \to \THH(A)$ is initial among all maps from $A$ to an $\E_\infty$-ring spectrum equipped with a $\T$-action (through $\E_\infty$-maps).
\end{proposition}

\begin{proof} We have to prove that $\THH(A)$ is the tensor of $A$ with $S^1$ in the $\infty$-category of $\E_\infty$-ring spectra. We use the simplicial model for the circle $S^1$ given as $S_\bullet=\Delta^1 / \partial \Delta^1$ which has $(n+1)$ different (possible degenerate) $n$-vertices $S_n$. Thus we have in the $\infty$-category $\calS$ of spaces the colimit $S^1 \simeq \colim_{\Delta^{op}} S_n$. Therefore we get that the tensor of $R$ with $S^1$ is given by
$$
\colim_{\Delta^{op}} R^{\otimes S_n} \simeq \THH(R)
$$
where we have used that $R^{\otimes S_n}=R^{\otimes (n+1)}$ is the $(n+1)$-fold coproduct in the $\infty$-category of $\E_\infty$-ring spectra, cf.~\cite[Proposition 3.2.4.7]{HA}.
\end{proof}

From the $\E_\infty$-map $A\to \THH(A)$ we get a $C_p$-equivariant $\E_\infty$-map $A\otimes\ldots \otimes A\to \THH(A)$ by taking the coproduct in the category of $\E_\infty$-algebras of all the translates by elements of $C_p \subseteq \T$ of the map (note that $A\otimes\ldots \otimes A$ is the induced $C_p$-object in $\E_\infty$-rings). This is the map that can also be described through the $p$-fold subdivision. Thus, in the commutative square
\begin{equation}\label{diagramFrobenius}
\xymatrix{
A \ar[r]\ar[d]^{\Delta_p} & \THH(A) \ar[d]^{\varphi_p} \\
(A\otimes\ldots \otimes A)^{tC_p} \ar[r] & \THH(A)^{tC_p}
}
\end{equation}
of $\E_\infty$-rings, the lower map is explicit, giving rise to an explicit map $A\to \THH(A)^{tC_p}$ of $\E_\infty$-rings, where $\T=\T/C_p$ acts on $\THH(A)^{tC_p}$. In view of the above universal property of $\THH(A)$ we conclude:
\begin{corollary}\label{uniqueness}
For an $\E_\infty$-ring $A$ the Frobenius $\varphi_p$ is the unique $\T$-equivariant $\E_\infty$-map $\THH(A) \to \THH(A)^{tC_p}$ that makes the diagram \eqref{diagramFrobenius}
commutative. $\hfill \Box$
\end{corollary}
Note that this can also be used to define $\varphi_p$ in this situation. Moreover, this observation can be used to prove that the various definitions of $\E_\infty$-ring structures on $\THH(A)$ in the literature are equivalent to ours.

A consequence Proposition \ref{thhcommutative} is that for $\E_\infty$-ring spectra $A$ there is always a map $\pi: \THH(A) \to A$ which is a retract of the map $A \to \THH(A)$. The map $\pi$ is by construction $\T$-equivariant when $A$ is equipped with the trivial $\T$-action, in contrast to the map in the other direction. 
\begin{corollary}
For an $\E_\infty$-ring $A$ the  composition 
$$A \to \THH(A) \xto{\varphi_p} \THH(A)^{tC_p} \xto{\pi^{tC_p}} A^{tC_p} $$
is equivalent to the Tate-valued Frobenius (see Definition \ref{deffrobenius}). 
\end{corollary}
\begin{proof}
We have a commutative diagram
\begin{equation}
\xymatrix{
A \ar[r]\ar[d]^{\Delta_p} & \THH(A) \ar[d]^{\varphi_p} \\
(A\otimes\ldots \otimes A)^{tC_p} \ar[r] & \THH(A)^{tC_p} \ar[r]^-{\pi^{tC_p}} & A^{tC_p} .
}
\end{equation}
Therefore it suffices to show that the composition
$$
A\otimes\ldots \otimes A \to \THH(A) \to A
$$
is as a $C_p$-equivariant map equivalent to the multiplication of $A$. The source is, as a $C_p$-object in $\E_\infty$-algebras induced from $A$. Therefore this amounts to checking that the map $A \to \THH(A) \to A$ is equivalent to the identity which is true by definition.
\end{proof}

A slightly different perspective on the construction of $\varphi_p$ in the commutative case is due to Jacob Lurie. To this end let $R = \THH(A)$ or more generally any $\E_\infty$-ring spectrum with $C_p$-action. Then we have a map of $\E_\infty$-ring spectra which can informally be described as 
$$
\tilde m: R \otimes\ldots \otimes R\to R : r_1\otimes\ldots\otimes r_p\mapsto \prod_{i=1}^p \sigma^i(r_i)\ ,
$$
where $\sigma \in C_p$ is a generator. The precise way of defining this map is by the observation that the left hand side is the induced $\E_\infty$-ring spectrum with an action by $C_p$. Then we get the desired map $\tilde m$ by the universal property. This also shows that the map $\tilde m$ is $C_p$-equivariant where the left hand side is equipped with the cyclic action and the right hand map with the given action. Now consider the composite map
$$
\tilde{\varphi}_p: R \xto{\Delta_p} (R \otimes\ldots \otimes R)^{tC_p} \xto{\tilde m^{tC_p}} R^{tC_p}\ . 
$$
This is an equivariant version of the Frobenius map discussed above. Now in the case $R=\THH(A)$ we not only have a $C_p$-action but a $\T$-action and the map $\tilde{\varphi}_p$ is by functoriality $\T$-equivariant where the target is equipped with the $\T$-action obtained from the $\T$-action on $R$ that extends the $C_p$-action. This action of $\T$ on $R^{tC_p}$ has the property that $C_p \subseteq \T$ acts trivially so that the $\T$-action factors over the residual $\T/C_p$-action on $R^{tC_p}$ that we have used several times before. As a result we get a factorization of the map $\tilde{\varphi}_p$ through the homotopy orbits of the $C_p$-action on $R$ in the $\infty$-category of $\E_\infty$-ring spectra. Writing $R=\THH(A)$ as the tensor of $A$ with $\T$ in the $\infty$-category of $\E_\infty$-rings, these homotopy orbits are the tensor of $A$ with $\T/C_p\cong \T$ in the $\infty$-category of $\E_\infty$-rings. Thus, they are equivalent to $\THH(A)$ itself, in a $\T\cong\T/C_p$-equivariant way. In total, we get a $\T$-equivariant map of $\E_\infty$-rings
\[
\varphi: \THH(A) \to \THH(A)^{tC_p}\ .
\]
It follows from the construction that it sits in a commutative diagram of $\E_\infty$-rings
$$\xymatrix{
A \ar[r]\ar[d]^{\Delta_p} & \THH(A) \ar[d]^{\varphi} \\
(A\otimes\ldots \otimes A)^{tC_p} \ar[r] & \THH(A)^{tC_p}\ ,
}
$$
so it must by Corollary \ref{uniqueness} agree with $\varphi_p$.

\section{Loop Spaces}\label{sec:loopspaces}

Finally, we will recover some specific computations using our formalism. We start with the historically first computation, namely B\"okstedt--Hsiang--Madsen's computation of $\TC$ for suspension spectra of based loop spaces. Recall that for any connected based space $(Y,\ast)\in \calS_\ast$, the (based) loop space $\Omega Y$ is an $\E_1$-group in the $\infty$-category of spaces.\footnote{Conversely, any $\E_1$-group is of this form: the category of $\E_1$-groups in spaces is equivalent to connected based spaces, cf.~\cite[Lemma 7.2.2.11]{HTT}.} As $\Sigma_+^\infty: \calS\to\Sp$ is symmetric monoidal with respect to $\times$ on $\calS$ and $\otimes$ on $\Sp$, this implies that its suspension spectrum
\[
\mathbb S[\Omega Y] := \Sigma_+^\infty \Omega Y
\]
is an $\E_1$-algebra in spectra.\footnote{We prefer the notation $\mathbb S[\Omega Y]$ to emphasize the role as group algebra.} Thus, we can form the cyclotomic spectrum $\THH(\mathbb S[\Omega Y])$. The first goal is to identify this spectrum with its cyclotomic structure explicitly. This is Corollary \ref{cor:thhloop} below.

By definition, $\THH(\mathbb S[\Omega Y])$ is the geometric realization of the cyclic spectrum
\[
N(\Lambda^\op)\xto{V^\circ} N(\Ass^\otimes_\act)\xto{\Omega Y^\otimes} \calS^\otimes_\act\xto{(\Sigma_+^\infty)^\otimes} \Sp_\act^\otimes\xto{\otimes} \Sp\ .
\]
As $\Sigma_+^\infty$ is symmetric monoidal and preserves colimits, we can write this as $\Sigma_+^\infty$ of the geometric realization of the cyclic space
\[
B^{\mathrm{cyc}}_\bullet \Omega Y: N(\Lambda^\op)\xto{V^\circ} N(\Ass^\otimes_\act)\xto{\Omega Y^\otimes} \calS_\act^\otimes\xto{\times} \calS\ .
\]
In order the identify the geometric realization of this space we first need an auxiliary result. We formulate this slightly more generally.

For an $\mathbb{E}_1$-monoid $M$ in $\calS$ we denote the analogue of $\THH$ by $B^\cyc M$. This is, as in the case $\Omega Y$ above,  given by the geometric realization of the cyclic bar construction $\B^\cyc_\bullet M$ and admits a canonical $\T$-action, see Proposition \ref{prop:geomrealtaction} in Appendix~\ref{app:cyclic}.
\begin{lemma}\label{unstablecyclotomic}Let $M$ be an $\mathbb{E}_1$-monoid in spaces.
\begin{altenumerate}
\item The cyclic bar construction $\B^\mathrm{cyc}M$ admits a canonical $\T$-action and a canonical $\T$-equivariant map
$$
\psi_p: \B^\mathrm{cyc}M \to (\B^\mathrm{cyc}M)^{hC_p}
$$
for every prime $p$. 
\item There is a commutative diagram
$$
\xymatrix{
M \ar[d]^{\Delta}\ar[r]^{i} &   \B^\mathrm{cyc}M\ar[d]^{\psi_p}  \\
(M \times \ldots \times M)^{hC_p} \ar[r] &  \B^\mathrm{cyc}M^{hC_p}
}
$$
where the upper horizontal map is the inclusion of $M$ into the colimit and the lower horizontal map is induced from the inclusion of the $[p]$-th object into the colimit. 
\item
Upon taking the suspension spectrum, the map $\psi_p$ refines the cyclotomic structure of $\THH(\bS[M]) = \Sigma^\infty_+ B^\cyc M$ in the sense that there is a commutative diagram
$$
\xymatrix{
\Sigma^\infty_+ B^\cyc M \ar[r]^-{\psi_p} \ar[d]^{\varphi_p}&\Sigma^\infty_+\left( B^\cyc M^{hC_p} \right) \ar[d] \\
\left(\Sigma^\infty_+ B^\cyc M\right)^{tC_p} &  \left(\Sigma^\infty_+ B^\cyc M\right)^{hC_p} \ar[l]_-{\can}.
}
$$
\end{altenumerate}
\end{lemma}

\begin{proof}
The existence of the $\T$-action on $B^\cyc M$ follows immediately from the results in Appendix~\ref{app:cyclic}. Thus we have to construct the ``cyclotomic structure'' on $B^\cyc M$ and show the compatibility. To this end we just repeat the steps in the construction of the cyclotomic structure on $\THH$, see Section \ref{sec:thhnaive}. We replace every occurrence of the Tate diagonal by the space level diagonal $M \to (M \times \ldots \times M)^{hC_p}$. More precisely we want to construct a map of cyclic objects
\begin{equation}\label{mapcyclicspaces}
\xymatrix{
\cdots \ar[r]<1.2ex> \ar[r]<-1.2ex>\ar[r]<0.4ex> \ar[r]<-0.4ex> & M^{\times 3} \ar[d]^{\Delta_p} \ar@(ul,ur)^{C_3}\ar[r]<0.8ex>\ar[r]<0ex>\ar[r]<-0.8ex> \ar[r]& M^{\times 2} \ar[d]^{\Delta_p}\ar@(ul,ur)^{C_2}
\ar[r]<0.4ex> \ar[r]<-0.4ex>& M \ar[d]^{\Delta_p} \\
\cdots \ar[r]<1.2ex> \ar[r]<-1.2ex>\ar[r]<0.4ex> \ar[r]<-0.4ex> & \big(M^{\times 3p}\big)^{hC_p} \ar@(dl,dr)_{C_3}\ar[r]<0.8ex>\ar[r]<0ex>\ar[r]<-0.8ex> \ar[r]& \big(M^{\times 2p}\big)^{hC_p}\ar@(dl,dr)_{C_2}
\ar[r]<0.4ex> \ar[r]<-0.4ex>& \big(M^{\times p}\big)^{hC_p}
}
\end{equation}
where the bottom cyclic objects is constructed as the one involving the Tate construction in Section~\ref{sec:thhnaive}. Thus analogous to the construction there this follows from the fact that there is a natural $BC_p$-equivariant transformation from the functor 
\[
I': N(\Free_{C_p})\times_{N(\Fin)} \calS^\otimes_\act\to \calS^\otimes_\act\xto{\times}\calS
\]
to the functor
\[
\tilde{H}_p: N(\Free_{C_p})\times_{N(\Fin)} \calS^\otimes_\act\to (\calS^\otimes_\act)^{BC_p}\xto{\otimes} \calS^{BC_p}\xto{-^{hC_p}} \calS\ .
\]
This transformation exists, because the analogue of Lemma \ref{lem:initialtech} is true for the functor $I'$: it is initial among all lax symmetric monoidal functors $N(\Free_{C_p})\times_{N(\Fin)} \calS \to \calS$. This follows from the combination of the argument in the proof of Lemma \ref{lem:initialtech} and the fact that the identity $\calS \to \calS$ is initial among all lax symmetric monoidal endofunctors, see \cite[Corollary 6.8 (1)]{Nik}. This proves part (i) and part (ii). For the third part we claim that there is a comparison of maps of cyclic spectra, between the suspension 
$$\Sigma^\infty_+ \B^\cyc_\bullet M \to \left((\Sigma^\infty_+ \B^\cyc_\bullet)^{\otimes p})\right)^{hC_p}$$
of the map \eqref{mapcyclicspaces} and the map 
$$\Sigma^\infty_+ \B^\cyc_\bullet M \to \left((\Sigma^\infty_+ \B^\cyc_\bullet)^{\otimes p})\right)^{tC_p}$$
that leads to the cyclotomic structure on $\THH(\bS[M])$ as a map of cyclic spectra. Clearly there is a natural comparison map between the targets given by the canonical map from homotopy fixed points to the Tate construction. Thus we have to compare two different maps from $\Sigma^\infty_+ \B^\cyc_\bullet M$ to $\left((\Sigma^\infty_+ \B^\cyc_\bullet)^{\otimes p})\right)^{tC_p}$. Both arise as lax symmetric monoidal transformations from the functor 
\[
I'': N(\Free_{C_p})\times_{N(\Fin)} \calS^\otimes_\act\to \calS^\otimes_\act\xto{\times}\calS \to \Sp
\]
to the functor 
\[
\tilde{T}_p: N(\Free_{C_p})\times_{N(\Fin)} \calS^\otimes_\act\to (\calS^\otimes_\act)^{BC_p}\xto{\otimes} \calS^{BC_p} \to \Sp^{BC_p} \xto{-^{tC_p}} \calS\ .
\]
Then the claim follows since the functor $I''$ is initial among lax symmetric monoidal functors $N(\Free_{C_p})\times_{N(\Fin)}\calS \to \Sp$. This follows again with the argument of Lemma \ref{lem:initialtech} and the fact that the suspension spectrum functor $\calS \to \Sp$ is initial as shown in \cite[Corollary 6.9 (2)]{Nik}.
\end{proof}

In the case that $M=\Omega Y$ is the loop space of a connected based space $Y$, one can identify $\B^\cyc M$ with the free loop space $LY = \Map(S^1,Y)$ of $Y$.

\begin{proposition}\label{prop:bcycloop} Assume that $M=\Omega Y$ is the loop space of a connected base space $Y$.
\begin{altenumerate}
\item There is a natural $\T$-equivariant equivalence
\[
\B^\cyc M\simeq LY=\Map(S^1,Y)\ ,
\]
where $\T$ acts on $LY$ through its action on $S^1$.
\item Under the equivalence $\B^\cyc M\simeq LY$, the $\T\cong \T/C_p$-equivariant map $\psi_p: \B^\cyc M\to (\B^\cyc M)^{hC_p}$ is identified with the $\T\cong \T/C_p$-equivariant map $LY\to LY^{hC_p}$ induced by the $p$-fold covering $S^1\to S^1$.
\end{altenumerate}
\end{proposition}

\begin{proof} This is a classical fact, see e.g.~\cite[Theorem 7.3.11]{MR1600246}. We give a direct proof in our language.

First we note that the functor $ \times\colon \calS^\otimes_\act\to \calS$ (which sends a lists of objects $(X_1,\ldots,X_n)$  to the product  $\prod X_i$ in $\calS$) extends to a functor $\times\colon \calS^\otimes \to \calS$. This extension is called a Cartesian structure by Lurie, see \cite[Definition 2.4.1.1]{HA} and exists by construction of Cartesian symmetric monoidal $\infty$-categories \cite[Proposition 2.4.1.5]{HA}. Note that this extension does not exist for general symmetric monoidal $\infty$-categories.

Now we construct a canonical, $\T$-equivariant map $\varphi: \B^\cyc  \Omega Y \to LY$. By adjunction such a map is the same as a map $\varphi': \B^\cyc \Omega Y \times S^1 \to Y$. Since the map $\varphi$ is supposed to be $\T$-equivariant, the map $\varphi'$ has to be $\T$-equivariant for the diagonal action on the source and the trivial action on the target. But such a map is the same as a non-equivariant map $\varphi'': \B^\cyc \Omega Y\to Y$. Now using the fact that $\calS$ is an $\infty$-topos we can realize $Y$ as the geometric realization $B \Omega Y$ of the bar construction of $\Omega Y$, i.e. the simplicial object given by the composition
$$
\B_\bullet \Omega Y: N(\Delta^\op) \xto{\mathrm{Cut}} N(\Ass^\otimes) \xto{ \Omega Y^\otimes} \calS^\otimes\xto{\times} \calS. 
$$
Here the first map $\mathrm{Cut}$ is given by the functor sending $[n] \in \Delta$ to the set of cuts of $[n]$ as described in Appendix~\ref{app:cyclic}. Now we claim that there is a map $j^*\B^{\mathrm{cyc}}_\bullet \Omega Y \to \B_\bullet \Omega Y$, where $j: \Delta^\op \to \Lambda^\op$ is the natural functor, see  Corollary \ref{cor:lambdavsdelta} in Appendix~\ref{app:cyclic}. This natural map is induced from a natural transformation in the (non-commutative) diagram
$$
\xymatrix{
N(\Delta^\op) \ar[d]^j \ar[r]^{\mathrm{Cut}} & N(\Ass^\otimes)\\
 N(\Lambda^{\op}) \ar[r]^{V^\circ} & N(\Ass^\otimes_\act) \ar[u]
 }
$$
To understand the construction of the transformation we first note that the counter-clockwise composition in the diagram is given by the functor that sends a finite ordered set $S \in \Delta^{\op}$ to the set $\Cut(S)_+$, which is obtained from $\Cut(S)$ by forgetting the existing basepoint and adding a new disjoint basepoint. This fact follows directly from unfolding the definitions given in Appendix~\ref{app:cyclic}: the functor $j$ sends a linearly ordered set $S$ the the linearly ordered set $\Z  \times S$ (with lexicographic order), the functor $-^\circ$ sends this to the set of non-empty cuts and then the last functor $V$ takes the quotient by the $\Z$ action which means that two cuts are identified if they differ by an integer. Therefore we can shift any cut into $S = S\times 0 \subseteq j(S)$ and thus obtain the set $\Cut(S)$ of cuts of $S$ in which the cuts $\emptyset\sqcup S$ and $S \sqcup \emptyset$ are identified by definition. The last functor then only adds a disjoint basepoint. 

Now the transformation in question is given by the canonical map $\Cut(S)_+ \to \Cut(S)$. Informally the map $j^*\B^{\mathrm{cyc}}_\bullet \Omega Y \to \B_\bullet \Omega Y$ is given by the morphism of simplicial objects\begin{equation}\label{diagramcyclictobar}
\xymatrix{
\cdots \ar[r]<1.2ex> \ar[r]<-1.2ex>\ar[r]<0.4ex> \ar[r]<-0.4ex> & \Omega Y \times \Omega Y \times \Omega Y \ar[d]^{\pi_{23}}\ar[r]<0.8ex>\ar[r]<0ex>\ar[r]<-0.8ex> \ar[r]& \Omega Y \times \Omega Y \ar[d]^{\pi_2}
\ar[r]<0.4ex> \ar[r]<-0.4ex>& \Omega Y \ar[d]^{!} \\
\cdots \ar[r]<1.2ex> \ar[r]<-1.2ex>\ar[r]<0.4ex> \ar[r]<-0.4ex> & \Omega Y \times \Omega Y \ar[r]<0.8ex>\ar[r]<0ex>\ar[r]<-0.8ex> \ar[r]& \Omega Y\ar[r]<0.4ex> \ar[r]<-0.4ex>& \pt
}
\end{equation}
in which the first factor in the upper lines corresponds to the old basepoint $[S \cup \emptyset]$ in $\Cut(S)_+$. 
After realization the map $j^*\B^{\mathrm{cyc}}_\bullet \Omega Y \to \B_\bullet \Omega Y$ gives us the desired map $\varphi''$ and thus also the $\T$-equivariant map $\varphi: \B^\cyc  \Omega Y \to LY$. 

In order to show that the map $\varphi$ is an equivalence in the $\infty$-category of spaces we try to understand the fiber of the map
$\varphi'': | \B^{\mathrm{cyc}}_\bullet \Omega Y | \to | \B_\bullet \Omega Y | \simeq Y$. As a first step we consider the sequence $\pt_+ \to \Cut(S)_+ \to \Cut(S)$ of functors from $\Delta^\op \to N(\Ass^\otimes)$ where the first term is the constant functor and the first map is given by sending $\pt$ to the old basepoint $[S \cup \emptyset] \in \Cut(S)$. The composition is given by the trivial map. This sequence gives rise to a diagram
$$(\Omega Y)_\bullet \to \B^{\mathrm{cyc}}_\bullet \Omega Y \to \B_\bullet \Omega Y$$
of simplicial spaces, where the first term is the constant simplicial space on $\Omega Y$ and the first map is given by the inclusion into the first factor (in the ordering chosen in the displayed diagram  \eqref{diagramcyclictobar}). The compositon of the two maps comes with a chosen nullhomotopy induced from the fact that the composition $\pt_+ \to \Cut(S)$ is trivial. With this structure the sequence $(\Omega Y)_\bullet \to \B^{\mathrm{cyc}}_\bullet \Omega Y \to \B_\bullet  \Omega Y$ is a fiber sequence in the $\infty$-category $\Fun(N\Delta^\op, \calS)$ of simplicial spaces. 
By construction, the map $\varphi$ induces a comparison from the realization of this fiber sequence to the fiber sequence  
$$
\Omega Y \to LY \to Y
$$
in $\calS$. 
Thus in order to show that $\varphi$ is an equivalence it suffices to show that the realization of the simplicial fiber sequence $(\Omega Y)_\bullet \to B^{\mathrm{cyc}}_\bullet  \Omega Y\to B_\bullet \Omega Y$ is a fiber sequence in the $\infty$-category $\calS$. There are criteria in terms of model categories to check this, which can be used to prove this fact, see e.g.  \cite{Anderson}, but we want to give a direct $\infty$-categorical argument.

We start with the following general observation:
consider a fiber sequence of (pointed) simplicial spaces $X_\bullet \to Y_\bullet \to Z_\bullet$ such that 
$Y_\bullet$ and $Z_\bullet$ are groupoid objects in $\calS$. That means for each partition $[n] = S_0 \cup S_1$ such that $S_0 \cap S_1$ consists of a single element $s$, the canonical map $Y_n \to Y_{S_0}\times_{Y_{\{s\}}} Y_{S_1}$ is an equivalence, and similarly for $Z_\bullet$, see \cite[Definition 6.1.2.7]{HTT}. We claim that in this situation the induced map from $|X_\bullet |$ into to fiber $F$ of the map  $|Y_\bullet | \to |Z_\bullet |$ is $(-1)$-connected, i.e.  an inclusion of connected components. To see this we use that, since $\calS$ is an $\infty$-topos, the groupoid objects $Y_\bullet$ and $Z_\bullet$ are effective. This means that the associated augmented simplicial objects $Y^+_\bullet$ and $Z_\bullet^+$, augmented over the geometric realizations, are \v{C}ech nerves. 
Concretely that means that the canonical map $Y_n \to Y_1 \times_{|Y_\bullet|} \ldots \times_{|Y_\bullet|} Y_1$ is an equivalence for every $n$ and similarly for $Z_\bullet$. Recall that we denote by $F$ the fiber of the map $|Y_\bullet | \to |Z_\bullet |$. Since pullbacks commute with taking fibers we deduce that for each $n$ the induced map
$$
X_n \to X_1 \times_F \ldots \times_F X_1  
$$
is an equivalence. We conclude that $X_\bullet$ is the \v{C}ech nerve of the map $X_0 \to F$. Thus the map $|X_\bullet| \to F$ is $(-1)$-connected by \cite[Proposition 6.2.3.4]{HTT}.

In our situation it is clear that $\B G_\bullet$ is a groupoid object. It is straightforward to check that  $\B^{\mathrm{cyc}}_\bullet$ is a groupoid objects as well (in fact it is equivalent to the inertia groupoid of $\B G_\bullet$). Thus we know that the map from $\Omega Y = |\Omega Y_\bullet|$ to the fiber $F$ of the map $|\B^{\mathrm{cyc}}_\bullet \Omega Y | \to |\B_\bullet  \Omega Y|$ is $(-1)$-connected. Therefore it suffices to check that it is surjective on $\pi_0$. 
We already know that the composition $\Omega Y \to F \to |\B^\cyc_\bullet  \Omega Y|$ is surjective on $\pi_0$, since it is the inclusion of the bottom cell into the geometric realization. Thus it is enough to show that $\pi_0(F) \to \pi_0|\B_\bullet^\cyc  \Omega Y|$ is injective. Using the  long exact sequence 
$$
\xymatrix{
\pi_1 |\B^{\mathrm{cyc}}_\bullet  \Omega Y| \ar[r]^f & \pi_1 |\B_\bullet \Omega Y| \ar[r] & \pi_0|F| \ar[r] & \pi_0|\B^{\mathrm{cyc}}_\bullet \Omega Y| \
}
$$
this follows from the surjectivity of $f$ which is a consequence of the fact that  the simplicial map $\B^\cyc_\bullet  \Omega Y\to B_\bullet \Omega Y$ admits a section. This section can be constructed explicitly. We can without loss of generality replace $\Omega Y$ by a strict topological group and then write down the section in the simplicial diagram \eqref{diagramcyclictobar}. We leave the details to the reader.

For part (ii), we need to identify two natural $\T\cong \T/C_p$-equivariant transformations
\[
LY\to LY^{hC_p} = \Map(S^1, Y)^{hC_p} = \Map(S^1/C_p, Y) \simeq LY\ ,\]
i.e.~equivalently two natural $\T$-equivariant maps $LY \to LY$. Such a map is the same as a non-equivariant map $LY \to Y$. Since the map is natural in $Y$ (in pointed maps) we can use the fact that $LY = \Map_*(S^1_+,Y)$ and $Y = \Map_*(\pt_+,Y)$ together with the Yoneda lemma to see that such maps are in bijective correspondence with pointed maps $\pt_+ \to S^1_+$. There are up to equivalence two such maps: the map which hits both connected components and the map which sends both points to the basepoint in $S^1_+$. The first map leads to the map $LY \to LY^{hC_p}$ as stated and the second map leads to the map $LY \to LY^{hC_p}$ which sends every loop in $Y$ to the constant loop at the basepoint of $Y$. We have to rule out the second possibility. By Lemma~\ref{unstablecyclotomic} we have a commutative square
$$
\xymatrix{
\Omega Y \ar[r]\ar[d]^{\Delta} & LY \ar[d]^{\psi_p} \\
((\Omega Y)^{\times p})^{hC_p} \ar[r] & LY^{hC_p} 
}
$$
The upper map is given by the canonical inclusion, since in the equivalence $\B^\mathrm{cyc} \Omega Y \simeq LY$ this corresponds to the inclusion of the bottom cell of the simplicial diagram $\B^\mathrm{cyc} \Omega_\bullet Y$. The lower line is also equivalent to the inclusion $\Omega Y \to LY$ under the obvious identifications, as one sees similarly. Under these identifications the left hand map corresponds to the identity map $\Omega Y \to \Omega Y$. As a result the map $LY \to LY^{hC_p} = LY$ cannot be the trivial map which sends every map to the constant map. Thus it has to be equivalent to the identity which finishes to proof.  
\end{proof}

Now we can prove the first result about the structure of $\THH(\bS[\Omega Y])$ for $Y$ a pointed connected space. 

\begin{corollary}\label{cor:thhloop} There is a natural $\T$-equivariant equivalence
\[
\THH(\mathbb S[\Omega Y])\simeq \Sigma_+^\infty LY\ ,
\]
where $LY = \Map(S^1,Y)$ denotes the free loop space of $Y$ with its natural $\T$-action. Under this equivalence, the $\T$-equivariant map
\[
\varphi_p: \THH(\mathbb S[\Omega Y])\to \THH(\mathbb S[\Omega Y])^{tC_p}
\]
is given by the composite
\[
\Sigma_+^\infty LY\to (\Sigma_+^\infty LY)^{hC_p}\to (\Sigma_+^\infty LY)^{tC_p}\ ,
\]
where the first map is induced by the map $LY\to LY$ coming from the $p$-fold covering $S^1\to S^1$, and the second map is induced by the projection $-^{hC_p}\to -^{tC_p}$.
\end{corollary}

\begin{proof} This is immediate from Lemma~\ref{unstablecyclotomic} and Proposition~\ref{prop:bcycloop}.
\end{proof}

In particular, in this case, the Frobenius map $X\to X^{tC_p}$ factors over $X^{hC_p}$. More generally, we say that a Frobenius lift on a $p$-cyclotomic spectrum $(X,\varphi_p)$ is a $C_{p^\infty}$-equivariant factorization of the morphism $\varphi_p: X \to X^{tC_p}$ as the composite of a map $\tilde{\varphi}_p: X\to X^{hC_p}$ and the projection $X^{hC_p}\to X^{tC_p}$.

Recall that if $X$ is a $p$-cyclotomic spectrum which is bounded below, then
\[
\TC(X)_p^\wedge = \TC(X^\wedge_p) = \TC(X^\wedge_p,p)\ .
\]
In order words, the $p$-completion of $\TC(X)$ can be computed in terms of $X^\wedge_p$ as either $\TC(X^\wedge_p)$ or $\TC(X^\wedge_p,p)$; this follows from the discussion in Section~\ref{sec:tccomp}. Note also that in this case the $\Tp$-action on $X^\wedge_p$ extends automatically to a $\T$-action.

\begin{proposition}
For a bounded below $p$-complete $p$-cyclotomic spectrum $X$ with a Frobenius lift $\tilde{\varphi}_p: X\to X^{hC_p}$ we obtain a pullback square of the form
\[\xymatrix{
\TC(X)\ar[r]\ar[d] & \Sigma X_{h\T}\ar[d]^{\tr} \\
X\ar[r]^{\id - \tilde{\varphi}_p} & X
}\]
\end{proposition}

\begin{proof} We have a pullback diagram
$$\xymatrix{
\TC(X,p)\ar[rr]\ar[d] &  & 0 \ar[d]\\
X^{h\T} \ar[r]^{\id - \tilde{\varphi}_p^{h\T}} & X^{h\T}\ar[r]^{\can^{h\T}} & (X^{tC_p})^{h\T}
}
$$
 by Proposition \ref{prop:formula} (which states that the outer square can be extended to a pullback) and the factorization of the lower maps by assumption. Now we take the pullback of the right hand square and claim that we obtain a diagram of the form
$$
\xymatrix{
\TC(X,p)\ar[r]\ar[d] &  \Sigma X_{h\T}    \ar[d]^{\tr} \ar[r] & 0 \ar[d]\\
X^{h\T} \ar[r]^{\id - \tilde{\varphi}_p^{h\T}} & X^{h\T}\ar[r]^{\can^{h\T}} & (X^{tC_p})^{h\T}
}
$$
in which all squares are pullbacks. As $\TC(X)=\TC(X,p)$, it remains to see that the fiber of the map $X^{h\T}\to (X^{tC_p})^{h\T}$ is given by $\Sigma X_{h\T}$, which follows from Lemma~\ref{lemtate2}. Thus we get a pullback square of the form
$$
\xymatrix{
\TC(X)\ar[r]\ar[d] &  \Sigma X_{h\T}    \ar[d]^{\tr} \\
X^{h\T} \ar[r]^{\id - \tilde{\varphi}_p^{h\T}} & X^{h\T}\ .
}
$$
Then the proposition is implied by the next lemma.
\end{proof}

\begin{lemma} For every $p$-complete $p$-cyclotomic spectrum $X$ with a Frobenius lift $\tilde{\varphi}_p: X\to X^{hC_p}$, the commutative square
\[\xymatrix{
X^{h\T} \ar[rr]^{\id - \tilde{\varphi}_p^{h\T}} \ar[d] && X^{h\T}\ar[d] \\
X  \ar[rr]^{\id - \tilde{\varphi}_p} && X
}\]
is a pullback of spectra.
\end{lemma}

\begin{proof} Note that the $\T\cong \T/C_p$-equivariant map $\tilde{\varphi}_p: X\to X^{hC_p}$ is equivalently given by a natural transformation $\psi_p$ of functors $B\T\to \Sp$ from $X\circ f_p: B\T\to B\T\to \Sp$ to $X: B\T\to \Sp$, where $f_p: B\T\to B\T$ denotes the map induced by the degree $p$ selfmap of $\T$. Under this identification, the map $\tilde{\varphi}_p^{h\T}: X^{h\T}\to X^{h\T}$ is given by the composite
\[
X^{h\T}\xto{f_p^\ast} (X\circ f_p)^{h\T}\xto{\lim_{B\T} \psi_p} X^{h\T}
\]
where $\lim_{B\T} \psi_p = \psi_p^{h\T}$. 
Now we use that for any functor $Y: B\T\to \Sp$, one can write
\[
Y^{h\T}=\lim_{B\T} Y\ ,\ B\T=\mathbb C P^\infty = \bigcup_n \mathbb C P^n
\]
as the inverse limit
\[
Y^{h\T} = \underleftarrow{\lim}_n \lim_{\mathbb C P^n} Y\ ,
\]
where one has natural fiber sequences for all $n\geq 0$,
\[
\Omega^{2n} Y\to \lim_{\mathbb C P^n} Y\to \lim_{\mathbb C P^{n-1}} Y\ ,
\]
where we also write $Y$ for the underlying spectrum. The map $f_p: B\T\to B\T$ can be modelled by the map raising all coordinates to the $p$-th power on $\mathbb C P^\infty$, and preserves the skeleta $\mathbb C P^n$. Thus,
\[
f_p^\ast: X^{h\T}\to X^{h\T}
\]
is an inverse limit of compatible maps
\[
f_p^\ast: \lim_{\mathbb C P^n} X\to \lim_{\mathbb C P^n} X\ ,
\]
which are homotopic to multiplication by $p^n$ on the homotopy fibers $\Omega^{2n} X$ (as $f_p$ is a map of degree $p^n$ on $\mathbb C P^n$). It follows that the composite map
\[
\tilde{\varphi}_p^{h\T}: X^{h\T}\xto{f_p^\ast} (X\circ f_p)^{h\T}\xto{\psi_p} X^{h\T}
\]
is an inverse limit of maps on $\lim_{\mathbb C P^n} X$, and it is enough to show that $\id - \tilde{\varphi}_p^{h\T}$ induces homotopy equivalences on all homotopy fibers $\Omega^{2n} X$ for $n\geq 1$. But by the above considerations, $\tilde{\varphi}_p^{h\T}$ is the composite of $f_p^\ast$ and $\lim_{B\T} \psi$, and therefore is divisible by $f_p^\ast = p^n$ on $\Omega^{2n} X$, which implies that $\id - \tilde{\varphi}_p^{h\T}$ is a $p$-adic homotopy equivalence on $\Omega^{2n} X$ for $n\geq 1$ (as can be checked after smashing with $\mathbb S/p$, where the map becomes homotopic to the identity), and thus a homotopy equivalence (as $X$ is $p$-complete).
\end{proof}

In particular, we recover the computation of $\TC(\mathbb S[\Omega Y])$ as given by B\"okstedt--Hsiang--Madsen.

\begin{thm}\label{thm:BHMformula}
For a connected based space $Y$, we have a pullback square
$$
\xymatrix{
\TC(\mathbb S[\Omega Y])\ar[r] \ar[d]& \Sigma (\Sigma_+^\infty LY)_{h\T} \ar[d]^{\tr}   \\
\Sigma^\infty_+ LY \ar[r]^{1-\tilde{\varphi}_p} & \Sigma^\infty_+ LY\ .
}
$$
after $p$-completion.$\hfill \Box$
\end{thm}

Note that there is a further simplification of this pullback square that has been made in \cite{MR1408537}.\footnote{We thank John Rognes for making us aware of this reference.} In particular, Proposition 3.9 of this article states that for a simply connected space $Y$ the further square
$$
\xymatrix{
\Sigma^\infty_+ LY \ar[rr]^{1-\tilde{\varphi}_p} \ar[d]^{\ev}&& \Sigma^\infty_+ LY \ar[d]^{\ev} \\
\Sigma^\infty_+ Y \ar[rr]^{(1-\tilde{\varphi}_p) = 0} &&   \Sigma^\infty_+ Y
}
$$ 
is also a pullback square after $p$-completion. Combining this with Theorem \ref{thm:BHMformula} they obtain that
$$
\TC(\bS[\Omega Y]) \simeq \Sigma^\infty_+ Y \oplus \fib(\ev\circ \tr). 
$$
Our results actually imply a related result concerning the Segal conjecture for loop spaces. This generalizes a theorem proven by Carlsson in \cite{MR1115825} for finite, simply connected CW complexes.

\begin{thm} Let $Y$ be a simply connected based space, and regard $\Sigma_+^\infty LY = \THH(\mathbb S[\Omega Y])$ as a cyclotomic spectrum. Then the Frobenius
\[
\varphi_p: \Sigma^\infty_+ LY \to (\Sigma^\infty_+ LY)^{tC_p}
\]
is a $p$-completion. Thus for every cyclic $p$-group $C_{p^n} \subseteq \T$ the induced map
$(\Sigma^\infty_+ LY)^{C_{p^n}} \to (\Sigma^\infty_+ LY)^{hC_{p^n}} $ is a $p$-adic equivalence.
\end{thm}

\begin{proof} The second assertion follows from the first by Corollary \ref{cor:segalinduct}. Thus we have to prove the first. To this end we write (as in the proof of Proposition \ref{thhcommutative}) the circle $\T$ as the homotopy colimit of discrete sets $\T \simeq \colim_{\Delta^{\op}} S_n$. Then we obtain that the free loop space $LY$ is given by the limit of a cosimplicial space $[n] \mapsto \Map(S_n, Y) \simeq Y^{n+1}$. If we replace the simplicial circle by its $p$-fold edgewise subdivision, then we similarly  get an equivalence $LY \simeq \lim_\Delta Y^{p(n+1)}$. Now we claim that the suspension spectrum functor preserves these limits, i.e.~that we have an equivalence
\[
\Sigma^\infty_+ LY \simeq \lim_\Delta (\Sigma^\infty_+ Y)^{\otimes p(n+1)}\ .
\]
To see this we claim more precisely that the fibers in the associated tower of this cosimplicial spectrum become highly connective. This is proven by Kuhn in \cite{MR2039768}, and uses that $Y$ is simply connected. See also \cite[Section 2.5]{malkiewich2015cyclotomic} for a short discussion.

It follows that this limit also commutes with the Tate construction $-^{tC_p}$; indeed, for the $-^{hC_p}$ part, this is clear, and for the $-_{hC_p}$-part it follows from the connectivity claim. Therefore, in addition to the equivalence
\[
\Sigma^\infty_+ LY \simeq \lim_{\Delta} (\Sigma^\infty_+ Y)^{\otimes (n+1)}
\]
from above (taking $p=1$), we get the identification
\[
(\Sigma^\infty_+ LY)^{tC_p} \simeq \lim_{\Delta} ((\Sigma^\infty_+ Y)^{\otimes p(n+1)})^{tC_p}\ .
\]
Now under these identifications the map $\varphi_p$ is given by the limit of the Tate diagonals $\Delta_p$. Thus the claim follows from Theorem \ref{thm:generalsegal} and the fact that limits commute with $p$-completion.
\end{proof}

\section{Rings of characteristic $p$}\label{sec:charp}

Now we revisit a few results about $\TC(A)$ if $A$ is a ring of characteristic $p$ (in the sense that $p=0$ in $\pi_0 A$). In particular, we give a complete description of the $\E_\infty$-algebra in cyclotomic spectra $\THH(H\bF_p)$.

We start with the case $A=H\bF_p$. Our goal is to give the computation of $\TC(H\bF_p)$ using as input only B\"okstedt's description of $\pi_\ast \THH(H\bF_p)$ as a polynomial algebra. Let us recall this result.

First, we compute Hochschild homology. Recall that for a classical associative and unital ring $A$, the Hochschild homology $\HH(A)$ is the geometric realization of the cyclic object
\[\xymatrix{
 \cdots \ar[r]<4.5pt>\ar[r]<1.5pt>\ar[r]<-4.5pt>\ar[r]<-1.5pt> & A \buildrel{\mathbb L}\over\otimes_{\mathbb Z} A \buildrel{\mathbb L}\over\otimes_{\mathbb Z} A  \ar[r]<3pt>\ar[r]\ar[r]<-3pt>  & A\buildrel{\mathbb L}\over\otimes_{\mathbb Z} A \ar[r]<1.5pt>\ar[r]<-1.5pt> & A\ .
}\]
If $A$ is a usual ring which is flat over $\Z$, this is the usual cyclic ring; if $A$ is not flat, the tensor products need to be derived, and one can define $\HH(A)$ more concretely as the realization of the simplicial object $\HH(A_\bullet)$ of a simplicial resolution $A_\bullet$ of $A$ by flat $\Z$-algebras.

\begin{proposition}\label{prop:hkr} Let $A$ be a commutative and unital ring. There is a descending separated filtration of $\HH(A)$ with graded pieces given by $(\bigwedge^i_A \mathbb L_{A/\Z})[i]$, where $\mathbb L_{A/\Z}$ denotes the cotangent complex, and the exterior power is derived.
\end{proposition}

\begin{proof} 
For every commutative ring $A$ the chain complex $\HH(A)$ is naturally associated with a simplicial commutative ring by definition of the Bar construction. Thus on homotopy we get a graded commutative $A$ algebra which has the property that elements in odd degree square to zero. Moreover the $\T$-action equips $\HH_*(A)$ with a differential that is compatible with the ring structure. As a result we get a map of CDGAs
\[
\Omega^*_{A/\Z} \to \HH_*(A) = \pi_* \HH(A) \ .
\]
This map is always an isomorphism in degree $* \leq 1$ and 
if $A=\Z[X_i,i\in I]$ is a free (or smooth) $\Z$-algebra, it is an isomorphism by the Hochschild--Kostant--Rosenberg theorem. In general, we can choose a simplicial resolution $A_\bullet\to A$ such that all $A_n$ are free $\Z$-algebras, in which case we filter $\HH(A_\bullet)$ by the simplicial objects $\tau_{\geq i} \HH(A_\bullet)$. One checks that this is independent of all choices, and the graded pieces are by definition $(\bigwedge^i_A \mathbb L_{A/\Z})[i]$.
\end{proof}

Moreover, there is a natural map $\THH(HA)\to H\HH(A)$.

\begin{proposition}\label{prop:thhtohhsmall} For any associative and unital ring $A$, the map
\[
\pi_i\THH(HA)\to \pi_i H\HH(A) = H_i \HH(A)
\]
is an isomorphism for $i\leq 2$.
\end{proposition}

\begin{proof} The truncation in degrees $\leq 2$ of $\THH(HA)$ depends only on
\[\begin{aligned}
\tau_{\leq 0}(HA\otimes HA\otimes HA)&\simeq \tau_{\leq 0}H(A\buildrel{\mathbb L}\over\otimes_\Z A\buildrel{\mathbb L}\over\otimes_\Z A)\ ,\\
\tau_{\leq 1}(HA\otimes HA) &\simeq \tau_{\leq 1}H(A\buildrel{\mathbb L}\over\otimes_\Z A)\ ,\\
\tau_{\leq 2}HA &\simeq  \tau_{\leq 2} HA\ ,
\end{aligned}\]
which gives the result.
\end{proof}

Now note that $\mathbb L_{\bF_p/\Z_p} = \bF_p[1]$ (with canonical generator given by $p\in I/I^2=H_1 \mathbb L_{\bF_p/\Z_p}$, where $I=\ker(\Z\to \bF_p)$), and $\bigwedge^i_k \mathbb L_{\bF_p/\Z_p} = \Gamma^i(\bF_p)[i]\cong \bF_p[i]$ for all $i\geq 0$, where $\Gamma^i$ denotes the divided powers. Thus, Proposition~\ref{prop:hkr} implies the following computation of Hochschild homology.

\begin{proposition}\label{prop:hhfp} The homology groups of $\HH(\bF_p)$ are given by
\[
H_i \HH(\bF_p) = \left\{\begin{array}{cl} \bF_p & i\geq 0\ \mathrm{even}\\ 0 & \mathrm{else}\ .\end{array}\right .
\]
Let $u\in H_2 \HH(\bF_p)$ denote the canonical generator. Then $H_\ast \HH(\bF_p)$ is isomorphic to the divided power algebra in $u$ over $\bF_p$.$\hfill \Box$\end{proposition}

By Proposition~\ref{prop:thhtohhsmall}, we have $\pi_1\THH(H\bF_p)=0$ and $\pi_2\THH(H\bF_p)=\bF_p\cdot u$. Now we can state the following theorem of B\"okstedt.

\begin{thm}[B\"okstedt] The homotopy groups of $\THH(H\bF_p)$ are given by
\[
\pi_i \THH(H\bF_p) = \left\{\begin{array}{cl} \bF_p & i\geq 0\ \mathrm{even}\\ 0 & \mathrm{else}\ .\end{array}\right .
\]
Moreover, $\pi_\ast \THH(H\bF_p)$ is isomorphic to the polynomial algebra in $u$ over $\bF_p$. $\hfill \Box$
\end{thm}

In other words, $\pi_{2i} \THH(H\bF_p) = \bF_p\cdot u^i$. This theorem is rather striking in the simple answer it gives (a priori it might involve the stable homotopy groups of spheres).

\begin{remark}\label{rem:boekstedtequiv} We cannot give new insights into the proof of B\"okstedt's theorem, which is proved by a computation with the B\"okstedt spectral sequence using Dyer-Lashoff operations in an essential way. However, there is a short argument, given in \cite{MR2684512, MR2651551}, deducing it from the following theorem of Mahowald--Hopkins.
\begin{altenumerate}
\item The free $\mathbb E_2$-algebra in $\Sp$ with a chosen nullhomotopy $p=0$ is given by $H\bF_p$.
\item Equivalently: One can write $H\bF_p$ as the Thom spectrum of the map $\Omega^2 S^3\to \mathrm{Pic}(\mathbb S^\wedge_p)$ of $\mathbb E_2$-groups\footnote{Recall that if $R$ is an $\E_\infty$-ring spectrum (often $\mathbb S$), $G$ is an $\mathbb E_n$-group in spaces and $\mathcal I: G\to \mathrm{Pic}(R)$ is a map of $\mathbb E_n$-groups, then the Thom spectrum is the $\mathbb E_n$-ring $R[G,\mathcal I]$ which is the twisted form $\colim_G \mathcal I$ of the group ring $R[G] = R\otimes_{\mathbb S} \mathbb S[G]$, where $\mathbb S[G]=\Sigma_+^\infty G$. For example, if $G$ is discrete, this is given by $R[G,\mathcal I]=\bigoplus_{g\in G} \mathcal I_g$, where $\mathcal I_g$ are invertible $R$-modules with $\mathcal I_g\otimes_R \mathcal I_h\cong \mathcal I_{gh}$, so that one can make $R[G,\mathcal I]$ in the obvious way into a ring.} induced by the map of pointed spaces
\[
S^1\to B\GL_1(\mathbb S^\wedge_p)\to \mathrm{Pic}(\mathbb S^\wedge_p)
\]
given by the element $1-p\in \pi_1(B\GL_1(\mathbb S^\wedge_p)) = \mathbb Z_p^\times$, noting that $\Omega^2 S^3$ is the free $\mathbb E_2$-group generated by $S^1$.
\end{altenumerate}
\end{remark}

In order to compute $\TC(H\bF_p)$, the most important ingredient is $\THH(H\bF_p)^{h\T}$. Note that there is a convergent spectral sequence
\[
E_2^{ij} = H^i(B\T,\pi_{-j}\THH(H\bF_p))\Rightarrow \pi_{-i-j} \THH(H\bF_p)^{h\T}
\]
(as usual, in cohomological Serre grading). Here, by B\"okstedt's theorem (and because $B\T=\mathbb C P^\infty$ is even), all contributions are in even total degree, so the spectral sequence will necessarily degenerate. In particular, it follows that $\pi_\ast \THH(H\bF_p)^{h\T}\to \pi_\ast \THH(H\bF_p)$ is surjective, and there are elements
\[
\tilde{u}\in \pi_2 \THH(H\bF_p)^{h\T}\ ,v \in \pi_{-2}\THH(H\bF_p)^{h\T}
\]
projecting to $u\in \pi_2 \THH(H\bF_p)$ and the natural generator of $H^2(B\T,\pi_0 \THH(H\bF_p)) = \bF_p$ that pulls back to the orientation class of $\C P^1$, respectively.

To compute $\pi_\ast \THH(H\bF_p)^{h\T}$, it remains to understand the extensions. This can actually be done by hand.

\begin{proposition} The homotopy groups of $\THH(H\bF_p)^{h\T}$ are given by
\[
\pi_i \THH(H\bF_p)^{h\T} = \left\{\begin{array}{cl} \Z_p & i\ \mathrm{even}\\ 0 & \mathrm{else}\ .\end{array}\right .
\]
More precisely, for $i\geq 0$,
\[
\pi_{2i} \THH(H\bF_p)^{h\T} = \Z_p\cdot \tilde{u}^i\ ,\ \pi_{-2i} \THH(H\bF_p)^{h\T} = \Z_p\cdot v^i\ ,
\]
and one can choose $\tilde{u}$ and $V$ so that $\tilde{u}v = p\in \pi_0 \THH(H\bF_p)^{h\T}$. Therefore,
\[
\pi_\ast \THH(H\bF_p)^{h\T} = \Z_p[\tilde{u},v] / (\tilde{u}v - p)\ .
\]
\end{proposition}

\begin{proof} We use that the spectral sequence
\[
E_2^{ij} = H^i(B\T,\pi_{-j}\THH(H\bF_p))\Rightarrow \pi_{-i-j} \THH(H\bF_p)^{h\T}
\]
is multiplicative. (As it degenerates at $E_2$, no subtleties regarding the multiplicative structure on the spectral sequence arise.) By Lemma~\ref{lem:unup} below, the image of $p\in \pi_0 \THH(H\bF_p)^{h\T}$, which necessarily lies in the first step of the abutment filtration, maps to the class
\[
uv\in E_2^{2,-2} = H^2(B\T,\pi_2 \THH(H\bF_p)) = \bF_p\ .
\]
By multiplicativity, this implies that for all $i\geq 1$, the image of $p^i\in \pi_0 \THH(H\bF_p)^{h\T}$, which lies in the $i$-th step of the abutment filtration, maps to the class
\[
u^iv^i\in E_2^{2i,-2i} = H^{2i}(B\T,\pi_{2i} \THH(H\bF_p)) = \bF_p\ .
\]
In particular, the powers of $p$ hit all contributions to total degree $0$, so
\[
\pi_0 \THH(H\bF_p)^{h\T} = \Z_p\ .
\]
Multiplying by powers of $\tilde{u}$, resp.~$v$, we get the given description of all homotopy groups. Changing $\tilde{u}$ by a unit, we can then arrange that $\tilde{u}v = p$.
\end{proof}

\begin{lemma}\label{lem:unup} The image of $p\in\pi_0 \THH(H\bF_p)^{h\T}$ in $H^2(B\T,\pi_2\THH(H\bF_p))$ is given by $uv$.
\end{lemma}

\begin{proof} The statement depends only on $\tau_{\leq 2} \THH(H\bF_p)\simeq \tau_{\leq 2} H\HH(\bF_p)$. Thus, it is enough to show that the image of
\[
p\in H_0 (\tau_{\leq 2} \HH(\bF_p))^{h\T}
\]
in $H^2(B\T,H_2\HH(\bF_p))\simeq \bF_p \cdot uv$ is $uv$. This is a standard computation in Hochschild homology, so we only sketch the argument.

It is enough to prove the same result for the image of $p$ in $H_0 \lim_{\C P^1} (\tau_{\leq 2} \HH(\bF_p))$, noting that $H^2(B\T, H_2\HH(\bF_p)) = H^2(\C P^1,H_2\HH(\bF_p))$. But we can understand the $\C P^1$-spectrum $\tau_{\leq 2} \HH(\bF_p)$. In fact, the truncation in degrees $\leq 2$ is equivalent to the truncation by the second filtration step in the filtration of Proposition~\ref{prop:hkr}. For a general commutative ring $A$, the quotient $\overline{\HH(A)}$ of $\HH(A)$ by this second filtration step is an extension
\[
\mathbb L_{A/\Z}[1]\to \overline{\HH(A)}\to A\ ,
\]
and its $\C P^1$-structure is given by a selfmap
\[
\overline{\HH(A)}\to \overline{\HH(A)}[-1]\ .
\]
If $A$ is smooth over $\Z$, this selfmap factors for degree reasons as
\[
\overline{\HH(A)}\to A\to \Omega^1_{A/\Z}\to \overline{\HH(A)}[-1]\ ;
\]
here, the middle map can be checked to be the derivative. By left Kan extension, this shows that in general, the selfmap is given by
\[
\overline{\HH(A)}\to A\to \mathbb L_{A/\Z}\to \overline{\HH(A)}[-1]\ .
\]
Also note that
\[
\lim_{\C P^1} \overline{\HH(A)} = \fib(\overline{\HH(A)}\to \overline{\HH(A)}[-1])\ .
\]
Applying this to $A=\bF_p$ (where $\overline{\HH(\bF_p)} = \tau_{\leq 2} \HH(\bF_p)$) shows that
\[
H_0 \lim_{\C P^1} \overline{\HH(\bF_p)} = (\bF_p\to \mathbb L_{\bF_p/\Z})\ .
\]
It is a standard computation in commutative algebra that the map $\bF_p\to \mathbb L_{\bF_p/\Z}\simeq (\bF_p\cdot u)[1]$ is given by the extension $\Z/p^2\Z$, with $p$ mapping to the generator $u$. This finishes the proof.
\end{proof}

Now it is easy to compute $\pi_\ast\THH(H\bF_p)^{t\T}$, which forms another ingredient in $\TC(H\bF_p)$.

\begin{corollary} The homotopy groups of $\THH(H\bF_p)^{t\T}$ are given by
\[
\pi_\ast \THH(H\bF_p)^{t\T} = \Z_p[v^{\pm 1}]\ .
\]
In particular, for all even $i\in \mathbb Z$, the map
\[
\pi_i\THH(H\bF_p)^{h\T}\cong \Z_p\to \pi_i \THH(H\bF_p)^{t\T}\cong \Z_p
\]
is injective. If $i\leq 0$, it is an isomorphism, while if $i=2j\geq 0$, it has image $p^j\Z_p$.
\end{corollary}

\begin{proof} There is the Tate spectral sequence
\[
E_2^{ij} = \pi_{-i} (\pi_{-j} \THH(H\bF_p))^{t\T}\Rightarrow \pi_{-i-j} \THH(H\bF_p)^{t\T}\ ,
\]
which is again multiplicative, and concentrated in even total degree. Comparing with the spectral sequence for $\THH(H\bF_p)^{h\T}$, we get the result in negative homotopical degree. Using multiplicativity, the result follows in positive homotopical degrees.
\end{proof}

Interestingly, we can also identify the Frobenius map
\[
\varphi_p^{h\T}: \THH(H\bF_p)^{h\T}\to \THH(H\bF_p)^{t\T}
\]
up to scalars.

\begin{proposition} For all even $i\in \mathbb Z$, the map
\[
\pi_i \varphi_p^{h\T}: \pi_i\THH(H\bF_p)^{h\T}\cong \Z_p\to \pi_i \THH(H\bF_p)^{t\T}\cong \Z_p
\]
is injective. If $i\geq 0$, it is an isomorphism, while if $i=-2j\leq0$, the image is given by $p^j\Z_p$.
\end{proposition}

\begin{proof} As $\tilde{u}v=p$ and the map is multiplicative, it follows that the maps must be injective, and that they are isomorphisms either in positive or in negative degrees. Assume that they are isomorphisms in negative degrees. Then we look at the commutative diagram
\[\xymatrix{
\pi_{-2} \THH(H\bF_p)^{h\T}\ar^{\pi_{-2} \varphi_p^{h\T}}[r]\ar[d] & \pi_{-2} \THH(H\bF_p)^{t\T}\ar[d]\ar[r] & \pi_{-2} H\bF_p^{t\T}\ar[d]\\
\pi_{-2} \THH(H\bF_p)\ar^{\pi_{-2} \varphi_p}[r] & \pi_{-2} \THH(H\bF_p)^{tC_p}\ar[r] & \pi_{-2} H\bF_p^{tC_p}\ .
}\]
Note that the lower left corner is equal to $0$, so if the upper left arrow is surjective, the map
\[
\Z_p\cdot v = \pi_{-2} \THH(H\bF_p)^{t\T}\to \pi_{-2} H\bF_p^{tC_p}\cong \bF_p
\]
must be zero. But $v$ maps to a nonzero class in $\pi_{-2} H\bF_p^{t\T}$ by construction, and this class maps to a nonzero element of $\pi_{-2} H\bF_p^{tC_p}$.
\end{proof}

\begin{corollary} The homotopy groups of $\TC(H\bF_p)$ are given by
\[
\pi_i \TC(H\bF_p) = \left\{\begin{array}{cl} \Z_p & i=0,-1\\ 0 & \mathrm{else}\ .\end{array}\right .
\]
\end{corollary}

In particular, note that this implies that $\THH(H\bF_p)$ is a $\T$-equivariant $\tau_{\geq 0} \TC(H\bF_p) = H\Z_p$-algebra, and the Frobenius map
\[
\varphi_p: \THH(H\bF_p)\to \THH(H\bF_p)^{tC_p}
\]
is a $\T\cong \T/C_p$-equivariant map of $\E_\infty$-$H\Z_p$-algebras.

\begin{proof} Note that $H\bF_p$ is $p$-complete, so $\TC(H\bF_p)=\TC(H\bF_p,p)$, which is $p$-complete, so there is a fiber sequence
\[
\TC(H\bF_p)\to \THH(H\bF_p)^{h\T}\xto{\can-\varphi_p^{h\T}} \THH(H\bF_p)^{t\T}\ .
\]
The second and third term have only even homotopy groups, so we get exact sequences
\[
0\to \pi_{2i} \TC(H\bF_p)\to \pi_{2i} \THH(H\bF_p)^{h\T}\xto{\can-\varphi_p^{h\T}} \pi_{2i} \THH(H\bF_p)^{t\T}\to \pi_{2i-1} \TC(H\bF_p)\to 0\ .
\]
But if $i\neq 0$, then the middle map is the difference of an isomorphism and a map divisible by $p$ between two copies of $\Z_p$; thus, an isomorphism. For $i=0$, we must have $\varphi_p^{h\T} = \id$, as everything is a $\Z_p$-algebra. The result follows.
\end{proof}

\begin{remark} Note that the map $\varphi_p: \THH(H\bF_p)^{h\T} \to \THH(H\bF_p)^{t \T}$ can be extended uniquely to a map of $\E_\infty$-rings
$$
(\THH(\bF_p)^{t\T})[p^{-1}] \simeq \THH(\bF_p)^{h\T}[p^{-1}] \xto{\varphi_p^{h\T}} \THH(\bF_p)^{t\T}[p^{-1}]
$$
by inverting $p$ in the source, which also inverts $v$. The homotopy groups of these spectra are given by the ring $\Q_p[v^{\pm 1}]$. This is the ``meromorphic extension'' used on the level of homotopy groups by Hesselholt in \cite[Proposition 4.2]{2016arXiv160201980H}.
\end{remark}

In fact, we can now completely identify $\THH(H\bF_p)$ as an $\E_\infty$-algebra in cyclotomic spectra. For this, we use the following formulas for $\T$-equivariant chain complexes.

\begin{lemma}\label{lem:chaincomplext} The following natural transformations of functors $\calD(\Z)^{B\T}\to \calD(\Z)$, induced by the respective lax symmetric monoidal structures, are equivalences:
\[\begin{aligned}
X^{h\T}\otimes_{\Z^{h\T}} \Z&\to X\ ;\\
X^{h\T}\otimes_{\Z^{h\T}} \Z^{t\T}&\to X^{t\T}\ ;\\
X^{t\T}\otimes_{\Z^{t\T}} \Z^{tC_n}&\to X^{tC_n}\ , n\geq 1\ .
\end{aligned}\]
\end{lemma}

\begin{proof} For the first, note that $\Z$ is compact as $\Z^{h\T}$-module, so the left-hand side commutes with all limits, so using Postnikov towers we can assume that $X$ is bounded above, say coconnective. Then the left-hand side commutes with filtered colimits and is exact, so it suffices to consider the statement for $X=\Z$, where it is clear.

Now, for the second, note that by Theorem~\ref{thm:genfarrelltate}, it is enough to check that the homotopy fiber of $X^{h\T}\otimes_{H\Z^{h\Z}} H\Z^{t\T}\to X^{h\T}$ commutes with all colimits. But this follows from the first, noting that the homotopy fiber $\Sigma \Z_{h\T}$ of $\Z^{h\T}\to \Z^{t\T}$ is a filtered colimit of perfect complexes.

For the third, note that by Lemma~\ref{lem:tates1cn}, one has an equivalence of chain complexes $\Z^{tC_n}\simeq \Z^{t\T}/n$. In particular, $\Z^{tC_n}$ is a perfect $\Z^{t\T}$-module. Thus, $- \otimes_{\Z^{t\T}} \Z^{tC_n}$ commutes with all limits and colimits. By Lemma~\ref{lem:tateconvergence}, we can then assume that $X$ is bounded, or even that $X$ is a flat $\Z$-module concentrated in degree $0$. Then it follows from a direct computation.
\end{proof}

We can now prove the following corollary.

\begin{corollary}\label{cor:computationthhfp} The $\T$-equivariant map of $\E_\infty$-algebras $H\Z_p\to \THH(H\bF_p)$ induces a $\T/C_p$-equivariant equivalence of $\E_\infty$-algebras
\[
H\Z_p^{tC_p}\simeq \THH(H\bF_p)^{tC_p}\ .
\]
In particular, $\pi_\ast \THH(H\bF_p)^{tC_p}\cong \bF_p[v^{\pm 1}]$.

Moreover, the $\T\cong \T/C_p$-equivariant map of $\E_\infty$-algebras
\[
\varphi_p: \THH(H\bF_p)\to \THH(H\bF_p)^{tC_p}
\]
identifies $\THH(H\bF_p)$ with the connective cover $\tau_{\geq 0} \THH(H\bF_p)^{tC_p}\simeq \tau_{\geq 0} H\Z_p^{tC_p}$.
\end{corollary}
\begin{proof} Using Lemma~\ref{lem:chaincomplext}, we see that
\[
\THH(H\bF_p)^{tC_p}\simeq \THH(H\bF_p)^{t\T}/p\ ,
\]
which implies that $\pi_\ast \THH(H\bF_p)^{tC_p}\simeq \bF_p[v^{\pm 1}]$. This is also the homotopy of $H\Z_p^{tC_p}$, and any graded ring endomorphism of $\bF_p[v^{\pm 1}]$ is an isomorphism.

For the second claim, it is enough to show that
\[
\pi_\ast \varphi_p: \pi_\ast \THH(H\bF_p)\to \pi_\ast \THH(H\bF_p)^{tC_p}\]
is an isomorphism in nonnegative degrees. But this follows from the commutative diagram
\[\xymatrix{
\THH(H\bF_p)^{h\T}\ar[r]^{\varphi_p^{h\T}}\ar[d] & \THH(H\bF_p)^{t\T}\ar[d]\\
\THH(H\bF_p)\ar[r]^{\varphi_p} & \THH(H\bF_p)^{tC_p}
}\]
and the explicit descriptions of the other three maps.
\end{proof}

\newcommand{\triv}{\mathrm{triv}}
\newcommand{\sh}{\mathrm{sh}}

This last corollary can be used to identify $\THH(\bF_p)$ as an $\E_\infty$-cyclotomic spectrum as follows. First note that for every spectrum $X$ there is a cyclotomic spectrum $X^\triv$ whose underlying spectrum is $X$ equipped with the trivial $\T$-action and the Frobenius given by the $\T$-equivariant composition $X \to X^{hC_p} \xto{\can} X^{tC_p}$ where the first map is pullback along $BC_p \to \pt$. For example $\bS^\triv$ is the cyclotomic sphere, c.f. Example \ref{examplesTHH}(ii). Another way of writing $X^\triv$ is as the $X$-indexed colimit of the constant diagram in cyclotomic spectra with value the cyclotomic sphere. This shows the following statement:

\begin{proposition} There is an adjunction
$$
\xymatrix{
-^\triv: \Sp \ar[r]<2pt> & \Cycn: \TC \ar[l]<2pt> \ .
}
$$ $\hfill \Box$
\end{proposition}

The computation of $\TC(H\bF_p)$ shows that there is a map $H\Z_p \to \TC(H\bF_p)$ which then by adjunction induces a map of cyclotomic spectra
$$
H\Z_p^\triv \to \THH(H\bF_p).
$$
\begin{construction}
Let $X$ be a connective cyclotomic spectrum. We construct a new connective cyclotomic spectrum $\sh_p X$ as follows: 

The underlying spectrum of $\sh_pX$ with $\T$-action is $\tau_{\geq 0} (X^{tC_p})$ (where as usual this carries the residual action). This spectrum is $p$-complete, thus the Frobenius maps $\varphi_l$ for $l \neq p$ are zero as the target is zero. The Frobenius $\varphi_p$ is induced by the initial Frobenius, which we interpret as a map $X \to \tau_{\geq 0} (X^{tC_p})$ (since $X$ is connective) by applying the functor $\tau_{\geq 0}(-^{tC_p})$ which commutes with itself.

The cyclotomic spectrum $\sh_p X$ comes with a natural map $X \to \sh_pX$ of cyclotomic spectra induced by $\varphi_p$. 
\end{construction}

\begin{corollary}\label{corfp}
In the diagram induced from the map $H\Z_p^\triv \to \THH(\bF_p)$ of cyclotomic spectra
$$
\xymatrix{
H\Z_p^\triv \ar[r]\ar[d] & \THH(\bF_p)\ar[d]^\simeq \\
\mathrm{sh}_p H\Z_p^\triv \ar[r]^\simeq & \mathrm{sh}_p \THH(\bF_p)
}
$$
the right and the lower map 
are equivalences of $\E_\infty$-cyclotomic spectra, i.e. $\THH(\bF_p) \simeq \mathrm{sh}_p (H\Z^\triv)$.$\hfill \Box$
\end{corollary}

In other words, the cyclotomic $\E_\infty$-ring $\THH(H\bF_p)$ is given by $\tau_{\geq 0} H\Z_p^{tC_p}$ with its remaining $\T=\T/C_p$-action, and the Frobenius $\varphi_p$ is given by realizing $\tau_{\geq 0} H\Z_p^{tC_p}$ as the connective cover of
\[
(\tau_{\geq 0} H\Z_p^{tC_p})^{tC_p} = H\Z_p^{tC_p}\ ,
\]
where the equality follows from the Tate orbit lemma (or the first part of the corollary). From this, one can deduce that one can choose $\tilde{u}\in \pi_2 \THH(H\bF_p)^{h\T}$ and $v\in \pi_{-2} \THH(H\bF_p)^{h\T}$ such that $\tilde{u}v=p$ and $\varphi_p^{h\T}(v)=pv$. In particular, this shows that the constants $\lambda_n$ in \cite[Proposition 5.4]{HesselholtMadsen} can be taken as $1$. Indeed, by Corollary~\ref{cor:segalinduct}, we see that for all $n\geq 1$,
\[
\THH(H\bF_p)^{C_{p^n}}\simeq \tau_{\geq 0} \THH(H\bF_p)^{hC_{p^n}}\simeq \tau_{\geq 0} H\Z_p^{tC_{p^{n+1}}}\ ,
\]
and all maps become explicit. 

\begin{remark} A consequence of Corollary~\ref{cor:computationthhfp} is that $H\Z_p^{tC_p}=H\Z^{tC_p}$ admits an $\E_\infty$-$H\bF_p$-algebra structure. Is there a direct way to see this, and is this also true if $p$ is not a prime?

As a word of warning, we note that this $H\bF_p$-algebra structure on $H\Z_p^{tC_p}$ is incompatible with its natural $Hb\Z_p$-algebra structure. In fact, there are at least three different maps of $\E_\infty$-algebras $H\Z_p\to H\Z_p^{tC_p}$: The natural one via $H\Z_p\to H\Z_p^{hC_p}\to H\Z_p^{tC_p}$, the Tate valued Frobenius, and the composition $H\Z_p\to H\bF_p\to H\Z_p^{tC_p}$.
\end{remark}

Now, finally, let $A$ be an $\E_2$-algebra of characteristic $p$, i.e.~$p=0$ in $\pi_0 A$. Fixing a nullhomotopy from $p$ to $0$, we get that $A$ is an $\E_2$-$H\bF_p$-algebra by Remark~\ref{rem:boekstedtequiv}. Then $\THH(A)$ is a module spectrum over $\THH(H\bF_p)$ compatibly with the cyclotomic structure. In particular, $\THH(A)$ is a module over $\tau_{\geq 0} \TC(H\bF_p) = H\Z_p$, compatibly with the cyclotomic structure. This implies that there is a fiber sequence
\[
\TC(A)\to \THH(A)^{h\T}\xto{\can-\varphi_p^{h\T}} \THH(A)^{t\T}
\]
even if $A$ is not bounded below, as for $H\Z$-module spectra, the Tate orbit lemma is always valid. Moreover,
\[
\THH(A)^{t\T}\simeq \THH(A)^{h\T}\otimes_{H\Z_p^{h\T}} H\Z_p^{t\T} = \THH(A)^{h\T}[v^{-1}]
\]
by Lemma~\ref{lem:chaincomplext}.

In particular, note that as $v$ divides $p$, we have automatically
\[
\THH(A)^{t\T}[p^{-1}] = \THH(A)^{h\T}[p^{-1}]
\]
via the canonical map, and so one can regard
\[
\varphi_p^{h\T}: \THH(A)^{h\T} \to \THH(A)^{t\T}
\]
as a self-map $\THH(A)^{t\T}[p^{-1}]\to \THH(A)^{t\T}[p^{-1}]$ after inverting $p$.

\setcounter{thm}{0}
\appendix
\setcounter{appendix}{0}
\setcounter{chapter}{18}
\chapter{Symmetric monoidal $\infty$-categories}\label{app:symmmon}

\renewcommand{\thesection}{A}

In this Appendix we recall the notion of a symmetric monoidal $\infty$-category from \cite{HA} which is used in an essential way throughout the whole paper. We also discuss Dwyer-Kan localizations of symmetric monoidal $\infty$-categories following Hinich \cite{Hinich}. 

In order to prepare for the definition of a symmetric monoidal $\infty$-category, recall that $\Fin_\ast$ is the category of finite pointed sets. We denote by $\langle n\rangle\in \Fin_\ast$ the set $\{0,1,\ldots,n\}$ pointed at $0$. For $i=1,\ldots,n$, we denote by $\rho^i: \langle n\rangle\to \langle 1\rangle$ the projection sending all elements to $0$, except for $i\in \langle n\rangle$.

\begin{definition}[{\cite[Definition 2.0.0.7]{HA}}]\label{defsymmon} A symmetric monoidal $\infty$-category is a coCartesian fibration
\[
\calC^\otimes\to N(\Fin_\ast)\ ,
\]
of simplicial sets, such that the functor 
\[
(\rho^i_!)_{i=1}^n: \calC^\otimes_{\langle n\rangle}\to \prod_{i=1}^n \calC^\otimes_{\langle 1\rangle}\ .
\]
is an equivalence for all $n\geq 0$.
\end{definition}


Given a symmetric monoidal $\infty$-category $\calC^\otimes$, we denote by $\calC=\calC^\otimes_{\langle 1\rangle}$ the ``underlying'' $\infty$-category, and will sometimes by abuse of notation simply say that $\calC$ is a symmetric monoidal $\infty$-category. Note that by the condition imposed in the definition, one has $\calC^\otimes_{\langle n\rangle}\simeq \calC^n$. For a general map $f: \langle n\rangle\to \langle m\rangle$, the corresponding functor $f_!: \calC^\otimes_{\langle n\rangle}\simeq \calC^n\to \calC^\otimes_{\langle m\rangle}\simeq \calC^m$ is given informally by $(X_1,\ldots,X_n)\in \calC^n\mapsto (\bigotimes_{j\in f^{-1}(i)} X_j)_i\in \calC^m$.

Another piece of notation that we need is the ``active part'' of $\calC^\otimes$, defined as
\[
\calC^\otimes_\act = \calC^\otimes\times_{N(\Fin_\ast)} N(\Fin)\ ,
\]
where $\Fin$ is the category of finite (possibly empty) sets, and the functor $\Fin\to \Fin_\ast$ is given by adding an additional base point. By \cite[Definition 2.1.2.1, Definition 2.1.2.3, Remark 2.2.4.3]{HA}, this agrees with the definition in \cite[Remark 2.2.4.3]{HA}, noting that $\Fin\subseteq \Fin_\ast$ is the subcategory with all objects and active morphisms.

A morphism $f: \langle n \rangle \to \langle m \rangle$ in $\Fin_\ast$ is called \emph{inert} if for every $0\neq i \in \langle m \rangle$ the preimage $f^{-1}(i)$ contains exactly one element 
\cite[Definition 2.1.1.8.]{HA}.
\begin{definition}
Let $p: \calC^\otimes \to N(\Fin_\ast)$ and $q: \calD^\otimes \to N(\Fin_\ast)$ be symmetric monoidal $\infty$-categories. A symmetric monoidal functor is a functor $
F^\otimes: \calC^\otimes \to \calD^\otimes$ such that $p = q \circ F^\otimes $ and such that $F^\otimes$ carries $p$-coCartesian lifts to $q$-coCartesian lifts. 

A lax symmetric monoidal functor is a functor $F^\otimes: \calC^\otimes \to \calD^\otimes$ such that $p = q \circ F^\otimes $ and such that $F^\otimes$ carries $p$-coCartesian lifts of inert morphisms in $N(\Fin_\ast)$ to $q$-coCartesian lifts.
\end{definition}

For a given (lax) symmetric monoidal functor $F^\otimes: \calC^\otimes \to \calD^\otimes$ we will write $F := (F^\otimes)_{\langle 1 \rangle} : \calC \to \calD$ and refer to it as the underlying functor. Abusively we will very often only say that $F$ is a (lax) symmetric monoidal functor. The $\infty$-category of lax symmetric monoidal functors is denoted by $\Fun_\lax(\calC,\calD)$ and defined as a full subcategory of $\Fun_{N(\Fin_\ast)}(\calC^\otimes, \calD^\otimes)$. Similarly we denote the full subcategory of symmetric monoidal functors by $\Fun_\otimes(\calC,\calD) \subseteq \Fun_\lax(\calC,\calD) $.

\begin{remark}
Lax symmetric monoidal functors are the same as maps between the underlying $\infty$-operads of symmetric monoidal $\infty$-categories, see \cite[Section 2.1.2]{HA}. These $\infty$-operad maps $\calC^\otimes \to \calD^\otimes$ are also called  $\calC$-algebras in $\calD$. This is reasonable terminology in the context of operads, but in the context of symmetric monoidal categories we prefer the term lax symmetric monoidal functors. But note that a lot of constructions done with lax symmetric monoidal functors are more naturally done in the context of $\infty$-operads.
\end{remark}

For the rest of the section we discuss Dwyer-Kan localizations of symmetric monoidal $\infty$-categories. Therefore assume that we are given a symmetric monoidal $\infty$-category $\calC^\otimes$ and a class of edges $W \subseteq \calC_1$ in the underlying $\infty$-category called weak equivalences. We define a new class $W^\otimes$ of edges in $\calC^\otimes$ consisting of all morphisms in $\calC^\otimes_{\langle n \rangle}$ lying over an identity morphism $\id_{\langle n\rangle}$ in $N(\Fin_\ast)$ and which correspond under the equivalence $\calC^\otimes_{\langle n \rangle} \simeq \calC^n$ to products of edges in $W$. By definition the functor $\calC^\otimes \to N(\Fin_\ast)$ sends edges in $W^\otimes$ to identities in $N(\Fin_\ast)$.

\begin{definition}\label{deflocsym}
We define an $\infty$-category $\calC[W^{-1}]^\otimes \to N(\Fin_\ast)$ together with a functor $i: \calC^\otimes \to \calC[W^{-1}]^\otimes$ over $N(\Fin_\ast)$ such that $i$ exhibits $\calC[W^{-1}]^\otimes$ as the Dwyer-Kan localization of $\calC^\otimes$ at the class $W^\otimes$ and such that $\calC[W^{-1}]^\otimes \to N(\Fin_\ast)$ is a categorical fibration.
\end{definition}

The definition determines $\calC[W^{-1}]^\otimes$ up to contractible choices, therefore we assume that a choice is made once and for all. We warn the reader that despite the notation we do not claim that in general for the $\infty$-category $\calC[W^{-1}]^\otimes$ the fiber over $\langle 1 \rangle \in \Fin_\ast$ is equivalent to the Dwyer-Kan localization $\calC[W^{-1}]$ nor that $\calC[W^{-1}]^\otimes \to N(\Fin_\ast)$ is a symmetric monoidal $\infty$-category. However, this will be the case in favourable situations.

\begin{proposition}[Hinich]\label{propHinich}
Assume that the tensor product $\otimes: \calC \times \calC \to \calC$ preserves weak equivalences separately in both variables. Then the following holds:
\begin{altenumerate}
\item
$\calC[W^{-1}]^\otimes \to N(\Fin_\ast)$ is a symmetric monoidal $\infty$-category.
\item
The functor $i\colon \calC^\otimes \to \calC[W^{-1}]^\otimes$ is symmetric monoidal.
\item
The underlying functor $i_{\langle 1 \rangle}$ exhibits $\calC[W^{-1}]^\otimes_{\langle 1 \rangle}$ as the Dwyer-Kan localization of $\calC$ at $W$.  
\item 
More generally: for every $\infty$-category $K$ equipped with a map $K \to N(\Fin_\ast)$ the pullback $\calC[W^{-1}] \times_{N(\Fin_\ast)} K \to K$ is the Dwyer-Kan localization of $\calC \times_{N(\Fin_\ast)} K \to K$ at the class of weak equivalences obtained by pullback from $W^\otimes$.
\item
For every other symmetric monoidal $\infty$-category $\calD$ the functor $i$ induces equivalences
$$
\Fun_\lax(\calC[W^{-1}], \calD) \to \Fun_\lax^W(\calC, \calD) \qquad \Fun_\otimes(\calC[W^{-1}], \calD) \to \Fun_\otimes^W(\calC, \calD) 
$$
where the superscript $W$ denotes the full subcategories of functors which send $W$ to equivalences.
\end{altenumerate}
\end{proposition}
\begin{proof}
The first three assertions are Proposition 3.2.2 in \cite{Hinich}. The claim (4) follows immediately from Proposition 2.1.4 in \cite{Hinich} since coCartesian fibrations are stable under pullback and equivalence can be tested fiberwise. The last assertion follows as follows: by construction we have that $i$ induces an equivalence 
$$
\Fun_{N(\Fin_\ast)}(\calC[W^{-1}]^\otimes,\calD^\otimes) \to  \Fun_{N(\Fin_\ast)}^W(\calC^\otimes,\calD^\otimes) 
$$
thus we only need to check that this equivalence respects functors which preserve coCartesian lifts of (inert) morphisms in $N(\Fin_\ast)$. But this is clear by the fact that $i$ is symmetric monoidal and the uniqueness of coCartesian lifts. 
\end{proof}

Note that claim (iv) of Proposition \ref{propHinich} in particular implies that $\calC[W^{-1}]^\otimes_\act$ is the Dwyer-Kan localization of $\calC^\otimes_\act$. We note that this Proposition also follows from the methods that we develop below, which are inspired by but logically independent of Hinich's results.

\begin{remark}
Under the same assumption as for Proposition \ref{propHinich} a symmetric monoidal Dwyer-Kan localization of $\calC^\otimes$ is constructed in \cite[Proposition 4.1.7.4]{HA} by different methods. It also satisfies the universal property $\Fun_\otimes(\calC[W^{-1}], \calD) \xto{\simeq} \Fun_\otimes^W(\calC, \calD)$. Thus the two constructions are equivalent. But we also need the lax symmetric monoidal statement (iv) of Proposition \ref{propHinich} which does not seem to follow directly from the construction given by Lurie. 
\end{remark}

Now finally we consider the case of a symmetric monoidal model category $\calM$. Recall that this means that $\calM$ is a closed symmetric monoidal category, a model category and the two structures are compatible in the following sense: the tensor functor $\otimes: \calC \times \calC \to \calC$ is a left Quillen bifunctor and for every cofibrant replacement $Q\mathbbm{1} \to \mathbbm{1}$ of the tensor unit $\mathbbm{1} \in \calC$ and every cofibrant object $X \in \calC$ the morphism
$ Q\mathbbm{1} \otimes X \to \mathbbm{1} \otimes X \cong X$ is a weak equivalence. Note that the latter condition is automatically satisfied if the tensor unit is cofibrant, which is also sometimes assumed for a symmetric monoidal model category. See for example \cite[Section 4]{MR1650134} for a discussion of symmetric monoidal model categories.

In the case of a symmetric monoidal model category $\calM$ the assumption of Proposition \ref{propHinich} is not satisfied since the tensor product $\otimes: \calM \times \calM \to \calM$ does not necessarily preserve weak equivalences in both variables separately.  It does however if we restrict attention to the full subcategory $\calM_c \subseteq \calM$ of cofibrant objects. As a result we get that $N(\calM_c)[W^{-1}]^\otimes \to N\Fin_\ast$ is a symmetric monoidal $\infty$-category. The underlying $\infty$-category $N(\calM_c)[W^{-1}]$ is equivalent to the Dwyer-Kan localization $N(\calM)[W^{-1}]$ by a result of Dwyer and Kan, see \cite[Proposition 5.2]{MR584566}. We now  prove a similar statement for the symmetric monoidal version. 

\begin{thm}\label{DKlocalizationsymmon}
Let $\calM$ be a symmetric monoidal model category.
\begin{enumerate}
\item\label{teins}
The functor $N(\calM)[W^{-1}]^\otimes \to N(\Fin_\ast)$ defines a symmetric monoidal $\infty$-category.
\item\label{tzwei}
The functor $i: N(\calM)^\otimes \to N(\calM)[W^{-1}]^\otimes$ is lax symmetric monoidal.
\item\label{tdrei}
The underlying functor $i_{\langle 1 \rangle}$ exhibits $N(\calM)[W^{-1}]^\otimes_{\langle 1 \rangle}$ as the Dwyer-Kan localization of $N(\calM)$ at $W$.  
\item\label{tvier}
More generally: for every $\infty$-category $K$ equipped with a map $K \to N(\Fin_\ast)$ the pullback $N(\calM)[W^{-1}]^\otimes \times_{N(\Fin_\ast)} K \to K$ is the Dwyer-Kan localization of $N(\calM)^\otimes \times_{N(\Fin_\ast)} K \to K$.
\item\label{tfuenf}
For every symmetric monoidal $\infty$-category $\calD$ the functor $i$ induces an equivalence
$$
\Fun_\lax(N(\calM)[W^{-1}], \calD) \to \Fun_\lax^W(N(\calM), \calD) 
$$
where the superscript $W$ denotes the full subcategory of functors which send $W$ to equivalences in $\calD$.
\item\label{tsechs}
The inclusion of the cofibrant objects $\calM_c \to \calM$ induces an equivalence
$
N(\calM_c)[W^{-1}]^\otimes \to N(\calM)[W^{-1}]^\otimes 
$ of symmetric monoidal $\infty$-categories.
\end{enumerate}
\end{thm}

The theorem will follow as a special case of a more general claim about Dwyer-Kan localizations in families. To this end we generalize the results given in \cite[Section 2]{Hinich}. For the next definition we will need the notion of an absolute right Kan extension. Recall first that a diagram of $\infty$-categories of the form
\begin{equation}\label{kanext}
\xymatrix{
X \ar[rd]^f\ar[d]_i  &  \\
X' \ar[r] _{g} \ar@{=>}[ur]-<22pt>& Y
}
\end{equation}
is said to exhibit $g$ as a right Kan extension of $f$ along $i$ if it is terminal in the $\infty$-category of all completions of the diagrams
$$
\xymatrix{
X \ar[rd]^f\ar[d]_-i &  \\
X'  & Y
}
$$
to a diagram as above. 
Note that diagram \eqref{kanext} does not commute, only up to a non-invertible 2-cell as indicated.  We say that $g$ is an absolute right Kan extension if for every functor $p: Y\to Y'$ to an $\infty$-category $Y'$ the induced diagram  
$$
\xymatrix{
X \ar[rd]^{pf} \ar[d]_-i & \\
X' \ar[r]_{pg} \ar@{=>}[ur]-<22pt>& Y'
}
$$
exhibits $pg$ as the right Kan extension of $pf$ along $i$. 
\begin{definition}\label{defderivable}
Let $p: X \to S$ be a coCartesian fibration where $S$ is an $\infty$-category equipped with a subset $W$ of edges in $X$. We say that $p$ is \emph{left derivable} if the following conditions are satisfied:
\begin{enumerate}
\item
The morphisms in $W$ are sent to identities by $p$.
\item
For every morphism $s: a \to b$ in $S$ the functor $s_!: X_a \to X_b$ is left derivable, i.e.~there exists an absolute right Kan extension $Ls_!$ in the diagram
$$
\xymatrix{
X_a \ar[r]^{s_!}\ar[d] & X_b \ar[d] \\
X_a[W_a^{-1}] \ar@{-->}[r]^{Ls_!} \ar@{}[ru]^(.30){}="a"^(.80){}="b" \ar@2{-->} "a";"b" & X_b[W^{-1}_b].
}
$$
This means that the morphism $X_a[W_a^{-1}] \to X_b[W^{-1}_b]$ is the absolute right Kan extension of $X_a \to X_b \to  X_b[W^{-1}_b]$ along $X_a \to X_a[W_a^{-1}]$.
\item
For every 2-simplex 
$$
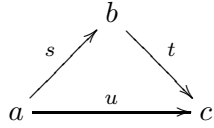
\xymatrix{
& b\ar[rd]^t & \\
a \ar[ru]^s\ar[rr]^{u} && c 
}
$$
in $S$ the canonical morphism $Lt_! \circ Ls_! \to Lu_!$ of functors $X_a[W_a^{-1}] \to X_c[W_c^{-1}]$ is an equivalence.
\end{enumerate} 
\end{definition}

\begin{example}\label{preserve}
Assume that for a coCartesian fibration $X \to S$ with a marking $W$ condition (1) is satisfied and that all the functors $s_!: X_a \to X_b$ have the property that  $s_!(W_a) \subseteq W_b$. Then $X \to S$ is left derivable. In this case the absolute right Kan extension 
$$
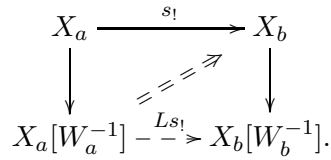
\xymatrix{
X_a \ar[r]^{s_!}\ar[d] & X_b \ar[d] \\
X_a[W_a^{-1}] \ar@{-->}[r]^{Ls_!} \ar@{}[ru]^(.30){}="a"^(.80){}="b" \ar@2{-->} "a";"b" & X_b[W^{-1}_b].
}
$$
is given by the factorization using the universal property of Dwyer-Kan localizations and the 2-cell is an equivalence. To see this we just note that right Kan extension is a right adjoint to the fully faithful restriction functor
$$
\Fun(X_a[W^{-1}], X_b[W^{-1}_b]) \simeq \Fun^W(X_a, X_b[W^{-1}_b]) \subseteq \Fun(X_a, X_b[W^{-1}_b]).
$$
Therefore if a functor $s_!$ already  lies in the subcategory of functors that preserve weak equivalences, then the right adjoint does not change it. The same remains true after postcomposition with another functor, so that it is in fact an absolute right Kan extension.
 
 For example if we have a symmetric monoidal $\infty$-category $\calC$ satisfying the assumptions of  Proposition \ref{propHinich} then $\calC^\otimes \to N\Fin_\ast$ with the class $W^\otimes$ (as defined before Definition \ref{deflocsym}) is left derivable.
\end{example}

\begin{example}\label{Quillenfunctor}
Let $\calM$ and $\calN$ be model categories and $F: \calM \to \calN$ a left Quillen functor. Then the coCartesian fibration $X \to \Delta^1$ classified by the functor $NF: N\calM \to N\calN$ is left derivable for the class of weak equivalences that are given by the weak equivalences in $\calM$ and $\calN$. 

To see this we recall the well known fact, that the derived functor in the sense of Quillen model categories is the (absolute) right Kan extension. We will give a quick proof of this. For simplicity we assume that $\calM$ has a functorial cofibrant replacement,  i.e. for every $X$ in $\calM$ there is a cofibrant replacement $X_c \to X$ that depends naturally on $X$. This is not necessary for the statement,  but simplifies the proof considerably and is in practice almost always satisfied, in particular in all our applications. The argument for the more general case follows the arguments given in \cite[Section 1]{Hinich} or \cite[Section 5]{MR584566}.

As a first step we consider the functor
$$
R\colon \Fun(N(\calM), \calD) \to \Fun(N(\calM),\calD) \qquad R(G)(X) = G(X_c)
$$
for a general $\infty$-category $\calD$. The functor $R$ comes with a natural transformation $R \to \id$ induced from the maps $X_c \to X$ and if $G$ lies in the full subcategory 
$$
\Fun^W(N(\calM),\calD) \subseteq \Fun(N(\calM),\calD) 
$$
then the transformation $R(G) \to G$ is an equivalence. In particular $R$ sends the full subcategory $\Fun^W(N(\calM),\calD)$ to itself. Now assume that we are given a functor $G$ with the property that $R(G)$ lies in $\Fun^W(N(\calM),\calD)$. Then we claim that the morphism $R(G) \to G$ exhibits  $R(G)$ as the right Kan extension of $G$ along $N(\calM) \to N(\calM)[W^{-1}]$ or said differently the reflection of $G$ into the full subcategory 
$\Fun^W(N(\calM),\calD) \subseteq \Fun(N(\calM),\calD)$. To see this we consider the subcategory 
$$\Fun'(N(\calM), \calD) \subseteq \Fun(N(\calM), \calD))$$
 given by all functors $G: N\calM \to \calD$ for which $RG$ lies in $\Fun^W(N(\calM),\calD)$. By what we have said before we have in particular 
 $$\Fun^W(N(\calM),\calD) \subseteq \Fun'(N(\calM),\calD).$$
 Now since $G$ lies by assumption in $\Fun'(N(\calM),\calD)$ it obviously suffices to check that $RG$ is the reflection from $\Fun'(N(\calM),\calD)$ into $\Fun^W(N(\calM),\calD)$ (since we only need to check the universal property against objects of $\Fun^W(N(\calM),\calD)$). But on the subcategory $\Fun'(N(\calM),\calD)$ the endofunctor $G$ defines a colocalization with local 
 objects the functors in $\Fun^W(N(\calM),\calD)$ which can be seen using \cite[Proposition 5.2.7.4]{HTT}. Since this proof works for all $\calD$ we see that $R(G)$ is in fact the absolute right Kan extension. 
 
 Now we specialise the case that we have considered before to the case where $\calD = N\calN[W^{-1}]$ and $G$ is the functor $N\calM \to N\calN \to N\calN[W^{-1}]$. By the properties of a left Quillen functor it follows that the functor $RG(X) = F(X_c): N\calM \to N\calN[W^{-1}]$ preserves weak equivalences. Thus it factors to a functor which is the absolute right Kan extension. But $RG$ is by definition exactly the left derived functor in the sense of Quillen.
 \end{example}

\begin{example}\label{Quillen-bifunctor}
As a variant of the example before assume that we have a Quillen bifunctor $B: \calM \times \calM' \to \calN$. Then the left derived bifunctor is also the absolute right Kan extension. This follows as in Example \ref{Quillenfunctor}. Thus the coCartesian fibration over $\Delta^1$ classified by $B$ with the obvious notion of weak equivalence is also left derivable.  
\end{example}

\begin{example}\label{composition}
Let $\calM, \calN, \calO$ be model categories and $F: \calM \to \calN$ and $G: \calN \to \calO$ be left Quillen functors. Then we claim that the resulting coCartesian fibration $X \to \Delta^2$ classified by the diagram$$
\xymatrix{
& N(\calN)\ar[rd]^G & \\
N(\calM)\ar[rr]^{GF}\ar[ur]^F && N(\calO)  
}
$$
is left derivable, where we use the obvious choice of weak equivalences. Since the functors $F$, $G$ and $GF$ are left derivable as we have seen in Example \ref{Quillenfunctor} we can deduce that condition (2) of Definition \ref{defderivable} is satisfied. To verify condition (3) we need to verify that the right Kan extension of the composite $GF$ is equivalent to the compositions of the right Kan extensions of $F$ and $G$. But this is obvious by the formula for right Kan extensions of model categories.
\end{example}

\begin{example}\label{example_symmon}
Let $\calM$ be a symmetric monoidal model category. We consider the coCartesian fibration $N(\calM^\otimes) \to N\Fin_\ast$ and equip it with the class of weak equivalences $W^\otimes$ as described before Definition \ref{deflocsym}. 
This is left derivable in the sense of Definition \ref{defderivable}. To see this we just observe that all the functors in question are a composition of inert and active maps in $N\Fin_\ast$. For the active maps we get multi-Quillen functors, which are left derivable by Example \ref{Quillen-bifunctor}. The inert maps just preserve weak equivalences.  It then follows as in Example \ref{composition} that also the compositions are left derivable and that the derived functors (i.e. absolute right Kan extensions) of these functors compose. This proves that the coCartesian fibration is left derivable.
\end{example}

Let $p: X \to S$ be a left derivable coCartesian fibration. Then we form the Dwyer-Kan localization $i: X \to X[W^{-1}]$ which comes with a functor $X[W^{-1}] \to S$ obtained by the universal property. We can arrange for this functor to be a categorical fibration which we assume from now on. It might be important to recall that if we choose two different equivalent categorical fibrations which are equivalent over $S$ then one is coCartesian if and only if the other is. This makes the following statement, which is the key result about  left derivable coCartesian fibrations, model independent.

\begin{proposition}\label{propositionleftderivable}
Let $p: X \to S$ be a left derivable coCartesian fibration over an $\infty$-category $S$. Then the following are true:
\begin{enumerate}
\item\label{peins} The functor $X[W^{-1}] \to S$ is a coCartesian fibration.
\item\label{pfuenf}  For every object $a \in S$ the morphism $i_a: X_a \to X[W^{-1}]_a$ exhibits $X[W^{-1}]_a$ as the Dwyer-Kan localization of $X_a$ at the weak equivalences $W_a$.
\item\label{pvier}
For every morphism $s: a \to b$ in $S$ the associated diagram 
$$
\xymatrix{
X_a \ar[r]^{s_!}\ar[d]^{i_a} & X_b \ar[d]^{i_b} \\
X[W^{-1}]_a\ar@{->}[r]^{s'_!} \ar@{}[ru]^(.30){}="a"^(.80){}="b" \ar@2{->} "a";"b" & X[W^{-1}]_b
}
$$
exhibits $s_!'$ as the absolute right Kan extension of the composition $X_a \xto{s_!} X_b \to X[W^{-1}]_b'$ along $X_a \to X[W^{-1}]_{a}$. Here $s_!'$ is the functor associated to the coCartesian fibration $X[W^{-1}] \to S$ and the $2$-cell is the natural transformation obtained from the functor $X \to X[W^{-1}]$ over $S$.
\end{enumerate}
\end{proposition}

Let us first assume this result and draw some consequences, in particular we deduce Theorem \ref{DKlocalizationsymmon} from it. 
\begin{corollary}\label{corollaryderivable} Assume that we are given a left derivable coCartesian fibration $p: X \to S$.
\begin{enumerate}
\item\label{pzwei} For every $\infty$-category $K$ with a morphism $K \to S$ the morphism $X \times_S K \to X[W^{-1}] \times_S K \to K$ exhibits $X[W^{-1}] \times_S K$ as the Dwyer-Kan localization of $X \times_S K$ at the edges in $X \times_S K$ that are mapped to $W$. Then:
\item\label{pdrei} If a morphism $s: a \to b$  in $S$ has the property that the functor $s_!: X_a \to X_b$ preserves weak equivalences, then the functor
$X \to X[W^{-1}]$ preserves $s$-coCartesian lifts.
\end{enumerate}
\end{corollary}
\begin{proof}
Since left derivable coCartesian fibrations are clearly stable under pullback, we can deduce  from \ref{propositionleftderivable} that $(X\times_SK)[W^{-1}] \to K$ is coCartesian, the fibers are Dwyer-Kan localizations of the fibers of $X \times_S K \to K$  and the induced functors are given by right Kan extensions of the corresponding functors for the coCartesian fibration $X \times_S K \to K$. Then the canonical map into the pullback $X[W^{-1}] \times_S K \to K$, which is also coCartesian, is a fiberwise equivalence and also compatible with base change. This implies that it is an equivalence, which proves part \eqref{pzwei}. 

To see property \eqref{pdrei} we note that under the assumption that $s_!$ preserves weak equivalences, the (absolute) right Kan extension of 
 $s_!: X_a \to X_b \to X[W^{-1}]_b$ is given by the extension using the universal property of $X[W^{-1}]_a$ as the Dwyer-Kan extension (see Example \ref{preserve} above for an argument). Therefore we get that
the 2-cell in the diagram  
 $$
\xymatrix{
X_a \ar[r]^{s_!}\ar[d] & X_b \ar[d] \\
X[W^{-1}]_a'\ar@{->}[r]^{s'_!} \ar@{}[ru]^(.30){}="a"^(.80){}="b" \ar@2{->} "a";"b" & X[W^{-1}]_b
}
$$
is invertible. But this implies the claim.
\end{proof}
\begin{proof}[Proof of Theorem \ref{DKlocalizationsymmon}]
By Example \ref{example_symmon} a symmetric monoidal model category $N\calM^\otimes \to N\Fin_*$ is a left derivable coCartesian fibration. As a result we deduce from Proposition \ref{propositionleftderivable} that the resulting functor $N\calM[W^{-1}]^\otimes \to N\Fin_*$ is a coCartesian fibration. The fact that it is a symmetric monoidal $\infty$-category easily follows from the fact that a Dwyer-Kan localization of a product category $\calC \times \calC'$ at a product class $W \times W'$ is by the canonical map equivalent to the product $\calC[W^{-1}] \times \calC'[W^{-1}]$. This shows \eqref{teins}. 

For \eqref{tzwei} we observe that the functor $f_!: N\calM^\otimes_{\langle n \rangle} \to N\calM^\otimes_{\langle m \rangle}$ associated to an inert morphism $f: \langle n \rangle \to \langle m \rangle$ is given up to equivalence by product projection since the inert morphisms are generated by the $\rho^i$ of Definition \ref{defsymmon}. Thus by the definition of $W$ the functor $f_!$ preserves weak equivalences. Now we can apply \eqref{pdrei} of Corollary \ref{corollaryderivable} to deduce that $i: N\calM^\otimes \to N\calM[W^{-1}]^\otimes$ preserves coCartesian lifts of inert morphisms. This shows that it is a lax symmetric monoidal functor. 

Assertion \eqref{tdrei} and \eqref{tvier} of  Theorem \ref{DKlocalizationsymmon} immediately follow from \eqref{pfuenf} of Proposition \ref{corollaryderivable} and \eqref{pzwei} of Corollary \ref{propositionleftderivable}. 

For \eqref{tfuenf} we first observe that it is immediate that we get for every symmetric monoidal $\infty$-category $\calC^\otimes \to \calS$ an equivalence
$$
\Fun_{N\Fin_*}(N\calM[W^{-1}]^\otimes, \calC^\otimes) \xto{i^*}  \Fun^W_{N\Fin_*}(N\calM^\otimes, \calC^\otimes) .
$$
In fact we get such an equivalence for every categorical fibration in place of $\calC^\otimes$ but we will not need this extra generality here. Now since the coCartesian lifts of inert morphisms in $N\calM[W^{-1}]^\otimes$ all come from $N\calM$ (as shown above) we get that the above equivalence restricts to an equivalence
$$
\Fun_\lax(\calC[W^{-1}], \calD) \to \Fun_\lax^W(\calC, \calD) 
$$
as desired. 

Finally to prove claim \eqref{tsechs} we make use of the following two observations: first the functor $N\calM_c[W^{-1}]^\otimes \to N\calM[W^{-1}]^\otimes$ over $N\Fin_*$ is fiberwise over $N\Fin_*$ an equivalence since it is fiberwise given by products of the functor $N\calM_c[W^{-1}] \to N\calM[W^{-1}]$ which is an equivalence and secondly it sends coCartesian lifts to coCartesian lifts, i.e.~is a symmetric monoidal functor. For the latter we use the universal property of $N\calM_c[W^{-1}]$ to see that it is symmetric monoidal precisely if the functor $N\calM_c^\otimes \to N\calM^\otimes \to N\calM[W^{-1}]^\otimes$ is symmetric monoidal which is true by construction of the derived functors (see Example \ref{example_symmon}). But this finishes the proof since a functor between coCartesian fibrations which takes  coCartesian lifts 
to coCartesian lifts and is a fiberwise equivalence is already an equivalence on total spaces.
\end{proof}

Now we only have to prove Proposition \ref{propositionleftderivable}. The main idea is to use descent in the base $S$ to do this. Therefore we first give the descent statement for coCartesian fibrations that we will use.

\begin{lemma}\label{descent}$ $
\begin{enumerate}
\item
Let $X \to S$ be a coCartesian fibration and let $S = \colim S_i$, i.e. $S$ is the colimit in $\Cat_\infty$ of a diagram $F: I \to \Cat_\infty$. Then $X$ is the colimit of $X \times_S S_i$. 
\item
Assume conversely that we are given a diagram $I \to (\Cat_\infty)^{\Delta^1}$ of coCartesian fibrations $X_i \to S_i$ so that for each morphism $i \to j$ in $I$ the square 
$$
\xymatrix{
X_i \ar[d]\ar[r] & X_j \ar[d] \\
S_i \ar[r] & S_j
}
$$
is a pullback. Choose  $X = \colim X_i \to S= \colim S_i $ to be a categorical fibration. Then $X \to S$ is coCartesian and the canonical maps $X_i \to X\times_S S_i$ are equivalences.
\end{enumerate}
\end{lemma}

\begin{proof}
 We first claim that for a coCartesian fibration $X \to S$  the base change functor
$$
(\Cat_\infty)_{/S} \to (\Cat_\infty)_{/X} 
$$ 
preserves colimits.\footnote{This is more generally true for flat categorical fibrations in the sense of \cite[Definition B.3.8]{HA} but not for any map between $\infty$-categories since $\Cat_\infty$ is not locally cartesian closed.} To see this we use that this functor can be modelled by the strict pullback functor from the category of simplicial sets over $S$ to the category of simplicial sets over $X$ (both equipped with the Joyal model structure). This functor preserves cofibrations and colimits. Since it also preserves weak equivalences by \cite[3.3.1.3]{HTT} it is a left Quillen functor which implies that it preserve all homotopy colimits and therefore the claim.

Now if we are in the situation of (1) then we write the identity functor $S \to S$ as a colimit of the functors $S_i \to S$ in $(\Cat_\infty)_{/S}$. By pullback along $X \to S$ we get the equivalence $\colim (S_i \times_S X) \simeq X$ as desired.  

For (2) we first use that the $\infty$-category of coCartesian fibrations over $S$ is equivalent to the functor $\infty$-category $\Fun(S, \Cat_\infty)$. This equivalence is natural in $S$ in the sense that pullback of coCartesian fibrations corresponds to pullback of functors, see \cite[Proposition 3.2.1.4]{HTT}. Using this translation we see that we have a family of functors $F_i: S_i \to \Cat_\infty$ corresponding to $X_i \to S_i$ such that the pullback of $F_j$ along $S_i \to S_j$ is given by $F_j$. Thus we can glue those functors together to a functor $F: S \to \Cat_\infty$ which corresponds to a coCartesian fibration $X' \to S$. By construction we know that the pullback $X' \times_S S_i \to S_i$ is equivalent to $X_i \to S_i$. Thus we can deduce by (1) that $X'$ is equivalent to $X = \colim X_i$. This shows part (2). 
\end{proof}

\begin{lemma}\label{generation}
Let $\calC \subseteq \Cat_\infty$ be the smallest full subcategory of the $\infty$-category $\Cat_\infty$ that contains $\Delta^0$, $\Delta^1$ and that is closed under all small colimits. Then the inclusion $\calC \subseteq \Cat_\infty$ is an equivalence of $\infty$-categories.
\end{lemma}
\begin{proof}
This is well known, but for convenience we give a quick proof. First we note that the $n$-simplex $\Delta^n$ is in $\Cat_\infty$ equivalent to the iterated pushout $\Delta^1 \cup_{\Delta^1} \cup \ldots \cup_{\Delta^0} \Delta^1$. In particular all simplicies $\Delta^n$ lie in $\calC$. Now we show that $\calC$ contains all $\infty$-categories represented by finite simplicial sets $S$. To see this we use induction on the dimension of $S$. But then it suffices by a further induction on the top dimensional simplices to show that if $S$ in $\calC$ then also $S \cup_{\partial \Delta^n} \Delta^n \in \calC$. But this follows from the induction hypothesis since $\partial \Delta^n$ is of dimension $n-1$ and $\Delta^n$ is in $\calC$ as remarked above. Now finally every arbitrary simplicial set is a filtered colimit of finite simplicial sets, which finishes the proof. 
\end{proof}

\begin{proof}[Proof of Proposition \ref{propositionleftderivable}]
The proof will be given in two steps: the first part is the proof of the statement of Proposition \ref{propositionleftderivable} for left derivable coCartesian fibrations $X \to S$ with $S = \Delta^0$ and $S = \Delta^1$. In the second part we prove that if Proposition \ref{propositionleftderivable} is true for all left derivable coCartesian fibrations over some family $S_i$ of $\infty$-categories which are the vertices of a diagram $I \to \Cat_\infty$, then the Proposition is also true for all left derivable coCartesian fibrations over  the colimit $S = \colim_I S_i$ in $\Cat_\infty$. Together these two assertions then imply the claim, using Lemma \ref{generation}.

In the case $S = \Delta^0$ there is nothing to show. For the case $S = \Delta^1$ we use that the $\infty$-category of coCartesian fibrations $X \to \Delta^1$ is equivalent to the $\infty$-category of diagrams $\phi: \Delta^1 \to \Cat_\infty$, i.e. functors $F: X_0 \to X_1$ for $\infty$-categories $X_0$ and $X_1$. The inverse of this equivalence is given by the mapping simplex construction $M(\phi) \to \Delta^1$ see \cite[Definition 3.2.2.6]{HTT}. Explicitly we have that $M(\phi)$ is given by the pushout$$
\xymatrix{
X_0 \ar[r]^-{\id \times \{1\}} \ar[d]_F & X_0 \times \Delta^1 \ar[d] \\
X_1\ar[r] & M(\phi)
}
$$
of simplicial sets. Since the upper vertical map is a cofibration this is also a pushout in the $\infty$-category $\Cat_\infty$.\footnote{This pushout is in fact a special case of a more general statement, namely that the total space of every coCartesian fibration is the oplax colimit of its classifying functor, see \cite{GHN}.}

Now to prove the Proposition for a given left derivable coCartesian fibration $X \to \Delta^1$, we present it as $M(\phi) \to \Delta^1$ for the associated diagram $X_0 \to X_1$. Since $X \to \Delta^1$ is left derivable we know that we can find an absolute right Kan extension diagram
$$
\xymatrix{
X_0 \ar[r]^{F}\ar[d] & X_1 \ar[d] \\
X_0[W_0^{-1}]\ar@{->}[r]^{LF} \ar@{}[ru]^(.30){}="a"^(.80){}="b" \ar@2{->} "a";"b" & X_1[W_1^{-1}].
}
$$
In particular this means that we have a natural transformation $h: X_0 \times \Delta^1 \to X_1[W_1^{-1}]$. When we denote the lower line of the diagram by $\phi'$ then we have an associated mapping simplex $M(\phi')$. Moreover we claim that there is a canonical map $M(\phi) \to M(\phi')$ over $S$ induced by the diagram above. More precisely we act on the $X_0 \times \Delta^1$ factor of $M(\phi)$ with the natural transformation $h$ and on the $X_1$ factor with the functor $X_1 \to X_1[W_1^{-1}]$ and the two maps fit together on  $X_0 \times \{1\}$. We claim that this functor $p: M(\phi) \to M(\phi')$ exhibits $M(\phi')$ as the Dwyer-Kan localization at the weak equivalences $W$.  To this end we first note that $p$ sends $W$ by construction to equivalences in $M(\phi')$. Then it suffices to show that it satisfies the universal property of the Dwyer-Kan localization. Thus consider an arbitrary $\infty$-category $\calD$ and consider the $\infty$-category 
$
\Fun( M(\phi'), \calD)
$. Decomposing this functor category into a pullback 
$$
\xymatrix{
\Fun( M(\phi'), \calD) \ar[r] \ar[d]& \Fun( X_0[W_0^{-1}] \times \Delta^1, \calD)\ar[d] \\
\Fun(X_1[W_1^{-1}],\calD) \ar[r] &  \Fun( X_0[W_0^{-1}], \calD). 
}
$$
we see that it is (on the nose) equivalent to the $\infty$-category of diagrams
$$
\xymatrix{
X_0[W_0^{-1}] \ar[rd]^-{G_0}\ar[d]_{LF} & \\
X_1[W_1^{-1}] \ar[r]_-{G_1}\ar@{<=}[ru]-<26pt> & \calD 
}
$$
with a not-necessarily invertible 2-cell (i.e.natural transformation) $G_0 \to G_1 \circ LF$. The vertical morphism $LF$ is an absolute right Kan extension (along the morphism $X_0 \to X_0[W_0^{-1}]$ that is not in the picture) which implies that $G_1 \circ LF$ is also a right Kan extension along the same morphism. Thus by the universal property of the right Kan extension we get that the space of natural transformations $G_0 \to G_1 \circ LF$  as in the diagram is equivalent to the space of natural transformations in the induced diagram
$$
\xymatrix{
X_0 \ar[r]\ar[d]_{F} & X_0[W_0^{-1}] \ar[rd]^{G_0} & \ \\
X_1 \ar[r]\ar@{}[ru]^(.30){}="a"^(.80){}="b" \ar@2{<-} "a";"b" & X_1[W_1^{-1}] \ar[r]_-{G_1} & \calD
}.
$$
Together this shows that the $\infty$-category of functors $\Fun( M(\phi'), \calD)$ is equivalent to the $\infty$-category that consists of a functor $G_0': X_0 \to \calD$ which sends $W$ to equivalences, a functor $G_1': X_1 \to \calD$ which sends $W$ to equivalences and a transformation $G_1'F \to G_0'$. But this is by exactly the category $\Fun^W(M(\phi), \calD)$ which follows as above using the colimit description of $M(\phi)$. 
Together this shows that $M(\phi')$ is a Dwyer-Kan localization. Therefore we conclude that the Dwyer-Kan localization is a coCartesian fibration with the desired properties and thus that  Proposition \ref{propositionleftderivable} for $S = \Delta^1$ is true.

As a next step we assume that Proposition \ref{propositionleftderivable} is true for $\infty$-categories $S_i$ and we are given a colimit diagram $I^\rhd \to \Cat_\infty$. We denote the cone point by $S$. If we have a left derivable coCartesian fibration $X \to S$ we deduce from Lemma \ref{descent} that we have an equivalence $X \simeq \colim_I (X \times_S S_i)$. By assumption we know that all the Dwyer-Kan localizations $(X \times_S S_i) [W_i^{-1}] \to S_i$ are coCartesian fibrations where $W_i$ denotes the pulled back class of weak equivalences.  Moreover we know that the fibers are given by the Dwyer-Kan localizations and that the induced morphisms are given by absolute right Kan extension. We claim that for each morphism $i \to j$ in $I$ the induced diagram
$$
\xymatrix{
(X \times_S S_i) [W_i^{-1}] \ar[r]\ar[d] & (X \times_S S_j) [W_j^{-1}] \ar[d]\\
S_i \ar[r] & S_j
}
$$
is a pullback. This follows as in the proof of Corollary \ref{corollaryderivable}: first we note that the pullback of $(X \times_S S_j) [W_j^{-1}]$ along $S_i \to S_j$ is also a coCartesian fibration, and then use that the map from $(X \times_S S_i) [W_i^{-1}] $ into this pullback is a fiberwise equivalence (since the fibers are by assumption Dwyer-Kan equivalences) and preserves coCartesian lifts (since these are given by the right Kan extension). 

We deduce from Lemma \ref{descent} that the colimit $\colim(X \times_S S_i) [W_i^{-1}] \to \colim S_i = S$ is a coCartesian fibration. Now we claim that the colimit $X  \simeq \colim(X \times_S S_i) \to \colim(X \times_S S_i) [W_i^{-1}]$ exhibits the target as the Dwyer-Kan localization of the source at the weak equivalences $W$. To see this we verify the universal property. Therefore let $\calD$ be an arbitraty $\infty$-category, then we get
\begin{align*}
\Fun(\colim(X \times_S S_i) [W_i^{-1}] , \calD) &\simeq \lim \Fun( (X \times_S S_i) [W_i^{-1}] ,\calD)  \\
& \simeq  \lim \Fun^{W_i}(X \times_S S_i),\calD) \simeq \Fun^W(X,\calD).
\end{align*}
In a more abstract language, we have used that the Dwyer-Kan localization functor from relative $\infty$-categories to $\infty$-categories preserves colimits. Together this shows that $X[W^{-1}] \to S$ is a coCartesian fibration.

Moreover we  deduce also from Lemma  \ref{descent}  that the pullback of $X[W^{-1}]$ along $S_i\to S$ is given by $(X \times_S S_i)[W_i^{-1}]$. By assumption we know for the latter that the fibers are given by Dwyer-Kan localizations. Since every object of the $\infty$-category $S = \colim S_i$ lies (up to equivalence) in some $S_i$ this implies that also the fibers of $X[W^{-1}] \to S$ are given by Dwyer-Kan localizations. Moreover the same reasoning immediately shows that for every morphism $s: a \to b$ in $S$ that factors through one of the $S_i$'s the square  
$$
\xymatrix{
X_a \ar[r]^{s_!}\ar[d] & X_b \ar[d] \\
X[W^{-1}]_a\ar@{->}[r]^{s'_!} \ar@{}[ru]^(.30){}="a"^(.80){}="b" \ar@2{->} "a";"b" & X[W^{-1}]_b
}
$$
exhibits an absolute right Kan extension. But then the same also follows for compositions of such morphisms by assumption (3) in  Definition \ref{propositionleftderivable} of a left derivable coCartesian fibration. Since every morphism in $S$ is a finite composition of morphisms that factor through one of the $S_i$'s this finishes the proof.
\end{proof}

\setcounter{thm}{0}
\setcounter{chapter}{19}
\chapter{Cyclic objects}\label{app:cyclic}

\renewcommand{\thesection}{B}

In this Appendix, we briefly recall conventions regarding Connes' cyclic category $\Lambda$. In fact, we will need several related combinatorial categories $\Delta$, $\Lambda$, $\Lambda_p$ for $1\leq p\leq \infty$, and the associative operad $\Ass^\otimes$, with its natural map to the commutative operad $\Fin_\ast$. Moreover, we recall the geometric realization of cyclic spaces, both in the topological case, and in the $\infty$-categorical case.

First, $\Fin_\ast$ is the category of pointed finite sets. For all $n\geq 0$, we write $\langle n\rangle\in \Fin_\ast$ for the finite set $\{0,1,\ldots,n\}$ pointed at $0$.
Next, there is the associative operad $\Ass^\otimes$. Its objects are given by pointed finite sets, and for $n\geq 0$, we write $\langle n\rangle_\Ass\in \Ass^\otimes$ for the finite pointed set $\langle n\rangle$ regarded as an object of $\Ass^\otimes$. The set of morphisms
\[
\Hom_{\Ass^\otimes}(\langle n\rangle_\Ass,\langle m\rangle_\Ass)
\]
is given by the set of all maps $f: \langle n\rangle\to \langle m\rangle$ of finite pointed sets together with a linear ordering on the inverse image $f^{-1}(i)\subseteq \langle n\rangle$ for all $i\in \{1,\ldots,m\}\subseteq \langle m\rangle$. To define composition, note that if $f: S\to T$ is map of finite sets with a linear ordering $t_1<\ldots<t_m$ on $T$ and on the individual preimages $f^{-1}(t_1),\ldots,f^{-1}(t_m)$, then one gets a natural linear ordering on $S$, ordering preimages of $t_1$ before preimages of $t_2$, etc., before preimages of $t_m$. By definition, there is a natural functor $\Ass^\otimes\to \Fin_\ast$ forgetting the linear ordering.

Note that if $\calC$ is a symmetric monoidal $\infty$-category with total space $\calC^\otimes\to N(\Fin_\ast)$, then associative algebras are given by certain functors $N(\Ass^\otimes)\to \calC^\otimes$ over $N(\Fin_\ast)$. Concretely, this amounts to an object $A\in \calC$ together with maps $\bigotimes_{j\in f^{-1}(i)} A\to A$ whenever $\langle n\rangle\to \langle m\rangle$ is a map in $\Ass^\otimes$ and $i\in \{1,\ldots,m\}\subseteq \langle m\rangle$. This explains the requirement of the linear ordering on $f^{-1}(i)$, as the multiplication morphisms in an associative algebras depend on the order of the factors.

As usual, $\Delta$ denotes the category of totally ordered nonempty finite sets, and for $n\geq 0$, we let $[n] = \{0,1,\ldots,n\}\in \Delta$. There is a natural functor $\Delta^\op\to \Fin_\ast$, which sends a totally ordered, nonempty finite set $S$ to the set of cuts, i.e.~the set of disjoint decomposition $S=S_1\sqcup S_2$ with the property that for all $s_1\in S_1$, $s_2\in S_2$, we have $s_1<s_2$, and we identify the disjoint decomposition $S\sqcup \emptyset$ and $\emptyset\sqcup S$. The set of cuts is pointed at the trivial decomposition $S=S\sqcup \emptyset$. Note that in particular, up to isomorphism, $[n]\in \Delta^\op$ maps to $\langle n\rangle\in \Fin_\ast$.

In fact, there is a natural functor $\Delta^\op\to \Ass^\otimes$ over $\Fin_\ast$. Recall that the functor $\Delta^\op\to \Fin_\ast$ was given by $S\mapsto \Cut(S)$ \footnote{Note that as a s pointed simplicial set the functor $\Delta^\op\to \Fin_\ast$ is isomorphic to $S^1 = \Delta^1 / \partial \Delta^1$}. If $f: S\to T$ is a map of totally ordered nonempty finite sets, and $S=S_1\sqcup S_2$ is a cut of $S$ with $S_1$ and $S_2$ nonempty, then the set of cuts of $T$ pulling back to $S_1\sqcup S_2$ is naturally a totally ordered set, as it is a subset of the  of cuts of $T$ into two nonempty sets which is totally ordered as follows: a cut $T_1 \sqcup T_2 = T$ is less are equal to a cut $T_1' \sqcup T_2' = T$ if $T_1 \subseteq T_1'$ and consequently $T_2 \supseteq T_2'$. This gives the required functor $\Delta^\op\to \Ass^\otimes$.

Now we construct Connes' cyclic category $\Lambda$. We start with the definition of the paracyclic category $\Lambda_\infty$. It is the full subcategory $\Lambda_\infty \subseteq \Z\PO$  consisting of all objects isomorphic to $\frac{1}{n}\Z$ for $n \geq 1$.
 Here $\PO$ is the category or partially ordered sets and non-decreasing maps, $\Z\PO$ is the category of objects in $\PO$ equipped with a $\Z$-action 
and we consider $\frac 1 n \Z$ as an object with its natural ordering and the $\Z$-action given by addition.
 We will use the notation
$$
[n]_{\Lambda_\infty} = \frac{1}{n}\Z\in \Lambda_\infty.
$$
There is an action of $B\Z$ on $\Lambda_\infty$ defined as restriction of the action of $\B\Z$ on the functor category $\Fun(\B\Z,\PO)$. Concretely, this means that there is an action of $\mathbb Z$ on the morphism spaces of $\Lambda_\infty$, with the generator $\sigma$ of $\mathbb Z$ sending a morphism $f: \frac{1}{n}\mathbb Z\to \frac{1}{m}\mathbb Z$ to $\sigma(f): \frac 1 n \mathbb Z\to \frac 1 m\mathbb Z$ given by $\sigma(f) = f+1$. 
Now, for any integer $p\geq 1$, we define a category $\Lambda_p$ as the category with the same objects as $\Lambda_\infty$, but with the morphism spaces divided by the action of $\sigma^p$. Equivalently,
\[
\Lambda_p = \Lambda_\infty / B(p\Z)\ .
\]
For $p=1$, we abbreviate $\Lambda=\Lambda_1$. Thus, the objects of $\Lambda_p$ are, up to isomorphism, still given by the integers $\frac 1 n \Z$ for $n \geq 1$, and we denote the corresponding object by $[n]_{\Lambda_p}\in \Lambda_p$. Note that by definition, there is a remaining $BC_p = B\Z/B(p\Z)$-action on the category $\Lambda_p$, and
\[
\Lambda = \Lambda_p/BC_p\ .
\]
We note that the category $\Lambda_\infty$ is self-dual.
 Namely, one can send a map $f: S \to T$ to the function $f^\circ: T\to S$ given by
\[
f^\circ(x)=\min\{y\mid f(y)\geq x\}\ .
\]
Then $f^\circ$ is non-decreasing, $\Z$-equivariant, i.e. $f^\circ(x + 1) = f^\circ(x) +1$, and for all $x \in T, y\in S$, one has $f(y)\geq x$ if and only if $y\geq f^\circ(x)$. In particular, one can recover $f$ from $f^\circ$ by
\[
f(y)=\max\{x\mid f^\circ(x)\leq y\}\ ;
\]
we warn the reader that this is not the same as $(f^\circ)^\circ$. Another description of the map $f^\circ$ is given by identifiying an object $S \in \Lambda_\infty$ with the set $S^\circ$ of nontrivial cuts of $S$. For this identification an element $s \in S$ is sent to the cut $S = S_{< s} \sqcup S_{\geq  s}$. Since the set of non-trivial cuts is naturally contravariant this can also be used to desribe $f^\circ: T \to S$ for  a given map $f: S \to T$. 
The self-duality $\Lambda_\infty\simeq \Lambda_\infty^\op$ is $B\Z$-equivariant, and thus descends to a self-duality of $\Lambda_p$ for all $1\leq p\leq \infty$, in particular of $\Lambda$.

We need to relate $\Lambda$ with the other combinatorial categories. For this, note that there is a natural functor
\[
V: \Lambda\to \Fin
\]
sending $T \in \Lambda_\infty$ to $T/\Z$. Concretely this functor sends $[n]_{\Lambda_\infty}$ to the finite set 
$\left\{0,\frac 1 n,\frac 2n,\ldots,\frac {n-1} n \right\}$.
In fact, one has
$
V = \Hom_\Lambda([1]_\Lambda,-):\Lambda\to \Fin
$
via the transformation taking a map $f: [1]_\Lambda\to [n]_\Lambda$ to the image of $V(f): V([1]_\Lambda)\to V([n]_\Lambda)$.

\begin{proposition}\label{prop:lambdaass} The functor $V: \Lambda\to \Fin$ refines to a functor
\[
V: \Lambda\to \Ass^\otimes_\act = \Ass^\otimes\times_{\Fin_\ast} \Fin
\]
still denoted $V$.
\end{proposition}
We note that we will mostly use the composite
\[
V^\circ: \Lambda^\op\simeq \Lambda\to \Ass^\otimes_\act
\]
with the self-duality of $\Lambda$.

\begin{proof} In order to lift the morphism $V: \Lambda\to \Fin$ through $\Ass^\otimes_\act$ we have to supply orderings on the preimages of points of $S/\Z \to T/\Z$ for an equivalence class of order preserving, $\Z$-equivariant morphisms $f: S \to T$. But if $i\in T/\mathbb Z$ with lift $\tilde{i}\in T$ then $f^{-1}(\tilde{i})\subseteq S$ is totally ordered (as a subset of $S$), and canonically independent of the choices of $\tilde{i}$ and the representative $f$, because of the $\Z$-equivariance..\end{proof}

In particular, we get a natural functor
\[
\Lambda_{/[1]_\Lambda}\to \Delta
\]
sending a map $f: [n]_\Lambda\to [1]_\Lambda$ to $V([n]_\Lambda)$ equipped with the total ordering induced by $f$ using Proposition \ref{prop:lambdaass}.
We claim that this is an equivalence. To see this, we construct an inverse functor. In fact, there is a natural functor $\Delta\to \Lambda_\infty$ sending $[n-1]=\{0,1,\ldots,n-1\}$ to $\frac 1 n \Z \cong \Z \times  [n-1] $ with the lexicographic ordering and the action given by addition in the first factor. This clearly extends to all finite linearly ordered sets with the same formula, i.e. we send $S$ to $\Z  \times S$ with lexicographic ordering. With this description the functoriality is also clear. As $\Delta$ has a final object $[0]$, this gives a map $\Delta\to \Lambda_{\infty/[1]_{\Lambda_\infty}}$,
which induces the desired functor $\Delta\to \Lambda_{/[1]_\Lambda}$. One checks easily that the two functors are inverse, resulting in the following corollary.

\begin{corollary}\label{cor:lambdavsdelta} The functors described above given an equivalence of categories $\Delta\simeq \Lambda_{/[1]_\Lambda}$.$\hfill \Box$
\end{corollary}

In particular, we get a natural functor $\Delta\to \Lambda$ sending $[n]$ to $[n+1]_\Lambda$, or equivalently a functor $j: \Delta^\op\to \Lambda^\op$. Note that by the considerations above, this actually factors over a functor $j_\infty: \Delta^\op\to \Lambda_\infty^\op$. The following result is critical.

\begin{thm}\label{thm:jinftycofinal} The functor $j_\infty: N\Delta^\op\to N\Lambda_\infty^\op$ of $\infty$-categories is cofinal.
\end{thm}

\begin{proof} By Quillen's Theorem A, \cite[Theorem 4.1.3.1]{HTT}, this is equivalent to the statement that for all $T \in \Lambda_\infty$, the simplicial set given by the nerve of the category $\calC = \Delta_{/T} = \Delta \times_{\Lambda_\infty} {\Lambda_\infty}_{/T}$ is weakly contractible, that means contractible in the Quillen model structure. 

We first establish a criterion to prove weak contractibility in terms of ``coverings''. Assume that for a given $\infty$-category $\calC$ there is a family of full, saturated\footnote{Saturated means that every object, morphism etc.  which is equivalent to one in the subcategory also lies in the subcategory. Thus this really means that we have a full simplicial subset in the sense that it is induced from a subset of the vertices and a higher simplex is in the subsimplicial set precisely if all its vertices are.} subcategories $U_i \subseteq \calC$ for $i \in I$ with the property that all finite intersections of the $U_i$'s are weakly contractible and that we have an equality of simplicial sets $\calC = \bigcup_{i \in I} U_i$ (this is not just a condition on the objects, but also on the higher simplices). Then we claim that $\calC$ is weakly contractible as well.  

To prove this criterion we first observe that $\calC$ is the filtered colimit of the finite partial unions, thus we can assume that $I$ is finite. Inductively we assume we know the claim for all $\infty$-categories with a covering consisting of $n$ contractible, full subcategories. We will show that $\calC = U_0 \cup \ldots \cup U_n$ is contractible. We consider the following pushout diagram of simplicial sets
$$
\xymatrix{
(U_0 \cup \ldots \cup U_{n-1}) \cap U_n \ar[r]\ar[d] & U_n \ar[d] \\
U_0 \cup \ldots \cup U_{n-1} \ar[r] & U_0 \cup \ldots \cup U_n
} 
$$
which is also a pushout in the Kan-Quillen model structure since all the maps involved are cofibrations. By assumption the three upper left corners are weakly contractible. It follows that $U_1 \cup \ldots \cup U_n$ is also weakly contractible which shows the claim.

We now apply this criterion to the $\infty$-category $\calC = \Delta_{/T} = \Delta \times_{\Lambda_\infty} {\Lambda_\infty}_{/T}$ with the ``covering'' given by
$$
U_t := \calC[t,t+1) \qquad t \in T
$$
where $\calC[a,b)$ for $a,b \in T$ is the full subcategory of $\calC$ which consists of pairs $(S,\varphi)$ where $S \in \Delta$ and $\varphi: j_\infty S \to T$ such that $\varphi(S) \subseteq [a,b) \subseteq T$. Here $S$ is considered as the subset $\{0 \} \times S \subseteq \Z \times S = j_\infty(S)$, which is a fundamental region for the $\Z$-action on $j_\infty(S)$. 

Finite intersections of $U_t$'s are always also of the form $\calC[a,b)$ where $b \leq a +1$. Thus in order to verify the assumptions of our criterion we have to show that all those $\calC[a,b)$ are weakly contractible. For every choice of such $a,b \in T$ we can choose an isomorphism $T \cong \frac{1} n \Z$ such that $a$ and $b$ correspond to $0$ and $\frac{k} n$ with $k \leq n$, so that we have to show that $\calC[0, \frac{k}{n}) \subseteq \Delta_{/\frac 1 n \Z}$ is weakly contractible. But this category is equivalent to $\Delta_{/[k-1]}$ which has a terminal object. 
\end{proof}

\begin{corollary}\label{cor:homotopytypeslambda} The $\infty$-category $N\Lambda_\infty^\op$ is sifted in the sense of \cite[Definition 5.5.8.1]{HTT}. Moreover, denoting by $|X|\in \calS$ the homotopy type of a simplicial set $X$, $|N\Lambda_\infty|$ is contractible, and
\[
|N\Lambda| = |N\Lambda_\infty|/B\Z = BB\Z = K(\Z,2) = B\T \simeq \mathbb C P^\infty\ .
\]
\end{corollary}

\begin{proof} The first part follows from Theorem~\ref{thm:jinftycofinal} and \cite[Lemma 5.5.8.4]{HTT}. Thus, $|\Lambda_\infty|$ is contractible by \cite[Lemma 5.5.8.7]{HTT}; we could have deduced this also more directly from \cite[Proposition 4.1.1.3 (3)]{HTT}. Now the final assertion follows from the definition of $\Lambda$ as a quotient and the observation that the $B\Z$-action on $\Lambda_\infty$ is free, thus the orbits are also the homotopy orbits after taking nerves.\end{proof}

Now we need to discuss geometric realizations. First, we do the case of $\infty$-categories.

\begin{proposition}\label{prop:geomreal} For any $\infty$-category $\calC$ admitting geometric realizations of simplicial objects, there is a natural functor
\[
\Fun(N(\Lambda^\op),\calC)\to \calC^{B\T}
\]
from cyclic objects in $\calC$ to $\T$-equivariant objects in $\calC$. The underlying object of $\calC$ is given by
\[
(F: N(\Lambda^\op)\to \calC)\mapsto \colim_{N(\Delta^\op)} j^\ast F\ .
\]
\end{proposition}

\begin{proof} The functor can be constructed as the composite
\[
\Fun(N(\Lambda^\op),\calC)\to \Fun^{B\Z}(N(\Lambda_\infty^\op),\calC)\to \Fun^{B\Z}(\pt,\calC) = \calC^{BB\Z} = \calC^{B\T}\ .
\]
Here, in the middle we use the $B\Z=\T$-equivariant functor
\[
\colim_{N(\Lambda_\infty^\op)} : \Fun(N(\Lambda_\infty^\op),\calC)\to \calC\ .
\]
By Theorem~\ref{thm:jinftycofinal}, the underlying functor to $\calC$ is given by $\colim_{N(\Delta^\op)} j^\ast$, which also shows existence of the colimit.
\end{proof}

Now we  discuss the geometric realization functor for cyclic topological spaces. Our strategy is to deduce these from a realization of paracyclic topological spaces in parallel to the $\infty$-categorical case. But first we recall the geometric realization of simplicial spaces.

\begin{construction}\label{cons:geomreal}
Let $\Top$ denote the category of compactly generated weak Hausdorff spaces. There is a natural functor
\[
|-|\colon \Fun(\Delta^\op,\Top)\to \Top\ 
\]
constructed as follows.
We have the topological $n$-simplex 
$$
|\Delta^n| = \Big\{ (x_0,\ldots,x_n) \in [0,1]^{n+1} \mid \sum_{i = 0}^n x_i = 1 \Big\} \cong \Big\{ (y_1,\ldots,y_n) \in [0,1]^n \mid y_1 \leq \ldots \leq y_n \Big\}
$$
where $y_i = x_0 + \ldots + x_{i-1}$. 
Together all the simplices $\Delta^n$ can be given the structure of a cosimplicial space, i.e. a functor $\Delta \to \Top$. We explain in a second how this is done. Using this cosimplicial space, the geometric realization of a simplicial space $X_\bullet$ is defined as the coend
$$
| X_\bullet | = \int_{S \in \Delta} X_S \times |\Delta^S| = \Big(\coprod_{n \in \N} X_n \times |\Delta^{n}|\Big) / \sim\ .
$$
Recall that a priori this colimit is taken in compactly generated weak Hausdorff spaces, but it turns out that the topological space is also given by the colimit of the same diagram taken in topological spaces, cf.~\cite[Appendix A, Proposition A.35]{Global}. Thus we do not need to take any weak Hausdorffification and the points of this space are what one thinks.

Now to the cosimplicial structure on $|\Delta^\bullet|$: this is usually constructed in barycentric coordinates by sending a map of ordered sets $s: [n] \to [m]$ to the unique affine linear map $|\Delta^n| \to |\Delta^m|$ given on vertices by $s$. We  present a slightly different way to do this using increasing coordinates, which has also been used by Milnor and Joyal.\footnote{We thank Peter Teichner for suggesting that this description extends to cyclic spaces and helpful discussions of this point.} First we say that a linearly ordered set is an interval if it has a minimal and a maximal element which are distinct. An example is given by the unit interval $[0,1] \subseteq \R$ with its natural ordering. A morphism of intervals is a non-decreasing map that sends the minimal element to the minimal element and the maximal element to the maximal element. There is a functor $-^\circ:  \Delta \to \mathrm{Int}$, where $\mathrm{Int}$ is the category of intervals, given as
$$
S^\circ = \Hom_\Delta(S,[1]) = \{ S = S_0 \sqcup S_1 \mid s \leq t \text{ for } s \in S_0, t \in S_1 \}
$$
 where $S^\circ$ has the natural ordering as a set of maps into a poset. It is not hard to see that $S^\circ$ is indeed an interval. In fact the functor $-^\circ$ exhibits the category of finite intervals as the opposite category of $\Delta$. 
Now we define for  $S \in \Delta$ the geometric realization
$$
\left|\Delta^S\right| := \Hom_{\mathrm{Int}}(S^\circ, [0,1])
$$
with the topology induced from $[0,1]$. For $S = [n]$ we have $S^\circ \cong [n+1]$ and  $|\Delta^S |$ recovers the description of $|\Delta^n|$ in increasing coordinates as above. 
The point is that this description of $|\Delta^{n}|$ makes the functoriality in order-preserving maps $[n] \to [m]$ clear.  This description also shows the (well known) fact that the group of continuous, orientation preserving homeomorphisms of the intervall $[0,1]$ acts canonically on the geometric realization of every simplicial set or space. 
\end{construction}

We will need the well-known fact that geometric realization models the homotopy colimit. Recall that a  simplicial topological space $X_\bullet$ is called proper, if for each $n$ the inclusion
$$
\bigcup_{i=0}^{n-1} s_i(X_{n-1}) \hookrightarrow X_n
$$
is a Hurewicz cofibration\footnote{Usually closed cofibration is required, but since we are working within weak Hausdorff compactly generated spaces this is automatic.}
where $s_i$ are the degeneracy maps $X_{n-1} \to X_n$. This is for example the case if it is ``good'' in the sense that all degeneracy maps are h-cofibrations, see \cite[Corollary 2.4]{Lewis}. We denote by $\Fun(\Delta^\op,\Top)_{\mathrm{prop}}$ the full subcategory of $\Fun(\Delta^\op,\Top)$ consisting of the proper simplicial spaces. Every space can be functorially replaced by a weakly equivalent good one, so that this inclusion induces an equivalences on the associated $\infty$-categories. 
\begin{lemma}\label{geometrichocolim} The diagram of functors
\[\xymatrix{
N\Fun(\Delta^\op,\Top)_{\mathrm{prop}}\ar[r]^-{|-|} \ar[d] & N\Top\ar[d]\\
\Fun(N(\Delta^\op),\calS)\ar[r]^-\colim & \calS
}\]
commutes (up to a natural equivalence) where the vertical maps are the Dwyer-Kan localization maps. 
\end{lemma}
\begin{proof}
The key fact is that the geometric realization functor, when restricted to proper simplicial spaces, preserves weak equivalences. This is a well known fact that we review in Proposition \ref{proper} of the next section. 

Now we use the Reedy model structure  on $\Fun(\Delta^\op, \Top)$: this is a model structure on $\Fun(\Delta^\op,\Top)$ for which the weak equivalences are levelwise and every cofibrant object is proper (but it has to satisfy more conditions since proper is only Reedy cofibrant for the Strom model structure and we are working with the standard Quillen model structure). For the Reedy model structure geometric realization becomes a left Quillen functor, see \cite{reedy1974homotopy} or \cite[Chapter 15]{Hirsch} for proofs and precise statements. It follows that geometric realization can be left derived to induce a left adjoint functor of $\infty$-categories. Since for proper simplicial spaces weak equivalences are already preserved it follows that the left derived functor is given by evaluation on a proper simplicial space and one does not have to replace cofibrantly in the Reedy model structure. The right adjoint is given by the right derived functor of the functor which sends a topological space $X$ to the simplicial space $X^{|\Delta^\bullet|}$. This functor clearly preserves all weak equivalences and is equivalent to the diagonal functor. Putting everything together this shows that the induced functor from the geometric realization as restricted to proper simplicial spaces is left adjoint to the diagonal functor, thus it is equivalent to the colimit. 
\end{proof}

\begin{remark}
One can get rid of the assumption that the simplicial space has to be proper by replacing the functor $|-|$ with the fat geometric realization, that is the coend over the category $\Delta^{\mathrm{inj}} \subseteq \Delta$ which has the same objects as $\Delta$ but only injective, order preserving maps as morphisms. This amounts to taking the geometric realization without quotenting by the degeneracy maps. The disadvantage is that this fat realization is much bigger then the ordinary one and does not preserve products anymore, for example the terminal object in simplicial spaces is sent to the infinite dimensional simplex $\Delta^\infty$.
\end{remark}

Now we discuss the geometric realization of paracyclic spaces. This will eventually also lead to the notion of geometric realization of cyclic spaces.  

\begin{construction}\label{concyclicrealization}
We  construct a $\B\Z$-equivariant functor
\[
| - |\colon \Fun(\Lambda_\infty^\op,\Top)\to \R\Top\ ,
\]
where $\R\Top$ denotes the 1-category of topological spaces equipped with a continuous action of the additive group of the reals. The action of $\B\Z$ on the source is induced from the action on $\Lambda_\infty^\op$ and on the target by the action on the reals (or rather on $\B\R$ as $\R\Top$ is the category of continuous functors from $\B\R$ to $\Top$). 
The functor $|-|$ will be constructed similarly to the geometric realization of simplicial sets. 

We consider the real line $\R$ as a poset with its standard ordering and as equipped with a $\Z$-action by addition. 
For every $T \in \Lambda_\infty$ we define 
$$
| \Lambda_\infty^T| := \{f: T^\circ \to \R \mid f \text{ non-decreasing and $\Z$-equivariant} \} \ .
$$
We topologize this as a subspace of the space of maps into $\R$ with its standard topology. This space also carries a natural action by the additive group of real numbers $\R$ given by postcomposition with a translation. 
Together this construction defines a $\B\Z$-equivariant functor $\Lambda_\infty \to \R\Top$. 

Now for an arbitrary paracyclic topological space $X_\bullet$ we define the geometric realization by the coend
$$
| X_\bullet | = \int_{T\in \Lambda_\infty} X_T \times |\Lambda_\infty^{T}|  = \coprod_{T} \left( X_T \times |\Lambda_\infty^{T}| \right) / \sim
$$
in $\R\Top$. This functor is $\B\Z$-equivariant, since the functor $T \mapsto |\Lambda_\infty^T|$ is. 
\end{construction}

Recall that a cyclic space is a functor $X_\bullet: \Lambda^\op \to \Top$, which, by definition of $\Lambda$, is the same as a $\B\Z$-equivariant functor $\Lambda_\infty^\op \to \Top$. Since we are in a 1-categorical setting this is just a condition on the functor, namely that is sends $f$ and $f+1$ to the same map of topological spaces for each morphism $f$ in $\Lambda_\infty^\op$. In this sense we consider the category of cyclic spaces as a full subcategory of paracyclic spaces. Now since the realization functor $| - |\colon \Fun(\Lambda_\infty^\op,\Top)\to \R\Top$ is equivariant it follows that for a cyclic space $X_\bullet$  the action of $\Z \subseteq \R$ on the realization $|X_\bullet|$ is trivial, thus the action factors to a $\T = \R/\Z$-action. This $\T$-space is what we refer to as the geometric realization of the cyclic space.\footnote{One could alternatively repeat Construction \ref{concyclicrealization} using cyclic spaces instead of paracyclic spaces and $S^1$ instead of $\R$. We leave the details to the reader. }
Similarly for a functor $\Lambda_p^\op \to \Top$ the action of $p\Z \subseteq \R$ is trivial so that we obtain a $\R / p\Z$-action. 
\begin{example}\label{examplecyclicrealization}
Consider an object of $\Lambda_\infty$ of the form $j_\infty(S) = \Z \times S$ for an object $S \in \Delta$ (in fact every object of $\Lambda_\infty$ is up to isomorphism of this form). 
There is a canonical morphism $$|\Delta^S| \to | \Lambda^{j_\infty(S)}_\infty|$$ given by sending a map $s: S^\circ \to [0,1]$ to the unique $\Z$-equivariant extension $\overline{s}: j_\infty(S)^\circ \to \R$ along the canonical inclusion $S^{\circ} \subseteq j_\infty(S)^\circ$. Now using this map and the fact that $\R$ acts on $ | \Lambda_\infty^{j_\infty(S)}|$ we get a canonical map
\begin{align*}
|\Delta^S| \times \R &\longrightarrow | \Lambda_\infty^{j_\infty(S)}|\\
 \big(s,\alpha\big) & \mapsto \overline{s} + \alpha \ .
\end{align*}
This map is a homeomophism as one easily checks by writing down the inverse. 
\end{example}
\begin{remark}\label{remarkexplicit}
For $S = [n] \in \Delta$ we get an  equivalence $| \Lambda_\infty^{n+1}| \cong |\Delta^n| \times \R$. One can use this to directly define the geometric realization. Then one has to equip the spaces $|\Delta^n| \times \R$ by hand with the structure of a coparacyclic space. Chasing through the identifications one finds that the cyclic operator (i.e. addition with $\frac{1}{n+1}$) acts as
$$(x_0,\ldots,x_{n}, \alpha)  \mapsto \left(x_{n},x_0,\ldots,x_{n-1}, \alpha-x_{n}\right)$$
where the $x_i$ are the barycentric simplex coordinates in $|\Delta^n|$ and $\alpha \in \R$. This shows that our geometric realization is compatible with the classical discussions, see e.g. the discussion in \cite{MR870737}.
\end{remark}
\begin{lemma}\label{realization}
For a paracyclic topological space $X_\bullet$ the underlying space of the geometric realization $|X_\bullet|$ is naturally homeomorphic to the geometric realization of the underlying simplicial space $j_\infty^*X_\bullet$.  
\end{lemma}
\begin{proof}
Since both are defined by coend formulas and $j^*$ preserves colimits it is enough to check that there is a natural homeomorphism $|\Lambda_\infty^T| \cong |j^*\Lambda_\infty^T|$ for $T \in \Lambda_\infty$. Here $j^*\Lambda_\infty^{T}$ is the simplicial set obtained from the representable $\Lambda_\infty^T \colon:  \Lambda_\infty^\op \to \mathrm{Set}$ by pullback along $j:\Delta \to \Lambda_\infty$.
By definition we have
$$
| j^*\Lambda_\infty^T | =  \coprod_{S \in \Delta} \big(|\Delta^S| \times \Hom_{\Lambda_\infty}(j_\infty(S),T) \big) / \sim \ .
$$
There is a natural map from this coend to $|\Lambda_\infty^T|$ given as in Example \ref{examplecyclicrealization}: it sends a pair $(s,f)$ consisting of an  element $s: S^\circ \to [0,1]$ in $|\Delta^S|$ and a map $f: j_\infty(S) \to T$ to the  composition of $f^\circ: T^\circ \to j_\infty(S)^\circ$ with the extenions $\overline{s}: j_\infty(S)^\circ \to \R$. This is well-defined and natural in $T$. Thus we have to show that this map $| j^*\Lambda_\infty^T |  \to |\Lambda_\infty^T|$ is  a homeomorphism for every  $T \in \Lambda_\infty$. 

To this end we construct an inverse map. For an element $t \in |\Lambda_\infty^T|$ given by a map $t: T^\circ \to \R$ we let $n$ be the unique natural number such that  $T \cong \frac 1 n \Z$, i.e. the cardinality of $T/\Z$. Then we consider the unique map
$$
f: \frac 1 {n+1} \Z \to T 
$$
with $f(\frac{1}{n+1}) = \min t^{-1}(\R_{\geq 0})$ (where we have identified $T$ and $T^\circ$ as usual) and $f(\frac {k} {n+1})$ is the successor of $f(\frac {k-1} {n+1})$ for $1 \leq k \leq n $.\footnote{Note that $k = 1$ is also included so that $f(0)$ is the successor of $\min t^{-1}(\R_{\geq 0}$). We recommend that the reader computes the map $f$ and $f^\circ$ under the isomorphism $T \cong \frac{1} n \Z$ which sends $\min t^{-1}(\R_{\geq 0})$ to $0$ as an illustration.  }
Using this map we obtain a factorization of $t$ as
$$T^\circ \xto{f^\circ} \frac{1} {n+1} \Z^\circ\xto{\overline{s}} \R$$
with $\overline{s}(0) = 0$ and $\overline{s}(1) = 1$, i.e. $\overline{s}$ is the extension of some $s: [n]^\circ \to [0,1]$. In fact this factorization determines $s$ uniquely since the elements $\frac{1}{n+1},\ldots , \frac{n}{n+1}$ all lie in the image of $f^\circ$, but not $0$ and $1$. Together this construction determines a map
\begin{equation}\label{inverse}
|\Lambda_\infty^T| \to |\Delta^{n}| \times\Hom_{\Lambda_\infty}(j_\infty[n],T) \to | j^*\Lambda_\infty^T |
\end{equation}
which is by construction left inverse to the map $| j^*\Lambda_\infty^T | \to |\Lambda_\infty^T|$ in question, i.e. the composition
$
|\Lambda_\infty^T| \to | j^*\Lambda_\infty^T | \to |\Lambda_\infty^T|
$
is the identity. We  show that it is also right inverse, i.e the composition 
$$
| j^*\Lambda_\infty^T | \to |\Lambda_\infty^T| \to | j^*\Lambda_\infty^T | 
$$
is the identity as well. 
We claim that this can be reduced to showing that the map
\begin{equation}\label{mapsurjective}
|\Delta^{n}| \times\Hom^{\mathrm{surj}}_{\Lambda_\infty}(j_\infty[n],T) \to | j^*\Lambda_\infty^T |
\end{equation}
is surjective, where $\Hom^{\mathrm{inj}}_{\Lambda_\infty}(j_\infty[n],T)$ denotes the set of those maps $j_\infty[n] = \frac{1}{n+1} \Z  \to T$ that are injective when restricted to $\{0, \frac 1 {n+1},\ldots, \frac n {n+1} \} \subseteq \frac{1}{n+1} \Z$  (and $n$ is still determined by $T \cong \frac{1} n \Z$). 
In this case the map $\frac{1}{n+1} \Z \to T$ is automatically surjective. 
The surjectivity of \eqref{mapsurjective} is indeed sufficient,  since for an object $(s,f)  \in |\Delta^{n}| \times\Hom^{\mathrm{inj}}_{\Lambda_\infty}(j_\infty[n],T) $ the factorization of the associated map $t = \bar s \circ f^\circ$ is given by the pair $(s,f)$ itself as one sees from the construction.

In order to see that the map \eqref{mapsurjective} is indeed surjective, we use the following more general criterion.
Let $X$ be a simplicial set. We call a simplex $\Delta^n \to X$ essential if it is non-degenerate and it is not the face of another non-degenerate simplex. We denote the subset of essential $n$-simplicies by $X_n^{\mathrm{ess}} \subseteq X_n$. Then the claim is that the map $\coprod_{[n] \in \Delta} \big( |\Delta^n| \times X_n^{\mathrm{ess}} \big)\to |X|$ given in the evident way is surjective. This fact is straightforward to check using the fact that the geometric realization can be built from the non-degenerate cells.
Now we identify the essential simplices of $j^*\Lambda_\infty^T$. The non-degenerate $k$-simplices are given by the maps 
$
\frac{1}{k+1} \Z \to T
$
that are injective when restricted to $\{0, \frac 1 {k+1},\ldots, \frac k {k+1} \} \subseteq \frac{1}{k+1} \Z  $ and the essential $k$-simplices are the ones for which moreover $k = n$, i.e. exactly the ones showing up in \eqref{mapsurjective}. This shows the surjectiviity of the map \eqref{mapsurjective}.

We have now shown that the two maps are inverse to each other. It remains to show that the composition \eqref{inverse} is continuous. This follows  since the only non-continuous part of the first map are the ``jumps'' of the minimum which are identified in the realization.  This finishes the proof. 

Note that one could alternatively investigate the non-degenerated cells of the simplicial set $j^*(\frac{1}n \Z) $ to see that this describes a cell structure of $\R \times \Delta^{n-1}$. This is more instructive (but a little harder to formalize) than our proof and is left to the reader.
\end{proof}

Now we obtain the analogue of Lemma \ref{geometrichocolim}. We call a paracyclic space $X_\bullet$, i.e. a functor $\Fun(\Lambda_\infty^\op,\Top)$ proper if the restriction $j^*X_\bullet$ is proper as a simplicial space. We denote by $\Fun(\Lambda_\infty^\op,\Top)_{\mathrm{prop}} \subseteq \Fun(\Lambda_\infty^\op,\Top)$ the full subcategory of proper paracyclic spaces.
\begin{proposition}\label{commutebz} The diagram
\[\xymatrix{
N\Fun(\Lambda_\infty^\op,\Top)_{\mathrm{prop}}\ar[r]^-{|-|}\ar[d] & N\R\Top\ar[d]\\
\Fun(N(\Lambda_\infty^\op),\calS)\ar[r]^-\colim & \calS
}\]
commutes as a diagram of $B\Z$-equivariant functors.\footnote{This means that it commutes up to a natural equivalence, which is specified up to a contractible choice.}
Here $\calS$ is equipped with the trivial $B\Z$-action and the right hand functor is, as a $B\Z$-equivariant functor, given by the Dwyer-Kan localization map $N\R\Top \to \Fun(\B\R,\calS)$ followed by the forgetful functor $\Fun(\B\R,\calS) \xto{\sim} \calS$.\footnote{Note that the 1-categorical forgetful functor $\R\Top \to \Top$ is not $\B\Z$-equivariant.}
\end{proposition}

\begin{proof}
Since the geometric realization of a paracyclic space is the geometric realization of the underlying simplicial set we can deduce from Lemma \ref{geometrichocolim} that the functor  $N\Fun(\Lambda_\infty^\op,\Top)_{\mathrm{prop}}\ \to N\R\Top$ preserves weak equivalences and also co\-limits. Thus it is a left adjoint functor. Then also similar to Lemma \ref{geometrichocolim} we see that the right adjoint is given by the right derived functor of the right adjoint of the realization. This functor is given by sending a topological space $X$ with $\R$-action to the paracyclic space 
$$
[n]_{\Lambda_\infty} \mapsto \Map_\R(|\Lambda_\infty^n|, X) \cong \Map(\Delta^{n-1}, X) \simeq X
$$
where we have used the equivalence of Remark~\ref{remarkexplicit}. It is as a $B\Z$-equivariant functor equivalent to the diagonal functor. This implies the claim.
\end{proof}

Passing to $B\Z$-fixed points now immediately gives the following important corollary, where a cyclic space is called proper if the underlying simplicial space is.

\begin{corollary}\label{prop:geomrealcomp} The diagram
\[\xymatrix{
N\Fun(\Lambda^\op,\Top)_{\mathrm{prop}}\ar[r]^-{|-|}\ar[d] & N\T\Top\ar[d]\\
\Fun(N(\Lambda^\op),\calS)\ar[r] & \calS^{B\T}
}\]
commutes, where the lower horizontal functor is the one constructed in Proposition \ref{prop:geomreal}.$\hfill \Box$
\end{corollary}

We also need a version of the previous proposition for spectra. For this, note that Construction~\ref{concyclicrealization} induces a functor
$$
\Fun(\Lambda_\infty^\op,\Sp^O)\to \R\Sp^O
$$
by applying it in every degree and taking basepoints into account (note that geometric realization preserves the point). 
We say that a simplicial or paracyclic spectrum is proper, if it is levelwise proper (after forgetting basepoints).

\begin{proposition}\label{commutebzsp} The diagram 
\[\xymatrix{
N\Fun(\Lambda_\infty^\op,\Sp^O)_{\mathrm{prop}}\ar[r]^-{|-|}\ar[d] & N\R\Sp^O\ar[d]\\
\Fun(N(\Lambda_\infty^\op),\Sp)\ar[r]^-\colim & \Sp
}\]
of $B\Z$-equivariant functors commutes.
\end{proposition}

\begin{proof}  This follows from Proposition \ref{commutebz} since ``colimits can be taken levelwise'': if we Dwyer-Kan localize $\Sp^O$ at the levelwise equivalences, then we get an $\infty$-category $\mathrm{PreSp}$ which we call the $\infty$-category of prespectra. By \cite[Proposition 4.2.4.4]{HTT} this is equivalent to the  functor category $\Fun(hcN\mathbf{O},\calS_*)$ where $\mathbf{O}$ is the topologically enriched category described in the footnote of Definition \ref{def:orthspectrum}. The objects of $\mathbf{O}$ are, up to equivalence, just the natural numbers and they correspond to the levels of the orthogonal spectrum. Now we use the fact that filtered colimits of pointed spaces are computed as the colimit of the underlying diagram of spaces and the fact that colimits in functor categories are computed pointwise to deduce that for $\mathrm{PreSp}$ in place of $\Sp$ the square in question evidently commutes.

Finally we use that $\Sp$ is a Bousfield localization of $\mathrm{PreSp}$, i.e. a reflective full subcategory. This implies that there is a commutative square
\[\xymatrix{
\Fun(N(\Lambda_\infty^\op),\mathrm{PreSp})\ar[r]^-{\colim}\ar[d] & \mathrm{PreSp}\ar[d]\\
\Fun(N(\Lambda_\infty^\op),\Sp)\ar[r]^-\colim & \Sp
}\]
of $B\Z$-equivariant functors. To see this we pass to the square of right adjoints. Pasting the two squares together then finishes the proof.
\end{proof}

\begin{corollary}\label{prop:geomrealcompspectra} The diagram
\[\xymatrix{
N\Fun(\Lambda^\op,\Sp^O)_{\mathrm{prop}}\ar[r]\ar[d] & N\T\Sp^O\ar[d]\\
\Fun(N(\Lambda^\op),\Sp)\ar[r] & \Sp
}\]
commutes.$\hfill \Box$
\end{corollary}

As a final topic, we need to discuss simplicial subdivision. Recall the category $\Lambda_p=\Lambda_\infty/B(p\Z)$, which comes with a natural functor $\Lambda_p\to \Lambda$, identifying $\Lambda=\Lambda_p/BC_p$. Geometrically, note that $|\Lambda_\infty|$ is contractible, while $|\Lambda_p|=|\Lambda_\infty|/B(p\Z)\simeq B(\R/p\Z) \simeq B\T$, and $|\Lambda_p|\to |\Lambda|$ is a $BC_p$-torsor.

On the other hand, there is another functor $\sd_p: \Lambda_p\to \Lambda$ which induces a homotopy equivalence of geometric realizations, and sends $[n]_{\Lambda_p}$ to $[pn]_\Lambda$. It is constructed as follows. There is an endofunctor $\sd_p: \Lambda_\infty\to \Lambda_\infty$ sending $[n]_{\Lambda_\infty}$ to $[pn]_{\Lambda_\infty}$, and a map $f: \frac{1}{n}\Z\to \frac{1}{m}\Z$  to the map $\frac{1}{p}f(p \cdot -): \frac{1}{pn}\Z\to \frac{1}{pm}\Z$. More generally for an arbitrary object $T \in \Lambda_\infty$ we introduce a new object $\sd_p(T) = \frac{1}{p} T$ with the same underlying order set as $T$ and the action of $n \in \Z$ given by multiplication with $pn$. 

The functor $\sd_p: \Lambda_\infty\to \Lambda_\infty$ passes to the quotient by $B\Z$ to give the desired functor
\[
\sd_p: \Lambda_p\to \Lambda\ .
\]
Note that as $\sd_p: \Lambda_\infty\to \Lambda_\infty$ necessarily induces a homotopy equivalence $|\Lambda_\infty|\simeq |\Lambda_\infty|$, so the preceding discussion implies that $\sd_p: \Lambda_p\to \Lambda$ also induces a homotopy equivalence $|\Lambda_p|\simeq |\Lambda|$.

\begin{proposition} The diagrams
\[\xymatrix{
\Lambda_\infty\ar[d]^{\sd_p}\ar[r]^\simeq & \Lambda_\infty^\op\ar[d]^{\sd_p} & & \Lambda_p\ar[d]^{\sd_p}\ar[r]^\simeq & \Lambda_p^\op\ar[d]^{\sd_p}\\
\Lambda_\infty\ar[r]^\simeq & \Lambda_\infty^\op & & \Lambda\ar[r]^\simeq & \Lambda^\op
}\]
commute.
\end{proposition}

\begin{proof} It is enough to consider the first diagram. On objects, commutation is clear, and on morphisms, both composite functors send $f$ to $f^\circ$.
\end{proof}

We need a compatibility between simplicial subdivision and the functor $\Lambda\to \Ass^\otimes_\act$. Indeed, using the category $\Free_{C_p}$ of finite free $C_p$-sets and its functor $S\mapsto \overline{S}=S/C_p$ to finite sets, there is a natural commutative diagram
\[\xymatrix{
\Lambda^\op\ar[r]^\simeq & \Lambda\ar^V[r] & \Ass^\otimes_\act\\
\Lambda_p^\op\ar[u]\ar[d]_{\sd_p}\ar[r]^\simeq & \Lambda_p\ar[u]\ar^{\!\!\!\!\!\!\!\!\!\!\!\!\!\!\!\!\!\!\!\!\! V_p}[r]\ar^{\sd_p}[d] & \Free_{C_p}\times_{\Fin} \Ass^\otimes_\act\ar[u]\ar[d]\\
\Lambda^\op\ar[r]^\simeq & \Lambda\ar^V[r] & \Ass^\otimes_\act\ ,
}\]
where the upper left vertical arrows are the projection $\Lambda_p\to \Lambda=\Lambda_p/(BC_p)$, and the upper right vertical arrow is the projection to $\Ass^\otimes_\act$. Moreover, for the lower right vertical arrow recall that $\Ass^\otimes_\act$ is the category of finite sets with maps given by maps with total orderings on preimages. Then the lower right vertical functor is given by the projection to the first factor $\Free_{C_p}$, with total orderings on preimages induced by the second factor. Finally, the functor $V_p$ sends $[n]_{\Lambda_p}$ to the pair of $(V(\sd_p([n]_{\Lambda_p})),V([n]_\Lambda))$, noting that $V(\sd_p([n]_{\Lambda_p}))$ has a natural $C_p$-action induced from the $C_p$-action on $[n]_{\Lambda_p}$, whose quotient is given by $V([n]_\Lambda)$.

Proposition~\ref{prop:geomreal} works with $\Lambda_p$ in place of $\Lambda$, as do Corollary~\ref{prop:geomrealcomp} and Corollary~\ref{prop:geomrealcompspectra}. However, there is an extra twist to the story here. Namely, there is a $BC_p$-action on $\Lambda_p$, and thus on
\[
\Fun(N(\Lambda_p^\op),\calC)
\]
for any $\infty$-category $\calC$. There is also a $BC_p$-action on $\calC^{B\T}$ by the natural action of $BC_p$ on $B\T$ induced by the inclusion $C_p\subseteq \T$.

\begin{lemma}\label{lem:geomrealbcpequiv} Let $\calC$ be an $\infty$-category that admits geometric realizations. The functors
\[\begin{aligned}
\Fun(N(\Lambda_p^\op),\calC)&\to \calC^{B\T}\ ,
\Fun(\Lambda_p^\op,\Top)&\to \T\Top\ ,
\Fun(\Lambda_p^\op,\Sp^O)&\to \T\Sp^O
\end{aligned}\]
are $BC_p$-equivariant, and the commutative diagrams
\[\xymatrix{
N\Fun(\Lambda_p^\op,\Top)_{\mathrm{prop}}\ar[r]\ar[d] & N\T\Top\ar[d]\\
\Fun(N(\Lambda_p^\op),\calS)\ar[r] & \calS^{B\T}\ ,
}\]
\[\xymatrix{
N\Fun(\Lambda_p^\op,\Sp^O)_{\mathrm{prop}}\ar[r]\ar[d] & N\T\Sp^O\ar[d]\\
\Fun(N(\Lambda_p^\op),\Sp)\ar[r] & \Sp^{B\T}
}\]
commute $BC_p$-equivariantly.
\end{lemma}

\begin{proof} First, the construction of the $BC_p$-equivariant functor 
$\Fun(N(\Lambda_p^\op),\calC)\to \calC^{B\T}$ works exactly as in Proposition \ref{prop:geomreal}: it is constructed as the composition 
\[
\Fun(N(\Lambda_p^\op),\calC)\to \Fun^{B(p\Z)}(N(\Lambda_\infty^\op),\calC)\to \Fun^{B(p\Z)}(\pt,\calC) = \calC^{BB(p\Z)} = \calC^{B\T}\ .
\]
All of the involved functors are $BC_p$-equivariant, since they are obtained as $B(p\Z)$ fixed points of functors that are $B\Z$-equivariant. 

Now the rest of the proposition follows analogously by taking fixed points since the constructions of geometric realizations come from $B\Z$-equivariant functors (Construction \ref{concyclicrealization}) and the diagrams already commute $B\Z$-equivariantly (Proposition \ref{commutebz} and Proposition \ref{commutebzsp}).  
\end{proof}

Using the functor  $\sd_p: \Lambda_p\to \Lambda$ we get for every cyclic object $X$ a subdivided $\Lambda_p$-object $\sd_p^* X$ and we recall the well-known fact that they have the same geometric realization.

\begin{proposition}\label{subdivisionrealization} Let $\calC$ be an $\infty$-category that admits geometric realizations.
\begin{altenumerate}
\item The diagram of functors
$$
\xymatrix{
\Fun(\Lambda^\op, \calC) \ar[d]^{\sd_p}\ar[r] & \calC^{B\Z} \ar@{=}[d] \\
\Fun(\Lambda_p^\op, \calC) \ar[r] & \calC^{B\Z}
}
$$
commutes.
\item For a proper paracyclic topological space (resp. spectrum) $X$ there is a natural $\T$-equivariant homeomorphism 
$$
|X|\cong |\sd_p^* X|\ ,
$$
compatible with the $\infty$-categorical equivalence under the comparison of Proposition \ref{commutebz} (resp. Proposition \ref{commutebzsp}).$\hfill \Box$
\end{altenumerate}
\end{proposition}

In our construction of the cyclotomic structure maps, we need to commute the Tate construction $-^{tC_p}$ with a geometric realization. Let us describe this abstractly; we consider the case $\calC=\Sp$ for concreteness. There is a natural functor
\[
\Fun(\Lambda_p^\op,\Sp)\to \Fun(\Lambda^\op,\Sp)
\]
intuitively given by sending $[n]_{\Lambda_p^\op}\mapsto X_n$ to $[n]_{\Lambda^\op}\mapsto X_n^{tC_p}$, using the natural $C_p$-action on $X_n$. Concretely, using the natural $BC_p$-action on $\Lambda_p^\op$, one has
\[
\Fun(\Lambda_p^\op,\Sp) = \Fun^{BC_p}(\Lambda_p^\op,\Sp^{BC_p})\ ,
\]
and composing with the $BC_p$-equivariant functor $-^{tC_p}: \Sp^{BC_p}\to \Sp$, one gets the functor
\[
\Fun(\Lambda_p^\op,\Sp) = \Fun^{BC_p}(\Lambda_p^\op,\Sp^{BC_p})\to \Fun^{BC_p}(\Lambda_p^\op,\Sp) = \Fun(\Lambda^\op,\Sp)\ .
\]

\begin{proposition}\label{prop:commutetaterealization} There is a natural transformation from the composite
\[
\Fun(\Lambda_p^\op,\Sp)\xto{-^{tC_p}} \Fun(\Lambda^\op,\Sp)\to \Sp^{B\T}
\]
to the composite
\[
\Fun(\Lambda_p^\op,\Sp)\to \Sp^{B\T}\xto{-^{tC_p}} \Sp^{B(\T/C_p)}\simeq \Sp^{B\T}\ .
\]
\end{proposition}

\begin{proof} Under the equivalence
\[
\Fun(\Lambda_p^\op,\Sp) = \Fun^{BC_p}(\Lambda_p^\op,\Sp^{BC_p}) = \Fun^{B\Z}(\Lambda_\infty^\op,\Sp^{BC_p})\ ,
\]
the first functor is the composite
\[
\Fun^{B\Z}(\Lambda_\infty^\op,\Sp^{BC_p})\xto{-^{tC_p}} \Fun^{B\Z}(\Lambda_\infty^\op,\Sp)\xto{\colim_{\Lambda_\infty^\op}} \Fun^{B\Z}(\pt,\Sp) = \Sp^{B\T}\ ,
\]
and the second functor is the composite
\[
\Fun^{B\Z}(\Lambda_\infty^\op,\Sp^{BC_p})\xto{\colim_{\Lambda_\infty^\op}} \Fun^{B\Z}(\ast,\Sp^{BC_p})\xto{-^{tC_p}} \Fun^{B\Z}(\ast,\Sp) = \Sp^{B\T}\ .
\]
Generally, if $\calC$, $\calD$ and $\calE$ are $\infty$-categories such that $\calD$ and $\calE$ admit $\calC$-indexed colimits and $G: \calD\to \calE$ is a functor, then there is a natural transformation of functors
\[
\Fun(\calC,\calD)\to \calE
\]
from
\[
(F: \calC\to \calD)\mapsto \colim_{\calC} G\circ F
\]
to
\[
(F: \calC\to \calD)\mapsto G(\colim_{\calC} F)\ ,
\]
given by the universal property of the colimit. Applying this to $\calC=\Lambda_\infty^\op$, $\calD=\Sp^{BC_p}$, $\calE=\Sp$, $G=-^{tC_p}: \Sp^{BC_p}\to \Sp$ and passing to $B\Z$-equivariant objects gives the result.
\end{proof}

We need a small variant of the above constructions in the case $BC_p$-equivariant functors
\[
\Lambda_p^\op\to \T\Top\ ,
\]
where $BC_p$ acts on $\T\Top$ via the natural inclusion $C_p\subset \T$. Namely, given such a functor, the usual geometric realization gives an object of $(\T\times \T)\Top$, but the $BC_p$-equivariance implies that it is actually an object of $((\T\times \T)/C_p)\Top$, where $C_p$ is embedded diagonally. Thus, we can restrict to the diagonal $\T/C_p\cong \T$, and get an object of $\T\Top$. We get the following proposition.

\begin{proposition}\label{prop:geomrealtactiontop} There are natural functors
\[
\Fun^{BC_p}(\Lambda_p^\op,\T\Top)\to \T\Top
\]
from the category $\Fun^{BC_p}(\Lambda_p^\op,\T\Top)$ of $BC_p$-equivariant functors $\Lambda_p^\op\to \T\Top$ to $\T\Top$, and
\[
\Fun^{BC_p}(\Lambda_p^\op,\T\Sp^O)\to \T\Sp^O
\]
from the category $\Fun^{BC_p}(\Lambda_p^\op,\T\Sp^O)$ of $BC_p$-equivariant functors $\Lambda_p^\op\to \T\Sp^O$ to $\T\Sp^O$.

On the subcategory of functors factoring over a functor $\Lambda^\op = \Lambda_p^\op/BC_p\to \Top\subseteq \T\Top$ (resp.~$\Lambda^\op = \Lambda_p^\op/BC_p\to \Sp^O\subseteq \T\Sp^O$), this is given by the usual geometric realization functor discussed in Construction \ref{concyclicrealization} and subsequently.$\hfill \Box$
\end{proposition}

This construction can also be done in $\infty$-categories.

\begin{proposition}\label{prop:geomrealtaction} For any $\infty$-category $\calC$ admitting geometric realizations, there is a natural functor
\[
\Fun^{BC_p}(N(\Lambda_p^\op),\calC^{B\T})\to \calC^{B\T}
\]
from the category $\Fun^{BC_p}(N(\Lambda_p^\op),\calC^{B\T})$ of $BC_p$-equivariant functors $N(\Lambda_p^\op)\to \calC^{B\T}$ to $\calC^{B\T}$. Restricting this construction to functors factoring as $\Lambda_p^\op/BC_p = \Lambda^\op\to \calC\to \calC^{B\T}$, this agrees with the functor from Proposition~\ref{prop:geomreal}.

Moreover, the diagrams
\[\xymatrix{
\Fun^{BC_p}(\Lambda_p^\op,\T\Top)_{\mathrm{prop}}\ar[r]\ar[d] & \T\Top\ar[d]\\
\Fun^{BC_p}(N(\Lambda_p^\op),\calS^{B\T})\ar[r] & \calS^{B\T}\ ,
}\]
\[\xymatrix{
\Fun^{BC_p}(\Lambda_p^\op,\T\Sp^O)_{\mathrm{prop}}\ar[r]\ar[d] & \T\Sp^O\ar[d]\\
\Fun^{BC_p}(N(\Lambda_p^\op),\Sp^{B\T})\ar[r] & \Sp^{B\T}
}\]
commute.
\end{proposition}

\begin{proof} The construction of the functor
\[
\Fun^{BC_p}(N(\Lambda_p^\op),\calC^{B\T})\to \calC^{B\T}
\]
is the same as above: As $\calC^{B\T}$ has geometric realizations, one has a functor
\[
\Fun(N(\Lambda_p^\op),\calC^{B\T})\to (\calC^{B\T})^{B\T} = \calC^{B(\T\times \T)}=\Fun(B(\T\times\T),\calC)\ .
\]
By Lemma~\ref{lem:geomrealbcpequiv}, this functor is $BC_p$-equivariant for the $BC_p$-action on the left-hand side given by the action on $\Lambda_p^\op$, and the $BC_p$-action on the right by acting on the second copy of $\T$. It follows that we get an induced functor
\[
\Fun^{BC_p}(N(\Lambda_p^\op,\calC^{B\T}))\to \Fun^{BC_p}(B(\T\times \T),\calC)\ ,
\]
where $BC_p$ acts on $B(\T\times \T)$ through the diagonal embedding (as $BC_p$ acts on the left also on $B\T$). But
\[
\Fun^{BC_p}(B(\T\times \T),\calC) = \Fun(B((\T\times \T)/C_p),\calC)\ ,\]
which has a natural map to $\calC^{B\T}$ by restricting to the diagonal $\T/C_p\subseteq (\T\times \T)/C_p$, and identifying $\T/C_p\cong \T$.

As all intervening functors in the construction commute with the functors $\Top\to \calS$ resp.~$\Sp^O\to \Sp$, we get the desired commutative diagrams.
\end{proof}

\setcounter{thm}{0}
\setcounter{chapter}{20}
\chapter{Homotopy colimits}\label{app:colim}

\renewcommand{\thesection}{C}

In this section we briefly recall the classical construction of homotopy colimits of (pointed) topological spaces and orthogonal spectra following Bousfield and Kan. The material in this section is well-known and standard but in the literature there are varying statements, especially the precise cofibrancy conditions that are imposed differ. We thank Steffen Sagave and Irakli Patchkoria for explaining some of the subtleties that arise and guidance through the literature. We also found the writeup \cite{malkiewich2014homotopy} very useful. 

As before we work in the category of compactly generated weak Hausdorff spaces throughout. We need the following well-known basic result, for which it is essential to work in compactly generated weak Hausdorff spaces.

\begin{proposition}\label{prop:geomfinlimit}
The geometric realization functor 
$$
| - | : \Fun(\Delta^\op , \Top) \to \Top
$$
of Construction~\ref{cons:geomreal} commutes with finite limits.
\end{proposition}

\begin{proof} First we note that $|-|$ commutes with finite products, which is shown in \cite[Proposition A.37 (ii)]{Global}. It remains to show that for a  pullback diagram of simplicial spaces the resulting diagram of realizations is again a pullback diagram. We  claim that the composition
\[
\Fun(\Delta^\op , \Top) \xto{| - | } \Top \xto{U} \mathrm{Set}
\]
preserves finite limits, where $U$ takes the underlying set of a topological space. To see this we use that the underlying set of the geometric realization is just the coend of the underlying sets of the spaces by \cite[Appendix A, Proposition A.35 (ii)]{Global}. In particular the composite functor $U \circ |-|$ is given by the composition
\[
\Fun(\Delta^\op , \Top) \xto{U_*} \Fun(\Delta^\op , \mathrm{Set}) \xto{U|-|} \mathrm{Set} \ .
\]
The first functor preserves all limits, and the second functor commutes with finite limits by \cite[Chapter III]{Gabriel-Zisman}\footnote{The argument breaks down to showing that the cosimplicial set $U|\Delta^\bullet|: \Delta \to \mathrm{Set}$ is a filtered colimit of corepresentables, which follows easily from our description in terms of intervals as in Construction \ref{cons:geomreal}}.

Thus, in the diagram
\[
\xymatrix{
\left| X_\bullet \times_{Y_\bullet} Z_\bullet \right| \ar[r] \ar[d] & |X_\bullet| \times_{|Y_\bullet|} |Z_\bullet| \ar[d] \\
\left| X_\bullet \times Z_\bullet \right| \ar[r] & |X_\bullet| \times |Z_\bullet|
}
\]
the upper horizontal map is a bijection and the lower horizontal map is a homeomorphism. But the vertical maps are closed embeddings, using \cite[Proposition A.35 (iii)]{Global} for the left vertical map, so the result follows.
\end{proof}

Recall that a map $A \to X$ of topological spaces is called a Hurewicz cofibration (or simply h-cofibration) if it has the homotopy extension property, i.e.~if the induced map
$$
A \times I \cup_{A \times \{0\}} X \times \{0\} \to X \times I
$$
admits a retract, where $I=[0,1]$ is the interval. Now we state the technical key lemma for all results in this section. The first reference for this is \cite[Proposition 4.8(b), Page 249]{MR0420609} but see also the discussion in \cite[Proposition 1.1]{malkiewich2014homotopy}.

\begin{lemma}[Gluing Lemma]
Consider a diagram of spaces
$$
\xymatrix{
B\ar[d]^\simeq & A\ar[l]\ar[r]\ar[d]^\simeq & C\ar[d]^\simeq \\
B'& A' \ar[l]\ar[r] & C'\ ,
}
$$
where the maps $A \to B$ and $A' \to B'$ are h-cofibrations and the vertical maps are weak homotopy equivalences. Then also the induced map on pushouts $B \cup_A C \to B' \cup_{A'} C'$ is a weak homotopy equivalence. $\hfill \Box$
\end{lemma} 

For the next statement recall that we call a simplicial space $X_\bullet$ proper if all the maps $L_n X \to X_n$ are $h$-cofibrations where 
$$L_n X = \mathrm{coEq}\left( \xymatrix{\coprod_{i=0}^{n-2} X_{n-2}  \ar[r]<-2pt>\ar[r]<+2pt> & \coprod_{i=0}^{n-1} X_{n-1}}  \right)
\cong  \bigcup_{i=0}^{n-1} s_i(X_{n-1}) \subseteq X_{n-1}$$
is the $n$-th latching object. The following proposition is well known, cf.~\cite[A.4]{May74}. This also follows from \cite[Theorem A.7 and Reformulation A.8]{MR2045835}.
We give a quick proof based on the Gluing lemma which also seems to be folklore (see e.g. \cite{malkiewich2014homotopy}).

\begin{proposition}\label{proper}
Let $f: X_\bullet \to Y_\bullet$ be a map of proper simplicial spaces that is a levelwise weak homotopy equivalence. Then the realization $|X_\bullet| \to |Y_\bullet|$ is a weak homotopy equivalence as well. 
\end{proposition}

\begin{proof}
Every simplicial space $X_\bullet$ admits a skeletal filtration 
$$
|\Sk^0 X_\bullet| \subseteq |\Sk^1 X_\bullet| \subseteq  |\Sk^2 X_\bullet| \subseteq \ldots \subseteq |X| 
$$
where each $|\Sk^{n} X_\bullet|$ is obtained by the following pushout
$$
\xymatrix{
L_n X \times |\Delta^n| \cup_{L_nX \times |\partial \Delta^n|} X_n \times  |\partial \Delta^n| \ar[r]\ar[d] & X_n \times \Delta^n\ar[d] \\
|\Sk^{n-1} X_\bullet| \ar[r] & |\Sk^{n} X_\bullet|\ .
}
$$
If $X_\bullet$ is proper the pushout product axiom implies that the upper morphism is an h-cofibration.

As a first step we  show that under the assumptions of the Proposition the induced map
$|\Sk^{n} X_\bullet| \to |\Sk^{n} Y_\bullet|$ is a weak equivalence. By the above considerations and the gluing lemma it suffices to show that for all $n$ the  maps $L_n X \to L_nY$ are weak homotopy equivalences. This can be seen by a similar induction over the dimension and the number of summands in the union  $\bigcup_{i=0}^{n-1} s_i(X_{n-1})$.

Now we know that all maps $|\Sk^{n} X_\bullet| \to |\Sk^{n} Y_\bullet|$ are weak equivalences and the result follows by passing to the limit. This does not create problems since all maps involved in this filtered limit are h-cofibrations. The  directed colimit of those is a homotopy colimit.
\end{proof}

Note that the last proposition is slightly surprising from the point of view of model categories since it mixes $h$-cofibrations and weak homotopy equivalences. In particular it can not be proved abstractly for any model category (it can not even be stated in this generality).

Now we can come to the discussion of homotopy colimits. Thus let $I$ be a small category and consider a functor $X: I \to \Top$. 

\begin{definition}\label{BK-formula}
The (Bousfield-Kan) homotopy colimit of $X$ is defined as the geometric realization
$$
\hocolim_I X := \Big| \coprod_{i_0 \to \ldots \to i_n} X(i_n) \Big|\ .
$$
\end{definition}

It is fairly straightforward to compute that the latching object of the simplicial space in question is a disjoint union of $X(i_n)$ over some subset of all strings $i_0 \to \ldots \to i_n$. Explicitly it is the subcategory where at least one of the morphisms is the identity. Cleary the inclusion of a component into a disjoint union is an $h$-cofibration. 

\begin{proposition}\label{model}
For a transformation $X \to X'$ between functors $X, X': I \to \Top$ such that every map $X(i) \to X'(i)$ is a weak homotopy equivalence the induced map
$\hocolim_I X  \to \hocolim_I X'$ is also a weak homotopy equivalence.  

Moreover the homotopy colimit models the $\infty$-categorical colimit in the sense that the diagram
$$
\xymatrix{
N\Fun(I, \Top) \ar[rr]^-{\hocolim_I} \ar[d] && N\Top\ar[d] \\
\Fun(NI, \calS) \ar[rr]^{\colim_I} && \calS
}
$$
commutes. The vertical maps are the Dwyer-Kan localizations at the (levelwise) weak homotopy equivalences.
\end{proposition}
\begin{proof}
Only the second part needs to be verified. But for this we use that the model categorical homotopy colimit is equivalent to the $\infty$-categorical colimit. The comparison between the model categorical colimit (with respect to the projective model structure) and the Bosufield Kan formula is classical, see for example \cite{dugger2008primer} under some more restrictive cofibrancy assumptions. But the first part of this proposition show that this is unnecessary since cofibrant replacement does not change the weak homotopy type of the homotopy colimit. 
\end{proof}

Now we  talk about pointed spaces. First of all, if we have a pointed simplicial space
$X_\bullet: \Delta^\op \to \Top_*$ then the geometric realization is again canonically pointed since the geometric realization of the constant simplicial diagram is again a point, so we do not need to modify the geometric realization. The situation is different for the homotopy colimit of a diagram $X: I\to \Top_*$. To see this let us compute the homotopy colimit of the constant diagram $I \to \Top$ with $i \mapsto \pt$ for each $i \in I$. By definition this geometric realization is the geometric realization of the simplicial set $NI$, the nerve of the category $I$. This geometric realization is not homeomorphic to the point (aside from trivial cases). This makes the following definition necessary:

\begin{definition}\label{red_colim}
Let $X_\bullet: I \to \Top_*$ be a diagram of pointed topological spaces. The reduced homotopy colimit is defined as the quotient
$$
\widetilde{\hocolim_I} X := \hocolim_I X / \hocolim_I \pt
$$
where $\hocolim_I X$ refers to the homotopy colimit of the underlying unpointed diagram in the sense of Definition \ref{BK-formula}.
\end{definition} 

Now recall that a based space $X$ is called well-pointed if the inclusion $\pt \to X$ of the basepoint is an h-cofibration. 
\begin{proposition}
For a transformation $X \to X'$ between functors $X, X': I \to \Top_*$ such that every map $X(i) \to X'(i)$ is a weak homotopy equivalence the induced map
$$\widetilde{\hocolim_I} X  \to \widetilde{\hocolim_I} X'$$
is also a weak homotopy equivalence provided that all spaces $X(i)$ and $X'(i)$ are well pointed.  
Moreover the reduced homotopy colimit models the $\infty$-categorical colimit of pointed spaces.
\end{proposition}
\begin{proof}
The reduced homotopy colimit is homeomorphic to the geometric realization of the simplicial space
$$
\widetilde{\hocolim_I} X \cong  \Big| \bigvee_{i_0 \to \ldots \to i_n} X(i_n) \Big|
$$
thus we have to show that this simplicial space is proper. But the latching inclusions are given by pushout against the basepoint inclusions so that they are h-cofibrations if all involved spaces are well-pointed. 
The statement about the $\infty$-categorical colimit follows as in Proposition \ref{model}.
\end{proof}

In practice we will abusively only write $\hocolim_I X$ instead of $\widetilde{\hocolim_I} X$ if it is clear from the context that we are working in a pointed context.

Finally we  discuss homotopy colimits in the category $\Sp^O$ of orthogonal spectra. If we just apply the results about pointed spaces levelwise and use that all colimits are computed levelwise, then we get that for levelwise well-pointed orthogonal spectra the Bousfield-Kan formula computes the correct $\infty$-categorical colimit. But we will see now that more is true: the Bousfield-Kan formula always computes the correct colimit, even for non-well pointed diagrams! This miracle has to do with stability and is in contrast to the case of pointed spaces. 

To prove this fact we have to input a version of the gluing lemma for orthogonal spectra which does not hold for pointed spaces. To this end recall that a map $A \to X$ of orthogonal spectra is called an h-cofibration if it has the homotopy extension property in the category of orthogonal spectra. Equivalently if there is a retract of the map
$$
A \wedge I_+ \cup_{A \wedge S^0} X \wedge S^0 \to X \wedge I_+
$$
where we have used that orthogonal spectra are tensored over pointed spaces. 
Then we have the following version of the Gluing Lemma which is proved in \cite[Theorem 7.4(iii)]{MMSS} for prespectra and in  \cite[Theorem 3.5(iii)]{MandellMay} for orthogonal $G$-spectra (just specialise to $G = \pt$):
\begin{lemma}[Gluing Lemma]\label{glueing2}
If we have a diagram of orthogonal spectra
$$
\xymatrix{
B\ar[d]^\simeq & A\ar[l]\ar[r]\ar[d]^\simeq & C\ar[d]^\simeq \\
B'& A' \ar[l]\ar[r] & C'
}
$$
where the maps $A \to B$ and $A' \to B'$ are h-cofibrations and the vertical maps are stable equivalences. Then also the induced map on pushouts 
$B \cup_A C \to B' \cup_{A'} C'$ is a stable equivalence. $\hfill \Box$
\end{lemma} 

\begin{remark}
Note that the analogous notion of h-cofibration for pointed spaces is what is called pointed h-cofibration. For a map of pointed spaces being a pointed h-cofibration is a weaker condition than for the underlying map of spaces being an h-cofibration. 
 For example the map
$\pt \to X$ is always a pointed h-cofibration for every pointed space $X$. But it is only an underlying h-cofibration if the space $X$ is well pointed. The crucial difference between pointed spaces and orthogonal spectra is that the analogous statement to the Gluing Lemma \ref{glueing2} is wrong in pointed spaces. We rather need that the underlying maps are h-cofibrations.
\end{remark}

\begin{definition}
For a diagram $I \to \Sp^O$ the (Bousfield-Kan) homotopy colimit $\hocolim_I X \in \Sp^O$ is defined by levelwise application of the reduced homotopy colimit of Definition \ref{red_colim}. 
\end{definition}

\begin{proposition}\label{orthogonalBK}
For a transformation $X \to X'$ between diagrams of orthogonal spectra $X, X': I \to \Sp^O$ such that every map $X(i) \to X'(i)$ is a stable equivalence the induced map
$${\hocolim_I} X  \to {\hocolim_I} X'$$
is also a stable equivalence. 
Moreover the homotopy colimit models the $\infty$-categorical colimit of orthogonal spectra. \footnote{We thank I. Patchkoria and C. Malkiewich  for making us aware of a little inaccuracy in an earlier version of the proof of this statement.}
\end{proposition}
\begin{proof}
This follows formally exactly as in the case of spaces (Proposition \ref{proper}) from the gluing lemma using the fact that the inclusions $X \to X \vee Y$ of wedge summands are h-cofibrations. 
Thus the relevant simplicial object is Reedy h-cofibrant. 
We now use the fact that there is a pushout-product axiom for h-cofibrations of orthogonal spectra when smashing with h-cofibrations of   spaces which is due to Schw\"anzel-Vogt \cite[Corollary 2.9]{MR1967263}, see also \cite[Proposition 4.3]{malkiewich2014homotopy} and the references there. We apply this to the smash with the maps $|\partial \Delta^n|_+ \to |\Delta^n|_+$, then the statement follows as in Proposition \ref{proper} from the inductive construction of the geometric realization and finally the fact that the filtered colimit along h-cofibrations again is the homotopy colimit.
\end{proof}

\begin{lemma}
The homotopy colimit functor $\hocolim: \Fun(I,\Sp^O)\to \Sp^O$ commutes with geometric realizations, i.e. for a diagram $X: \Delta^{\op} \to \Fun(I,\Sp^O)$ there is a canonical homeomorphism
\begin{align*}
|\hocolim_I (X_i)| \cong \hocolim_I |X_i| \ .
\end{align*}
\end{lemma}

\begin{proof} Geometric realization commutes with colimits since it is a coend and the smash product commutes with colimits in both variables seperately. As a result the statement can by definition of homotopy colimits be reduced to showing that geometric realization commutes with another geometric realization. But this means that taking vertical and horizontal geometric realization of a bisimplicial space is the same as the other way around. The latter fact is standard and can be seen by a formal reduction to simplices where it is obvious. 
\end{proof}

Now we consider the category of orthogonal $G$-spectra for $G$ a finite group $G$ (or a compact Lie group). By continuity of the functor $\hocolim$ we get that for a diagram $I \to G\Sp^O$ the homotopy colimit $\hocolim_I X$ gets an induced $G$-action, thus can be considered as a orthogonal $G$-spectrum itself. The following lemma will be crucial:

\begin{lemma}\label{fixedcommutes}
The fixed points functor $-^G: G\Sp^O \to \Sp^O$ commutes with geometric realizations and homotopy colimits, i.e. for diagrams $Y: \Delta^{\op} \to G\Sp^O$ and  $X: I \to G\Sp^O$ the canonical maps
\begin{align*}
|Y_\bullet^G| \to | Y_\bullet|^G \qquad \text{and} \qquad 
\hocolim_I (X_i^G) \to (\hocolim_I X_i)^G
\end{align*}
are homeomorphisms. 
The same is true if we replace $\Sp^O$ by spaces or pointed spaces. 
\end{lemma}
\begin{proof}
We first show that geometric realization commutes with taking fixed points in spaces. The statements about spectra (and pointed spaces) then follows since it is just applied levelwise.
Fixed points for a finite group are a finite limit. Fixed points for a compact Lie group $G$ can be computed by taking a finite set of topological generators of $G$ and then taking fixed points for those. This way we see that it is also essentially a finite limit. Now the result follows from the observation that geometric realization as a functor from simplicial objects in $\Top$ to $\Top$ commutes with finite limits, cf.~Proposition~\ref{prop:geomfinlimit}.

Finally the homotopy colimit is the geometric realization $|Y_\bullet|$ where 
$Y_\bullet$ is the simplicial object in $\Sp^O$ with
$$
Y_n = \bigvee_{i_0 \to ...\to i_n} X_{i_n}.
$$
Thus it suffices to show that wedge sums commute with taking fixed points if the action is basepoint preserving. But this is obvious. 
\end{proof}

We note that the fixed points functor itself needs to be derived, thus the above lemma a priori does not have any homotopical meaning (of course it can be derived but we shall not need this here). Also we note that taking fixed points does not commute with general colimits of spaces.

\begin{proposition}
For a transformation $X \to X'$ between diagrams of orthogonal $G$-spectra $X, X': I \to G\Sp^O$ such that every map $X(i) \to X'(i)$ is an equivariant equivalence (i.e.~a stable equivalence on all geometric fixed points, see Definition \ref{def:genuineequiv}) the induced map
$${\hocolim_I} X  \to {\hocolim_I} X'$$
is also an equivariant equivalence of $G$-spectra.
Moreover the homotopy colimit models the $\infty$-categorical colimit of orthogonal $G$-spectra.
\end{proposition}
\begin{proof}
By definition we have to check that all maps 
$$
\Phi^H({\hocolim_I} X_i) \to \Phi^H({\hocolim_I} X'_i) 
$$
for $H \subseteq G$ are stable equivalences. Since geometric fixed points is by definition given by taking fixed points after an index shift it follows from Lemma \ref{fixedcommutes} that it commutes with homotopy colimits, so that this above map is isomorphic to the map
$$
\hocolim_{i \in I}\Phi^H(X_i) \to \hocolim_{i \in I} \Phi^H(X'_i)\ . 
$$
This map is a stable equivalence by Proposition \ref{orthogonalBK} since by assumption all the maps $\Phi^H(X_i)  \to \Phi^H(X'_i)$ are stable equivalences. 
\end{proof}

Note that the last statement can also be proven more directly from the gluing lemma for $G$-orthogonal spectra as proven in \cite[Theorem 3.5(iii)]{MandellMay} but we need Lemma \ref{fixedcommutes} independently.

\bibliographystyle{amsalphaabrvd}
\bibliography{CyclotomicSpectra}

\end{document}